\documentclass[a4paper,12pt]{amsbook}

\usepackage{amssymb}
\usepackage{epsfig}
\usepackage[utf8]{inputenc} 
\usepackage{stix}
\usepackage[all]{xy}
\usepackage{mathdots}

\newtheorem{theorem}{Theorem}[chapter]
\newtheorem{lemma}[theorem]{Lemma}
\newtheorem{corollary}[theorem]{Corollary}

\newtheorem{proposition}[theorem]{Proposition}
\newtheorem{deftheorem}[theorem]{Definition and Theorem}
\newtheorem{deflemma}[theorem]{Definition and Lemma}
\newtheorem{defcorollary}[theorem]{Definition and Corollary}

\theoremstyle{definition}
\newtheorem{definition}[theorem]{Definition}
\newtheorem{example}[theorem]{Example}

\theoremstyle{remark}
\newtheorem{remark}[theorem]{Remark}

\numberwithin{section}{chapter}
\numberwithin{equation}{chapter}

\makeindex

\numberwithin{equation}{section}

\newcommand{\Z}{{\mathbb Z}}

\newcommand{\disc}{{\rm disc}}

\newcommand{\F}{{\mathbb F}}
\newcommand{\N}{{\mathbb N}}
\newcommand{\R}{{\mathbb R}}
\newcommand{\Q}{{\mathbb Q}}
\newcommand{\C}{{\mathbb C}}
\newcommand{\A}{{\mathbb A}}
\newcommand{\vol}{{\mathfrak v}}

\newcommand{\gen}{{\rm gen}}
\newcommand{\spn}{{\rm spn}}

\newcommand{\nullvek}{\mathbf 0}
\newcommand{\x}{{\bf x}}
\newcommand{\f}{{\bf f}}
\newcommand{\e}{{\bf e}}
\newcommand{\z}{{\bf z}}
\newcommand{\y}{{\bf y}}

\newcommand{\rk}{{\rm rk}}
\newcommand{\barx}{\bar{x}}
\newcommand{\barb}{\bar{b}}
\newcommand{\barQ}{\bar{Q}}
\newcommand{\bary}{\bar{y}}
\newcommand{\barw}{\bar{w}}
\newcommand{\barX}{\bar{X}}

\newcommand{\bart}{\bar{t}}
\newcommand{\barz}{\bar{z}}
\newcommand{\wtbp}{\widetilde{b'}}
\newcommand{\tr}{{\rm tr}}

\newcommand{\cB}{{\mathcal B}}

\newcommand{\cG}{{\mathcal G}}

\newcommand{\mfH}{{\mathfrak H}}
\newcommand{\fo}{{\mathfrak o}}
\newcommand{\mfa}{{\mathfrak a}}
\newcommand{\mfb}{{\mathfrak b}}
\newcommand{\mfc}{{\mathfrak c}}
\newcommand{\mfp}{{\mathfrak p}}

\newcommand{\End}{\operatorname*{End}}

\newcommand{\Hom}{\operatorname*{Hom}} 

\newcommand{\reteil}{\operatorname*{Re}}

\newcommand{\mcF}{{\mathcal F}}
\newcommand{\mcS}{{\mathcal S}}
\newcommand{\mcX}{{\mathcal X}}
\newcommand{\mcD}{{\mathcal D}}
\newcommand{\pr}{{\rm pr}}

\newcommand{\cha}{{\rm char}}
\newcommand{\cl}{{\rm cls}}
\newcommand{\rad}{\operatorname{rad}}
\newcommand{\sym}{\operatorname{sym}}

\newcommand{\id}{\operatorname*{Id}}
\newcommand{\ind}{\operatorname*{ind}}

\newcommand{\spin}{{\rm Spin}} 
\newcommand{\pin}{{\rm Pin}} 
\makeatletter
\providecommand*{\bigcupdot}{%
  \mathop{%
    \vphantom{\bigcup}%
    \mathpalette\@bigcupdot{}%
  }%
}
\newcommand*{\@bigcupdot}[2]{%
  \ooalign{%
    $\m@th#1\bigcup$\cr
    \sbox0{$#1\bigcup$}%
    \dimen@=\ht0 %
    \advance\dimen@ by -\dp0 %
    \sbox0{\scalebox{2}{$\m@th#1\cdot$}}%
    \advance\dimen@ by -\ht0 %
    \dimen@=.5\dimen@
    \hidewidth\raise\dimen@\box0\hidewidth
  }%
}
\makeatother
\begin{document}

\frontmatter

\title{Lecture Notes on Quadratic Forms and their
  Arithmetic}


\author{Rainer Schulze-Pillot}
\address{}
\curraddr{}
\email{schulze-pillot@math.uni-sb.de}
\thanks{}



\date{\today}


\maketitle


\setcounter{page}{5}

\tableofcontents

\chapter*{Introduction}

Although there is no need for yet another book on the arithmetic of
quadratic forms, after my retirement I couldn't resist the temptation
to  make these notes available to people who are interested in the
subject. They were compiled over the years from the  notes for my
courses on this topic and are influenced very much by the lecture
notes  ``Quadratische Formen'' from 1974 of my teacher Martin Kneser
(revised and edited as the book \cite{kneserbook} in 2002 by him and
Rudolf Scharlau). Other major influences come from the books of
Eichler \cite{eichler_qfog}, O'Meara \cite{omeara}, Cassels
\cite{cassels} (in chronological order). I thank Gabriele Nebe and
Rudolf Scharlau for sharing their respective unpublished lecture notes with me.

Perhaps I will extend these notes by a few additional sections or
chapters, e.g. by a section about representations by spinor genera or
a chapter about automorphic forms on orthogonal groups and theta
liftings. If you would like to see such an addition don't hesitate to
tell me, this would increase my motivation to go ahead with it.  

Native speakers of English will undoubtedly find many linguistic
errors. I will be grateful for any suggestions for corrections of such
errors. 

I didn't add a glossary or an index. If you would like to find a
particular word in the text your pdf-reader will be a more reliable
search tool  than an index could be. 

\mainmatter
\chapter{Basics}
\section{Quadratic Spaces and Modules}
Throughout these lecture notes $R$ will denote a commutative ring with $1\ne 0$ and
$S$ an $R$-module (usually contained in some $R$-algebra $S'$). Most
often one is interested in the case that $R=S$ is an integral domain
or even a field $F$; these assumptions make some things simpler and we
will mention such possibilities for simplifications. The $R$-torsion
elements of $S$ (i.e, the $s \in S$ such that one has $rs=0$ for some
$0\ne r \in R$) will also be called zero divisors, we say that $2$ is
not a zero divisor in $S$ if $S$ has no $2$-torsion, i.e., if $2s=0$ implies
$s=0$ for $s \in S$ . We will also mention
simplifications arising if one excludes the case that $2=0$ or
more generally requires $2$ to be a unit in $R$  or at least not a zero
divisor in $S$.

\subsection{Quadratic Spaces}
For the reader's convenience we treat the case that $R=F$ is a field
of characteristic not $2$ separately before discussing the general
situation. The reader annoyed by this redundancy may safely omit this subsection.

\begin{definition}\label{quadratic_space}
 Let $V$ be a vector space over the field  $F$. A map $Q:V\to F$ is called an $F$-valued
 {\em quadratic form} \index{quadratic!form}\index{form!quadratic} on $V$ if one has
 \begin{enumerate}
 \item $Q(ax)=a^2Q(x)$ for all $a\in F, x \in V$.
\item $b(x,y):=Q(x+y)-Q(x)-Q(y)$ defines a symmetric bilinear form on $V$.
 \end{enumerate}

The symmetric bilinear form in b) is called the {\em bilinear form
associated to $Q$}\index{associated bilinear form}\index{bilinear form!associated}. The pair $(V,Q)$
is called a {\em quadratic space}\index{quadratic!space}\index{space!quadratic}
over $F$.
\end{definition}
\begin{remark}
  We have $b(x,x)=2Q(x)$ for all $x\in V$. If $F$ is not of
  characteristic $2$ it is more usual to put  $B:=b/2$ and work with
  the bilinear form $B$ instead
  of $b$.
One  obtains then  $Q(x)=B(x,x)$, which looks more convenient. In
particular, one sees that $Q$ and $B$ determine each other uniquely.
The notation chosen above is more adequate in the general situation,
i.e. for quadratic forms over fields of characteristic $2$ or more
general over rings in which $2$ is not invertible.
In particular, this applies when one is dealing with
  integral quadratic forms and their reductions modulo the prime $2$
  (or a 
  prime ideal dividing $2$ in the number field case). 
\end{remark}
\begin{definition}
Let $V$ be a finite dimensional vector space over the field $F$ with
basis $\cB=\{v_1,\ldots,v_m\}$ and quadratic form $Q$ on $V$ with
associated bilinear forms $b, B=\frac{1}{2}b$.
\begin{enumerate}
\item The (homogeneous) {\em quadratic
    polynomial}\index{quadratic!polynomial}\index{polynomial!quadratic}
  associated to $V,Q,\cB$ is the homogeneous polynomial $P_{Q,\cB}\in
  F[X_1,\ldots,X_m]$ of degree $2$ given by
  \begin{equation*}
   P_{Q,\cB}[X_1,\ldots,X_m]=Q(\sum_{i=1}^m
   X_iv_i)=\sum_{i=1}^mQ(v_i)X_1^2+\sum_{1\le i <j\le
     m}b(v_i,v_j)X_iX_j. 
  \end{equation*}
Such a polynomial is also called a quadratic form \index{quadratic!form in $m$ variables} in $m$ variables.
\item The {\em Gram matrix}\index{Gram matrix} associated to $M,b,\cB$
  is the symmetric matrix $M_{\cB}(b)=(m_{ij}) \in M_m^{\sym}(F)$ with
  $m_{ij}=b(v_i,v_j)$ and analogously for $B$.
\end{enumerate}
\end{definition}

\begin{remark}
\begin{enumerate}
\item If $\cha(F)\ne 2$, one has $$Q(\sum_{i=1}^mx_iv_i)=P_{Q,\cB}(x_1,\ldots,x_m)={}^t\x M_{\cB,B}\x=\frac{1}{2}{}^t\x M_{\cB,b}\x$$ for
$\x=(x_1,\ldots,x_m)\in F^m$. The theory of finite dimensional
quadratic forms and spaces over such a field $F$ (and with careful formulation
also over subrings of such a field) is therefore the same as the
theory of symmetric matrices over $F$ resp.\  over such a subring, a point of view which is taken
in part of the literature on the subject, in particular in the
groundbreaking work of Carl Ludwig Siegel. The language of quadratic
spaces has been introduced in 1937 by Ernst Witt and has become the usual description since then.  
\item If $E$ is an extension field of $F$, the quadratic form $Q$
  along with its associated bilinear form $b$ extends in a unique way to the
  $E$-vector space $V\otimes_F E$ (extension of scalars); with respect
  to a basis of type $\{v_i\otimes 1\}$ the extended
  form has the same Gram matrix and the same associated homogeneous
  polynomial as the original one with respect
  to the $v_i\in V$ and will also be
  denoted by $Q$. 
\end{enumerate}
\end{remark}
\begin{deflemma}
Let $(V,Q)$ be a quadratic space over the field $F$ with $\cha(F)\ne
2$, with associated symmetric bilinear forms $b$ and $B=\frac{1}{2}b$,
let $\cB=\{v_1,\ldots,v_m\},\cB'=\{w_1,\ldots,w_m\}$ be bases of $V$
with $w_j=\sum_{i=1}^mt_{ij}v_i$, let $T=(t_{ij})\in GL_m(F)$ be the
matrix describing this change of basis.
Then
\begin{equation*}
 M_{\cB'}(b)={}^tTM_{\cB}(b)T,\quad M_{\cB'}(B)={}^tTM_{\cB}(B)T.
\end{equation*}

Symmetric matrices $A, A'$ related in this way are called {\em
  congruent} \index{congruent!matrices} over $F$ or {\em $F$-equivalent}.
\index{$F$-equivalent!matrices}
Similarly, the homogeneous quadratic polynomials
$P_{\cB,Q},P_{\cB',Q}$ are called equivalent over $F$.
\index{equivalent!quadratic
  polynomials}\index{$F$-equivalent!polynomials}

The square class $\det( M_{\cB}(b))(F^\times)^2\subseteq F$ in $F$ of the determinant of a
Gram matrix of $(V,Q)$ with respect to some basis $\cB$ of $V$ is
called the {\em determinant}\index{determinant!of quadratic space}
$\det_Q(V)=\det_b(V)$ of the quadratic space 
$(V,Q)$ (or the bilinear space $(V,b)$); one often writes $\det(V)$ if
$Q$ is understood. It is independent of the choice of basis. We write
$\det_B(V)$ for the 
square class of the determinant of a Gram matrix of
$B=\frac{1}{2}b$. Sometimes the determinant is also called the
(unsigned) discriminant \index{discriminant} of the space.
\end{deflemma}

\begin{deflemma}\label{space_radical}
Let $(V,Q)$ be a quadratic space over the field $F$ with $\cha(F)\ne
2$, with associated symmetric bilinear forms $b$ and $B=\frac{1}{2}b$.
The {\em radical} \index{radical} of $(V,Q)$ is
\begin{equation*}
  \rad(V)=\rad_b(V)=\rad_B(V)=V^\perp=\{x \in V\mid b(x,v)=0 \text{ for all } v
  \in V\}.
\end{equation*}
$\rad(V)$ is a subspace of $V$ with $Q(\rad(V))=\{0\}$, and on the quotient space 
  $V_0:=V/\rad(V)$ one can define a quadratic form $Q_0$ by
  $Q_0(x+\rad(V)):=Q(x)$, its associated  symmetric
  bilinear form $b_0$ is non degenerate and satisfies
  $b_0(x+\rad(V),y+\rad(V))=b(x,y)$ for all $x,y \in V$. 

If $U\subseteq V$ is a subspace complementary to $\rad(V)$, i.e., with $V=U\oplus \rad(V)$, the
restriction of $b$ to $U$ is non-degenerate too.
\end{deflemma}
\begin{proof}
  Exercise.
\end{proof}

\subsection{Quadratic Modules}
We return now to the general situation.
\begin{definition}
 Let $M$ be an $R$-module. A map $Q:M\to S$ is called an $S$-valued
 {\em quadratic form} \index{quadratic!form}\index{form!quadratic} on $M$ if one has
 \begin{enumerate}
 \item $Q(ax)=a^2Q(x)$ for all $a\in R, x \in M$.
\item $b(x,y):=Q(x+y)-Q(x)-Q(y)$ defines a symmetric bilinear map
  $b:M\times M \to S$, which we will also call a bilinear form on $M$
  with values in $S$.
 \end{enumerate}
If $R=S$ the phrase ``$R$-valued'' is omitted.

The symmetric bilinear form in b) is called the {\em bilinear form
associated to $Q$}\index{associated bilinear form}\index{bilinear form!associated}. The pair $(M,Q)$
is called a {\em quadratic module}\index{quadratic!module}\index{module!quadratic} over
$R$ (with values in $S$), a {\em quadratic space}\index{quadratic!space}\index{space!quadratic} in the case that $R$ is a field.
\end{definition}
\begin{remark}
  \begin{enumerate}
  \item If  $S$ has no $2$-torsion and $b(M,M)\subseteq 2S$
    holds one can write $b(x,y):=2B(x,y)$ with a unique symmetric
    bilinear form $B$; one has then $Q(x)=B(x,x)$ for all $x \in
    M$. This notation is in particular usual if $S$ is a field of
    characteristic not $2$, it is also often used for $R=\Z$ with
    $S=\Z$ or $S=\frac{1}{2}\Z$ or similarly for rings of integers in
    a number field.
\item Conversely, for any symmetric bilinear form $B$ on $M$ (with
  values in $S$) we can define a quadratic form on $M$ with associated
  symmetric bilinear form $2B$ by putting
  $Q(x):=B(x,x)$ for all $x \in M$, one has then
  $Q(x+y)=Q(x)+Q(y)+2B(x,y)$.
  This convention is used in many
  books on quadratic forms which focus on situations in which one has $R=S$
  and $2$ is  not a zero divisor, the form $b$ does then usually not
  appear at all. If $2$ is not a unit the quadratic form is then
  called even or even integral if it takes values in $2R$ and
  $B(M,M)\subseteq R$ holds. Obviously, this convention is inadequate
  if $2$ is allowed to be a zero divisor, in particular, if one wants
  to treat fields without excluding the case of characteristic $2$.
  \end{enumerate}
\end{remark}
\begin{example}\label{example1}
  \begin{enumerate}
  \item For $M=R^n$ and $R=S$ we define for $\x=(x_1,\ldots,x_n)$ by
    $Q(\x)=\sum_{i=1}^n x_i^2$ a quadratic form with
    associated bilinear form $b(\x,\y)=2\sum_{i=1}^nx_iy_i$ and
    $B(x,y)=\langle x,y\rangle=\sum_{i=1}^nx_iy_i$ equal to the
    standard bilinear form on $R^n$.
\item If $2$ is not a zero divisor in $R$ it is invertible in some
  extension ring of $R$. We can then set
  $S=\frac{1}{2}R$ and define $Q:R^n\to S$ by
  $Q(\x)=\frac{1}{2}\sum_{i=1}^n x_i^2$ to obtain a quadratic form
  with associated symmetric bilinear form given by
  $b(\x,\y)=\sum_{i=1}^nx_iy_i$.
\item For $R=S=\F_2, M=R^2$ and $Q(\x)=x_1x_2$ we have
  $b(\x,\y)=x_1y_2+x_2y_1$. The form $B$ can then not be defined. In
  fact, there exists no symmetric bilinear form $B'$ on $M$ with
  $B'(\x,\x)=Q(\x)$ for all $\x\in M$: for a symmetric bilinear form
  $B'$ with $B'((1,0),(1,0))=a,
  B'((0,1),(0,1))=c$ we obtain $B'(\x,\x)=ax_1^2+cx_2^2$, so $a=c=0$
  implies that $B'(\x,\x)$ is identically zero. More generally, if $2$ is a zero
  divisor we can not write all quadratic forms as $Q(x)=B(x,x)$ with a
  symmetric bilinear form $B$, and not all symmetric bilinear forms
  are attached to some quadratic form.
  \end{enumerate}
\end{example}
A generalization of the last example merits a separate definition, as
it gives a fundamental building block of the theory:
\begin{definition}\label{hyperbolic_module}
  Let $M$ be an $R$-module with dual module $M^*=\Hom(M,R)$, put
  $H(M)=M\oplus M^*$ (external direct sum). On $H(M)$ we define a
  quadratic form $Q$ by $Q(v+v^*)=v^*(v)$ for $v\in M, v^*\in M^*$,
  with associated symmetric bilinear form
  $b(v+v^*,w+w^*)=v^*(w)+w^*(v)$.

The quadratic $R$-module $(H(M),Q)$ is called the hyperbolic module
over $M$.

More generally, for an $R$-module $S$ we can define an $S$-valued
hyperbolic quadratic module $(H_S(M),Q_S)$ by setting $H(M)_S=M\oplus
\Hom(M,S)$ and $Q(v,\varphi)=\varphi(v)\in S$ for all $v\in M, \varphi
\in \Hom(M,S)$.  This construction is mainly of use if $S$ is a
quotient module of $R$.
\end{definition}
\begin{example}
  If $M$ above is a free module with basis $(v_1,\ldots,v_n)$ and
  $(v_1^*,\ldots,v_n^*)$ is the dual basis of $M^*$, we have
  $Q(\sum_ix_iv_i+\sum_iy_iv_i^*)=\sum_ix_iy_i$, which in the case
  $R=\R$ and $n=1$ explains the use of the word ``hyperbolic''.
\end{example}
\begin{lemma}\label{bilinearforms}
Let $Q$ be a quadratic form on the $R$-module $M$.
\begin{enumerate}
\item $Q$ is uniquely determined by its associated symmetric bilinear
  form $b$ if $2$ is not a zero divisor in $S$.
\item Let $\beta$ be any (not necessarily symmetric) $S$-valued bilinear form on
  $M$ and define $Q:M\to S$ by $Q(x):=\beta(x,x)$. Then $Q$ is a
  quadratic form whose associated symmetric bilinear form is the
  symmetrization $b(x,y)=\beta(x,y)+\beta(y,x)$ of $\beta$.
\item If $\beta_1,\beta_2$ are bilinear forms on $M$ inducing the same
  quadratic form $Q$ on $M$ as above, the form $\beta_1-\beta_2$ is
  alternating. 
\item If $M$ is finitely generated projective, every quadratic form
  $Q$ on $M$ can be written as $Q(x)=\beta(x,x)$ for some not
  necessarily symmetric bilinear form $\beta$.
\end{enumerate}
\end{lemma}
\begin{proof}
  Exercise. For d), the case of a finitely generated free module is
  easy: If $(v_1,\ldots,v_n)$ is a basis, one puts
  $\beta(v_i,v_i):=Q(v_i)$, $\beta(v_i,v_j)=b(v_i,v_j)$ for $i<j$,
  $\beta(v_i,v_j)=0$ for $i>j$ and extends bilinearly. For a
  projective module one supplements it first by a suitable 
  direct summand to obtain a free module, extend $Q$ to that free module and restrict the $\beta$
  obtained there to the original module.
\end{proof}
\begin{example}
If we consider the module $H(M)=M\oplus M^*$ as in definition
\ref{hyperbolic_module} and equip it with the not symmetric bilinear
form $\beta_1$ given by $\beta(v+v^*,w+w^*)=v^*(w)$, we obtain as $Q$
and $b$ the same forms as in that definition. The same result is
obtained by using $\beta_2(v+v^*,w+w^*)=w^*(v)$. The difference
$A:=\beta_1-\beta_2$ is the standard alternating form on the module $H(M)$.  
\end{example}
\begin{definition}
Let $(M,Q)$ be a quadratic module over $R$ with values in the
commutative $R$-algebra $S$ and associated symmetric bilinear form
$b$; 
assume that $M$ is a finitely generated $R$-Module with a generating
set ${\mathcal G}=\{v_1,\ldots,v_m\}$.
\begin{enumerate}
\item The (homogeneous) {\em quadratic
    polynomial}\index{quadratic!polynomial}\index{polynomial!quadratic}
  associated to $M,Q,\cG$ is the homogeneous polynomial $P_{Q,\cG}\in
  S[X_1,\ldots,X_m]$ of degree $2$ given by
  \begin{equation*}
   P_{Q,\cG}[X_1,\ldots,X_m]=Q(\sum_{i=1}^m
   X_iv_i)=\sum_{i=1}^mQ(v_i)X_1^2+\sum_{1\le i <j\le
     m}b(v_i,v_j)X_iX_j. 
  \end{equation*}
\item The {\em Gram matrix}\index{Gram matrix} associated to $M,b,\cG$
  is the symmetric matrix $M_{\cG}(b)=(m_{ij}) \in M_m^{\sym}(S)$ with $m_{ij}=b(v_i,v_j)$.
\end{enumerate}
\end{definition}
\begin{remark} 
\begin{enumerate}
\item Whenever  $B=\frac{1}{2}b$ is
defined,  one has
$P_{Q,\cG}(x_1,\ldots,x_m)={}^t\x M_{\cG,B}\x=\frac{1}{2}{}^t\x M_{\cG,b}\x$ for
$\x=(x_1,\ldots,x_m)\in R^m$.  Otherwise, there is no such direct
connection between the polynomial $P_{Q,\cG}$ and the Gram matrix.
\item Let $P\in S[X_1,\ldots,X_m]$ be a homogeneous quadratic
  polynomial and 
$M$ a  free $R$-module of rank $m$  with basis $\cB$. Then there is a
unique quadratic form $Q$ on $M$ such that $P=P_{\cB,Q}$.
\item If $R'$ is any $R$-algebra, the quadratic form $Q$ extends in a natural way to the
  $R'$-module  $M\otimes_R R'$ (extension of scalars) with values in
  $R'\otimes_RS$; 
this extension will also be
  denoted by $Q$.
With respect
  to a generating set of type  $\{v_i\otimes 1\}$ the Gram matrix of
  the extended
  form is obtained by tensoring all entries of the original  Gram
  matrix with $1$, and the  associated homogeneous
  polynomial is obtained in the same way from the original one with respect
  to the $v_i\in M$.  
\end{enumerate}
\end{remark}

\begin{lemma}\label{modules_matrixversion}
Let $(M,Q)$ be a quadratic module over the ring $R$ with
generating sets $\cG=(v_1,\ldots,v_m)$ and  $\cG'=(w_1,\ldots w_n)$
satisfying $w_j=\sum_{i=1}^mt_{ij}v_i$ for $1\le j\le n$,
put   $T=(t_{ij})\in M_{m,n}(R)$.
\begin{enumerate}
\item One has  
$$M_{\cG'}(b)= {}^t\!TM_{\cG}(b)T. $$
\item With $x_i=\sum_{j=1}^nt_{ij}y_j$ one has 
$$P_{Q,\cG}(x_1,\ldots,x_m)=P_{Q,\cG'}(y_1,\ldots,y_n).$$ 
\item If $M$ is a free module over $R$ with a finite basis
  $\cB=(v_1,\ldots,v_m)$ and $S$ is contained in an $R$-algebra $S'$,
  the square class 
  $\det(M_{\cB}(b))(R^\times)^2$ in $S'$ is independent of the choice of
    basis $\cB$ and is called the determinant (sometimes also the
    (unsigned) discriminant) of the quadratic module. 
\end{enumerate}
\end{lemma}
\begin{proof}
For c), since the rank of a finitely generated free module is well
defined over any commutative ring, the transformation matrix between
any two bases is an invertible square matrix and has invertible
determinant. a) and b) 
are easy exercises: for a) we put $A=(a_{ij})=M_{\cG'}(b)$ and have 
$b(w_k,w_l)=\sum_{i,j}t_{ik}a_{ij}t_{jl}$, for b) we have
$Q(\sum_ix_iv_i)=Q(\sum_{i,j}t_{ij}y_jv_i)=Q(\sum_jy_jw_j)$.    
\end{proof}
\begin{definition}
\begin{enumerate}
  \item Let $(M,Q), (M',Q')$ be quadratic $R$-modules.
An injective $R$- linear map $\phi:M\to M'$ is called {\em
  isometric}\index{isometric} or an {\em isometry}\index{isometry} or
a {\em representation of $(M, Q)$ by $(M',Q')$} if $Q'(\phi(x))=Q(x)$ holds for all
$x\in M$. If such an isometric linear map exists one says that $(M,Q)$
is {\em represented} \index{represented!for modules} by $(M',Q')$.
The quadratic modules $(M,Q)$ and $(M',Q')$ are called
{\em isometric}\index{isometric} if there  exists an isometric  linear isomorphism $\phi:M\to
M'$.
\item An isometric linear automorphism of 
 $(M,Q)$ is called an {\em orthogonal map}\index{orthogonal map} of the
 quadratic module. The set of all orthogonal maps of $(M,Q)$ is the
 {\em orthogonal group}\index{orthogonal group}
 $O(M,Q)=O_Q(M)=O_{(M,Q)}(R)=O_M(R)$ of the quadratic module. (We will
 use these variants of notation in the sequel. The last version is the
 one from algebraic group theory, it will
 be preferred when we deal with extensions of the ground ring or field.) 
\item Let $P_1\in S[X_1,\ldots,X_m], P_2\in S[X_1,\ldots,X_n]$ be
  homogenous quadratic polynomials. One says that $P_2$ is represented
  \index{represented over $R$!for quadratic polynomials}
  by $P_1$ over $R$, if there exists a matrix $T=(t_{ij})\in M_{m,n}(R)$ such
  that one has $P_1(y_1,\ldots,y_m)=P_2(x_1,\ldots,x_n)$ for all
  $\x=(x_1,\ldots,x_n)\in R^n$, where we put
  $y_i=\sum_{j=1}^mt_{ij}x_j$ for $1\le i\le n$. If each of $P_1,P_2$
  is represented over $R$ by the other one we say that they are {\em
    $R$-equivalent} or equivalent over $R$\index{equivalent!over $R$}.
\item A symmetric $n\times n$ matrix $B$ over $S$ is 
said to be {\em represented over $R$} \index{represented over $R$! for
  matrices} by the $m\times m$-matrix $A$, if one
has $B={}^t TAT$ for some $T\in M_{m,n}(R)$. The matrices $A,B$ above
are  
called {\em
    $R$-congruent}\index{$R$-congruent} or {\em $R$-equivalent} if each of
  them is represented over $R$ by the other one.
\end{enumerate}
\end{definition}
\begin{remark}
  \begin{enumerate}
\item The condition of injectivity is sometimes omitted in the
  definition of isometry or of representation.
  \item An isometry is also compatible with the bilinear forms
    $b,b'$ associated to $Q,Q'$. If $2$ is not a zero divisor in $S$
    the reverse direction is also true. 
\item  It is easily seen that $R$ equivalent matrices must have the same
 size and that $R$-equivalent polynomials must have the same number of variables. 
  \end{enumerate}
\end{remark}
\begin{proposition}
  \begin{enumerate}
  \item Let $(M,Q),(M',Q')$ be  quadratic modules with finite
    generating systems $\cG, \cG'$ and with associated symmetric
    bilinear forms $b,b'$ such that $(M',Q')$ is represented
    by $(M,Q)$. 

Then the polynomial $P_{\cG',Q'}$
    is represented by  $P_{\cG,Q}$ over $R$, and the Gram
    matrix $M_{\cG'}(b')$ is represented by $M_{\cG}(b)$.
\item Let $P_1\in S[X_1,\ldots,X_m], P_2\in S[X_1,\ldots,X_n]$ be
  homogenous quadratic polynomials such that $P_2$ is represented by $P_1$ over $R$.

Let $M,M'$ be free $R$-modules $M$ of rank $m$ and $M'$ of rank $n$
with quadratic forms $Q,Q'$ and bases $\cB,\cB'$ such that
$P_1=P_{\cB,Q}, P_2=P_{\cB',Q'}$. 

Then $(M',Q')$ is represented by $(M,Q)$.
  \end{enumerate}
\end{proposition}



We recall from linear algebra that a  bilinear form $\beta$ (with
values in $R$) on the
$R$-module $M$ induces an $R$-linear map $\tilde{\beta} :v\mapsto
\tilde{\beta}(v)=\tilde{\beta}_v \in M^*=\Hom(M,R)$ from $M$ 
to $M^*$ by setting $\tilde{\beta}_v(x):=\beta(v,x)$ for all $x \in M$. If the form
$\beta$ is symmetric it 
is {\em non degenerate} \index{non degenerate} if and only if this map is
injective, it is {\em regular} \index{regular!bilinear form}
if and only if it is an isomorphism of $R$-modules. These two notions
coincide if $R$ is a field and $M$ a finite dimensional vector space
over $R$. 
For example, for $R=\Q$ in a) of Example \ref{example1} the bilinear form
$b$ is regular, for $R=\Z$  it is non degenerate
but not regular, whereas the forms $b$ in b) and in c) of that example are
regular.
If the module is free of finite rank, the symmetric bilinear form $b$ is regular if
and only if the determinant of the Gram matrix with respect to a basis $\cB$
is invertible, and the matrix of the linear map $\tilde{\beta}:M\to M^*$ with
respect to the basis $\cB$ of $M$ and the dual basis $\cB^*$ of $M^*$
is the Gram matrix of $b$ with respect to the basis $\cB$. 
The 
bilinear form is non-degenerate if and only if the determinant of the
Gram matrix with respect to a basis is not a zero divisor (to prove
this latter fact, use that a linear map $R^n\to R^n$ is injective if
and only if its determinant is not a zero divisor; a fact of (multi-)linear
algebra not found in many places outside Bourbaki's Alg\`ebre).

Conversely, if $\lambda:M\to M^*$ is an $R$-linear map it induces a
bilinear form $\beta$ on $M$ by setting $\beta(v,x)=\lambda(v)(x)$;
this bilinear form is symmetric if and only if $\lambda$ is equal to
the pullback of its  transpose
to $M$ under the canonical mapping $\iota:M\to M^{**}$.

In the more general situation that $Q$ and $b$ take values in the
$R$-module $S$, we obtain similarly a map $\tilde{b}:M\to \Hom(M,S)$ and
call $b$ non degenerate if this map is injective. The notion of
regularity will only be used if  one has $R=S$, it doesn't make much
sense otherwise.

If $N$ is a submodule of the quadratic module $M$ we write as usual $N^\perp=\{m\in M
\mid b(m,N)=0\}$ (without mentioning $b$ in the notation) for the
orthogonal complement \index{orthogonal complement}
\index{complement!orthogonal} of $N$ in $M$ with respect to $b$. 
\begin{example}
  The hyperbolic module $H(M)$ over some $R$-module $M$ is non
  degenerate if and only if for every $\nullvek \ne x\in M$ there
  exists $f \in M^*$ with $f(x)\ne 0$; such a module is also called a
  torsionless module (which is not the same as a torsion free module).
If $M$ is finitely generated projective, it is also reflexive and one
sees that $H(M)$ is regular.

A regular quadratic module which is free is necessarily of finite
rank, since otherwise  a basis of the dual module has higher
cardinality than a basis of the module itself.

If $R=F$ is a field and not of characteristic $2$, every one
dimensional quadratic space $Fx$ with $Q(x)\ne 0$ is regular. 
\end{example}
\begin{definition}\label{defi_radical}
Let $M$ be an $R$-module with symmetric bilinear form $b$.
The {\em radical} \index{radical} of $(M,b)$ is
\begin{equation*}
  \rad(M)=\rad_b(M)=M^\perp=\{x \in M\mid b(x,m)=0 \text{ for all } m
  \in M\}.
\end{equation*}
If $(M,Q)$ is a quadratic module with associated symmetric bilinear
form we call the radical of $(M,b)$ the {\em bilinear
  radical}\index{radical!bilinear} and define the radical of the
quadratic module to be 
\begin{equation*}
  \rad_Q(M):=\{x \in \rad_b(M)\mid Q(x)=0\}.
\end{equation*}
\end{definition}
\begin{remark}
  If $(M,Q)$ is a quadratic module and $2$ is a not zero divisor one
  has $\rad_Q(M)=\rad_b(M)$.  However, if $2$ is a zero divisor, the
  restriction of the quadratic form $Q$ to $\rad_b(M)$ is not
  necessarily zero and $\rad_Q(M)$ may become a proper subset of $\rad_b(M)$. For example, for $M=R=S=\F_2$ and $Q(x)=x^2$ we have
  $\rad_b(M)=M$ and $\rad_Q(M)=\{0\}$. As another example, put
  $R=M=\Z$ and $S=\Z/6\Z$, with $Q(x)=x^2+6\Z$. Then $\rad_b(M)=3\Z$
  and $\rad_Q(M)=6\Z$. 
\end{remark}
\begin{lemma} Let $(M,Q)$ be a quadratic module with associated
  symmetric bilinear form $b$.
  \begin{enumerate}
  \item 
$\rad_b(M)$ is a submodule of $M$, and on the quotient module 
  $M_0:=M/\rad_b(M)$ one can define a non degenerate symmetric
  bilinear form $b_0$ by $b_0(x+\rad_b(M),y+\rad_b(M)):=b(x,y)$. 
\item The restriction of $Q$ to $rad_b(M)$ is $\Z$-linear
  (equivalently: additive)
  and  $\rad_Q(M)$ is a submodule of $\rad_b(M)$ which equals
  $\rad_b(M)$ if $2$ is not a zero divisor in $S$. On the quotient
  module $M_1=M/\rad_Q(M)$ one can define a quadratic form $Q_1$ by
  $Q_1(x+\rad_Q(M)):=Q(x)$. 
\item If $N_0\subseteq M$ is a submodule with $M=N_0\oplus \rad_b(M)$,
the natural map $n_0\mapsto n_0+\rad_b(M)$ from $N_0$ to $M/\rad_b(M)$
is a linear isomorphism compatible with the bilinear forms
$b\vert_{N_0}$ and $b_0$. In particular, the
restriction of $b$ to $N_0$ is non-degenerate too.
\item If $N_1\subseteq M$ is a submodule with $M=N_1\oplus \rad_Q(M)$,
the natural map $n_1\mapsto n_1+\rad_Q(M)$ from $N_1$ to $M/\rad_Q(M)$
is an isometry.
  \end{enumerate}
\end{lemma}
\begin{proof}
  Exercise.
\end{proof}
\begin{remark}
 If $R$ is a field, complementary subspaces $N_0,N_1$ in the above sense to $\rad_b(M),\rad_Q(M)$
    always exist, for an arbitrary ring $R$ this is not necessarily
    the case.  
\end{remark}

\section{Orthogonal splittings and orthogonal groups}
\begin{definition} Let $M$ be an $R$-module with ($S$-valued) symmetric bilinear form $b$.
  \begin{enumerate}
  \item If $M=N_1\oplus N_2 $ is a direct sum decomposition with
    $b(N_1,N_2)=\{0\}$, one writes $M=N_1\perp_b N_2=N_1\perp
    N_2$
    and says that $M$ is the {\em orthogonal sum}\index{orthogonal
      sum}\index{sum!orthogonal}
 (with respect to $b$)  of  $N_1$ and $N_2$.
\item For a submodule $N$ of $M$ the {\em orthogonal
    complement}\index{orthogonal
      complement}\index{complement!orthogonal} (with respect to $b$) is $N^{\perp_b}=N^\perp:=\{x \in
    M\mid b(x,N)=\{0\}\}$.
\item If $N\subseteq M$ is a submodule with $M=N\perp N^\perp$ one says
  that $N$ {\em splits off} \index{split off} in $M$ or splits $M$
  orthogonally.
\item Analogously one writes $$M=N_1\perp \dots \perp
  N_r=\perp_{i=1}^r N_i$$ if one has $M=\bigoplus_{i=1}^r N_i$ with
  $b(N_i,N_j)=\{0\}$for $i \ne j$.
\item A basis $(v_1,\ldots,v_m)$ of $M$ with $b(v_i,v_j)=0$ for $i \ne
  j$ is called on orthogonal basis of $V$ (with respect to $b$). Its
  Gram matrix with respect to this basis is then a diagonal matrix
  with entries $2Q(v_i)$, and we write $(M,Q)\cong
  [Q(v_1),\ldots,Q(v_m)]$ as a short notation for the isometry class
  of the quadratic module.
  \end{enumerate}
\end{definition}
\begin{remark}
  If $(M_1,b_1),(M_2,b_2)$ are modules with ($S$-valued) symmetric
  bilinear forms one can in an obvious way form the external
  orthogonal sum of $M_1,M_2$. We will usually not distinguish between
  external and internal orthogonal sums.
\end{remark}
\begin{lemma}
  Let $(M,Q)$ be a quadratic module and assume
  that $M=N_1\perp N_2$ for submodules  $N_1,N_2$ of $M$.
Then  $O(N_1,Q\vert_{N_1})\times O(N_2,Q\vert_{N_2})$ can be naturally
embedded into $O(M,Q)$ by mapping $(\varphi_1,\varphi_2)$
to the map $\psi=\varphi_1\perp\varphi_2$ given by
$\psi(v_1+v_2=\varphi_1(v_1)+\varphi_2(v_2)$ for all $v-1\in N_1, v_2
\in N_2$. 
In particular, $O(N_1,Q\vert_{N_1})$ is embedded into $O(M,Q)$ by
mapping $\varphi \in O(N_1,Q\vert_{N_1})$ to $\varphi \perp
\id_{N_2}$. The image of this embedding is contained in
$O(M,N_2,Q):=\{ \psi \in O(M,Q)\mid \psi \vert_{N_2}=\id_{N_2}\}$ and
equal to it if $(N_2,Q\vert_{N_2})$ is non degenerate; one writes  
just $O(M,N_2)$ if $Q$ is understood.
\end{lemma}
\begin{proof}
Only the characterization of the image of the embedding of
$O(N_1,Q\vert_{N_1})$ into $O(M,Q)$ is not obvious. For this assume
$(N_2,Q)$ to be non degenerate, let
$\psi \in O(M,N_2,Q)$ and write $\psi(v)=w_1+w_2$ for $v \in
N_1$. Then $0=b(v,w)=b(\psi(v), \psi(w))=b(w_1+w_2,w)=b(w_2,w)$ for all
$w \in N_2$, hence $w_2 \in N_2^\perp\cap N_2=\{\nullvek\}$ by
assumption, so we have $\psi(N_1)\subseteq N_1$ and
$\psi=\varphi\perp\id_{N_2}$ with $\varphi=\psi\vert_{N_1}\in
O(N_1,Q)$. The other inclusion is obvious.  
\end{proof}
\begin{lemma}
  Let $(M,b)$ be an $R$-module with symmetric bilinear form and assume
  that $M=N_1\perp N_2$ for submodules  $N_1,N_2$ of $M$.

Then $(M,b)$ is non degenerate resp.\ regular if and only if both
$(N_1, b\vert_{N_1})$, $(N_2,b\vert_{N_2})$ have the respective property.
\end{lemma}
\begin{proof}
  Exercise.
\end{proof}
\begin{theorem}\label{orthogonalcomplements}
Let $(M,b)$ be a module with $S$-valued symmetric bilinear form $b$,
let $N\subseteq M$ be a submodule. For $v \in M$ let
$\tilde{b}^{(N)}(v)=\tilde{b}_v^{(N)}\in \Hom(M,S)$ be the linear map
given by $\tilde{b}^{(N)}(v)(x)=b(v,x)$ for $x \in N$.
\begin{enumerate}
\item $M/N^\perp \cong \tilde{b}^{(N)}(M)\subseteq N^*$.
\item If $(N,b)$ is non degenerate with
  $\tilde{b}^{(N)}(M)=\tilde{b}^{(N)}(N)$ one has $M=N\perp N^\perp$.
\item If $R=S$ and $(N,b)$ is regular one has $M=N\perp N^\perp$.
\item If $R=S=F$ is a field and $(M,b)=(V,b)$ is finite dimensional and
  non degenerate one has $\dim(N)+\dim(N^\perp)=\dim(M)$ and $(N^\perp)^\perp=N$.
\end{enumerate}
\end{theorem}
\begin{proof}
  \begin{enumerate}
\item This follows from the homomorphy theorem of linear algebra.
  \item 
 By assumption, for all $v\in M$ the the linear functional
  $\tilde{b}^{(N)}(v)$ on $N$
  can also be obtained as $\tilde{b}^{(N)}(w_1)$ for some $w_1\in N$,
  i.e., there is $w_1\in N$ with $b(w_1,x)=b(v,x)$ for
  all $x\in N$, and one has $w_2:=v-w_1\in N^\perp$. Since we also
  have $N\cap N^\perp=\{0\}$ by assumption, $M=N\perp
  N^\perp$ holds as asserted.
\item This is a special case of b).
\item In this case non degeneracy implies regularity, and the linear
  map $\tilde{b}^{(N)}:M\to N^*$ 
  surjective with kernel $N^\perp$, which shows the asserted equation
  for the dimensions. Since $(N^\perp)^\perp\supseteq N$ is obvious,
  this implies  $(N^\perp)^\perp=N$.
  \end{enumerate}
\end{proof}
\begin{example}
Let $M=R^2$ with the quadratic form $Q(x_1,x_2)=x_1x_2$, this
quadratic module is called the {\em hyperbolic plane}\index{hyperbolic
  plane}\index{plane!hyperbolic} (over $R$), it is obviously regular. If $2$ is
invertible in $R$, the submodule $N=R\bigl(
\begin{smallmatrix}
  1\\1
\end{smallmatrix}\bigr)$ is regular and can be split off with
orthogonal complement $N^\perp=R\bigl(
\begin{smallmatrix}
  1\\-1
\end{smallmatrix}\bigr)$. The submodule $N'=R\bigl(
\begin{smallmatrix}
  1\\0
\end{smallmatrix}\bigr)$ is not regular, in fact, it is its own
orthogonal complement, so one sees that the dimension formula and
$((N')^\perp)^\perp=N'$ are indeed satisfied.
\end{example}
\begin{deflemma} Let $(M,Q)$ be a quadratic module with values in $S$.
  Let $z \in M$ with $b(z,M)\subseteq Q(z) \cdot R$ and such that
  $Q(z)$ is not a zero divisor in $S$.
The {\em symmetry}\index{symmetry} or {\em
  reflection}\index{reflection} with respect to $z$ is the linear map
$M\ to M$ given by 
\begin{equation*}
  \tau_z(x)=x-\frac{b(x,z)}{Q(z)} z \text{ for all } x \in M.
\end{equation*}
One has
\begin{enumerate}
\item $\tau_z^2=\id_M$.
\item $\tau_z \in O(M,Q)$.
\item $\tau_z(z)=-z$ and $\tau_z\vert_{(Rz)^\perp}=\id_{(Rz)^\perp}$.
\item If $x \in M$ satisfies $Q(x)=Q(z)$ and $Q(x-z)\ne 0$, one has $\tau_{x-z}(x)=z,\tau_{x-z}(z)=x$.
\end{enumerate}
\end{deflemma}
\begin{proof}
We have
\begin{eqnarray*}
  \tau_z(\tau_z(x))&=&\tau_z(x)-\frac{b(\tau_z(x),z)}{Q(z)}z\\
&=&x-\frac{b(x,z)}{Q(z)}z-\frac{b(x,z)-\frac{b(x,z)b(z,z)}{Q(z)}}{Q(z)}z\\
&=&
x-\frac{b(x,z)}{Q(z)}z-(-\frac{b(x,z)}{Q(z)}z)\\
&=&x.
\end{eqnarray*}
The rest is similarly shown by direct computation.
\end{proof}
\begin{theorem}[Witt's generation Theorem]\label{SatzvonWitt}\index{Witt's
    generation Theorem}
 Let $R=S=F$ be a field, $\cha(F)\ne 2$, let $(V,Q)$ be a non
 degenerate finite dimensional quadratic space over $F$.

Then the orthogonal group $O(V)=O(V,Q)$ is generated by symmetries. 

More precisely: Each $\sigma \in O(V)$ can be written as a product of
at most $\dim(V)$ symmetries.  
\end{theorem}
\begin{proof}
  We prove only the first part, using induction on $n=\dim(V)$, for the assertion about the number of
  symmetries see e.g.\ 
{Dieudonne}.

The case $n=1$ is trivial. Let $\dim(V)=n>1$, assume the assertion to be
proven for all non degenerate $(W, Q')$ with $\dim(W)<n$ and let
$\sigma \in O(V)$. Since $V$ is non degenerate there exists $x\in V$
with $Q(x)\ne 0$.
We have $Q(x-\sigma(x))=2Q(x)-b(x,\sigma(x)),
Q(x+\sigma(x))=2Q(x)+b(x,\sigma(x))$ and hence $Q(x+\sigma(x))\ne 0$
or $Q(x-\sigma(x))\ne 0$.  In the latter case, we put $y=x-\sigma(x)$
and have $\tau_y(\sigma(x))=x$ and therefore
$\tau_y\circ\sigma((Fx)^\perp)=(Fx)^\perp=:U$ with $\dim(U)=n-1$. By
the inductive assumption we can write
$\tau_y\circ\sigma\vert_U=\tau_{z_1}\circ\dots\circ \tau_{z_r}$ wit $r
\in \N$ and vectors $z_i\in U$ with $Q(z_i)\ne 0$. Since the $\tau_{z_i}$
satisfy $\tau_{z_i}(x)=x$, we have
$\tau_y\circ\sigma(x)=(\tau_{z_1}\circ\dots\circ \tau_{z_r})(x)$ as
well and hence $\tau_y\circ\sigma=\tau_{z_1}\circ\dots\circ
\tau_{z_r}$, which gives $\sigma=\tau_y\circ \tau_{z_1}\circ\dots\circ
\tau_{z_r}$.

In the other case we put $y=x+\sigma(x)$ and obtain in the same way as
above $\tau_y\circ\sigma\vert_U=(\tau_{z_1}\circ\dots\circ
\tau_{z_r})\vert_U$ and $\tau_y\circ\sigma(x)=-x$. From this we get $\sigma=\tau_x\circ\tau_y\circ\tau_{z_1}\circ\dots\circ
\tau_{z_r}$.
\end{proof}
\section{Hyperbolic Modules, Primitivity and Regular Embeddings}
\begin{definition}
  Let $M$ be an $R$-module.
  \begin{enumerate}
  \item A submodule $N\subseteq M$ is called a {\em primitive}
    submodule \index{primitive}\index{submodule!primitive} if $N$ is a
    direct summand in $M$, i.e., $M=N\oplus N'$ for some submodule
    $N'$ of $M$.
\item Let $b$ be an $R$-valued symmetric bilinear  form on $M$. Then a
  submodule $N\subseteq M$ is called {\em sharply primitive} or {\em
    $b$-primitive} or {\em
    regularly embedded} \index{regularly embedded}\index{sharply
    primitive} with respect to $b$ if one has
  $\tilde{b}^{(N)}(M)=N^*$, i.e., if for every $\varphi \in N^*$ there
  exists $v \in M$ with $b(v,w)=\varphi(w)$ for all $w \in N$.
  \end{enumerate}
\end{definition}
\begin{remark}\label{regularly_embedded_remark}
Obviously, a monogenic submodule $N=Rw$ is regularly embedded if and only if one
has $b(w,M)=R$.

If $M$ is regular, every primitive submodule is regularly embedded.
In particular, if $R=F=S$ is a field and $(V,Q)$ is a regular quadratic
space over $F$, every subspace is regularly embedded. If the space
$(V,Q)$ is not regular, a subspace $U$ is regularly embedded if and
only if the restriction to $U$ of the  projection $\pi$ to the regular
space $\bar{V}=V/\rad_b(V)$ is injective, i.e., $U\cap
\rad_b(V)=\{\nullvek\}$. 
We can then view it
as a subspace of ``the regular part'' of $V$ in the sense that for
each subspace $W$ of $V$ which is
complementary 
to $\rad_b(V)$ (and hence regular) we obtain a
natural embedding of $U$ into $W$. This explains the
terminology ``regularly embedded''.

To see this, assume first
that there is $\nullvek \ne u\in U\cap\rad(V)$. Then all linear forms
in the image of $\tilde{b}^{(U)}$ are zero on $u$, so
$\tilde{b}^{(U)}(V)=U^*$ can not hold. Conversely, assume $\pi\vert_U$
to be injective. Then its transpose $\bar{V}^*\to U^*$ is surjective,
and since $\bar{V}$ is regular we see that for every $f \in U^*$ there
exists $y\in V$ with $b(y,x)=g(\pi(x))=f(x)$ for all $x \in U$, where
we denote by $g$ a preimage under the transpose of $\pi\vert_U$ of
$f$. The subspace  $U$ is therefore regularly embedded in $V$. 

 Conversely,  if $N\subseteq M$ is finitely generated projective and regularly
  embedded, it is a primitive submodule of $M$. This is trivial in the
  case of vector spaces and  follows for modules from the
  fact that $\tilde{b}^{(N)}(M)=N^*$ implies that every linear functional on
  $N$ can be extended to $M$, i.e., the map $M^*\to N^*$ given by
  restriction to $N$ is surjective. For finitely generated projective
  modules this latter property is well known to be equivalent to $N$
  being a direct summand in $M$. To see this if $N$ is finitely generated free
  with basis $(v_1,\ldots, v_n)$  consider the dual basis
  $(v_1^*,\ldots,v_n^*)$ of $N^*$ and pick $\phi_1,\ldots,\phi_n \in
  M^*$ with $\phi_1\vert_N=v_i^*$.
We can then write $v\in M$ as
$(\sum_{i=1}^n\phi_i(v)v_i)+(v-(\sum_{i=1}^n\phi_i(v)v_i)$ and
denote by $P$ the set of all the $(v-(\sum_{i=1}^n\phi_i(v)v_i$ for
$v \in M$, this give as $M=N\oplus P$.
To prove it for general finitely generated projective $N$, you can e.g. consider
  the (by projectivity split) exact sequence $0\to N^\perp=(M/N)^*\to
  M^*\to N^*\to 0$,
  dualize it and use reflexivity.
\end{remark}
\begin{lemma}
  Let $N_1$ be a finitely generated projective submodule of the bilinear
  $R$-module $(M,b)$ which is regularly embedded.

 Then there exists a finitely generated projective submodule $N_2$ of $M$
 such that $\tilde{b}^{(N_1)}\vert_{N_2}N_2\to N_1^*$ is an
 isomorphism.

Moreover, $N_2$ is regularly embedded into $M$ and
$\tilde{b}^{(N_2)}\vert_{N_1}:N_1 \to N_2^*$ is an isomorphism.
\end{lemma}
\begin{proof}
  Since $N_1$ is projective, the exact sequence $0\to N_1^\perp\to
  M \to N_1^*\to 0$ splits, i.e., there exists a module homomorphism
  $h:N_1^*\to M$ with $b^{(N_1)}\circ h=\id_{N_1^*}$,  and we can take $N_2$ as the image of the
  splitting homomorphism $N_1^*\to M$. If $N_1$ is finitely generated
  free, we can mor directly take $N_2$ to be the linear span of any set
  of preimages of a set of basis vectors of $N_1^*$ under the map  $b^{(N_1)}$.

$N_1$, being finitely generated projective, is a reflexive module
(i.e., $(N_1^*)^*$ is naturally isomorphic to $N_1$), and
we can view the transpose (or dual map) of   $\tilde{b}^{(N_1)}\vert_{N_2}N_2\to
N_1^*$ as an isomorphism $N_1 \to N_2^*$; it is easily checked that
this transpose is nothing but $\tilde{b}^{(N_2)}\vert_{N_1}$.
\end{proof}
\begin{definition}
  Let $(M,Q)$ be a quadratic module.
  \begin{enumerate}
  \item $\nullvek\ne x \in M$ is called {\em isotropic}\index{isotropic} if $Q(x)=0$,
    {\em anisotropic}\index{isotropic} otherwise.
\item $(M,Q)$ is called isotropic if it contains an isotropic vector,
  anisotropic otherwise.
\item A nonzero submodule $N\subseteq M$ with $Q(N)=\{0\}$ is called
  {\em totally isotropic}\index{totally isotropic} or {\em singular}\index{singular}.
  \end{enumerate}
\end{definition}
\begin{example}
  In the hyperbolic $R$-module $H(M)=M\oplus M^*$ over some $R$-module
  $M$ both $M$ and $M^*$ are totally isotropic submodules (if nonzero). 
\end{example}
\begin{theorem}\label{isotropic_hyperbolic}
  Let $(M,Q)$ be a quadratic module, $N\subseteq M$ a finitely
  generated projective submodule which is totally isotropic and
  regularly embedded. Then there is a submodule $N'$ of $M$ such that
  $N\oplus N'$ with the restriction of $Q$ as quadratic form is
  isometric to the hyperbolic module over $N$.

In particular, every totally isotropic subspace of a regular quadratic
space $(V,Q)$ over a field $F$ can be supplemented to a hyperbolic
space in which it is a maximal totally isotropic subspace.  
\end{theorem}
\begin{proof}
  We find first a finitely generated projective submodule $P$ of $M$
  such that $\tilde{b}^{(N)}\vert_P:P\to N^*$ and
  $\tilde{b}^{(P)}\vert_N:N \to P^*$ are isomorphisms. There is a (not
  necessarily symmetric) bilinear form $\beta$ on $P$ satisfying
  $Q(x)=\beta(x,x)$ for all $x \in P$, and using the isomorphism
  $\tilde{b}^{(P)}\vert_N:N \to P^*$  we can define a map $f:P
  \to N$ by requiring $b(y,f(x))=\beta(y,x)$ for all $x,y \in
  P$. Obviously, $f$ is linear and one has $Q(x-f(x))=0$ for all $x
  \in P$. The submodule $N'=\{ x-f(x)\mid x \in P\}$ is then as desired.
\end{proof}
\begin{remark}
  If $N$ is free with basis $(v_1,\ldots,v_n)$, it is easy to
  construct $N'$ explicitly: Since $N$ is regularly embedded one finds
  first $w_1,\ldots,w_n \in M$ with $b(v_i,w_j)=\delta_{ij}$ and
  denotes by $P$ the linear span of the $w_j$ (which are obviously
  linearly independent). The map $f$ above is then
  given by linear continuation of $w_j \mapsto
  Q(w_j)v_j+\sum_{i=1}^{j-1}b(w_i,w_j)v_i \in N$. The linearly
  independent vectors
  $v'_j=w_j-f(w_j)$ span then the  free and totally isotropic
  submodule $N'$ and satisfy $b(v_i,v_j')=\delta_{ij}$ so that
  $N\oplus N'$ is isometric to the hyperbolic module $H(N)$.
\end{remark}
\begin{corollary}
  Let $R=F=S$ be a field of characteristic different from $2$ and
  $(V,Q)$ a two dimensional regular quadratic space over $F$.
Then the following are equivalent:
\begin{enumerate}
\item $(V,Q)$ is a hyperbolic plane.
\item $(V,Q)$ is isotropic.
\item $\det(V,Q)=-1 \cdot (F^\times)^2$.
\end{enumerate}
\end{corollary}
\begin{proof}
  The equivalence of a) and b) follows from the previous theorem. For
  the equivalence of b) and  c)
  let $(v,w)$ be a basis of $V$.
The quadratic equation $0=Q(x+cy)=Q(x)+cb(c,y)+c^2Q(y)$ is then
solvable with $c \in F$
if and only if the discriminant $-4Q(v)Q(w)+b(v,w)^2$ is a square.
\end{proof}
\section{Witt's Theorems and the Witt Group}
The following basic theorems of Witt can be generalized to quadratic
(or rather bilinear) modules over an arbitrary local ring, as has been observed by
Kneser. For the
reader's convenience we treat the (original) case of quadratic spaces
over a field ov characteristic different from $2$ separately and
return to the general case at the end of this section. \begin{theorem}[Witt's extension theorem]\label{extension_theorem}
  Let $R=F=S$ be a field with $\cha(F)\ne 2$ and $(V,Q)$ a finite dimensional quadratic
  space over $F$.

Let $U_1,U_2\subseteq V$ be regular isometric subspaces with an
isometric isomorphism
$\rho:U_1 \to U_2$.

Then $\rho$ can be extended to all of $V$, i.e., there is $\sigma \in
O(V,Q)$ with $\sigma\vert_{U_1}=\rho$.
\end{theorem}
\begin{proof}
  We use induction on $r=\dim(U_1)=\dim(U_2)$.

For $r=1$ we have $U_i=Fx_i (i=1,2)$ with $Q(x_i)\ne 0$ and
$\rho(x_1)=x_2$. We can then put $\sigma=\tau_{x_1-x_2}$ or
$\sigma=\tau_{x_2}\circ \tau_{x_1+x_2}$ as above.

Let $r>1$ and assume the theorem to be proven for subspaces of
dimension $<r$. 
There is $x_1\in U_1$ with $Q(x_1)\ne 0$, we put $x_2=\rho(x_1)\in
U_2$ and set $W_1:=(Fx_1)^\perp\subseteq U_1,
W_2=\rho(W_1)=(Fx_2)^\perp$, so that $U_1=Fx_1\perp W_1, U_2=Fx_2\perp
W_2$ holds.

By the inductive assumption there exists $\sigma_1 \in O(V)$ with
$\sigma\vert_{W_1}=\rho\vert_{W_1}$, so that we have
$\sigma_1^{-1}\circ \rho\vert_{W_1}=\id_{W_1}$. The vector  
$y_1:=(\sigma_1^{-1}\circ \rho)(x_1)$ is in $W_1^\perp$ since it satisfies
$b(y_1,w_1)=b(\rho(x_1),\rho(w_1))=0$ for all $w_1\in W_1$, and we can
find as in the case of $r=1$ above a map $\sigma_2 \in O(V)$ with
$\sigma_2(x_1)=y_1$, and since $\sigma_2$ is a product of symmetries with
respect to vectors in $W_1^\perp$ it satisfies
$\sigma_2\vert_{W_1}=\id_{W_1}$.

Setting $\sigma=\sigma\circ\sigma_2$  we see that we have
$\sigma\vert_{W_1}=\rho\vert_{W_1}$ as well as
$\sigma(x_1)=\rho(x_1)$, so that $\sigma$ is as desired.
\end{proof}
\begin{corollary}[Witt's cancellation theorem]\index{Witt's
    cancellation theorem}
 Let $F$ be a field of characteristic different from $2$ and $(V,Q)$ a
 quadratic space over $F$ with orthogonal splittings 
 \begin{equation*}
   V=U_1\perp W_1=U_2\perp W_2,
 \end{equation*}
where $U_1,U_2$ are regular isometric subspaces of $V$.

Then $W_1$ is isometric to $W_2$.
\end{corollary}
\begin{proof}
  By the previous theorem there is $\sigma \in O(V)$ with
  $\sigma(U_1)=U_2$. Since by regularity of $U_1,U_2$ we have
  $U_i^\perp =W_i$ for $i=1,2$, one sees that
  $\sigma(W_1)=\sigma(U_1^\perp)=U_2^\perp=W_2$ must hold so that
  $W_1,W_2$ are isometric.
\end{proof}
\begin{corollary}
  Let $F$ be a field of characteristic different from $2$ and $(V,Q)$
  a regular 
 quadratic space over $F$ with isometric and regularly embedded subspaces $U_1,U_2$.

Then there exists $\sigma \in O(V)$ with $\sigma(U_1)=U_2$.
\end{corollary}
\begin{proof}
  For $i=1,2$ we write $U_i=U_i^{(0)}+\rad(U_i)$ with regular
  subspaces $U_i^{(0)}$. 

By applying Theorem
  \ref{isotropic_hyperbolic} to the (regular) orthogonal complement
  $V_1\subseteq V$ of $U_1^{(0)}$ in $V$ we obtain a hyperbolic subspace
  $W_1\subseteq V$ which is orthogonal to $U_1^{(0)}$ and contains
  $\rad(U_1)$ as a maximal totally isotropic subspace. In the same way
  find  a hyperbolic space $W_2$ orthogonal to $U_2^{(0)}$ in which
  $\rad(U_2)$ is a maximal totally isotropic subspace.

 In particular,
  $W_1,W_2$ are hyperbolic spaces of the same dimension and hence
  isometric, with an isometry that sends $\rad(U_1)$ to
  $\rad(U_2)$, which implies that there is an isometry from the regular
  quadratic space $\tilde{U}_1:=U_1^{(0)}\perp W_1$ to $U_2^{(0)}\perp
  W_2$ mapping $U_1$ onto $U_2$. By Theorem \ref{extension_theorem}
  this isometry can be extended to all of $V$.
\end{proof}
With a little more work Theorem \ref{extension_theorem} can be proven
in a more general situation:

\begin{theorem}\label{extension_theorem_localring}
  Let $R$ be a local ring and $(M,Q)$ be a quadratic module over $R$
  (with values in $R$). Let $N_1,N_2$ be 
  regularly embedded isometric
  submodules which are free of finite rank with an isometry
  $\sigma:N_1\to N_2$.

Then $\sigma$ can be
extended to an element of the orthogonal group $O(M,Q)$.

Moreover, if $X$ is a submodule of $M$ with
$\tilde{b}^{(N_i)}(X)=N_i^*$ for $i=1,2$ and such that
$\sigma(v)\equiv v \bmod X$ for all $v \in N_1$ holds, the extension
$\tilde{\sigma}$ can be chosen such that it is trivial on $X^\perp$
and satisfies $\tilde{\sigma}(v)\equiv v \bmod X$ for all $v \in M$.

In particular, any isometry between regularly embedded subspaces of a
finite dimensional quadratic space $(V,Q)$ over a field $F$ (of any characteristic) can be
extended to an element of the orthogonal group $O(V,Q)$ of the space
and Witt's cancellation theorem holds for finite dimensional quadratic spaces over an
arbitrary field.
\end{theorem}
\begin{proof}  We follow the proof from \cite{kneserbook}.

Let $P$ denote the maximal ideal of $R$ and let $k=R/P$ be the
residue field. By $\bar{Y}=Y/PY$ we denote the reduction of a module
$Y$ modulo $P$, by $\bar{Q}, \bar{b}$ the reductions of the quadratic
and the bilinear form.

Following \cite{kneserbook} we adapt the idea of the proof of Theorem \ref{extension_theorem}
suitably to the present more general situation. It will turn out that
this is easier if we assume that in the case $k\not\cong \F_2$ one has
$\bar{Q}(\bar{X})\neq \{0\}$ whereas for  $k\cong \F_2$ one has
$\bar{Q}(\bar{X}^\perp)\neq\{0\}$ (with the orthogonal complement
taken inside $\barX$); we will dispose of the two
exceptional cases in the end.

We proceed by
induction, starting with $N_i=Rw_i$ being of rank $1$,
$\sigma(w_1)=w_2=w_1+x$ with $x \in X$.

If $Q(x) \in R^\times $ holds we have
$\tau_x(w_1)=w_2$ for the reflection $\tau_x$, and we are done with $\tilde{\sigma}:=\tau_x$.

Otherwise we have $Q(x) \in P$ and $\tau_x(w_1)=w_2$ implies
\begin{equation*}
Q(x)=-b(w_1,x)=b(w_2,x)\in P, 
\end{equation*}
hence
\begin{equation*}
  \barQ(\barx)=0=\barb(\barw_1,\barx)=\barb(\barw_2,\barx)=0.
\end{equation*}
We write $\bar{X}_i=\{\bar{x}\in
\bar{X}\mid \bar{b}(\bar{w}_i,\bar{x})=0\}$ and proceed to show that
$\bar{Q}(\bar{X}\setminus \bar{X}_1\cup \bar{X}_2)= \{0\}$ would
contradict our assumptions. Indeed, if that was the case  we had
\begin{equation*}
\bart^2\barQ(\bary)+\bart\barb(\bary,\barz)=\barQ(\bart\bary+\barz)=0
\end{equation*}
for all $\bart\in k, \bary\in \barX_1\cap \barX_2, \barz\in
\barX\setminus (\barX_1\cup \barX_2)$. If we have $k\not\cong \F_2$
with $\barQ(\barX)\ne\{0{}$
this implies $\barQ(\bary)=0=\barb(\bary,\barz)$. 
Since we have $\barx\in \barX_1\cap\barX_2$ we see in particular
$\barb(\barx, \barX\setminus (\barX_1\cup \barX_2)=\{0\}$.
Moreover, an
easy exercise in linear algebra shows that $\bar{X}\setminus
\bar{X}_1\cup \bar{X}_2$ generates $\barX$ in this case, and we see
$\barb(\barx,\barX)=\{0\}$,  and $\barw_2=\barw_1+x$ implies
$\barX_1=\barX_2$.
But then $\barX_1\cap\barX_2=\barX_1\cup\barX_2=\barX_1$ and we obtain
$\barQ(\barX)=\{0\}$, which contradicts our assumptions. If, on the other
hand, we have $\barQ(\barX^\perp)\ne \{0\}$, we notice first that we
have
\begin{equation*}
  \barQ(\bary)=\barQ(\bary+\barz)=0
\end{equation*}
for all
\begin{equation*}
  \bary\in \barX^\perp\cap \barX_1\cap \barX_2, \barz \in
  \barX^\perp\cap(\barX\setminus(\barX_1\cup \barX_2)).
\end{equation*}

Obviously, we have $\barX^\perp\cap\barX_1=\barX^\perp\cap\barX_2$,
which implies that $\barX^\perp$ is the union of
$\barX^\perp\cap\barX_1\cap\barX_2$ and $\barX\setminus(\barX_1\cup
\barX_2)$, and we get $\barQ(\barX^\perp)=\{0\}$ contrary to our
assumptions.

Summing up this discussion, we have established that
$\barQ(\barX\setminus (\barX_1\cup \barX_2))\neq \{0\}$ holds. We may
therefore choose $y\in X$ with $Q(y)\in R^\times, b(w_1,y)\in
R^\times, b(w_2,y)\in R^\times$. Writing $w_2=\tau_y(w_1)+z$ we have
\begin{eqnarray*}
  z&=&b(w_1,y)Q(y)^{-1}y+x\\
  Q(z)&=&b(w_1,y)b(w_2,y)Q(y)^{-1}+Q(x),
\end{eqnarray*}
hence $Q(z)\in R^\times$ and $\tau_z\tau_y(w_1)=w_2$, so that $\tilde{\sigma}:=\tau_z\tau_y$
is as desired and the assertion for $N_i$ of rank $1$ is proven.

Let now $r>1$. By the assumptions of the theorem we can choose a basis
$(w_1,\dots,w_r)$ of $N_1$ and vectors $x_1,\dots,x_r\in X$ with $b(w_i,x_j)=\delta_{ij}$ and have
$X=\sum_{i=1}^rRx_i \oplus (X\cap N_1^\perp)$.
With the inductive step and our additional assumptions on $X$ in mind we make this choice as follows:
We take  a vector $x\in X$ with $Q(x)\not \in P$ and
   $\bar{x}\in \barX^\perp$ in the case $k=\F_2$ and choose $x_r \in
   X\setminus (N_1^\perp\cap X)$ in such a way that one has $\barx \in
   k\barx_r+\overline{N_1^\perp\cap X}.$ The vector $\barx_r$ can be
   extended to a basis $\barx_1,\dots,\barx_r$ of $\barX$ modulo
   $N_1^\perp\cap X$. We choose representatives $x_i\in X$ of the
   $\barx_r$ and let $(w_1,\dots,w_r)$ be the basis dual to
   $(x_1,\dots,x_r)$ of $N_1$.

By the
inductive assumption we find  a product  $\rho$ of 
reflections in vectors of $x\in X$ with $Q(x)\in R^\times$ that 
satisfies 
$\rho \vert \sum_{i=1}^{r-1}w_i=\sigma \vert \sum_{i=1}^{r-1}w_i$.
Replacing $\sigma$ by $\rho^{-1}sigma$ we may assume that one has
 $\sigma(w_i)=w_i$ for $1\le i \le r-1$, which implies
 $b(\sigma(x)-x,w_i)=0$ for $1\le 1\le r-1$.
and $\sigma(x)-x \in
Rx_r\oplus \oplus (X\cap N_1^\perp)=:\tilde{X}$. 
We consider now $Rf_r:=\tilde{F}$ instead of $F$ and and $\tilde{X}$ instead of $X$. 
It is clear that the conditions of the theorem are satisfied in this
situation, and by our choice of the $x_i,w_i$ the additional
conditions
 $Q(\tilde{X})\not\in P$ in the case $k\not\cong \F_2$ and
 $\barQ(\overline{\tilde{X}})\neq \{0\}$ in the case $k=\F_2$ are
 satisfied too. 

We may therefore 
apply the case of rank 1 to this situation and obtain a reflection
$\tau_y\in O(M,Q)$ satisfying $\tau_y\vert N_1=\sigma$ which is
congruent to the identity modulo $X$ and is trivial on $X^\perp$ and
finish the induction.

To complete our proof we have to consider the situations where our
additional assumptions are violated. We take then a hyperbolic plane
$H=Re+Rf$ (so $Q(e)=Q(f)=0,b(e,f)=1$) and set $M':=M\perp H, N_1'=N_1
\perp Re, N_2'=N_2+Re, X'=X\perp R(e+f), \sigma'=\sigma\perp
\id_{Re}$. We have $Q(e+f)=1$ and in addition $e+f \in {X'}^\perp$ in
the case $k=\F_2$, so our additional assumptions are satisfied in this
situation and we obtain an extension $\rho$ of $\sigma'$ which fixes
$e$ and 
the vector $e-f\in {X'}^\perp$, hence also  $e+f$. We can therefor
write $\rho=\tilde{\sigma}\perp \id_{R(e+f)}$ and obtain the desired
extension $\tilde{\sigma}$ of $\sigma$. 
\end{proof}
\begin{remark}
If one of the two additional conditions stated in the beginning of the
proof is satisfied, the proof above shows that the extension of $\sigma$ can be chosen to be a
product of reflections $\tau_y$ in vectors $y\in X$ with $Q(y)\in R^\times$.   
\end{remark}

\begin{corollary}
  Let $(V,Q)$ be a finite dimensional quadratic space over the field
  $F$. Then all maximal totally isotropic regularly embedded subspaces
  of $V$ have the same dimension. This dimension is also equal to the
  maximal number of hyperbolic planes that can be split off
  orthogonally in $V$, and one can write $V$ as $V=\perp_{i=1}^r H
  \perp \rad_b(V)\perp U$, where $r$ is the Witt index, $H$ denotes a
  hyperbolic plane, and $U$ is a regular  anisotropic quadratic space.
\end{corollary}
\begin{proof}
  By remark \ref{regularly_embedded_remark} we have to treat only the
  case that $(V,Q)$ is regular. If $U_1,U_2$ are maximal totally
  isotropic spaces with $\dim(U_1)\le \dim(U_2)$, we let $W_1$ be a
  subspace of $U_2$ of dimension $\dim(U_1)$. This space is obviously
  isometric to $U_1$, one can therefore find $\sigma \in O(V)$ with
  $\sigma(U_1)=W_1$. But then $\sigma^{-1}(U_2)\supseteq U_1$ is a
  totally isotropic subspace containing $U_1$, so by maximality equal
  to $U_1$ and we see $\dim(U_1)=\dim(U_2)$. The rest of the assertion
  follows immediately.
\end{proof}
\begin{definition}[Witt index]\label{Witt_index_def}\index{Witt index}
  Let $(V,Q)$ be a finite dimensional quadratic space over the field
  $F$. Then the dimension of a maximal totally isotropic regularly
  embedded subspace is called the {\em Witt index} of $(V,Q)$.
\end{definition}
\begin{theorem}[Witt decomposition]\index{Witt decomposition}\label{Witt_decomposition}
  Let $F$ be a field and $(V,Q)$ a  finite dimensional quadratic space
  over $F$ of Witt index $r$. 

Then $V$ has an orthogonal decomposition 
\begin{equation*}
  V=V_{\rm an}\perp H_r\perp \rad_b(V),
\end{equation*}
where $V_{\rm an}$ is anisotropic regular and  $H_r$ is a hyperbolic space
of dimension $2r$.

The isometry class of the space $V_{\rm an}$ is uniquely determined. 

If $\cha(F)\ne 2$, the restriction of $Q$ to $\rad_b(V)$ is
identically zero.
\end{theorem}
\begin{proof}
  Replacing $V$ by a (regular) subspace compIementary to $\rad_b(V)$
  and using that this space is isometric to the quotient space
  $V/\rad_b(V)$ we see that it  suffices to assume that $(V,Q)$ is regular. In that
  case we split off orthogonally a maximal $2r$-dimensional hyperbolic
  subspace whose orthogonal complement $V_{\rm an}$ is necessarily
  anisotropic. If we have another such splitting $V=V'_{\rm an}\perp
  H'_r$, the hyperbolic spaces $H_r,H'_r$ of dimension $2r$ are
  isometric, and by the Witt cancellation theorem the anisotropic
  spaces $V_{\rm an},V'_{\rm an}$ are isometric as well.
\end{proof}
\begin{definition} The up to isometry unique subspace $V_{\rm an}$
  with $Q$ restricted to it in
  the decomposition of the theorem (as well as its isometry class) is
  called the anisotropic kernel of $(V,Q)$. 
  \end{definition}
  \begin{lemma}\label{minusmodule}
    Let $(M,Q)$ be a finitely generated projective regular quadratic module. Then the (external)
    orthogonal sum $(M,Q)\perp (M,-Q)$ is isometric to the hyperbolic
    module $H(M)$ over $M$. 
  \end{lemma}
  \begin{proof}
    We denote by $\beta$ a bilinear form on $M$ with $\beta(x,x)=Q(x)$
    for all $x \in M$ and define maps $\phi:(M,Q)\perp (M,-Q) \to
    H(M), \psi:H(M)\to (M,Q)\perp (M,-Q)$ by
    \begin{eqnarray*}
      \phi((x,y))&=&(x+y,\tilde{\beta}(x)+\tilde{\beta}(y)-\tilde{b}(y))\\
\psi((z,z^*)&=&((x,y)) \text{ where }
                \tilde{b}(y)=\tilde{\beta}(z)-z^*,\quad x=z-y;
    \end{eqnarray*}
notice that the regularity of $(M,Q)$ implies that the vector $y$ in the
definition of $\psi$ exists and is uniquely determined.

A direct calculation shows that $\phi,\psi$ are mutually inverse isometries.
  \end{proof}
  \begin{deftheorem}
    Let $R$ be a ring, denote by $\tilde{W}(R)$ the set of isometry
    classes of finitely
    generated projective regular quadratic modules $(M,Q)$ over $R$.

Denote by $\sim$ the relation on $\tilde{W}(R)$ given by
$(M_1,Q_1)\sim (M_2,Q_2)$ (where we do not distinguish in the
notation between the quadratic module and its isometry class) if and only if there exist finitely
generated projective hyperbolic modules $H_{1},H_{2}$  such that $(M_1,Q_1)\perp H_{1}$ is isometric
to $(M_2,Q_2)\perp H_{2}$.

Then the relation $\sim$ is an equivalence relation which is
compatible with forming orthogonal sums, 
and the set $W(R)$ of equivalence classes with addition defined by
taking orthogonal sums of representatives is an abelian group, the
{\em Witt group}\index{Witt group} of $R$. The neutral element of the
group is the class of hyperbolic spaces, the inverse to the class of
$(M,Q)$ is the class of $(M,-Q)$.

If $R=F$ is a field, each class in the Witt group is represented by the
isometry class of a uniquely determined regular anisotropic quadratic space.
  \end{deftheorem}
  \begin{proof}
  It is obvious that $\sim$ is an equivalence relation and that it is
  compatible with forming orthogonal sums, so that one obtains an
  addition  on the set $W(R)$ of equivalence classes induced by taking
  orthogonal sums of quadratic   modules. It is also obvious that this
  addition is associative and commutative and that the class of
  hyperbolic modules is a neutral element for this addition. Finally,
  Lemma \ref{minusmodule} shows that each class has an additive
  inverse in $W(R)$.

If $R=F$ is a field, the Witt decomposition theorem shows that each
class in $W(R)$ has a representative which is anisotropic, and the
Witt cancellation theorem gives the uniqueness of the isometry class
of an anisotropic representative. 
  \end{proof}
\section{Orthogonal Bases}
In this section we let $F$ be a field.
\begin{theorem}
 Let $(V,Q)$ be a finite dimensional regular quadratic space over $F$.
 \begin{enumerate}
 \item If $\cha(F)\ne 2$, the space $V$ has an orthogonal basis.
\item If $\cha(F)=2$, the space $F$ has a decomposition into an
  orthogonal sum of (regular) $2$-dimensional subspaces. In
  particular, $\dim(V)$ is even.
 \end{enumerate}
\end{theorem}
\begin{proof}
  The regular space $V$ contains a vector $x$ with $Q(x)\ne 0$. If
  $\cha(F)\ne 2$, the one dimensional space $Fx$ is regular and can be
  split off orthogonally, so the assertion follows by induction.
If $\cha(F)=2$, by regularity there is a vector $y\in V$ with
$b(x,y)=1$, and $b(x,x)=2Q(x)=0$ implies that $x,y$ are linearly
independent. The Gram matrix of the two dimensional  subspace
$W=Fx+Fy\subseteq V$ with respect to
the basis $(x,y)$ has then determinant $-(F^\times)^2$, so this space
is regular and can be split off orthogonally, and the assertion
follows again by induction on the dimension of $V$.
\end{proof}
\begin{remark}
  \begin{enumerate}
  \item If $\cha(F)\ne 2$, an arbitrary basis of the space $V$ can be
    transformed into an orthogonal basis by the well known
    Gram-Schmidt algorithm of linear algebra, with slight
    modifications to admit isotropic vectors in the original  basis.
\item If $\cha(F)=2$, an odd dimensional quadratic space $V=U\perp Fx$
  is called {\em half regular}\index{half regular space} if $U$ is
  regular and $Q(x) \ne 0$ holds. Such a space has one dimensional
  bilinear radical $\rad_b(V)$ and trivial $Q$-radical
  $\rad_Q(V)$. It follows that all isotropic vectors are regularly
  embedded. These spaces can also (see Kneser's book \cite{kneserbook}) be characterized by the non
  vanishing of the so called half determinant:  Consider for odd $n$ 
the  symmetric matrix
  $A_X=(a_{ij})$ over the ring $\Z[\{X_{ij}\mid 1 \le i\le j\le n\}]$
  with even diagonal elements $a_{ii}=2X_{ii}$ and
  $a_{ij}=a_{ji}=X_{ij}$ for $i < j$. Then it is an exercise in linear
  algebra   to prove that $\det(A_X)$ as a polynomial over $\Z$ in the
  variables $X_{ii}$ for $1\le i\le n$ and $X_{ij}$ for $1\le i<j\le
  n$ has even coefficients, so that $\frac{1}{2}\det(A_X)$ is still an
  integral polynomial. The half determinant of a finitely generated free
quadratic module of odd dimension over the ring $R$ with basis
$(v_1,\ldots,v_n)$  is
then the square class of the value of the polynomial
$\frac{1}{2}\det(A_X)$ upon insertion of the $Q(v_i)$ for $X_{ii}$ and
the $b(v_i,v_j)$ for $X_{ij}$ for $i<j$. The half regular spaces over
the field $F$ of characteristic $2$ are then the spaces whose half
determinant (with respect to any basis) does not vanish.  
  \end{enumerate}
\end{remark}

\section{Lattices and their duals}
\begin{definition}\label{lattice_definition}
  Let $R$ be an integral domain with field of fractions $F$, let
  $V$ be a vector space over $F$ of finite
  dimension $n$.

An $R$-submodule $\Lambda$ of $V$ is called an $R$-lattice in $V$ if there is a
basis  $(v_1,\ldots,v_n)$ of $ V$ with $\Lambda \subseteq
\oplus_{j=1}^n Rv_j$. It is called a $R$-lattice on $V$ if the subspace of
$V$ generated by $\Lambda$ over $F$ is equal to $V$, it is called a
lattice of rank $r\le n$ in $V$, if that subspace has dimension $r$
over $F$.
\end{definition}
\begin{remark}
Lattices are usually considered over a ground ring $R$ which is
noetherian, which implies that they are finitely generated
$R$-modules. Over a general integral domain this condition is
sometimes added to the definition.
\end{remark}
\begin{lemma}\label{lattice_compare}
Let $R,F, V$ be as above, let $\Lambda$ be an $R$-lattice on $V$.

An $R$-submodule $M$ of $V$ is an $R$-lattice in $V$ if and only if there exists
$0\ne a \in R$ with $aM\subseteq \Lambda$.   

In particular, for two $R$-lattices $\Lambda_1, \Lambda_2$ on $V$
there exist nonzero $a,b \in F$ with $a\Lambda_2 \subseteq \Lambda_1
\subseteq b\Lambda_2$.
\end{lemma}
\begin{proof}
If $M$ is a lattice, let $(v_1,\ldots v_n)$ be a basis of $V$ with
$M\subseteq \oplus_{j=1}^n Rv_j$, let $w_1,\ldots, w_n$ be $n$
linearly independent vectors in $\Lambda$ and write
$v_j=\sum_{i=1}^nc_{ij}w_i$ with $c_{ij}\in F$. If we let $a$ be the
product of the denominators of the $c_{ij}$ when these are written as
fractions of elements in $R$, we have $av_j \in \Lambda$ for all $j$
and hence $aM\subseteq \Lambda$.

If conversely $M$ is an $R$-submodule of $V$ and $0\ne a \in R$ with $aM
\subseteq \Lambda$, let $\Lambda \subseteq \oplus_{j=1}^nRv_j$ for a
suitable basis $(v_1,\ldots,v_n)$ of $V$. Then the $w_j=a^{-1}v_j$ form a
basis of $V$  with $M\subseteq\oplus_{j=1}^nRw_j$.
\end{proof}
\begin{lemma}
  \begin{enumerate}
  \item $R$-submodules of $R$-lattices in $V$ are $R$-lattices in $V$.
\item If $U$ is a subspace of $V$, an $R$-submodule of $U$ is a
  lattice in $U$ if and only if it is a lattice in $V$.
  \end{enumerate}
\end{lemma}
\begin{proof}
 Obvious.
\end{proof}
\begin{definition}\label{dual_lattice_definition}
Let $R,F$ be as above and let $(V,Q)$ be a finite dimensional non degenerate quadratic
space over $F$ with associated symmetric bilinear form $b$, let
$\Lambda$ be an $R$-lattice on $V$. 
 
The dual  of $\Lambda$ is
\begin{equation*}
  \Lambda^\#:=\{v\in V \mid b(v,\Lambda)\subseteq R\}.
\end{equation*}
 \end{definition}
 \begin{remark}
The dual $\Lambda^\#$ of $\Lambda$ is the image of $\Lambda^*=\Hom_R(\Lambda,R)
\subseteq V^*$  under the isomorphism $\tilde{b}^{-1}:V^*\to V$. 
 \end{remark}
 \begin{lemma}
  Let the notations be as above.
  \begin{enumerate}
  \item $\Lambda^\#$ is a lattice on $V$.
\item $\Lambda\subseteq \Lambda^\#$ if and only if
$b(\Lambda,\Lambda)\subseteq R$.
\item If $b(\Lambda,\Lambda)\subseteq R$ holds, the factor module
  $\Lambda^\#/\Lambda$ is a torsion module. 
\item If $M$ is another $R$-lattice on $V$, one has $(M\cap
  \Lambda)^\#=M^\#+\Lambda^\#$ and $(M+\Lambda)^\#=M^\#\cap \Lambda^\#$. 
  \end{enumerate}
 \end{lemma}
 \begin{proof}
  For a), let  $(w_1,\ldots,w_n)$ be a basis of $V$ contained in
  $\Lambda$, let $(w_1^\#,\ldots, w_n^\#)$ be the dual basis of  $V$
  with respect to $b$, i.e, $b(w_i,w_j^\#)=\delta_{ij}$. Then one has
  $\Lambda^\#\subseteq \oplus_{j=1}^nRw_j^\#$, so $\Lambda^\#$ is a
  lattice in $V$. If $(v_1,\ldots,v_n)$ is a basis of $V$ with
  $\Lambda \subseteq \oplus_{j=1}^nRv_j$ and dual basis
  $(v_1^\#,\ldots,v_n^\#)$ of $V$ with respect to $b$, the $v_j^\#$ are
  in $\Lambda^\#$ and generate $V$ over $F$, so $\Lambda$ is a lattice
  on $V$.

Assertions b) and d) are obvious. If we have $\Lambda \subseteq
\Lambda^\#$, we let $(w_1,\ldots,w_n)$ and $(w_1^\#,\ldots, w_n^\#)$
be as in the proof of a). Let $c$ be such that $cA \in M_n(R)$, where
$A$ is the matrix obtained by expressing the $w_j^\#$  as linear
combinations of the basis vectors $w_i$. Then we have
$c\Lambda^\#\subseteq \Lambda$ and obtain c).

For e), we have $\Lambda=\oplus_{j=1}^nRv_j$ for a basis of $V$ and
hence $\Lambda^\#=\oplus_{j=1}^nRv_j^\#$, where again the $v_j^\#$
form the dual basis of $V$ with respect to $b$. From this we see that
$(\Lambda^\#)^\#=\Lambda$ holds.
 \end{proof}
 \begin{definition}\label{relative_determinant}\label{R-index}
  Let $\Lambda_1,\Lambda_2$ be two free $R$-lattices on the finite
  dimensional vector space $V$ of dimension $n$ over $F$. 

The 
{\em $R$-index} $\ind_R(\Lambda_1/\Lambda_2)=(\Lambda_1:\Lambda_2)_R$ is the
class $\det(T)R^\times$, where $T=(t_{ij})\in  M_n(F)$
is any matrix expressing an $R$- basis $(w_1,\ldots,w_n)$ of
$\Lambda_2$ by the vectors $(v_1,\ldots,v_n)$
of a basis of $\Lambda_1$, i.e, with $w_j=\sum_{i=1}^n t_{ij}v_i$ for
$1 \le j \le n$.

We will also denote any  such $\det(T)$ by
$\ind_R(\Lambda_1/\Lambda_2)$ or $(\Lambda_1:\Lambda_2)_R$.
 \end{definition}
 \begin{remark}
 If $\Lambda_2 \subseteq \Lambda_1$ is a sublattice of $\Lambda_1$,
 the group index $(\Lambda_1:\Lambda_2)$ is equal to the order of the
 factor ring 
 $R/\ind_R(\Lambda_1/\Lambda_2) R$. In particular for $R=\Z$ it is equal
 to the absolute value of $ \ind_R(\Lambda_1/\Lambda_2) $.   
 \end{remark}
 \begin{lemma}\label{determinantrelation_lattices}
Let $\Lambda_1,\Lambda_2,V$ be as in the definition above and let $b$
be a non degenerate symmetric bilinear form on $V$.

Then 
\begin{equation*}
  {\det}_b(\Lambda_2)=((\Lambda_1:\Lambda_2)_R)^2 {\det}_b(\Lambda_2).
\end{equation*}
In particular, if one has $\Lambda_1\subseteq \Lambda_2$, the lattices
have equal determinant if and only if they are equal.
    \end{lemma}
    \begin{proof}
      Obvious.
    \end{proof}
\begin{lemma}
  Let $\Lambda,\Lambda_1,\Lambda_2$ be  reflective $R$- lattices on the regular finite dimensional quadratic
  space $(V,Q)$ over $F$ with associated symmetric bilinear form $b$.
  \begin{enumerate}
  \item $(\Lambda^\#)^\#=\Lambda$.
\item One has 
$\Lambda_1 \subseteq \Lambda_2$ if and only if one has $\Lambda_2^\#\subseteq \Lambda_1^\#$.
\item If $\Lambda$ is free of finite rank, one has \begin{equation*}
\det(\Lambda)R^\times=(\Lambda^\#:\Lambda)_R,\quad \det(\Lambda^\#)=(\det(\Lambda))^{-1} 
\end{equation*}

\end{enumerate}
\end{lemma}
\begin{proof} 
a) is easily seen to follow from the reflectivity of $\Lambda$ and the
fact that $\tilde{b}:\Lambda \to \Lambda^*$ is an isomorphism of
modules.

In b) the direction from left to right is obvious, the reverse
direction then follows from a).

For the computation of the
determinants in c) we may assume, multiplying $\Lambda$ by a suitable $a\in
R$, that $\Lambda \subseteq \Lambda^\#$ holds. 
Furthermore, we see that $b(v_i,v_k)=b(\sum_{j=1}^mb(v_i,v_j) v_j^\#,v_k)$ holds
for all $i,j,k$, which implies $v_i=\sum_{j=1}^mb(v_i,v_j)v_j^\#$ for
all $i$. In other words, the Gram matrix of $b$ with respect to the
basis of $\Lambda$ consisting of the $v_i$ is the matrix 
obtained upon expressing the $v_i$ as linear combinations of the
$v_j^\#$. By the definition of  $(\Lambda^\#:\Lambda)_R$ we see that 
$\det(\Lambda)=(\Lambda^\#:\Lambda)_R(R^\times)^2$ is true. For the
determinant of $\Lambda^\#$ exchange the roles of $\Lambda,
\Lambda^\#$ in the above argument.

\end{proof}
\section{Lattices and orthogonal decompositions}
 \begin{lemma}\label{ortho_decompositions}
  Let $(V,Q)$ be a finite dimensional regular quadratic space over $F$
  with an orthogonal decomposition $V=U_1\perp U_2$, let $\Lambda$ be
  a lattice of full rank on $V$ with $b(\Lambda,\Lambda)\subseteq R$
  (equivalently, $\Lambda\subseteq \Lambda^\#$).
Let $L_i=\Lambda \cap U_i (i=1,2)$ and denote by $\pi_i$ the
orthogonal projection $\pi_i:V\to U_i$.
Then one has
\begin{enumerate}
\item $U_i\cap \Lambda^\# =(\pi_i(\Lambda))^\#$ for $i=1,2$ (where the dual on the
  right hand side is taken inside $U_i$). 
\item $\pi_1(\Lambda)/L_1 \cong \pi_2(\Lambda)/L_2$ as $R$-modules.
\end{enumerate}
 \end{lemma}
 \begin{proof}
    \begin{enumerate}
  \item Let $u \in U_1$ be given, write $x\in\Lambda$ as
    $x=\pi_1(x)+\pi_2(x)$.

Then $u \in \Lambda^\#$ is equivalent to $b(u,\pi_1(x))=b(u, x)\in R$
for all $x \in \Lambda$, hence to $b(u,\pi_1(\Lambda))\subseteq R$,
hence to $u \in (\pi_1(\Lambda))^\#$.
\item For  $u_1=\pi_1(x) \in \pi_1(\Lambda)$ there exists $u_2 \in
  \pi_2(\Lambda)\subseteq U_2$ with $u_1+u_2 \in \Lambda$ (e.g. $\pi_2(x)$), and the coset
  $u_2+L_2$ is independent of the choice of $u_2$ since $u_1+u_2' \in
  \Lambda$ implies $u_2-u_2'\in U_2\cap \Lambda=L_2$. The $R$-linear map
  from $\pi_1(\Lambda)$ to $\pi_2(\Lambda)/L_2$ given by $u_1 \to u_2$ as
  above has kernel $\{\pi_1(x)\mid x\in \Lambda, \pi_2(x)\in
  \Lambda\}=L_1$ and is obviously surjective.
\end{enumerate}
 \end{proof}
 \begin{definition}
   \begin{enumerate}
   \item Let the notations be as above. The lattice $\Lambda$ is
     called unimodular if $\Lambda=\Lambda^\#$, it is called even
     unimodular if in addition $Q(\Lambda)\subseteq  R$ (equivalently,
     $b(x,x)\in 2R$ for all $x \in \Lambda$) holds.
\item The lattice $R$ is called $I$-modular for an ideal $I\subseteq
  R$ if one has $\Lambda=I\Lambda^\#$, it is called even
     $I$-modular if in addition $Q(\Lambda)\subseteq  I$ holds.
It is called (even) modular if it is (even) $I$-modular for some ideal
$I\subseteq R$. 
   \end{enumerate}
 \end{definition}
 \begin{remark}
 The lattice $\Lambda$ is even unimodular if and only if $(\Lambda,Q)$
 is a regular quadratic $R$-module  
 \end{remark}
 \begin{lemma}
 Let $\Lambda$ be an $I$-modular  $R$-lattice on $V$ which is a free
 $R$-module.
 \begin{enumerate}
 \item $\tilde{b}^{(\Lambda)}(\Lambda)=\Hom_R(\Lambda,I)$.
\item $\Lambda=\{v\in V\mid b(v,\Lambda)\subseteq I\}$.
 \end{enumerate}
 \end{lemma}
 \begin{proof}
   a) is obvious. For b) let $v$ be in the set on the right hand
   side. Then $\tilde{b}^{(\Lambda)}(v) \in \Hom_R(\Lambda,I)$, so by
   a) and since $\tilde{b}$ is injective we have $v \in \Lambda$. The
   other inclusion is trivial. 
 \end{proof}
 \begin{lemma}\label{modular_split}
Let $R$ be an integral domain with field of fractions $F$, let $(V,Q)$
be a regular quadratic space over $F$, let $\Lambda$ be an $R$-lattice
on $V$   and let $K\subseteq \Lambda$ be an $I$-modular  sublattice of $\Lambda$ for
some ideal $I\subseteq R$, i.e., $K$ generates over $F$ a regular
subspace $U$ of $V$ and is an $I$-modular lattice on $U$.

Then $K$ splits off orthogonally in $\Lambda$ if and only if one has
$b(K,\Lambda)\subseteq I$.
\end{lemma}
\begin{proof}
Since $K$ is $I$-modular, we have $\tilde{b}^{(K)}(K)=\Hom_R(K,I)$,
the assumption $b(K,\Lambda)\subseteq I$ gives
$\tilde{b}^{(K)}(\Lambda)\subseteq \Hom_R(K,I)$, hence
$\tilde{b}^{(K)}(\Lambda)=\tilde{b}^{(K)}(K)$, and 
Theorem \ref{orthogonalcomplements} implies that $K$ splits off orthogonally.

The other direction is trivial.  
\end{proof}

 \begin{remark}
 The results of this section remain valid if we omit the condition
 that $R$ is an integral domain and replace $F$ with the total ring of
 fractions of $R$ and the vector space $V$ by a free module over
 $F$. The quadratic module $(V,Q)$ is then required to be regular.  
 \end{remark}


\chapter{Special Ground Rings}
Obviously the theory of quadratic modules over a fixed ground ring
depends very much on the ring. Our main concern being the arithmetic
of quadratic forms we are particularly interested in the case where
the ground ring is a number field or a number ring, i.e., the ring of integers in such a
field. For these, important information is contained in the study of
the real and complex embeddings of the field and of the reduction modulo
prime ideals of the ring of integers, hence in the study of finite
fields. Since many interesting properties of forms over the rational
integers or over number rings
can be studied in the more general context of principal ideal domains
respectively  of Dedekind domains, we
will also study the basic properties of forms over these in this chapter.
\section{Real and Complex Numbers}
\begin{theorem}
  $$W(\C)\cong \Z/2\Z.$$
\end{theorem}
\begin{proof}
Since all elements of $\C$ are squares, every $2$-dimensional regular
quadratic space is hyperbolic.  
\end{proof}
\begin{remark}
  Obviously, the same results holds for all {\em quadratically closed}
  fields $F$, i.e., fields admitting no quadratic extension.
\end{remark}
\begin{theorem}[Sylvester]
$$W(\R)\cong \Z.$$  
\end{theorem}
\begin{proof}
A regular quadratic space 
 $(V,Q)$ has an orthogonal basis consisting of  $r_+$
  vectors $x_j$ with  $Q(x_j)=1$ and $r_-$ vectors $y_j$ with
  $Q(y_j)=-1$; it is hyperbolic if and only if 
  $s=r_+-r_-$ is zero.
\end{proof}
\section{Finite Fields}
Let $F=\F_q$ be the finite field of characteristic $p$ with $q=p^s$
elements.
\begin{theorem}[Chevalley-Warning]
Let $f\in F[X_1,\ldots,X_n]$ be a homogeneous polynomial of degree
$d>n$. 

Then $f$ has a nontrivial zero in $F$.  
\end{theorem}
\begin{proof}
  With $g:=1-f^{q-1}$ one has (since $a^{q-1}=1$ for all $a \ne 0$ in $F$) 
  \begin{equation*}
    \sum_{\x\in F^n}g(\x)=\sum_{\{\x \in F^n\mid f(\x)=0\}}1_F=N_f \cdot 1_F,
  \end{equation*}
where $N_f$ is the number of zeroes of $f$ in $F^n$.

On the other hand, let $m=X_1^{e_1}\dots X_n^{e_n}$ be a monomial
occurring in $g$. Since the degree of $f$ is less than $n$, all
monomials occurring in $g$ have degree less than $n(q-1)$, so at least
one of the $e_i$ satisfies $e_i<q-1$, w.l.o.g we assume $e_1<q-1$.
The contribution of $m$ to $ \sum_{\x\in F^n}g(\x)$ is then
\begin{equation*}
  \sum_{\x \in F^n}m(\x)=\sum_{x_1\in
                            F}x_1^{e_1}\sum_{(x_2,\ldots,x_n)}x_2^{e_2}\dots
                            x_n^{e_n}
\end{equation*}
with $\sum_{x_1\in F}x_1^{e_1}=q\cdot 1_F=0$ if $e_1=0$ and 
\begin{equation*}
  \sum_{x_1\in F}x_1^{e_1}=\sum_{j=0}^{q-2}(\alpha^j)^{e_1}=\frac{(\alpha^{e_1})^{q-1}-1}{\alpha^{e_1}-1}=0
\end{equation*}
for $0<e_1<q-1$, where $\alpha$  is a generator of the (cyclic) multiplicative group $F^\times$
of $F$ and where $\alpha^{e_1}\ne 1$ because of $0<e_1<q-1$.

So all monomials contribute $0$ to $\sum_{\x\in F^n}g(\x)$, and we
must have $p \mid N_f$. Since $f$, being homogeneous, has at least the
trivial zero $\nullvek$, there must be at least $p-1$ nontrivial zeroes.
\end{proof}
\begin{corollary}\label{isotropy_finitefields}
  Every  quadratic space $(V,Q)$ of dimension $m\ge 3$ over the
  finite field $F$ is isotropic.
\end{corollary}
\begin{proof}
  This is an obvious consequence of the Chevalley-Warning theorem
  since the number of variables in the associated homogeneous
  quadratic 
  polynomial (with respect to any basis) over $F$ is at least $3$.
\end{proof}
\begin{definition}
  A quadratic module $(M,Q)$ over $R$ with values in $R$ is called
  universal if $Q(M)=R$ is true.
\end{definition}
\begin{example}
  The hyperbolic module $H(M)$ over a non zero finitely generated free
  module $M$ is universal since for any $z^*\in M^*$ with $1 \in z^*(M)$ the
  values $z^*(z)$ for the $z\in M$ run through all of $R$.
\end{example}
\begin{corollary}\label{universality_finitefields}
  Every $2$-dimensional regular quadratic space over a finite field
  $F$ is universal.
\end{corollary}
\begin{proof}
If 
$(V,Q)$ is isotropic, it is hyperbolic and hence
universal. Otherwise let $0\ne a \in F$ and consider the space $W:=V\perp
Fy$ (external orthogonal sum), where $Fy$ is a $1$-dimensional
quadratic space with $Q(y)=-a$. Then $W$ is isotropic, since it has
dimension $3$, and an isotropic vector $z\in W$ must be of the form
$z=x+cy$ with $c \ne 0$ since we assumed $(V,Q)$ to be anisotropic. We
have therefore $Q(x)=c^2a$, hence $Q(\frac{x}{c})=a$ as desired.   
\end{proof}
\begin{lemma}
  Let $q$ be odd.

Then two regular quadratic spaces over $F$ are isometric if and only
if they have the same dimension and the same determinant. 

In particular, in each dimension there are precisely two isometry
classes of regular quadratic spaces over $F=\F_q$.
\end{lemma}
\begin{proof}
  The assertion is trivial in dimension $1$. In dimension $2$ let a
  regular space 
  $(V,Q)$ be given and put $W:=V\perp Fy$ with a $1$-dimensional space
  $Fy$ with $Q(y)=-1$. Then $W$ is isotropic and splits as $W=H\perp
  Fz$ with a hyperbolic plane $H$, and we have
  $-2\det(V)=\det(W)=-2Q(z)(F^\times)^2$.
Moreover, since $\cha(F)\ne 2$ we can split $H$ as $H=Fx_1\perp Fx_2$
with $Q(x_1)=-1, Q(x_2)=1$. By the Witt cancellation theorem we obtain 
$V \cong Fx_2\perp Fz$ with $Q(x_2)=1, Q(z)\in \det(V)$, the isometry
class of which depends
only on $\det(V)$. The assertion for arbitrary dimension follows by
induction since all regular spaces of dimension $\ge 3$ are isotropic
and split off a hyperbolic plane.
\end{proof}
\begin{theorem}
  The Witt group of $F$ for $\cha(F) = 2$ is isomorphic to $\Z/2\Z$.

For $\cha(F)\ne 2$ we have
\begin{equation*}
  W(F)\cong
  \begin{cases}
    \Z/4\Z & q\equiv -1 \bmod 4\\
\Z/2\Z \times \Z/2\Z & q\equiv  1 \bmod 4. 
  \end{cases}
\end{equation*}
\end{theorem}
\begin{proof}
  Let $q$ be odd. Since all anisotropic spaces are of dimension $1$ or
  $2$ we see that $W(F)$ has order $4$. If $q\equiv -1 \bmod 4$, the
Witt classes of the   $1$-dimensional regular spaces have order $4$
since $-a^2$ is not a square for $a \ne 0$, otherwise they have
  order $2$, and they generate the group.

Assume now $\cha(F)=2$, let $(V,Q)$ be an anisotropic regular
quadratic space of dimension $2$. Since $(V,Q)$ is universal and
regular we have a basis $(v,w)$ of $V$ with $Q(v)=1=b(v,w), Q(w)=a \ne
0$. Moreover, since $(V,Q)$ is anisotropic we have $0\ne
Q(xv+w)=x^2+x+a$ for all $x \in F$, hence $a$ is not in the additive subgroup $S:=\{x^2+x\mid
x \in F\}$ of $F$ of index $2$. Another binary anisotropic space
$(V',Q')$ is
then of the same type with basis $(v',w')$ and an $a'\not\in S$. But
then there exists $x \in F$ with $a=x^2+x+a'$, and the linear map
sending $v$ to $v'$ and  $w$ to $xv'+w'$ is an isometry from $(V,Q)$
onto $(V',Q')$. Since every class in the Witt group of $F$ is
represented by some $2$-dimensional space, the assertion is proven in
this case too. 
\end{proof}

\section{Dedekind domains}\label{dedekindsection}
A Dedekind domain $R$ is a noetherian integrally closed integral domain in
which every non zero prime ideal is maximal. In particular, all
principal ideal domains are Dedekind domains, and the ring
$\fo_F$ of all over $\Z$ integral elements of an algebraic number
field $F$ is a Dedekind domain. 
In this section  $R$ is a Dedekind domain with field of quotients
$F$. 

The theory of lattices over Dedekind
domains is very similar to the theory over a principal ideal domain
since all localizations are principal ideal domains. We therefore treat both
cases together in this section.

We summarize first some basic facts from \cite{chevalley,curtisreiner,bourbaki}.

\medskip
A finitely generated module $M$ over the Dedekind domain $R$ is projective if and only if
it is torsion free, if $R$ is a principal ideal domain, it is free.
In
particular, all $R$ - submodules of a finite dimensional vector space
$V$ over $F$ are projective, free in the PID-case.
A submodule $N$ of the finitely generated projective module $M$ is a direct summand if and
only if $M/N$ is torsion free.
The localizations $M_P$ of the finitely generated projective module
$M$ with respect to the
maximal ideals $P$ of $R$ are  free $R_P$-modules of rank $r$
independent of $P$, the number $r$ (possibly $\infty$) is  called
the rank of $M$. If $M$ is a finitely 
generated projective module over $R$ there are linearly independent vectors
$x_1,\ldots,x_r\in M$ and a fractional ideal $\mfa$ of $R$ such that
$M=Rx_1\oplus \dots Rx_{r-1}\oplus \mfa x_r$, where $r$ is the rank of
$M$. The class of $\mfa$ modulo principal fractional ideals is uniquely
determined and is called the Steinitz class of $M$. If one has
$M=\mfa_1y_1\oplus \dots \oplus \mfa_r y_r$ with linearly independent
vectors $y_i \in M$ and fractional ideals $\mfa_i$ of $R$ the Steinitz
class of $M$ is equal to the class modulo principal fractional ideals of
$\prod_{i=1}^r \mfa_i$. Two finitely
generated projective modules over $R$ are isomorphic as $R$-modules if
and only if they have the same rank and the same Steinitz class.
If $V$ is a finite dimensional vector space over $F$ the $R$-lattices
in $V$ are precisely the finitely generated projective $R$-modules
contained in $V$, such a module is a lattice on $V$ if and only if its
rank equals the dimension of $V$.

Let $\Lambda_1,\Lambda_2$ be lattices on the finite dimensional vector
space $V$ over $F$. Then there exist a basis $(v_1,\ldots,v_n)$ of $V$
and fractional ideals $\mfa_1,\ldots,\mfa_n,\mfb_1,\ldots,\mfb_n$ with
$\mfb_1\supseteq \mfb_2\supseteq \dots\supseteq \mfb_n$ and 
\begin{equation*}
 \Lambda_1=\bigoplus_{i=1}^n \mfa_i v_i,\quad
 \Lambda_2=\bigoplus_{i=1}^n \mfb_i \mfa_i v_i. 
\end{equation*}
The $\mfb_i$ are unique, they are called the invariant factors of
$\Lambda_2$ in $\Lambda_1$. We will denote by ${\mathfrak
  {ind}}_R(\Lambda_1/\Lambda_2)$ the product of these invariant factors and
call it the {\em $R$-index ideal}  in analogy to the case of free
modules; if both lattices are free it is the ideal generated by
$\ind_R(\Lambda_1/\Lambda_2)=(\Lambda_1:\Lambda_2)_R$.

\begin{lemma}\label{lattices_localglobal_dedekind}
Let $\Lambda_1,\Lambda_2$ be lattices on the finite dimensional vector
space $V$ over $F$. 

For a prime ideal ${\mathfrak p}$ of $R$ let $R_{\mfp}:=\{\frac{a}{b} \in F\mid
  a,b \in R, b \not \in \mfp\}\subseteq F$ denote  the localization of $R$ at
the prime ideal $\mfp \subseteq R$ and denote by $(\Lambda_i)_{\mfp}$ the
$R_{\mfp}$ module generated by $\Lambda_i$ in $V$ (called the
localization of $\Lambda_i$ at $\mfp$).     

\begin{enumerate}
\item One has $(\Lambda_1)_{\mfp}=(\Lambda_2)_{\mfp}$ for almost all
prime ideals $\mfp$ of $R$ and $\Lambda_1=\Lambda_2$ if and only if
$(\Lambda_1)_{\mfp}=(\Lambda_2)_{\mfp}$ holds for all
prime ideals $\mfp$ of $R$. 
\item
If conversely $\Lambda$ is an $R$-lattice on $V$ and one is given
finitely many prime ideals $\mfp_i\subseteq R$ ($1 \le i \le n)$ and
$R_{\mfp_i}$-lattices $L_{i}$ on $V$ for these $\mfp_i$, there is a unique
$R$-lattice $\Lambda'$ on $V$ such that $\Lambda'_{\mfp_i}=L_i$ for
$1\le i\le n$ and $\Lambda'_{\mfp}=\Lambda_{\mfp}$ for all prime ideals
$\mfp$ different from the $(\mfp_i)$. 
\end{enumerate}
\end{lemma}
\begin{proof}
\begin{enumerate}
\item  By the general theory of lattices there are $a,b \in R\setminus\{0\}$
 such that one has $a\Lambda_2 \subseteq \Lambda_1 \subseteq b
 \Lambda_2$. For all maximal ideals $\mfp$ of $R$  with $ab \not\in \mfp$ (and hence for
 almost all $\mfp$) one has then
 $(\Lambda_1)_{\mfp}=(\Lambda_2)_{\mfp}$. The second part of the
 assertion is clear.
\item
Let the $\mfp_i, L_i$
be given, let 
$(v_1,\ldots,v_m)$ be a basis of $\Lambda$.
Dividing $\Lambda$ by a suitable element of $R$ (divisible by high
enough powers of the $\mfp_i$) if necessary  we can assume without loss of generality
$L_i\subseteq \Lambda_{\mfp_i}$ for $1 \le i \le n$. 
For $1\le i \le n$ one has
$c_{jk}^{(i)}\in R_{\mfp_i}$ such that the  $w_k^{(i)}:=\sum_{j=1}^mc_{jk}^{(i)}v_j$
for $1\le k \le m$ form a basis of the $R_{\mfp_i}$-module $L_i$, and
the $c_{jk}^{(i)}$ can even be chosen to be in $R$ since their
denominators are units in $R_{\mfp_i}$.

By the chinese remainder theorem we can then find $c_{jk}\in R$ which
are congruent to the $c_{jk}^{(i)}$ modulo $\mfp_i^{r}$ for all $i$ and
for an integer $r$ which is large enough to imply $w_k^{(i)}-w_k\in
\mfp_iL_i$ for $1\le i \le n$, where we set   $w_k:=\sum_{j=1}^mc_{jk}v_j$
for $1\le k \le m$. The matrix over $R_{\mfp_i}$ expressing the $w_k$ as linear
combinations of the $w_l^{(i)}$ has therefore 
determinant in $R_{\mfp_i}^\times$, i.e., the $w_k$ form a basis of the $R_{\mfp_i}$-module
$L_i$ for $1 \le i \le r$. Let $M\subseteq \Lambda$ be the $R$-lattice
generated by the $w_k$.  We can find an ideal $\mfc \subseteq R$ which
is a product of powers of the  $\mfp_i$ such that
$\mfc\Lambda_{\mfp_i}\subseteq L_i=M_{\mfp_i}$ for $1 \le i \le n$ and have
$\Lambda_{\mfp}=\mfc\Lambda_{\mfp}$  for all $\mfp$ different from all 
the $\mfp_i$. The Lattice $\Lambda':=\mfc\Lambda+M$ is then as desired. 
\end{enumerate}
\end{proof}
\begin{remark}
Instead of the localizations $R_\mfp,\Lambda_\mfp$ we could also use the $v$-adic
completions $R_v,\Lambda_v$ for the non archimedean valuations
$v=v_\mfp$ associated to the maximal ideals $\mfp$ of $R$, the chinese
remainder theorem is then replaced by the strong approximation theorem
in the proof.   
\end{remark}
\begin{definition}
  Let $V$ be a  finite dimensional vector space over $F$ with non degenerate
  symmetric bilinear form $b$, let $\Lambda$ be an $R$-lattice on
  $V$. The fractional $R$-ideal generated by 
  the determinants of Gram matrices with respect to $b$ of free sublattices  of $\Lambda$ 
  of full rank is called the volume ideal ${\mathfrak v_b\Lambda}={\mathfrak v\Lambda}$ of $\Lambda$.
\end{definition}
\begin{remark}
 For any maximal ideal $\mfp$ of $R$ the $\mfp$-part of the volume
  ideal ${\mathfrak v}\Lambda$ is given as the power of $\mfp$ in the
 determinant of the free local lattice $\Lambda_\mfp$.  
\end{remark}

\begin{lemma}
  Let $\Lambda$ be an $R$- lattice on the finite dimensional 
  vector space $V$ over $F$. A submodule $M$ of $\Lambda$ which generates over $F$
  the subspace $W=FM$ of $V$ is a primitive submodule of $\Lambda$ if
  and only if one has 
$M=W\cap\Lambda$. 

In particular, a vector $x \in \Lambda$ can be
extended to an $R$-basis of $\Lambda$ if and only if it is a primitive
vector in the sense that $cx\in \Lambda$ for some $c \in F$ implies $c
\in R$.
\end{lemma}
\begin{proof}
The condition of the lemma is equivalent to $\Lambda/M$ being torsion free.
  
\end{proof}
\begin{theorem}\label{duals_and_indices}
  Let $(V,Q)$ be a finite dimensional regular quadratic space over $F$
  with an orthogonal decomposition $V=U_1\perp U_2$, let $\Lambda$ be
  a lattice of full rank on $V$ with $b(\Lambda,\Lambda)\subseteq R$
  (equivalently, $\Lambda\subseteq \Lambda^\#$).
Let $L_i=\Lambda \cap U_i (i=1,2)$ and denote by $\pi_i$ the
orthogonal projection $\pi_i:V\to U_i$.
Then one has
\begin{enumerate}
\item
  \begin{equation*}
    \Lambda^\#/(L_1\perp U_2\cap \Lambda^\#)\cong L_1^\#/L_1.
  \end{equation*}
\item $vol(\pi_2(\Lambda))\cdot \vol(L_1)=\vol(\Lambda)$, \\
$  \vol(L_1)\vol(L_2)=\vol(\Lambda) (\pi_2(\Lambda):L_2)_R^2=\vol(\Lambda) (\pi_1(\Lambda):L_1)_R^2$.
\item $\vol(\Lambda) \vol(L_1)=\vol(L_2)(L_1^\#:\pi_1(\Lambda))_R^2$.\\
In particular,  if $U_1=Fx$ has dimension $1$ with $L_1=Rx$, one has with $b(x,\Lambda)=:aR$ 
  \begin{eqnarray*}
\vol(\Lambda) &=&b(x,x)\vol(\pi_2(\Lambda))\\
 b(x,x)\vol(\Lambda)&=&a^2\vol(L_2). 
\end{eqnarray*}
\item If $\Lambda$ is unimodular, one has 
  \begin{eqnarray*}
    L_i^\#&=& \pi_i(\Lambda) \quad (i=1,2)\\
L_1^\#/L_1 &\cong & L_2^\#/L_2 \text{ as $R$-modules }\\
\vol(L_1)&=&\vol(L_2).
  \end{eqnarray*}
\end{enumerate}
\end{theorem}
\begin{proof}
Since all parts of the assertion are true if and only if they are true
for all localisations it is sufficient to prove them for principal
ideal domains, replacing $\vol$ by $\det$.  
  \begin{enumerate}
\item  By assertion a) of Lemma \ref{ortho_decompositions} we have
  $U_i\cap \Lambda^\#=(\pi_i(\Lambda))^\#$, and from
  $(\Lambda^\#)^\#=\Lambda$ one sees that  $L_i=U_i\cap
  \Lambda=(\pi_i(\Lambda^\#))^\#$ holds. Dualizing both sides of this
  equality we obtain   
$L_i^\#=\pi_i(\Lambda^\#)$
  for $i=1,2$. We obtain from this a surjective $R$-linear map
  $\Lambda^\# \to L_1^\#/L_1$, whose kernel is $\{x \in
  \Lambda^\#\mid \pi_1(x) \in L_1\}$. The latter set is nothing but
  $L_1 \perp (U_2\cap\Lambda^\#)$ as asserted.
\item  Let $(v_1,\ldots,v_n)$ be a basis of $\Lambda$ for which
  $(v_1,\ldots,v_r)$ is a basis of $L_1$. The linear map $\phi:V \to
  V$ with $\phi(v_i)=v_i$ for $1 \le i \le r$ and
  $\phi(v_i)=\pi_2(v_i)$ for $i>r$ has determinant $1$ and image $L_1
  \perp \pi_2(\Lambda)$,
which implies the first part of the assertion. The second part follows
from $\det(L_2)=\det(\pi_2(\Lambda))(\pi_2(\Lambda):L_2)_R^2$ and
the first part.
\item From b) we obtain 
\begin{eqnarray*}
\det(\Lambda)\det(L_1)&=&\det(L_1)^2\det(L_2)(\pi_1(\Lambda):L_1)_R^{-2}\\
&=&\det(L_2)(d_{L_1^\#/L_1})^2 (\pi_1(\Lambda):L_1)_R^{-2}\\
&=&\det(L_2)(L_1^\#:\pi_1(\Lambda))_R^2.
\end{eqnarray*}
For $\dim(U_1)=1$ we have $(\pi_1(\Lambda):L_1)_R=b(x,x)a^{-1}R^\times$ (with
$b(x,\Lambda)=aR$) and obtain the second formula. 
The first formula for this case follows from the
first formula in b).
\item The assertions are a), b) of Lemma \ref{ortho_decompositions}
  and  c) with $\Lambda^\#=\Lambda$ inserted. 
  \end{enumerate}
\end{proof}

\begin{definition}
  Let $(V,Q)$ be a finite dimensional regular quadratic space over
  $F$.

An $R$-lattice $\Lambda$ of full rank on $V$ is called maximal if
$Q(\Lambda)\subseteq R$ holds and if $\Lambda$ is maximal with respect
to this property.

For a fractional $R$-ideal $I\subseteq F$ an $I$-maximal lattice is
defined analogously. 
\end{definition}
\begin{lemma}\label{maximal_lattice}
  Let $I$ be a fractional $R$-ideal in $F$.
  Every $R$-lattice $\Lambda$ with $Q(\Lambda)\subseteq I$ on the
  finite dimensional regular quadratic space 
  $(V,Q)$ over $F$ is contained in an $I$- maximal lattice.
\end{lemma}
\begin{proof} If $\Lambda$ is not $I$-maximal it has a strict overlattice
  $\Lambda_1$ with $Q(\Lambda_1)\subseteq I$, and proceeding in the same
  way as long as possible we obtain an ascending  chain of lattices
  $L_i$ with $Q(\Lambda_i)\subseteq I$.
 
  From $Q(\Lambda)\subseteq I$ we see that $\Lambda \subseteq
  I\Lambda^\#$, and in the same way all the $\Lambda_i$ above
  are contained in
  $I\Lambda^\#$.
Since $I\Lambda^\#$ is noetherian any ascending
  chain of sublattices of it becomes stationary after finitely many
  steps, and we obtain the assertion.
\end{proof}
\begin{theorem}\label{hyperbolic_maximal}
  Let $(V,Q)$ be an $m$-dimensional isotropic regular quadratic space over
  $F$ and $\Lambda$ a maximal lattice on $V$. Then $\Lambda=H\perp
  \Lambda_1$, where $H$ is a regular hyperbolic $R$-lattice of rank
  $2r\le m$ and $\Lambda_1$ is a maximal lattice on an anisotropic
  subspace of $V$.
\end{theorem}
\begin{proof}
  Let $x$ be an isotropic vector of $V$,
we have ${\mathfrak a}x^{-1} =Fx\cap
\Lambda^\#$ for some ideal ${\mathfrak
  a}\subseteq R$.
  The lattice $L:=\Lambda+{\mathfrak a}^{-1}x$
  satisfies then $Q(L)\subseteq R$, by maximality of $\Lambda$ we have
  $L=\Lambda$ so that ${\mathfrak a}^{-1}x \subseteq \Lambda$. Since
  we have ${\mathfrak a}^{-1}x=Fx \cap\Lambda^\#$ we see that
  ${\mathfrak a}^{-1}x$ is a regularly embedded totally
  isotropic submodule of $\Lambda$, so by Theorem
  \ref{isotropic_hyperbolic} there is a regular hyperbolic
  sublattice of rank $2$ of $\Lambda$ containing ${\mathfrak a}^{-1}x$ which can be split off
  orthogonally in $\Lambda$. The assertion follows by induction on $m$. 
\end{proof}
\begin{remark}
An analogous result is valid for an $I$-maximal lattice, where
$I$ is a fractional $R$-ideal. The hyperbolic module $H$ above is then
replaced by a module isometric to $M\oplus IM^\ast$.
\end{remark}

\chapter{Reduction Theory of Positive Definite Quadratic Forms}
In this chapter and the next one we deal with geometric and
computational properties of spaces of
real valued quadratic forms over the
rational integers. Reduction theory is concerned with the task of selecting
special forms in an equivalence class or equivalently special bases of
a given quadratic module. The computational aim is to obtain Gram matrices
with small entries in order to facilitate computations, the geometric
aim is the study of the geometric properties of the action of the
group $GL_n(\Z)$ on the space, in particular the construction of
fundamental domains for this action. We deal with 
positive definite forms in this chapter and with indefinite forms in
the next chapter.
Our presentation in this chapter follows for the most part \cite{cassels,vdw}. The
general study of the properties of quadratic forms over number rings
will be continued after this interlude.
\section{Minkowski reduced forms and successive minima}
A lattice   in $\R^n$ of rank  $r$   is a $\Z$-module  $L=\Z f_1+\ldots +\Z
f_r\subseteq \R^n$ , where the $f_i$   are $\R$-linearly independent
vectors in $\R^n$. Let $Q$ be a positive definite quadratic form on
$\R^n$ with associated bilinear form $b$, let $B=b/2$ as usual.

\begin{definition}\label{def_minkowski_reduziert} Let $L\subseteq
  \R^n$ be a  $\Z$-lattice of full rank  $n$ with basis $f_1,\ldots, f_n$.\\
Put put $L_0: = \{0\}$, for $1 \leq i \leq n$ put $L_i: = L_i(\{f_j\}): = \Z f_1+
\cdots + \Z f_i$. 

The basis $\cB=(f_1,\ldots,f_n)$ of $L$  is called
Minkowski reduced with respect to $Q$ (or $B$) if one has for  $1 \leq i \leq n$: 

For all $x \in L$ for which 
 $L'_i(x) := L_{i-1} + \Z x$
is primitive in $L$  one has 
 $Q(x) \geq Q(f_i)$.

Its Gram matrix $M_{\cB}(B)$ is then called a Minkowski reduced matrix
and the quadratic polynomial $P(X_1,\dots,X_n)=\sum_{i,j}m_{ij}X_iX_j$
with $M_{\cB}(B)=(m_{ij})$ a Minkowski reduced quadratic form.
\end{definition}
\begin{remark}
  \begin{enumerate}
\item An equivalent formulation is: A basis  $f_1,\ldots,f_n $ of $L$
is Minkowski reduced if $Q(f_i)$ is minimal among the $Q(x)$ for
which $x+L_{i-1}$ is primitive in $L/L_{i-1}$ for all $i$. 
\item The positive definite matrix $A\in M_n^{\sym}(\R)$ is Minkowski
  reduced if and only if the standard basis $(\e_1,\dots,\e_n)$ of
  $\Z^n$ is a reduced basis of the lattice $\Z^n$ with the quadratic
  form $Q_A$ given by $Q_A(\x)={}^t\x A \x$.
  \end{enumerate}
\end{remark}
\begin{lemma}
  \begin{enumerate}
  \item The positive definite matrix $A\in M_n^{\sym}(\R)$ is Minkowski
reduced if and only  for $1\le i \le n$ one has $a_{ii}\le {}^t\x A \x$ for all $\x\in
\Z^n$ with $\gcd(x_i,\ldots,x_n)=1$. 
\item   A Minkowski reduced matrix $A=(a_{ij})$ satisfies
\begin{align*}
  0 < a_{11}  \leq  a_{22} \leq \cdots \leq a_{nn},\\
  |2a_{ij}|  \leq  a_{ii} \quad \text{for all }1 \leq i < j \leq n.
 \end{align*}
 \end{enumerate}
\end{lemma}
\begin{proof}
The first assertion follows from the fact that 
$\Z\e_1+ \dots +\Z\e_{i-1}+\Z\x$ is a primitive sublattice of $\Z^n$
if and only if one has $\gcd(x_i,\ldots,x_n)=1$. The second assertion
follows from taking $\x=\e_j $ respectively $\x=\e_i+\e_j$ for $i<j$.
\end{proof}
For the rest of this chapter we consider the standard scalar product
$\langle\quad ,\quad \rangle=:b(\quad,\quad)$ on the space $\R^n$.
As usual we put $Q(\x)=b(x,x)/2=B(x,x)$.

Instead of   ${\det}_b(L)$ as in earlier chapters  we will usually
consider ${\det}_B(L)=2^{-n}{\det}_b(L)$.

Restricting attention to $Q$ instead of an arbitrary positive definite
quadratic form on $\R^n$ doesn't lose generality
since all positive definite quadratic forms on $\R^n$ are equivalent
over $\R$ to $Q$.
The lattice $L$ is called integral if all the $b(f_i,f_j)=\langle
f_i,f_j\rangle$ are in $\Z$. If one has in addition $Q(L)\subseteq \Z$
the lattice is called even integral.

Since in this case the classes modulo squares of units in $\Z$ consist
of one element only we may treat ${\det}_B(L)$ as a real number
instead of a square class when this is convenient.

\begin{theorem} Every lattice has a Minkowski reduced
  basis. Equivalently, for every Minkowski reduced positive definite
  symmetric matrix $A$ there exists an integrally equivalent Minkowski
  reduced matrix. 
\end{theorem}
\begin{proof}
Obvious.
\end{proof}
\begin{remark}
We will see in Section \ref{geometry_reduced} that for fixed dimension
$n$ already  a
finite set of the reduction conditions $a_{ii}\le {}^t{\bf x}A{\bf x}$ suffices
to characterize Minkowski reduced matrices.
Such finite sets of reduction conditions are explicitly known for $n\le 7$
(see \cite{Tammela} for the case $n=7$). 
Nevertheless, for  $\rk(L)=n>4$ it is in general computationally difficult to determine
a Minkowski reduced basis of a given lattice effectively.
\end{remark}
For the cases
$n\le 4$ Minkowski has shown in \cite{minkowski_quaternaire} that
already the vectors ${\bf x}$ with components $x_i\in \{0,\pm 1\}$
suffice, see also Lemma 12.1.2 of \cite{cassels}, whose proof we follow:

\begin{example}
For  $2 \le n \le 4$ a matrix  $Y \in {\mathcal P}_n$ is Minkowski reduced
if and only if 
 \begin{enumerate}
 \item $0 \leq y_{11} \leq y_{22} \leq \cdots \leq y_{nn}$
 \item  $P_Y({\bf x}):={}^t{\bf x} Y{\bf x}$ satisfies $P_Y({\bf x}) \geq y_{jj}$ for all ${\bf x} \in \Z^n$
with  $x_i \in \{0,\pm 1\}$, $x_j = 1$, $x_i = 0$ for $i > j$.
 \end{enumerate}
\end{example}
\begin{proof}
It is obvious that  a reduced matrix satisfies the conditions a), b).
Assume conversely that $Y$ satisfies a) and b). We have to show that
any $\x \in \Z^n$
with $\gcd(x_r,\dots,x_n)=1$ for some $1\le r \le n$ satisfies
$P_Y(\x) \ge y_{rr}$, in fact we claim that (under our assumption
$n \le 4$) this even holds for all $\x$ with $x_r \ne 0$. 
Since deleting the $j$-th row and column from $Y$ and the $j$th
coordinate $x_j$ from $\x$ for all $j$ with
$x_j=0$ and multiplying the $j$-th row and column of $Y$ and the entry
$x_j$ of $\x$ by $-1$ for
all $j$ with $x_j<0$ doesn't change the validity of our claim we may
assume that $x_i>0$ for $1\le i \le n$ and have to show $P_Y(\x)\ge y_{nn}$.
We proceed by induction on $S:=\sum_{i=1}^nx_i$, the case $S=1$ being
a trivial consequence of b) and $n=1$ being trivial as well. Let now $S>1$ and set $c:=\min(x_i)>0$, let
$k$ be maximal with $x_k=c$. We set $z_i=x_i-c$ for $i\ ne k$ and
$z_k=x_k$ and let $\z=(z_1,\ldots,z_n)$. We have $z_n>0$ and
$\sum_{i=1}^nz_i<\sum_{i=1}^nx_i$, so by the inductive assumption we
have $P_Y(\z)\ge y_{nn}$. We are done if we can show $P_Y(\x)\ge P_Y(\z)$.
Indeed we have, using $y_{ij}=y_{ji}$ for all $i,j$, $x_i-z_i=c$ for $i\ne k$, and $x_k-z_k=0$:
\begin{equation*}\begin{split}
P_Y(\x)-P_Y(\z)&=\sum_jy_{kj}(cx_j-cz_j)+\sum_{i\ne
                                          k}y_{ik}(x_ic-z_ic)\\  &\quad+\sum_{i\ne
                                          k}\sum_{j\ne
                                          k}y_{ij}(x_ix_j-z_iz_j)\\
                                        &=2\sum_{j\ne
                                          k}c^2y_{kj}+\sum_{i\ne
                                          k}\sum_{j\ne
                                          k}y_{ij}(c^2+cz_i+cz_j)\\
                                        &=c^2(P_Y(1,\dots,1)-y_{kk})+2c\sum_{i\ne
                                          k}z_i\sum_{j\ne k}y_{ij}\\
                                        &\ge 2c\sum_{i\ne
                                          k}z_i\bigl((2-\frac{n}{2})y_{ii}+\frac{1}{2}\sum_{
                                          j\ne
                                          k, j\ne
                                          i}(y_{ii}+2y_{ij})\bigr)\\
                                            &\ge 0.
\end{split}
\end{equation*}
               
 \end{proof}

\begin{definition}
With notations a before
 $$\mu_1 := \min\{Q(x)~|~x\in L \setminus \{0\}\}$$
is called the minimum of the lattice $L$ and a vector $X$ satisfying
$Q(x)=\mu_1$ is called a minimum vector.\\
For $2 \leq j \leq n$ the $j$-th successive minimum $\mu_j$ of $L$ is
the smallest $a\in \R$ for which there are linearly independent
vectors  $x_1,\ldots,x_j \in L$  satisfying  $Q(x_1) \leq \cdots \leq
Q(x_j) \leq a$.

for $1\le j \le n$ linearly independent vectors  $x_1,\ldots,x_j \in L$ satisfying 
 $Q(x_i) = \mu_i$ for $1\le i \le j$ 
are called successive minimum vectors. 
\end{definition}
\begin{remark}
  \begin{enumerate}
  \item For each $C>0$ the set  $\{x \in \R^n~|~Q(x) \leq C\}$  is
    compact and can hence contain only finitely many points of the
    discrete set $L$. The minimum $\mu_1$ and the successive minima
    can therefore indeed be defined as minima instead of infima.
\item Let $x_1,\ldots,x_j$ be successive minimum vectors of $L$ and
  let $y_1,\ldots,y_{j+1}$ be linearly independent vectors with
  $Q(y_i)\le \mu_{j+1}$ for $1\le i\le j+1$.
Then at least one of the $y_i$ is not in the linear span of the $x_i$,
hence
\begin{equation*}
  \min\{Q(x)\mid \{x_1,\ldots,x_j,x\} \text{ is linearly
    independent}\}\le \mu_{j+1}
\end{equation*} 
holds. Since the reverse inequality is obvious, we can
also define the successive minima recursively, defining $\mu_{j+1}$ by
fixing already given successive minimum
vectors $x_1,\ldots,x_j$ and requiring the equation above to hold for
all $x$ which are linearly independent from the given vectors.
\item The successive minima too are in general difficult to determine effectively.
  \end{enumerate}
\end{remark}
\begin{theorem}\label{hermites_theorem_definite}
There exists a constant (called {\em Hermite's constant})
$\gamma_n$ depending only on $n$ so that each positive definite
$\Z$-lattice of rank $n$ with associated symmetric bilinear forms
$b,B=b/2$ satisfies 
 $$\mu_1 \cdots \mu_n \leq \gamma_n^n {\det}_B(\Lambda).$$
In particular one has 
$$\mu_1 \le \gamma_n ({\det}_B(\Lambda)^{1/n}.$$
\end{theorem}
\begin{proof} 
We show first by induction on $n$ that one can find a constant
$\gamma_n$ such that $\mu_1(\Lambda) \leq \gamma_n({\det}_B(\Lambda)^{1/n}$ 
is true for all positive definite lattices $\Lambda$ of rank $n$, the
case $n=1$ being trivial with $\gamma_1=1$.

Let then $n>1$, let $e_1 \in \Lambda$ with $Q(e_1) = \mu_1$, let $L'$
be the orthogonal projection of $\Lambda$ onto    $e_1^{\perp}$.
By b) of Theorem \ref{duals_and_indices} we have
${\det}_B(L')=\frac{{\det}_B(\Lambda)}{\mu_1}$, notice that this part
of the proof of that theorem does not need the bilinear form to take
values in $\Q$.

By the inductive assumption we find $x'\in L'$ satisfying 
 $$Q(x') \leq \gamma_{n-1} (\frac{{\det}_BL}{\mu_1})^{\frac{1}{n-1}}.$$
e can then find $\alpha \in \R$ with  $0 \leq |\alpha| \leq \frac{1}{2}$ and
$x := \alpha e_1+x' \in \Lambda$, hence
 $$\mu_1 \leq Q(x) \leq \frac{\mu_1}{4} + \gamma_{n-1}(\frac{{\det}_B(\Lambda)}{\mu_1})
  ^{\frac{1}{n-1}}.$$
From this we obtain
 $$\begin{array}{rll}
  \frac{3\mu_1}{4} & \leq & \gamma_{n-1}(\frac{{\det}_BL}{\mu_1})^{\frac{1}{n-1}}
   \vspace{0.2cm}\\
  (\frac{3}{4})^{n-1} \mu_1^n & \leq & \gamma_{n-1}^{n-1} {\det}_B(\Lambda)
   \vspace{0.2cm}\\
  \mu_1^n & \leq & (\frac{4}{3} \gamma_{n-1})^{n-1}{\det}_B(\Lambda)
 \end{array}$$
From $\gamma_1 = 1$ we obtain recursively that 
 $$\gamma_n: = (\frac{4}{3})^{\frac{n-1}{2}}$$
is as desired:
Inserting  $\gamma_{n-1} = (\frac{4}{3})^{\frac{n-2}{2}}$ in the
equation above we get 
 $$\begin{array}{rll}
  \mu_1^n & \leq & (\frac{4}{3})^{n-1}(\frac{4}{3})^{\frac{(n-2)(n-1)}{2}}{\det}_B(\Lambda)
   \vspace{0.2cm}\\
  & = & (\frac{4}{3})^{(n-1)(1+\frac{n-2}{2})}{\det}_B(\Lambda)
   \vspace{0.2cm}\\
  & = & (\frac{4}{3})^{\frac{(n-1)}{2}n}{\det}_B(\Lambda)
   \end{array}$$
and hence
  $$\mu_1 \leq  (\frac{4}{3})^{\frac{n-1}{2}}({\det}_BL)^{\frac{1}{n}}.$$
To show the first inequality of the theorem for any $\gamma_n$
satisfying the inequality for $\mu_1$ we let  $f_1,\ldots,f_n\in \Lambda$ 
be successive minimum vectors and denote by $\{f'_i\}$ the orthogonal
basis of $V=\R^n$ obtained from the $f_i$ by the Gram-Schmidt
orthogonalization procedure, i.\ e., 
 $$\begin{array}{lll}
  f'_1 & = & f_1 \vspace{0.2cm}\\
  f'_2 & = & f_2 - {\displaystyle \frac{B(f_1,f_2)}{B(f_1,f_1)}} f_1
 \end{array}$$
und recursively
 $$f'_j = f_j-\sum_{i=1}^{j-1} \frac{B(f_j,f'_i)}{B(f'_i,f'_i)} f'_i.$$
 In particular one has
 $$\R f'_1+\cdots + \R f'_i = \R f_1+ \cdots + \R f_i,$$
and  the Gram matrix of $B$ with respect to $f'_1,\ldots, f'_n$ has
the same determinant as that taken with respect to
$f_1,\ldots f_n$.\\
We define then a quadratic form $Q'$ on $V$ with associated symmetric
bilinear forms $b', B'=b'/2$ by requiring the 
$f_i'$ to be pairwise orthogonal with respect to $B'$ and setting 
 $$Q'(f'_i) = \frac{Q(f'_i)}{\mu_i}.$$
The determinants of $\bigoplus_{i=1}^n\Z f_i$ with respect to $B$ and
to 
$B'$ differ by the same factor as  ${\det}_B(\Lambda)$ and
${\det}_{B'}(\Lambda)$, and we see
 $${\det}_{B'}(\Lambda) = \frac{{\det}_B(\Lambda)}{\mu_1\cdots\mu_n}.$$
Let now  $y \in \Lambda$, $y \not= 0$ and  $j$ minimal, such that 
$y = \sum_{i=1}^{j} y_if'_i \in \R f'_1+\cdots + \R f'_j$ is true.
Then
 $$\begin{array}{lll}
  Q'(y) & = & \sum_{i=1}^{j} y_i^2
   \frac{Q(f'_i)}{\mu_i}
   \vspace{0.2cm}\\
  & \geq & \mu_j^{-1} \sum_{i=1}^{j} y_i^2 Q(f'_i)
   \vspace{0.2cm}\\
  & = & \mu_j^{-1} Q(y) \geq 1,
 \end{array}$$
since $y$ is linearly independent of the  $f_1,\ldots,f_{j-1}$. 

This implies
 $$1 \leq {\textstyle \min_{Q'}(L)} \leq \gamma_n
 \left(\frac{{\det}_B(\Lambda)}{\mu_1\cdots \mu_n}\right)^
  {\frac{1}{n}},$$
and finally  $\mu_1 \cdots \mu_n \leq \gamma_n^n {\det}_B(\Lambda)$. 
\end{proof}

\begin{lemma}\label{gramschmidt_koordinaten}  Let $(v_1,\ldots,v_n)$
  be a basis of the $\Z$-lattice $\Lambda\subseteq \R^n$ with 
  Gram-Schmidt orthogonalization $(v'_1,\ldots,v'_n)$, let $1\le j \le
  n$ and  
  $L_j=\oplus_{i=1}^j\Z v_i$, let  $z \in \R L_j$. 

Then there is 
 $y \in L_j$ such that one has 
$z-y = \sum_{i=1}^{j} \alpha_iv'_i$ with  $|\alpha_i| 
\leq \frac{1}{2}$ (for $1 \leq i \leq j$).
\end{lemma}
\begin{proof}
Write  $z = \sum_{i=1}^{j} \beta_iv'_i$
and choose $\alpha'_j \in \Z$
with  $|\beta_j-\alpha'_j| \leq \frac{1}{2}$.
We have then $z -\alpha_jv_j = (\beta_j-\alpha_j)v'_j +z'$ with 
$z' \in \R L_{j-1}$.

We can then find  $y' \in L_{j-1}$ with  $z'-y' = \sum_{i=1}^{j-1} \alpha_iv'_i$,
$|\alpha_i| \leq \frac{1}{2}$.
Setting  $y = y'+\alpha_jv_j$ one sees
$z-y=(\beta_j-\alpha'_j)v'_j+z'-y'$, and with 
$\alpha_j:=\beta_j-\alpha'_j$ the vector  $z-y$ is as desired.
\end{proof}

\begin{theorem}\label{comparison_reduction_constants} Let ${\mathcal B} = \{v_1,\ldots,v_n\}$ be a Minkowski reduced
basis of $\Lambda$, let $\{v'_1,\ldots,v'_n\}$ denote the 
Gram-Schmidt orthogonalization of ${\mathcal B}$ and $\mu_i$ the successive
minima of $\Lambda$.  Then one has
 \begin{enumerate}
 \item $Q(v'_j) \leq Q(v_j)$
\item $\mu_j \leq Q(v_j)$ 
\item  There exist constants   $c_1(j)$ independent of $\Lambda$ and $Q$ 
such that  $Q(v_j) \leq c_1(j)\mu_j$ for  $1\le j\le n$.
 \item There exist constants  $c_2(n)$ independent  of $\Lambda$ und $Q$ 
   with  $\mu_j \leq c_2(n) Q(v'_j)$ for $1\le j\le n$.   
\end{enumerate}
\end{theorem}
\begin{proof}
 \begin{enumerate}
 \item $v'_j = v_j-p_{j-1}(v_j)$, where $p_i$ is the orthogonal
   projection onto the space spanned by $v'_1,\ldots,v'_i$i.\ e.,
 $$\begin{array}{lll}
  Q(v_j) & = & Q(v'_j)+Q(p_{j-1}(v_j))\vspace{0.2cm}\\
  & \geq & Q(v'_j)
 \end{array}.$$
\item Obvious. 
\item We use induction on $j$, the start $j=1$ being trivial.
 Let linearly independent vectors $w_1,\ldots,w_n$ with $\mu_i =
 Q(w_i)$ be given. 
Let $1 \leq k \leq j$ be such that 
$\{v_1,\ldots,v_{j-1},w_k\}$ are linearly independent and 
choose $v_j^{\ast} \in \Lambda$ such that
$v_1,\ldots,v_{j-1},v_j^{\ast}$ form a 
basis of $(\R L_{j-1}+\R w_k) \cap \Lambda = (L_{j-1}+\R w_k)\cap \Lambda$.
That is possible since 
$L_{j-1}$ is a primitive sublattice of this lattice.

In particular, $L_{j-1}+ \Z v_j^{\ast}$ is then a primitive
sublattice of  $\Lambda$, and we have
$Q(v_j) \leq Q(v_j^{\ast})$.

\parindent=0pt
We change $v_j^*$ by a vector  of $L_{j-1}$ or multiply it by $-1$ if
necessary to achieve 
 $$w_k = \sum_{i=1}^{j-1} t_iv_i+s v_j^{\ast},\: 0<s \in \Z,\:
  t_i \in \Z$$ with  $|t_i| \leq \frac{s}{2}$.

Using $\sqrt{Q(x+y)} \leq \sqrt{Q(x)}+\sqrt{Q(y)}$ we see then
 $$\sqrt{Q(v_j)} \leq \sqrt{Q(v_j^{\ast})} \leq \frac{1}{s} \sqrt{\mu_k}
 + \frac{1}{2} \sum_{i=1}^{j-1} \sqrt{Q(v_i)}
\leq \sqrt{\mu_j}+\frac{1}{2} \sum_{i=1}^{j-1} \sqrt{c_1(i)} \sqrt{\mu_j},$$
and notice that $\sqrt{Q(v_i)} \leq \sqrt{c_1(i)}\sqrt{\mu_i} \leq
\sqrt{c_1(i)}\sqrt{\mu_j}$) holds by the inductive assumption.

$c_1(j)^{\frac{1}{2}} = 1+\frac{1}{2} \sum_{i=1}^{j-1}
c_1(i)^{\frac{1}{2}}$ is then as desired.
 \item  We have $\prod_{i=1}^{n} Q(v'_i)={\det}_B(\Lambda)$ and 
$\prod_{i=1}^{n} \mu_i \leq \gamma_n^n {\det}_B(\Lambda)$.  This implies
 \begin{eqnarray*}
  \gamma_n^n & \geq & \prod_{i=1}^{n} \frac{\mu_i}{Q(v'_i)}\\
  & \geq & \frac{\mu_j}{Q(v'_j)} \prod_{i\not= j} \mu_i\frac{
    c_1(i)^{-1}}{\mu_i}\\
&=&  \frac{\mu_j}{Q(v'_j)}\prod_{i\not= j} 
     c_1(i)^{-1},
\end{eqnarray*}
and we have 
  $$\mu_j \leq Q(v'_j) \underbrace{\gamma_n^n \max_j
   \{\prod_{i\not=j} c_1(i)\}}_{=c_2(n)}.
  $$
\end{enumerate}
\end{proof}

\begin{corollary} Let ${\mathcal B} = \{v_1,\ldots,v_n\}$ be a Minkowski reduced
basis of the lattice $\Lambda$, denote by $\mu_j$ the successive
minima of $\Lambda$.

Then for $1\le j\le 4$ one has $Q(v_j)=\mu_j$, for $j\ge 5$ one has 
$Q(v_j)\le  \bigl(\frac{5}{4}\bigr)^{j-4} \mu_j$.

For $1\le j \le 3$ any choice of  first $j$ minimum vectors can be extended to a
Minkowski reduced basis.
\end{corollary} 
\begin{proof}
Proceeding as in the proof of the theorem we 
want to show 
that one can choose $c_1(j)=1$ for
$1\le j \le 4$ and $c_1(j)=\bigl(\frac{5}{4}\bigr)^{j-4}$ for
$j>4$. 

We obtain first 
$w_k = \sum_{i=1}^{j-1} t_iv_i+s v_j^{\ast},\: 0<s \in \Z,\:
  t_i \in \Z.$\\
If one has here $s=1$, the lattice  $L_{j-1}+\Z w_k$ is primitive
in $\Lambda$ and we have 
$Q(v_j) \leq Q(w_k) = \mu_k \leq \mu_j$. 

Otherwise we decompose 
$v_j^{\ast} = z_1 +z_2$ with  $z_1 \in \R L_{j-1}$,
$z_2 \perp L_{j-1}$  and have 
 $$w_k = \left( \sum_{i=1}^{j-1} t_iv_i + sz_1\right) 
  + sz_2.$$
One has then
 $$Q(w_k) \geq s^2Q(z_2),$$
hence
 $$Q(z_2) \leq \frac{1}{4} \mu_k \leq \frac{1}{4} \mu_j.$$

By the previous lemma we can change $v_j^{\ast}$ modulo $L_{j-1}$  in
such a way that
$z_1 = \sum_{i=1}^{j-1} \alpha_iv'_i$ with
$|\alpha_i| \leq \frac{1}{2}$ holds and see  $Q(z_1) \leq \frac{1}{4}
\sum_{i=1}^{j-1} Q(v'_i)$.\\

Putting things together we get 
 $$\begin{array}{lll}
  Q(v_j) \leq Q(v_j^{\ast}) & \leq & \frac{1}{4} \mu_j + \frac{1}{4} 
   \sum_{i=1}^{j-1} Q(v'_i)\vspace{0.2cm}\\
  & \leq & \frac{1}{4} \mu_j + \frac{1}{4} \sum_{i=1}^{j-1} Q(v_i)
   \vspace{0.2cm}\\
  & \leq & \frac{1}{4} \mu_j + \frac{1}{4} \sum_{i=1}^{j-1} c_1(i)\mu_i
   \vspace{0.2cm}\\
  & \leq & \frac{1}{4} \mu_j (1+\sum_{i=1}^{j-1} c_1(i))
 \end{array}$$
and hence 
 $$Q(v_j) \leq \mu_j \cdot
{\max\{1,\frac{1}{4}+\frac{1}{4}
 \sum_{i=1}^{j-1}c_1(i)\}}.
$$
With $c_1(1)=1$ we see in particular that  one can choose $$c_1(2) =
c_1(3) = c_1(4) = 1.$$

Using  induction on $j$ we see finally
$c_1(j) = (\frac{5}{4})^{j-4}$ for  $j \geq 4$.

For $1\le j \le 3$ the argument above shows that   $s>1$
would imply $Q(v_j)<\mu_j$, a contradiction which shows the second
part of the assertion. 
\end{proof}
\begin{example}
  \begin{enumerate}
  \item Let $\Lambda \subseteq \R^5$ be the lattice with basis\\
    $v_1=(1,0,0,0,0)$,$v_2=(0,1,0,0,0)$,$v_3=(0,0,1,0,0)$,\\
    $v_4=(0,0,0,1,0),v_5=(\frac{1}{2},\frac{1}{2},\frac{1}{2},\frac{1}{2},\frac{1}{2})$. 

Then it is easily seen that this basis is Minkowski reduced and that
all five successive minima are equal to $1$, so we have
$Q(v_5)=\frac{5}{4}>1=\mu_5$.
\item Let $\Lambda \subseteq \R^4$ be the lattice with basis
  $(1,-1,0,0),(0,1,-1,0)$,\\$(0,0,1,-1),(0,0,1,1)$ (the $D_4$-lattice in
  the root lattice terminology). The vectors
  $(1,-1,0,0),(1,1,0,0)$,
$(0,0,1,-1),(0,0,1,1)$ are successive minimum
  vectors but do not generate the lattice. 
   \end{enumerate}
\end{example}
\section{Siegel domains}
We recall from linear algebra:
\begin{lemma}
Let $\cB=\{v_1,\ldots,v_n\}$ be a basis of the vector space $V$ and
$\{v'_1,\ldots,v'_n\}$ the 
Gram-Schmidt orthogonalization. Then the matrix $C=(c_{ij})$ given by 
$v_j = \sum_{i=1}^{n} c_{ij} v'_j$ is upper triangular with diagonal
entries $1$  and for the Gram matrix $M_{\cB}$ of the symmetric
  bilinear form $B$ with respect  to $\cB$ one has 
$$M_{\mathcal B} = C^t 
 \begin{pmatrix}
  h_1 & & & & \\
  & \cdot & & &\\
  & & \cdot & &  \\
  & & & \cdot&\\
   & & & & h_n
 \end{pmatrix} C $$
with $h_i = Q(v'_i)$ and 
 $$Q(\sum_{i=1}^{n} x_iv_i) = \frac{1}{2} \sum_{i=1}^{n} h_i \xi_i^2,
  \quad \xi_i = \sum_{j=1}^{n} c_{ij}x_j.$$
The matrix $C$ and the  $h_i$ with these properties are uniquely determined.
\end{lemma}

\begin{definition} Let $\delta >1$, $\epsilon >0$ be given.
The Siegel domain ${\mathcal S}_n(\delta,\epsilon)={\mathcal
  S}(\delta,\epsilon)$ is the set of all positive definite symmetric
$(n\times n)$-matrices  $M$ for which the decomposition 
$$M= C^t 
 \begin{pmatrix}
  h_1 & & & & \\
  & \cdot & & &\\
  & & \cdot & &  \\
  & & & \cdot&\\
   & & & & h_n
 \end{pmatrix} C $$
satisfies
 $$\begin{array}{rll}
  0 < h_j & \leq & \delta h_{j+1},\\
  |c_{ij}| & \leq & \epsilon \quad \forall i,j.
 \end{array}$$
\end{definition}
\begin{theorem} \label{reduziert_siegelbereich} There exist $\delta = \delta(n)$, $\epsilon = \epsilon(n)$,
such that all Minkowski reduced bases  ${\mathcal B} = 
\{v_1,\ldots,v_n\}$
satisfy
 $$M=M_{\mathcal B} \in {\mathcal S}(\delta(n),\epsilon(n)).$$
\end{theorem}
\begin{proof} With  $Q(v_j)=m_{jj}$ and  $Q(v_J')=h_j$ we have 
$$h_j \leq m_{jj} \le  m_{j+1,j+1} \leq  
c_1(j+1)c_2(n) h_{j+1}. \quad\quad (*)$$
For $i<j$ we see moreover
 $$m_{ij} = h_i c_{ij} +\sum_{k<i} h_k c_{ki} c_{kj},$$
hence 
 $$|c_{ij}| \leq \frac{|m_{ij}|}{h_i} + \sum_{k<i} \frac{h_k}{h_i}
  |c_{ki} c_{kj}|$$ for $i<j$.

We have  $|m_{ij}| \leq \frac{1}{2} m_{ii} \leq \frac{1}{2} c_1(i) c_2(n)
h_i$ and $\frac{h_k}{h_i} \leq c_3(n)^{i-k}$ for $k < i$ with
$c_3(n) = c_2(n) \max_{j \leq n} c_1(j)$. With $H_k = \max_{\ell>k} |c_{k\ell}|$
it follows that 
 $$|c_{ij}| \leq \frac{1}{2} c_3(n) + \sum_{k<i} c_3(n)^{i-k}H_k^2$$
 holds.
Since the matrix $M_{\mathcal B}$ is reduced,  $|c_{1j}| = |\frac{m_{1j}}{m_{11}}|
\leq \frac{1}{2}$, hence $H_1 \leq \frac{1}{2}$. By induction 
the $H_i$ are bounded by a constant depending only on $n$, which
proves the assertion.
\end{proof}
\begin{lemma} There exists a constant $C_1=C_1(n,\delta,\epsilon)$
  depending only on $n, \delta,\epsilon$ such that all  $M \in S(\delta,\epsilon)$ satisfy:
For $1\le j \le n$ all ratios of the 
$h_j,m_{jj},\mu_j$ are bounded by $C_1$, where the $\mu_j$ are the
successive minima of $\Z^n$ with the quadratic form $Q=Q_M$ given by
$2Q(\x)={}^t \x M \x$.
\end{lemma}
\begin{proof} $m_{jj} = h_j + \sum_{i<j} h_ic_{ij}^2$ implies
$m_{jj} \leq c_4 h_j$ with a constant $c_4 = c_4(n,\delta,
\epsilon)$, in the opposite direction $m_{jj}\ge h_j$ is trivial. We
see moreover
 $$\begin{array}{lll}
  \mu_j \leq \max_{i \leq j} m_{ii} & \leq & c_4 \max_{i\leq j} h_i\\
  & \leq & c_5 h_j\:\mbox{ (with a constant 
   $c_5 = c_5(\delta, \epsilon, n)$)}.\end{array}$$ 
Let finally  $v_1,\ldots,v_n$ be the standard basis vectors of 
$\Lambda=\Z^n$ and $v'_1,\ldots,v'_n$ their
Gram-Schmidt orthogonalization with respect to  $Q$. If ${v}\in \Lambda$
is linearly independent of $v_1,\ldots,v_{j-1}$ we write 
$v=\sum_{i=1}^n x_iv_i=\sum_{i=1}^n\xi_iv'_i$ with $x_i
\in \Z$, $x_i \ne 0$ for at least one  $i\ge j$. If $i_0$ is maximal
with $x_{i_0}\ne 0$ one has  $\xi_{i_0}=x_{i_0}\in \Z$, hence 
$Q(v)\ge Q(v'_{i_0})=h_{i_0}$ with $i_0\ge j$. This estimate also
holds for the $j$-th minimum vector, and we see 
 $\mu_j \geq \min(h_j,\ldots,h_n) \geq c_7 h_j$ with a constant 
$c_7 = c_7(\delta,n)$.
\end{proof}

\begin{theorem} \label{endlichkeit_siegelbereich} Let $n\in \N$ and
  $\delta,\eta>0$ be given.
There exists a constant $C(n,\delta,\eta)$ such that one has 
$$|t_{ij}| \leq C(n,\delta,\eta)$$
for all
$M_1,M_2 = T^tM_1T$ in $S_n(\delta,\eta)$  with  $T \in GL_n(\Z)$.

In particular, let $M_1$ und $M_2 = T^tM_1T$  be equivalent Minkowski reduced
matrices with $T =(t_{ij})\in GL_n(\Z)$ and let $\delta$ and $\epsilon$ be as
in Theorem $|t_{ij}| \leq
C(n,\delta,\epsilon)$.
\end{theorem}
For the proof of the Theorem we need an auxiliary lemma:

\begin{lemma}\label{coefficients_successiveminima} Let  $n,\delta,\epsilon$ be given. There exists
$C' = C'(n,\delta,\epsilon)$ with the following property:

If  ${\mathcal B} =
\{v_1,\ldots,v_n\}$ is a basis of $V$ with Gram matrix $M=M_{\mathcal B} \in S(\delta,\epsilon)$ 
and if 
$w_i = \sum_{j=1}^{n} \alpha_{ji} v_j$  are successive minimum vectors
of $\Z^n$ with respect to $Q=Q_M$, one has 
$|\alpha_{ji}| \leq C'(n,\delta,\epsilon)$ for all $i,j$.
\end{lemma}
\begin{proof}
We denote by $\{v'_1,\ldots,v'_n\}$ the  Gram-Schmidt
orthogonalization of ${\mathcal B}$, put  $w_i = \sum \beta_{ji} v'_j$
and fix $i,j$.

If we have 
 $Q(v_k) < \mu_i$ for all $ k \leq j$,
we must have $j<i$ and $v_1,\ldots,v_j,w_i$ must be linearly independent. Putting 
 $$w_i^{\ast} = w_i+\sum_{\ell =1}^{j} t_{\ell} v_{\ell} \in
 w_i+L_j$$ with arbitrary $t_{\ell} \in \Z$, 
one must have 
$Q(w_i^{\ast}) \geq \mu_i = Q(w_i)$ since otherwise  $w_i$ were a 
linear combination of the linearly independent vectors
$v_1,\ldots,v_j,w_i^*$ which have length strictly smaller than
$\mu_i$, which contradicts the defining property of successive minimum
vectors.

By Lemma
\ref{gramschmidt_koordinaten} we can choose such a vector $w_i^{*}$
satisfying 
 $$w_i^{\ast} = \sum_{\ell} \beta_{\ell i}^{\ast} v'_{\ell} \quad \mbox{ with }
 |\beta_{\ell i}^{\ast}| \leq \frac{1}{2} \mbox{ for }
 1 \leq \ell \leq j.$$
One has then 
$\beta_{\ell i} = \beta_{\ell i}^{\ast}$ for  $\ell > j$ since 
$w_i-w_i^{\ast}\in L_j$ holds and have therefore the same projection
onto $L_j^{\perp}$.

From this we get 
 $$\begin{array}{lll}
  \sum_{\ell \leq j} h_{\ell} \beta_{\ell i}^2 & = & Q(w_i) -\sum_
   {\ell > j} h_{\ell} \beta_{\ell i}^2
    \vspace{0.2cm}\\
  & \leq & Q(w_i^{\ast}) - \sum_{\ell > j} h_{\ell} \beta_{\ell i}^2
    \vspace{0.2cm}\\
  & = & \sum_{\ell \leq j} h_{\ell} \beta_{\ell i}^{\ast 2}
   \vspace{0.2cm}\\
  & \leq & \frac{1}{4} \sum_{\ell \leq j} h_{\ell} \leq \frac{1}{4}
   (\delta^{j-1}+\cdots + 1) h_j,
 \end{array}$$
and hence  $\beta_{ji}^2 \leq \frac{1}{4} (\delta^{j-1}+\cdots + 1)$.

\medskip
If on the other hand there exists $1\le k\le j$ satisfying
$\mu_i \leq Q(v_k)$, one has 
 $$\begin{array}{lll}
  \mu_i \leq Q(v_k) & \leq & c_8Q(v'_k) \leq \delta^{j-k} c_8h_j\\
  & & \hspace*{0.8cm} \| \\
 & & \hspace*{0.8cm} h_k\\
 & \leq & \delta^{n-1} c_{8} h_j,
 \end{array}$$
and hence  $\mu_i = \sum \beta_{\ell i}^2 h_\ell \leq
\delta^{n-1}c_{8}h_j$.
From this we see 
$\beta_{j i}^2h_j \leq \delta^{n-1} c_{8}h_j$, $\beta_{j i}^2\leq
\delta^{n-1} c_{8}$ so that the 
$\beta_{j i}$ are bounded.
Since we can express the coordinates with respect to the $v_i$ by
those with respect to the $v_i'$ with coefficients bounded by the
constants of the Siegel domain considered, the assertion follows.
\end{proof}
We can now prove the theorem:
\begin{proof}[Proof of Theorem \ref{endlichkeit_siegelbereich}]
Let $\cB=(v_1,\ldots,v_n),
\tilde{\cB}=(\tilde{v_1},\ldots,\tilde{v_n})$ be two bases of the
lattice 
$\Lambda=\Z^n$ with Gram matrices $M_1,M_2$ in ${\mathcal
  S}(\delta,\epsilon)$ and let  $w_i = \sum_{j=1}^{n}
  \alpha_{ji} v_j = \sum_{j=1}^{n}
  \widetilde{\alpha_{ji}}\tilde{v_j}$ be successive minimum vectors of 
  $\Lambda$.

By the lemma the $\widetilde{\alpha_{ji}}$ and the $\alpha_{ji}$ are
bounded.

Then the coefficients of the $\tilde{v_j}$ with respect to the  $w_i$
and therefore also the coefficients of the $\tilde{v_j}$ with respect
to the $v_i$ are bounded, and the assertion is proved.
\end{proof}
\begin{corollary} There exists a constant 
$C(n)$ such that each equivalence class of positive definite symmetric
matrices contains at most $C(n)$ Minkowski reduced matrices.
\end{corollary}
\begin{proof}
Obvious.
\end{proof}
\section{The geometry of the space of positive definite matrices}\label{geometry_reduced}
We will now study the geometry of the space of positive definite real
symmetric matrices.

\begin{lemma} Let ${\mathcal P}_n = \{Y \in M_n^{\rm sym}(\R)~|~
Y>0\}$. Then ${\rm GL}_n(\R)$ acts on  ${\mathcal P}_n$ by
 $$g. Y:= {^t}g^{-1}Yg^{-1}=:Y[g^-1].$$
The stabilizer of $1_n$ is $O_n(\R)$, and elements of
${\mathcal P}_n$ belong to the same  ${\rm GL}_n(\Z)$-orbit if and
only if they are integrally equivalent.
\end{lemma}
\begin{remark} For $Y = \bigl(
  \begin{smallmatrix}
    a&b\\b&c
  \end{smallmatrix}\bigr)
\in {\mathcal P}_2$  let
  $z$ be the unique zero of  $az^2+2bz+c=0$ in the upper half plane
$H$, we have then
$\reteil(z) = \frac{-b}{a}$ and $\vert z \vert= c/a$.

For  $\bigl(
  \begin{smallmatrix}
    \alpha&\beta\\ \gamma&\delta
  \end{smallmatrix}\bigr)
\in {\rm GL}_2(\R)$
one has 
\begin{align*}
(\gamma z+\delta)^2&
  \left(a\left(\frac{\alpha z +\beta}{\gamma z+\delta}\right)^2+2b\left(
   \frac{\alpha z+\beta}{\gamma z+\delta}\right)+c\right)\\
&=  a(\alpha z+\beta)^2+2b(\alpha z+\beta)(\gamma z+\delta)+
   c(\gamma z+\delta)^2\\
  &= z^2(a\alpha^2+2b\alpha \gamma + c\gamma^2)+
    z(2a \alpha\beta+2b(\alpha \delta +\gamma \beta)+2c\gamma \delta)\\
&\quad\quad+
     a\beta^2 + 2b\beta \delta +c\delta^2.
 \end{align*}
Hence, if $\frac{\alpha z+\beta}{\gamma z+\delta}$ is the point in $H$
corresponding to 
$\bigl(
  \begin{smallmatrix}
    a&b\\b&c
  \end{smallmatrix}\bigr)$, we see that $z$ is the point corresponding to 
  $$
  \begin{pmatrix}
\alpha&\beta\\ \gamma&\delta
  \end{pmatrix}
\begin{pmatrix}
    a&b\\c&d
  \end{pmatrix}
\begin{pmatrix}
\alpha&\beta\\ \gamma&\delta
  \end{pmatrix}.$$

That is, if we let 
${\rm SL}_2(\R)$ act on $H$ via
$\bigl(
  \begin{smallmatrix}
    \alpha&\beta\\ \gamma&\delta
  \end{smallmatrix}\bigr).z = \frac{\alpha z+\beta}
{\gamma z+\delta}$ by fractional linear transformations, the map
 $
 \Phi:\:  {\mathcal P} \longrightarrow H $ with
 $Y \longmapsto z$ as above
is compatible with the 
${\rm SL}_2(\R)$-actions:
$\Phi(g . Y) = g . \Phi(Y)$.
\end{remark}
\begin{theorem} ${\mathcal P}_n$ is a convex open set in
$M_n^{\rm sym}(\R) = \R^{n(\frac{n+1}{2})}$. The closure
$\overline{\mathcal P}_n$ of ${\mathcal P}_n$ in $M_n^{\rm sym}(\Bbb
R)$ is the set of positive semidefinite matrices in $M_n^{\rm sym}(\R)$.
\end{theorem}
\begin{proof} Obviously, ${\mathcal P}_n$ is convex and open
and the closure of ${\mathcal P}_n$ in
$M_n^{\rm sym}(\R)$ is contained in the set of  positive
semidefinite matrices. On the other hand, if 
$Y$ is positive definite and $Y_1$ is positive semidefinite, 
$\lambda Y+(1-\lambda)Y_1$is positive definite for all
$0 < \lambda \leq 1$, so that 
$Y_1$ is in the closure of ${\mathcal P}_n$.
\end{proof}
 
\begin{definition} A Minkowski reduced basis ${\mathcal B}$
and the  Gram matrix $M_{\mathcal B} \in {\mathcal P}_n$ associated to
it are called strictly reduced if the inequalities  $Q(x) \geq Q(v_i)$
in Definition \ref{def_minkowski_reduziert}  
are satisfied with strict inequality  for all admissible  $x \not= \pm v_i$.
\end{definition}

\begin{theorem} Let  $1 \leq j \leq n$,
 $${\mathcal W}_j := \{{\bf x} = (x_1,\ldots,x_n) \in \Bbb
 Z^n\setminus\{\pm {\bf e}_j\}\mid {\rm gcd}
  (x_j,\ldots,x_n) =1\}$$
be the set of reduction conditions. Let the set of
satisfiable reduction conditions be defined as the 
 set  ${\mathcal W}_j^{\ast}$ of all
${\bf x} \in {\mathcal W}_j$ for which there exists a Minkowski
reduced matrix
$Y$ satisfying  ${^t{\bf x}}Y{\bf x} = y_{jj}$.
Then
 \begin{enumerate}
 \item The set  ${\mathcal R}_n^0$ of strictly reduced matrices 
in ${\mathcal P}_n$ is equal to
 $$ \{Y \in M_n^{\rm sym}(\R)~|~y_{11} > 0,
 y_{jj}<{}^t \x Y\x \quad\text{ for all } 1\le j \le n, \x \in {\mathcal W}_j^*\}.$$

 \item For the set  ${\mathcal R}_n$ of Minkowski reduced matrices in
   ${\mathcal P}_n$  one has 
 $${\mathcal R}_n = \{Y \in M_n^{\rm sym}(\R)~|~y_{11}>0,
 y_{jj}\le {}^t\x Y \x \quad \text{ for all } 1\le j\le n, \x \in {\mathcal W}_j^*\}.$$
 \item ${\mathcal W}_j^{\ast}$ is finite.
 \end{enumerate}
\end{theorem}

\begin{proof} 
For a  vector $\x$ in  ${\mathcal W}_j^{\ast}$ there exists a Minkowski
reduced matrix $Y$ and a $T \in {\rm GL}_n(\Z)$ having $\x$ as its
$j$-th column such that ${^tT}YT = :Y'$ is again Minkowski reduced. By
Theorem \ref{endlichkeit_siegelbereich} $T$ and hence $\x$ has then
bounded entries, which shows c).

For a), we want to show for a matrix  $Y_0 \in M_n^{\rm \sym}(\R), Y_0
\not\in {\mathcal R}_n^0$ that at least one of the satisfiable
reduction conditions does not hold with strict inequality for $Y$, i.\
e., that for some $j$ there exists  ${\bf x} \in {\mathcal W}_j^*$ with
${^t{\bf x}} Y{\bf x} \leq y_{jj}$.

For $Y_1 \in {\mathcal R}_n^0$ and  $0 \leq \lambda \leq 1$ put 
 $$(y_{ij}^{(\lambda)}) = Y_{\lambda} = (1-\lambda)Y_0+\lambda Y_1
 \mbox{ and } w_{\lambda}({\bf x}) = {^t{\bf x}}Y_{\lambda}{\bf x}
 \mbox{ for } {\bf x} \in \Z^n.$$

Let $\mu = {\rm sup}\{\lambda \in [0,1]~|~Y_{\lambda}\not\in {\mathcal R}_n^0\}$.
Then  $Y_{\mu} \in \overline{\mathcal R}_n^0$ and   $Y_{\mu} \not\in
{\mathcal R}_n^0$ since the set of strictly reduced matrices is open.

Since $Y_\mu$ is in the closure of ${\mathcal R}_n^0$ it satisfies all
reduction conditions.
If $Y_{\mu}$ is positive definite, it is hence reduced but not
strictly reduced, hence there exist 
$j$ and ${\bf x} \in {\mathcal W}_j^{\ast}$ with ${^t{\bf x}}Y_{\mu}{\bf x} 
= y_{jj}^{(\mu)}$.
One has ${^t{\bf x}} Y_1{\bf x} >y_{jj}^{(1)}$ since 
$Y_1 \in {\mathcal R}_n^0$ , and we see that 
${^t{\bf x}} Y_0{\bf x} \leq y_{jj}^{(0)}$ must hold. Thus, 
 $Y_0$ is not in the set described in a). 

If $Y_{\mu}$ is semidefinite but not definite, one has
$y_{11}^{(\mu)}=0$ since the radical of the quadratic form has to
intersect $\Z^n$. We see then as above  $y_{11}^{(0)} \leq 0$, so that
$Y_0$ is not in the set described in a)
in this case too. 

On the other hand,  ${\mathcal R}_n^0$ is obviously contained in this
set.

\medskip
For b),  ${\mathcal R}_n$ is obviously contained in the set $R'$ on the right
hand side and the closure of ${\mathcal R}_n^0$ contains $R'$. Since
$R'\subseteq  {\mathcal P}_n$ by definition and  $\overline{\mathcal
  R}_n^0 \cap {\mathcal P}_n\subseteq  {\mathcal R}_n$ holds, we see
that  ${\mathcal R}_n=R'$ as asserted.
\end{proof}

\begin{corollary} ${\mathcal R}_n^0$ is a convex open subset of 
${\mathcal P}_n$. More precisely, ${\mathcal R}_n^0$ is a convex open
cone bounded by finitely many hyperplanes.
\end{corollary}
\begin{proof} This follows from the theorem.
\end{proof}
\begin{theorem} The set ${\mathcal P}_n$  has a decomposition ${\mathcal P}_n = \bigcup_{T\in {\rm GL}_n(\Z)}
T({\mathcal R}_n)$ (where we put  $T({\mathcal R}_n): = \{T.Y=Y[T^{-1}] ~|~ Y \in
{\mathcal R}_n\}$), and one has
 \begin{enumerate}
 \item If  $T_1,T_2 \in GL_n(\Z)$ satisfy $T_1({\mathcal R}_n^0)\cap T_2({\mathcal R}_n) \not= \emptyset$,
the matrix $T_2^{-1}T_1$ is diagonal with entries  $\pm 1$.
 \item If  $T({\mathcal R}_n) = {\mathcal R}_n$ holds for some $T\in
   GL_n(\Z)$ the matrix  $T$ is diagonal with entries  $\pm 1$.
 \end{enumerate}
\end{theorem}
\begin{proof}
 \begin{enumerate}
 \item Let $Y$ be strictly reduced and  $T\langle Y\rangle = Y_1$
   reduced, let 
$({\bf e}_1,\ldots,{\bf e}_n)$ be the standard basis of  $\R^n$.

then
${\bf e}_1$ is a minimum vector for  $Y_1$, hence  $T^{-1}{\bf e}_1$ a
minimum vector for $Y$, and since $Y$ is strictly reduced we must have 
$T^{-1}{\bf e}_1 = \pm {\bf e}_1$. In the same way we see $T^{-1}{\bf e}_j
= \pm {\bf e}_j$ for $1 \leq j \leq n$.
 \item follows from  a).
 \end{enumerate}
\end{proof}

\begin{theorem} $\Sigma = \{T \in {\rm GL}_n(\Z)~|~
T({\mathcal R}_n) \cap {\mathcal R}_n \not= \emptyset\}$ is a finite set.
\end{theorem}
\begin{proof} This follows from Theorem
  \ref{reduziert_siegelbereich}.  
\end{proof}
\begin{remark}
 We have shown that the set ${\mathcal R}_n$ of Minkowski reduced
 matrices is a fundamental domain for the action of $GL_n(\Z)$ on
 ${\mathcal P}_n$ in the sense that the translaes $T({\mathcal R}_n)$
 cover ${\mathcal P}_n$ and meet only in boundary points. 
\end{remark}
\begin{theorem} The translates $T({\mathcal R}_n)$ of the set ${\mathcal R}_n$ of Minkowski reduced
 matrices form a locally finite  covering of ${\mathcal P}_n$, i.\ e.,
 for any compact set $K \subseteq
{\mathcal P}_n$ there are only finitely many translates $T({\mathcal
  R}_n)$ with $T({\mathcal R}_n) \cap K \not= \emptyset$.
\end{theorem}
\begin{proof}
We choose $\delta,\epsilon > 0$ with ${\mathcal R}_n
\subseteq {\mathcal S}_n(\delta,\epsilon)$.  For $\delta' > \delta$,
$\epsilon' > \epsilon$ the set  ${\mathcal R}_n$ lies then in the
interior 
$S_0$ of ${\mathcal S}_n(\delta',\epsilon')$. The covering of  $K$ by
the open sets  $TS_0$ for $T \in {\rm GL}_n(\Z)$) has then a
finite subcover by certain  $T_jS_0, 1 \leq j \leq r$.
Each of these  $T_j S_0$ meets only finitely many   $TS_0$ by Theorem
\ref{reduziert_siegelbereich}, which implies the assertion.
\end{proof}
\begin{theorem} Let  $\Sigma_1$ denote the set of all $T \in {\rm GL}_n(\Z)$,
for which $T{\mathcal R}_n \cap {\mathcal R}_n$ has codimension $1$ in
${\mathcal R}_n$, hence dimension $\frac{n(n+1)}{2}-1$. 

Then $\Sigma_1$ generates ${\rm GL}_n(\Z)$.
\end{theorem}

\begin{proof} Let  $T \in {\rm GL}_n(\Z)$, let $Y_0 \in {\mathcal R}_n^0$
be a strictly reduced matrix. Let $K \subset T{\mathcal R}_n$ be a
closed ball, let $L := \{(1-\lambda)Y_0+\lambda Y_1~|~0 \leq \lambda
\leq 1,\: Y_1 \in K\}$. 
Since  $L$ is compact the previous theorem implies that there are
only finitely many $U \in {\rm GL}_n(\Z)$ with
$U{\mathcal R}_n \cap L \not= \emptyset$.

Let $D$ be the union of all $U_1{\mathcal R}_n \cap U_2{\mathcal R}_n$
with 
$U_1,U_2 \in {\rm GL}_n(\Z)$, $U_1{\mathcal R}_n \cap L \not= \emptyset
\not= U_2{\mathcal R}_n \cap L$ for which  $U_1{\mathcal R}_n \cap
U_2{\mathcal R}_n$ has codimension $\geq 2$ in ${\mathcal P}_n$.

Then $D$ too  has codimension $\geq 2$ in ${\mathcal P}_n$, there is
$Y_1 \in K$such that $L_1:= \{(1-\lambda)Y_0 + \lambda Y_1~|~
0 \leq \lambda \leq 1\}\cap D=\emptyset $.

The line $L_1$ meets only finitely many $T_j{\mathcal R}_n$ ($T_j \in 
{\rm GL}_n(\Z)$), each of these intersections is an interval,
since the  $T_j({\mathcal R}_n)$ are convex sets. We write $L_1 = (T_1({\mathcal R}_n) \cap L_1)
\cup \cdots \cup (T_r({\mathcal R}_n) \cap L_1)$ and number the $T_j$
in such a way that  the intervals $T_j({\mathcal R}_n)$ and
$T_{j+1}({\mathcal R}_n)$ intersect and 
$T_1 = {\rm Id},T_r = T$.

We have then $T_j^{-1} T_{j+1} \in \Sigma_1$, hence  $T_{j+1} = T_jU_j$ with
$U_j \in \Sigma_1$. Putting things together we obtain $T = U_1 \cdots
U_r$ with $U_i\in \Sigma_1$ as asserted. 
\end{proof}
\section{Extreme forms and sphere packings}
\begin{definition} A matrix $Y \in {\mathcal R}_n$ with $y_{11} =1$
is called {\em extreme} if  the determinant has a local minimum  in
$Y$  as function on the hypersurface $\{Y' \in {\mathcal R}_n~|~y'_{11} =1\}$
hat. It is called {\em absolutely extreme} if it is extreme and the
minimum is an absolute minimum.
\end{definition}
\begin{remark} 
  \begin{enumerate}
  \item If $Y$ is an extreme matrix, the quadratic form $Q_Y$ given by
    $Q_Y(\x)=\frac{1}{2}{}^t\x Y\x$ is called an extreme form and
    lattice with this quadratic form is called an extreme lattice. The
    same term is often used for all matrices or forms proportional to
    $Y$ resp.\ $Q_Y$.
\item Let $L: = \bigoplus_{i=1}^n \Z v_i$ be a lattice on
  $\R^n$ with Gram matrix 
$Y$ with respect to the basis $v_1,\ldots,v_n$ and the standard scalar
product and put spheres of radius $\frac{1}{2}$around the lattice
points. The spheres form then a (non overlapping) lattice sphere
packing in $\R^n$. A fundamental parallelotope of $L$ has then volume
${\det}(Y)$, and the part of it that is covered by the spheres has
volume  $(\frac{1}{2})^n
\frac{V_n}{\sqrt{\det Y}}$,where $V_n$ is the volume of the
$n$-dimensional unit sphere. A lattice whose Gram matrix is absolutely
extreme gives therefore the densest possible lattice sphere packing in $\R^n$.
  \end{enumerate}
\end{remark}

\begin{theorem}\label{existenz_extremform}
There exists an absolutely extreme $Y \in {\mathcal R}_n$.
\end{theorem}

\begin{proof} We fix a matrix $Y_0 \in {\mathcal R}_n$ with $D_0 = \det(Y_0),
y_{11}^{(0)} = 1$ and put 
${\mathcal R}_n^{(1)} = \{Y \in {\mathcal R}_n~|~y_{11} =1\}$. 

The set 
$M = \{Y \in {\mathcal R}_n^{(1)}~|~\det Y \leq D_0\}$ is compact:
For $Y \in M$ we have  $y_{11} \cdots y_{nn} \leq C \cdot D_0$ with a
constant 
$C$, and therefore
 $$y_{nn} \leq \frac{C \cdot D_0}{y_{22} \cdots y_{n-1,n-1}} \
  \leq C\cdot D_0.$$
All $y_{ij}$ are then bounded by  $C \cdot D_0$.

Since $\det$ is a continuous function it has to assume a minimum in $M$.
\end{proof}

\begin{lemma}[Concavity of the determinant] Let $Y_1\ne Y_2
\in {\mathcal P}_n$ with  $\det Y_1 = \det Y_2 = 1$, let  $0 < \lambda
< 1$.

Then
$\det((1-\lambda)Y_1+\lambda Y_2) > 1$.
\end{lemma}
\begin{proof} $Y_1$ and $Y_2$ can be simultaneously diagonalized
by $Y_j \longmapsto {^tT} Y_j T$ for some  $T \in {\rm GL}_n(\R)$,
i.e., we may assume $Y_1,Y_2$ to be diagonal matrices with entries
$a_i$ resp.\ $b_i$.

If we put 
\begin{equation*}
  D(\lambda):=\det((1-\lambda)Y_1+\lambda Y_2) = \prod_{i=1}^{n} 
  (a_i+\lambda(b_i-a_i)),
\end{equation*}
we have $D(0) = D(1) = 1$ and  
\begin{equation*}
\frac{d^2}{d\lambda^2}\log D(\lambda) = -\sum_{i=1}
^{n} \left(\frac{b_i-a}{a_i+\lambda(b_i-a_i)}\right)^2 < 0
\end{equation*}  for $\lambda \in (0,1)$.

This implies $\log D(\lambda) > 0$ in $(0,1)$, hence $D(\lambda) > 1$ in $(0,1)$.
\end{proof}
\begin{corollary}[Arithmetic geometric mean inequality]
For $1 \leq j \leq n$ let  $\alpha_j,\beta_j,\gamma_j > 0$,
$(\alpha_1,\ldots,\alpha_n)$ and $(\beta_1,\ldots,\beta_n)$ not
proportional, the $\gamma_j$ not all equal. We have then 
 \begin{enumerate}
 \item $\prod_{i=1}^{n}(\alpha_i+\beta_i)^{\frac{1}{n}} >
\prod_{i=1}^{n} (\alpha_i)^{\frac{1}{n}} + \prod_{i=1}^{n} (\beta_i)^
{\frac{1}{n}}$
 \item $\frac{1}{n} \sum_{i=1}^{n} \gamma_i > (\prod_{i=1}^{n} \gamma_i)^
{\frac{1}{n}}$
 \item  The hypersurface  $\det Y = 1$ lies strictly above the tangent
   hyperplane at any point $Y_0$.
\end{enumerate}
\end{corollary}
\begin{proof} for a) let $A = (\prod_{i=1}^{n}\alpha_i)^{\frac{1}{n}}$,
$B = (\prod_{i=1}^{n} \beta_i)^{\frac{1}{n}}$ and let  $t\in [0,1]$ be
such that 
$\frac{B}{A} = \frac{t}{1-t}$. 
There is $c > 0$ with  $B = ct, A = c(1-t)$, we put 
 $a_i = \frac{\alpha_i}{c(1-t)}$ and  $b_i = \frac{\beta_i}{ct}$.

Then $\prod_{i=1}^{n} a_i = \prod_{i=1}^{n} b_i = 1$,
$(a_1,\ldots,a_n) \not= (b_1,\ldots,b_n)$, and by the previous lemma
we have 
\begin{eqnarray*}
 \left(\prod_{i=1}^{n} (\alpha_i+\beta_i)\right)^{\frac{1}{n}} & = &
  c \cdot \prod_{i=1}^{n} ((1-t)a_i+tb_i)^{\frac{1}{n}}\\
&=&c=A+B\\
&=& \left(\prod_{i=1}^{n}\alpha_i\right)^{\frac{1}{n}}
  + \prod_{i=1}^{n}(\beta_i)^{\frac{1}{n}}.
\end{eqnarray*}

For b) we consider a) with  $(\beta_1,\ldots,\beta_n) = (1,\ldots,1)$,
$\alpha_i = t\gamma_i$. We have then for all $t$
 $$\prod_{i=1}^{n}(1+t\gamma_i) > \left(1+t\prod_{i=1}^{n} (\gamma_i)
  ^{\frac{1}{n}}\right)^n.$$
We view this as polynomial in $t$ and see for $t \to 0$ by comparing
the terms at $t$ that 
 $$\gamma_1+ \cdots + \gamma_n > n\prod_{i=1}^{n} (\gamma_i)^{\frac{1}{n}}.$$
For c) we consider $Y_0,Y_1$ with  $\det Y_0 = \det Y_1=1$,
$Y_0 \not= Y_1$. 
As above we may assume that $Y_0,Y_1$ are diagonal matrices with
entries $a_i,b_i$ respectively.

The gradient of $\det$ in $Y_0$ is the diagonal matrix with entries
$\frac{1}{a_i}$ which has scalar product $n$ with $Y_0$, so that the
tangent hyperplane at $Y_0$ is given as
$\{X = (x_{ij})~|~ \sum_{i=1}^{n} \frac{x_{ii}}{a_i}=n\}$.

By b) we have 
 $$\frac{1}{n}\sum_{i=1}^{n} \frac{b_i}{a_i} > \left(\prod_{i=1}^{n}
  \frac{b_i}{a_i}\right)^{\frac{1}{n}} =1,$$
hence
 $$\sum_{i=1}^{n} \frac{b_i}{a_i} > n,$$
so that $Y_1$ lies above the tangent hyperplane.
\end{proof}
\begin{theorem} An extreme matrix  $Y \in {\mathcal R}_n^1$ is a
  corner of  ${\mathcal R}_n^1$.
\end{theorem}

\begin{proof} If $Y$ were not a corner one could find $Y_1 \not= Y_2$
  in ${\mathcal R}_n^1$ such that  $Y = (1-\lambda)Y_1+\lambda Y_2$
  holds for some 
$0 < \lambda < 1$, and we can choose $Y_1,Y_2$ as close to $Y$ as we want.

We may again assume $Y_1,Y_2$ to be diagonal matrices with entries
$(a_1,\ldots,a_n)$ resp.\  $(b_1,\ldots,b_n)$  and have 
\begin{eqnarray*}
  (\det Y)^{\frac{1}{n}} &=& \prod_{i=1}^{n} ((1-\lambda)a_i+\lambda b_i)^
  {\frac{1}{n}}\\
 & > & (1-\lambda)\prod_{i=1}^{n} a_i^{\frac{1}{n}} +
  \lambda \prod_{i=1}^{n} b_i^{\frac{1}{n}}\\
  & = & (1-\lambda)(\det Y_1)^{\frac{1}{n}} + \lambda(\det Y_2)^{\frac{1}{n}}.
 \end{eqnarray*}
One sees that one of $\det(Y_1),\det(Y_2)$ must be smaller than 
$\det Y$, in contradiction to the assumption that $Y$ is extreme.
\end{proof}
\begin{corollary} An extreme matrix (which is known to exist by
  Theorem \ref{existenz_extremform}) has rational coefficients.
\end{corollary}
\begin{proof} Being a corner of ${\mathcal R}_n^1$ means to lie in a
  one-dimensional intersection of the finitely many hyperplanes
  bounding   the set  ${\mathcal R}_n$ of reduced matrices. Since all
  these hyperplanes are given by homogeneous linear equations with coefficients in
  $\Z$ we see that an extreme matrix $Y$ is a multiple of a rational
  matrix. Since $y_{11}=1$, the matrix $Y$ itself has to be rational.
\end{proof} 
\begin{corollary} Let $Y \in {\mathcal R}_n^1$ be extreme, let 
$\{{\bf x}_1,\ldots,{\bf x}_N\mid {^t{\bf x}_j} Y{\bf x}_j
=1\}\subseteq  \Z^n$  be the set of all
minimum vectors.

Then $Y$ is the only solution in $M_n^{\sym}(\R^n)$ of the system of equations ${^t{\bf x}_j} Y
{\bf x}_j = 1$ for $1 \leq j \leq n$. In particular, the minimum
vectors ${\bf x}_1,\ldots,
{\bf x}_N$ span the space  $\Bbb Q^n$.
\end{corollary}
\begin{proof}  It the set of solutions of this system of linear
  equations was at least one dimensional, we could again find positive
  definite reduced matrices 
$Y_1$ und $Y_2$ such that $Y$ is on the line connecting them and such
that both of them satisfy all $N$ 
equations ${^t{\bf x}_j} Y_i {\bf x}_j=1$.

If $Y_1$ and $Y_2$ are sufficiently close to  $Y$ the forms $Q_{Y_i}$
  given by 
$2Q_{Y_i}(\x)={^t{\bf x}} Y_i {\bf x}$ assume their minimum on one of
the 
 ${\bf x}_j$ for $Y$, in particular they have minimum $1$ and are in
 ${\mathcal R}_n^1$,  which contradicts the fact that $Y$ is a corner
 of ${\mathcal R}_n^1$.

If finally the ${\bf x}_j$ spanned a genuine subspace of $\Bbb Q^n$,
the system 
${^t{\bf x}_j}Y{\bf x}_j=1$ for $1 \le j \le N$ had rank smaller than
$\frac{n(n+1)}{2}$.
\end{proof}
\section{The LLL algorithm}
\begin{definition}\label{hkz_reduziert}  Let ${\mathcal B} =
  (v_1,\ldots,v_n)$ be a basis of the
 $\Z$-lattice $L$ and  let $L_i:=\Z v_1+\dots+\Z v_{i}$ for $1\le
 i \le n$, $L_0:=\{\nullvek\}$. 

For $x \in V$ denote by $\tilde{p}_i(x)$ 
the orthogonal projection of 
$x$ onto $(\R L_i)^{\perp}$  (with $\tilde{p}_0=\id$) and by 
$\{\tilde{p}_{j-1}v_j =: v'_j\}$ 
the  Gram-Schmidt orthogonalization of ${\mathcal B}$, write 
$v'_j = v_j-\sum_{i=1}^{j-1} c_{ij}v'_i$.

The basis ${\mathcal B}$ is called  Hermite-Korkine-Zolotarev reduced 
(HKZ-reduced), if the following conditions are satisfied:
 \begin{enumerate}
 \item $|c_{ij}| \leq \frac{1}{2}$ for all  $i,j$.
 \item $Q(\tilde{p}_{j-1}v_j)$ is minimal in 
$Q(\tilde{p}_{j-1}(L))\setminus\{0\}$ for 
$1 \leq j \leq n$.
\end{enumerate}
A symmetric matrix $A$ is called HKZ-reduced if it is the Gram matrix
of an HKZ-reduced basis of a lattice.
\end{definition}
To see that HKZ-reduced bases exist we need the following lemma:
\begin{lemma}
 Let $R$ be an integral domain with field of quotients $F$, let
 $\Lambda$ be a lattice of full rank on the $n$-dimensional vector
 space $V$ over $F$.
Let $V=U_1\oplus \dots \oplus U_r$ be a direct sum decomposition and
put $W'_j=U_1\oplus\dots\oplus U_j, W_j=U_{j+1}\oplus\dots\oplus U_r$
for $0\le j <r$, denote by $p_j$ the
projection of $V$ onto $W_j$ with respect to the decomposition
$V=W'_j\oplus W_j$. 

For $1\le j \le r$ let $L_j$ be a lattice on $U_j$ such that
$p_j(L_{j+1})\subseteq W_j$ is a primitive sublattice of
$p_j(\Lambda)\subseteq W_j$ for $0\le j <r$, in particular, $L_1$ is a
primitive sublattice of $U_1$.

Then $\Lambda=L_1\oplus \dots\oplus L_r$ holds.

In particular, if $r=n$ and $L_j=Rv_j$ is free of rank $1$, the vectors $v_1,\ldots,v_n$
form a basis of $\Lambda$.
\end{lemma}
\begin{proof}
 We prove this by induction on $r$, beginning with $r=2$.
Let $v\in \Lambda, v=u+p_1(v)$ with $u_1 \in U_1, p_1(v)\in U_2 \cap
p_1(\Lambda)=p_1(L_2)$, hence $p_1(v)=p_1(v')$ with $v'\in L_2$. We
have then $v-v'\in \ker(p_1)\cap\Lambda=U_1\cap\Lambda=L_1$, hence
$v\in L_1\oplus L_2$.

Let now $r>2$ and assume that the assertion is true for decompositions
into $r'<r$ subspaces.
  
Applying this inductive assumption to the lattice $p_1(\Lambda)$ and
the sublattices $p_1(L_j)$ for $2\le j \le r$ we see  that
one has $p_1(\Lambda)=p_1(L_2)\oplus\dots\oplus  p_1(L_r)=p_1(L_2\oplus\dots\oplus
L_r)$.

But then the case $r=2$ with the decomposition $V=U_1 \oplus
(U_2\oplus\dots\oplus U_r)$ shows that indeed $\Lambda=L_1\oplus  (L_2\oplus\dots\oplus
L_r)=L_1\oplus\dots\oplus L_r$ is true.
\end{proof}
\begin{theorem}
  
Let $\Lambda$ be a positive definite $\Z$-lattice. Then $\Lambda$ has
an $HKZ$-reduced basis.
\end{theorem}
\begin{proof}
 With notations as in Definition \ref{hkz_reduziert} we let
 $v_1,\ldots,v_n$ be vectors in $\Lambda$ for which $Q(\tilde{p}_{j-1}(v_j))$
 is minimal in $Q(\tilde{p}_{j-1}(\Lambda))\setminus \{0\}$ and  hence
 primitive.
By the Lemma the $v_j$ form a basis of $\Lambda$. Subtracting a
vector $y \in L_{j-1}=\oplus_{i=1}^{j-1}\Z v_i$ from $v_j$ doesn't
change $v'_j=\tilde{p}_{j-1}(v_j)$, and by Lemma
\ref{gramschmidt_koordinaten} we can choose the vector $y$ in such a
way that the modified $v_j$ satisfy $v'_j = v_j-\sum_{i=1}^{j-1}
c_{ij}v'_i$ with $\vert c_{ij}\vert \le \frac{1}{2}$ for $1\le i \le j-1$.
\end{proof}
\begin{theorem} The Gram matrix of a  Hermite-Korkine-Zolotarev
reduced basis is in ${\mathcal S}_n(\frac{4}{3},\frac{1}{2})$.
\end{theorem}
\begin{proof} 
By 
$\tilde{p}_j(v_{j+1})=v_{j+1}-c_{j,j+1}v'_j-\sum_{i=1}^{j-1}c_{i,j+1}v'_i$
we have  $\tilde{p}_j(v_{j+1})+c_{j,j+1}\tilde{p}_{j-1}v_j=v_{j+1}-y$
with $y \in L_{j-1}$.

Thus
\begin{eqnarray*}
   Q(\tilde{p}_{j-1}v_j) & \leq & Q(\tilde{p}_{j-1}v_{j+1})\\
   & = & Q(\tilde{p}_jv_{j+1}+c_{j,j+1}\tilde{p}_{j-1}v_j)\\
   & \leq & Q(\tilde{p}_jv_{j+1})+\frac{1}{4} Q(\tilde{p}_{j-1}v_j),
\end{eqnarray*}
which gives $\frac{3}{4} Q(\tilde{p}_{j-1}v_j) \leq Q(\tilde{p}_jv_{j+1})$ or
$h_j \leq \frac{4}{3} h_{j+1}$.
The bound on the $c_{ij}$ is part of the definition of HKZ-reduced.
\end{proof}
A slight weakening of the conditions for being HKZ-reduced has the
advantage of being algorithmically easier to use, it  gives the
conditions of Lenstra, Lenstra and Lovasz for their concept of 
reduction \cite{lll}.
\begin{definition} Let $\frac{1}{4} < \alpha \leq 1$ be given. 
The basis ${\mathcal B} = (v_1,\ldots,v_n)$ of the
  lattice 
$\Lambda$ is called 
$\alpha$-LLL reduced if 
 \begin{enumerate}
 \item $|c_{ij}| \leq \frac{1}{2}$ for all  $i,j$.
\item  $$ \alpha Q(\tilde{p}_{j-1}v_j) \leq Q(\tilde{p}_{j-1}v_{j+1}) \quad
  \text{ for } 1 \leq j \leq n-1.$$
\end{enumerate}
A symmetric matrix $A$ is called $\alpha$-LLL reduced if it is the Gram matrix
of an $\alpha$-LLL reduced basis of a lattice.
\end{definition}
\begin{theorem} The Gram matrix of an  $\alpha$-LLL reduced basis
is  in the Siegel domain ${\mathcal
  S}_n(\frac{1}{\alpha-\frac{1}{4}},\frac{1}{2})$.
Moreover, an $\alpha$-LLL reduced basis satisfies
 \begin{enumerate}
 \item $Q(v_1) \leq (\alpha-\frac{1}{4})^{1-n} \mu_1$.
 \item $Q(v_1) \cdots Q(v_n) \leq (\alpha-\frac{1}{4})^{-n(n-1)/2}
{\det}_B(\Lambda)$.
 \item $Q(v_1) \leq (\alpha-\frac{1}{4})^{\frac{1-n}{2}}
({\det}_B(\Lambda))^{\frac{1}{n}}$.
\end{enumerate}
\end{theorem}
\begin{proof} From
  \begin{eqnarray*}    
  Q(\tilde{p}_{j-1}v_j) & \leq & \alpha^{-1}Q(\tilde{p}_{j-1}v_{j+1})\\
      & = & \alpha^{-1}Q(\tilde{p}_jv_{j+1}+c_{j,j+1}\tilde{p}_{j-1}v_j)\\
  & \leq & \alpha^{-1}(Q(\tilde{p}_jv_{j+1})+\frac{1}{4} Q(\tilde{p}_{j-1}v_j))
\end{eqnarray*}
 we obtain  
 \begin{eqnarray*}
  Q(\tilde{p}_{j-1}v_j)(1-\frac{\alpha^{-1}}{4}) &\leq&
   \alpha^{-1}Q(\tilde{p}_jv_{j+1}),\\
   Q(\tilde{p}_{j-1}v_j) &\leq& \frac{1}{\alpha-\frac{1}{4}}
   Q(\tilde{p}_jv_{j+1}),
 \end{eqnarray*}
which says that the Gram matrix is indeed in
${\mathcal S}_n((\alpha-\frac{1}{4})^{-1},\frac{1}{2})$.
We see further 
 $$Q(\tilde{p}_{j-1}v_j) \geq (\frac{1}{\alpha-\frac{1}{4}})^{-j+1}Q(v_1),$$
hence
 $$\min_j Q(v'_j) \geq (\alpha-\frac{1}{4})^{n-1}Q(v_1)$$
or
$ Q(v_1)  \leq  (\alpha-\frac{1}{4})^{1-n} \min_j Q(v'_j)$.

Let now $\nullvek\ne x \in L$ be any non zero vector.
Let  $j$ be such that  $x=\sum_{i=1}^j a_iv_i=a_jv_j+y_1$ holds with 
$a_j \ne 0, y_1 \in L_{j-1}$.
We have then  $x=a_jv'_j+y_2$ with  $y_2 \in \R L_{j-1}$, hence 
$Q(x)=a_j^2Q(v'_j)+Q(y_2) \ge Q(v'_j)$. 
This implies $\min_j Q(v'_j)\le \mu_1$, and the assertion in a) is proven.

For b) and c) notice that 
 $$Q(v_1) \leq  (\alpha-\frac{1}{4})^{1-j}Q(\tilde{p}_{j-1}v_j)$$
implies 
 \begin{eqnarray*}
 Q(v_1)^n & \leq & (\alpha-\frac{1}{4})^{-\sum_{i=0}^{n-1}i} 
   \prod_{j=1}^{n} Q(v'_j)\\
   & = & (\alpha-\frac{1}{4})^{-{\frac{n(n-1)}{2}}} {\det}_B(\Lambda).
  \end{eqnarray*}
We have therefore
\begin{equation*}
  Q(v_1)  \leq  (\alpha-\frac{1}{4})^{-\frac{n-1}{2}}
   {\det}_B(\Lambda)^{\frac{1}{n}}.
\end{equation*}
Moreover, 
\begin{eqnarray*}
    Q(v_j) & \leq & Q(\tilde{p}_{j-1}v_j)+\frac{1}{4} \sum_{i=1}^{j-1}
   Q(\tilde{p}_{i-1}v_i)\\
  & \leq & Q(\tilde{p}_{j-1}v_j) (1+\frac{1}{4} 
   \sum_{i=1}^{j-1}(\alpha-\frac{1}{4})^
   {-j+i}),
\end{eqnarray*}
which implies
\begin{eqnarray*}
    Q(v_j) & \leq & Q(v'_j) (1+\frac{1}{4} \frac{(\alpha-\frac{1}{4})^{1-j}
   -1}{1-(\alpha-\frac{1}{4})})\\
  & = & Q(v'_j)(1+\frac{(\alpha-\frac{1}{4})^{1-j}-1}{5-4\alpha})\\
  & \leq & Q(v'_j)(\alpha-\frac{1}{4})^{1-j}.
\end{eqnarray*}
From this we get
 $$\prod_{i=1}^{n} Q(v_i) \leq \det(M_{\mathcal B}) \cdot (\alpha-
  \frac{1}{4})^{\frac{-n(n-1)}{2}}$$
as asserted.
\end{proof}
\begin{remark} If one just says LLL-reduced without specifying an
  $\alpha$, one usually means the case 
  $\alpha = \frac{3}{4}$. In that case the estimate for the length of
  the first basis vector becomes
$Q(v_1) \leq 2^{n-1}\mu_1$.
\end{remark}
\begin{proposition} Let  $v_1,\ldots,v_n$ be  a basis of the lattice
  $\Lambda$ with integral Gram matrix.
Then for $\frac{1}{4}< \alpha < 1$ the following algorithm produces an
 $\alpha$-LLL reduced
basis $w_1,\ldots, w_n$ in a number of steps which is polynomial in 
$n$ and in $\max_{1\leq i\leq n}\log Q(v_i)$. 
 \begin{itemize}
 \item[I)] $w_i = v_i$ \quad for $i=1,\ldots,n$
 \item[II)] For $i=1,\ldots,n$ write
 $$w_i-\tilde{p}_{i-1} w_i = \sum_{j=1}^{i-1} \alpha_{ji} w_j,$$
where $\tilde{p}_j$ denotes the  orthogonal projection onto $(\R w_1+\cdots + \R w_j)
 ^{\perp}$.

Let $\tilde{\alpha}_{ji}$ be the closest integer to  $\alpha_{ji}$, 
replace $w_i$ by $w_i-\sum_{j=1}^{i-1} \tilde{\alpha}_{ji} w_j$.

(After this step $w_i-\tilde{p}_{i-1}w_i = 
\sum_{j=1}^{i-1} \alpha_{ji}w_j$
with $|\alpha_{ji}| \leq \frac{1}{2}$ for all $i,j$.)
 \item[III)] If $Q(\tilde{p}_{j-1}w_j) \leq 
\alpha^{-1}Q(\tilde{p}_{j-1}w_{j+1})$
holds for all  $j$, terminate, output the $w_i$ as an $\alpha$-LLL
reduced basis.

Otherwise, let $k$ be the first
index with  $Q(\tilde{p}_{k-1}w_{k+1})< \alpha Q(\tilde{p}_{k-1}w_k)$.

Interchange  $w_k$ and $w_{k+1}$ and return to the beginning of step
II).

(We have then  for $i \leq k$  $|\alpha_{ji}|\leq \frac{1}{2}$ for
$j \leq i-1$, and for $i=k+1$ one has  $|\alpha_{j,k+1}| \leq \frac{1}{2}$
for all $j < k$. If one performs the transformation 
$w_{k+1} \longmapsto w_{k+1}-\sum_{j=1}^{k} \tilde{\alpha}_{j,k+1} w_j$ 
the vector  $\tilde{p}_{k-1}w_{k+1}$ changes and it is possible that
the comparison of
 $Q(\tilde{p}_{j-1}w_j)$ with  
$\alpha^{-1}Q(\tilde{p}_{j-1}w_{j+1})$ in
step  III) may stop at  $j=k$ again.)
 \end{itemize}
\end{proposition} 
\begin{proof}
We put 
 $$\begin{array}{lll}
  D(w_1,\ldots,w_n) & = & \prod_{i=1}^{n} Q(\tilde{p}_{i-1}w_i)^{n-i+1}
   \vspace{0.2cm}\\
  & = & \prod_{k=1}^{n} \det((B(w_i,w_j))_{i,j=1}^k) = \prod_{k=1}^{n} D_k
 \end{array}$$
Step II of the algorithm doesn't change $D$. 

From $D_j = \prod_{i=1}^{j} Q(\tilde{p}_{i-1}w_i)$ and $Q(\tilde{p}_{j-1}w_j) >
\alpha^{-1}Q(\tilde{p}_{j-1}w_{j+1})$ we see that 
step III  replaces $D_j$  by  $ D'_j <
\alpha D_j$.

Each application of step II therefore reduces 
$D(w_1,\ldots,w_n)$ at least by the factor  $\alpha<1$. Since 
$D(w_1,\ldots,w_n) \geq 1$ holds, step III can be applied at most 
$\log_{\alpha}D$ times, where we put  $D: = D(v_1,\ldots,v_n)$.

On the other hand, we have 
\begin{eqnarray*}
   D_j(v_1,\ldots,v_n) & = & \prod_{i=1}^{n} Q(v'_j)\\
  & \leq & \prod_{i=1}^{j} Q(v_j),
\end{eqnarray*}
which implies
 $$D(v_1,\ldots,v_n) \leq ((\max_{i} Q(v_i))^{\frac{n(n+1)}{2}}.$$
Step III is therefore applied at most 
 $$C \cdot \frac{n(n+1)}{2} \log(\max_{i} 2Q(v_i))$$
times.
Since Step II consists of
 $\sim n^2$ operations, we see that the total number of operations is
 indeed polynomial in $n$ and in the $Q(v_i)$.  The running time of
 each of these operations depends on the size of the $Q(v_j'),Q(w_j)$ occurring. 
A computation of the orthogonal projections occurring shows that all
the coefficients $\alpha_{ji}$ and hence all the  
$Q(v_j'),Q(w_j)$ can be bounded in the same way.
 \end{proof}

We consider the following application:

\begin{theorem} Let $\alpha_1,\ldots,\alpha_n \in \Bbb Q$ and 
$0 < \epsilon < 1$ be given. 
Then one can find 
 $p_1,\ldots,p_n$, $q \in \Z$ satisfying 
 $$|p_i-q\alpha_i| \leq \epsilon \quad \mbox{and} \quad 
 0 < q \leq 2^{\frac{n(n+1)}{4}} \epsilon^{-n}$$
in a number of steps which is polynomial in $\frac{1}{\epsilon}$ and
the logarithms of the $\vert \alpha_i\vert$.
\end{theorem}

\begin{proof}
consider the matrix
 $$Y = \begin{pmatrix}
  1 & 0 & \cdots & 0 & \alpha_1\\
  & \ddots & & & \vdots\\
  0 &  & & 1 & \alpha_n\\
  & & & & \epsilon/ Q
 \end{pmatrix} $$
for some  $Q > 0$
and let 
 $\Lambda$ be the lattice spanned by the columns  $v_1,\ldots,v_{n+1}$
of $Y$.
 
For a vector  ${w} = \sum_{i=1}^{n+1}p_iv_i$ with 
$\sqrt{Q(w)} \leq \epsilon$ one has 
$\sqrt{Q(w)} \geq |p_i-q\alpha_i|$ ($i = 1,\ldots,n$) 
with  $q=-p_{n+1}$ and
$\epsilon \geq \sqrt{Q(w)} \geq \epsilon q/Q$, hence $q \leq Q$.

Using  $Q=2^{\frac{n(n+1)}{4}} \epsilon^{-n}$ we know by the
proposition that we can find 
$w \in \Lambda$ satisfying 
 $$\begin{array}{lll}
  Q(w) & \leq & 2^{\frac{(n-\alpha)}{2}} \det Y^{\frac{2}{n}}
  \vspace{0.2cm}\\
 & \leq & 2^{\frac{n}{2}} (\frac{\epsilon}{Q})^{\frac{2}{n+1}}
 \end{array}$$
in polynomial time.
We have then $\sqrt{Q(w)} \leq \epsilon$, hence  $q \leq 2^{\frac{n(n+1)}{4}}
\epsilon^{-n}$, and the inequalities $|p_i-q\alpha_i| \leq \epsilon $
are satisfied.
\end{proof}


\chapter{Reduction Theory of Indefinite Quadratic Forms}
The reduction theory of indefinite quadratic forms is more complicated
than that of definite forms because the finiteness arguments used in
the latter case are not valid here. Hermite's idea was to play the
problem back to the study of associated definite quadratic forms.
We formulate most of the theory in terms of matrices. As usual we
write $A[T]:=^t{}T A T$ for $A\in M_n^{\sym}(\R), T \in M(n\times r,
\R)$, in particular for $T=\x \in \R^n=M(n\times 1, \R)$. We also
write $Q_A(\x)=\frac{1}{2}A[\x]$.

\section{The space of majorants}
\begin{definition}
 Let $A\in M_n^{\sym}(\R)$ be a non singular symmetric matrix. A positive
 definite matrix $P\in M_n^{\sym}(\R)$ is called a {\em (Hermite) majorant} of
 $A$ if 
 \begin{equation*}
A[\x]={\vert} ^t{}\x A \x \vert \le ^t{}\x P \x=P[\x]  \text{ for all } \x \in \R^n.
 \end{equation*}
Let the positive definite matrices in $M_n^{\sym}(\R)$ be partially
ordered by
\begin{equation*}
  P_1\le P_2 \text{ if and only if }P_1[\x]\le P_2[\x] \text{ for all
  } \x \in \R^n.
\end{equation*}
Then a majorant $P$ of $A$ is called a {\em minimal majorant} if and only if
it is minimal with respect to this ordering among the majorants of $A$.

The set of minimal majorants of $A\in M_n^{\sym}(\R)$ is denoted by
$\mfH(A)$ and is called the {\em space of (minimal) majorants} of $A$.  
\end{definition}
\begin{lemma}
  With notations as above the following statements are equivalent:
  \begin{enumerate}
  \item $P\in \mfH(A)$.
\item There is a decomposition $\R^n=V_1\oplus V_2$ of $\R^n$ into
  subspaces $V_1,V_2$ such that $V_1,V_2$ are orthogonal to each other
  with respect to both $Q_A,Q_P$ (respectively their associated
  symmetric bilinear forms $b_A, b_P$) and such that 
  \begin{eqnarray*}
    Q_A\vert_{V_1}&=&Q_P\vert_{V_1}\\
Q_A\vert_{V_2}&=&-Q_P\vert_{V_2}.
  \end{eqnarray*}
\item There exists $T\in GL_n(\R)$ satisfying
  \begin{equation*}
    A[T]=E_{a,b}, P[T]=E_n,
  \end{equation*}
where $(a,b)$ is the signature of $A$ and $E_{a,b}$ denotes the
diagonal matrix with $a$ entries $+1$ and $b$ entries $-1$.
\item $(AP^{-1})^2=E_n$.
  \end{enumerate}
\end{lemma}
\begin{proof}
By linear algebra we can diagonalize $A$ and $P$ simultaneously by a
suitable change of basis $A \mapsto A[T], P\mapsto P[T]$, we may
therefore assume 
\begin{eqnarray*}
 Q_M(\x)&=&\sum_{i=1}1^a x_i^2-\sum_{j=1}^bx_{i+a}^2\\ 
Q_P(\x)&=&\sum_{i=1}^n c_ix_i^2
\end{eqnarray*}
with $c_i \in \R, c_i>0$. It is then clear that $P$ is a minimal
majorant of $A$ if and only if all $c_i$ are $1$, it is equally clear
that this condition is equivalent to the validity of both b) and c).

If c) is satisfied with $T\in GL_n(\R)$ we have 
\begin{eqnarray*}
A P^{-1}AP^{-1}&=&(^t{}T^{-1}E_{a,b}T^{-1})(T
                   ^t{}T)(^t{}T^{-1}E_{a,b}T^{-1})(T ^t{}T)\\
&=&^t{}T^{-1}E_n ^t{}T\\
&=&E_n,
\end{eqnarray*}
hence d). If conversely d) holds it is also true for $A[T], P[T]$ in
place of $A,T$, and we can again assume without loss of
generality that $A$ and $P$ are diagonal as above. But then the
condition d) implies that all the $c_i^2$ and hence the $c_i$ are $1$,
so that c) is satisfied.
\end{proof}
\begin{corollary}
  \begin{enumerate}
  \item  There is a natural bijection between $\mfH(A)$ and the set of
 subspaces $V_1$ of $\R^n$ which are maximal positive definite with
 respect to $Q_A$. It is given by associating to $P\in \mfH(A) $ the
 radical of the quadratic form $Q_A-Q_P$ and to $V_1$ the matrix $P$
 associated to the quadratic form $Q_P(\x)=Q_A(p_1\x)-Q_A(p_2\x)$,
 where $p_1,p_2$ are the orthogonal projections onto $V_1$ and onto
 the orthogonal complement $V_2$ of $V_1$ with respect to $Q_A$
 respectively. 
\item The orthogonal group $O_{Q_A}(\R)$ acts transitively on
  $\mfH(A)$, the stabilizer of $P\in \mfH(A)$ is conjugate to
  $O_{a}(\R)\times O_b(\R)$, where $O_m(\R)$ denotes the orthogonal
  group of the standard scalar product on $\R^m$. 
  \end{enumerate}
\end{corollary}
\begin{proof}
  The first assertion follows directly from the lemma and its proof.
  For the second assertion we use Witt's extension theorem (Theorem
  \ref{extension_theorem}) to find for given totally positive spaces
  $V_1, V_1'$ a $\phi \in O_{Q_A}(\R)$ with
  $\phi(V_1)=V_1'$. Considering the splitting $\R^n=V_1\oplus V_2$
  associated to $P\in \mfH(A)$ we may choose bases of $V_1,V_2$ with
  respect to which $Q_M$  has matrix $E_a$ resp.\ $-E_b$, in the
  basis of $V$ obtained from these the stabilizer of $P$ has the required shape. 
\end{proof}
\begin{remark}
 The stabilizer in the corollary
 is obviously a compact subgroup of the orthogonal group
 $O_{Q_A}(\R)$ of the real quadratic form of signature $(a,b)$. In
 fact it is maximal compact, and every compact subgroup $K$ is contained
 in a 
 conjugate of it. To see this, denote by $\langle,\rangle$   a
  scalar product on $V=\R^n$ which is invariant under the action of
  $K$, such a scalar product can be obtained from the standard scalar
  product by integration over $K$. Since the symmetric bilinear form
  $b_A$ associated to $Q_A$ is non degenerate, there exists a unique 
  isomorphism $f:V \to V$ satisfying $\langle f(v),w \rangle=b_A(v,w)$
  for all $v,w \in V$.
In view of 
\begin{equation*}
  \langle f(v),w \rangle =b_A(v,w)=b_A(w,v)=\langle f(w),v \rangle
\end{equation*}
the map $f$ is self adjoint with respect to $\langle,\rangle$ so that
$V$ is the orthogonal (with respect to $\langle,\rangle$) sum of the
eigenspaces $V_\lambda$  of $f$.
 Moreover,  we
have for $k\in K$ and $v,w \in V$
\begin{equation*}
  \langle kf(v),w\rangle=\langle f(v),k^{-1}w\rangle=b_A(v,
  k^{-1}w)= b_A(kv,w),
\end{equation*}
so that $kf(v)=f(kv)$ for all $k \in K,v\in V$. In particular, $k$
leaves the eigenspaces $V_\lambda$ invariant. 

On $V_\lambda$  we have $2Q_A(v)=b_A(v,v)=\langle f(v),v
\rangle=\lambda \langle v,v\rangle$, which implies that $Q$ is
positive definite on the sum $V_+$ of the $V_\lambda$ for $\lambda>0$,
negative definite on the sum $V_-$ of the $V_\lambda$ with
$\lambda<0$.
It follows that $K$ is contained in the group $O_{V_+}(\R)\times
  O_{V_-}(\R)$ as asserted. 

The space $mfH(A)$ of minimal majorants of $A$ can therefore be
viewed as a homogeneous space $K\backslash O_{Q_A}(\R)$ with a maximal
compact subgroup $K$. 
\end{remark}
\section{Hermite reduction of indefinite forms}
\begin{definition}
With notations as before the matrix $A$ is called Hermite reduced if
the space $\mfH(A)$ of its minimal majorants contains a Minkowski
reduced matrix. 
\end{definition}
\begin{lemma}
  Every non degenerate real symmetric matrix $A$ is integrally
  equivalent to a Hermite reduced matrix.
\end{lemma}
\begin{proof}
  Let $P$ be a minimal majorant of $A$ and choose $T\in GL_n(\Z)$ so
  that $P[T]$ is Minkowski reduced. Then $A[T]$ is Hermite reduced.
\end{proof}
We want to prove a finiteness result for integral Hermite reduced matrices. For
this we need first an auxiliary lemma.
\begin{lemma}
Let $\delta>0, \epsilon>0, n\in \N$ be given. Then there exists
$\epsilon_1$ depending on $n,\epsilon$ such that for any positive
definite symmetric matrix $P$ in the Siegel domain
$\mcS_n(\delta,\epsilon)$ one has $JP^{-1}J \in
\mcS_n(\delta,\epsilon_1)$, where 
$J=\biggl(\begin{smallmatrix}
  0&\dots&0&1\\
&&\iddots& \\
&\iddots&&\\
1&0&\dots&0
\end{smallmatrix}\biggr)$.
\end{lemma}
\begin{proof}
 We write $P=H[C]$ where the entries $h_j$ of the diagonal matrix $H$
 satisfy $h_j \le \delta h_{j+1}$ and the entries $c_{ij}$ of the
 upper triangular matrix $T$ are bounded in absolute value by
 $\epsilon$.
with $C_1:=^t{}C^{-1}[J], H_1=H^{-1}[J]$ we have
$P^{-1}[J]=H_1[C_1]$. The coefficients of the upper triangular matrix
$C_1$ are then bounded in absolute value by a suitable $\epsilon_1>0$,
and the diagonal entries $h'_1=h_n^{-1},\ldots,h'_n=h_1^{-1}$ of
$H_1$ satisfy $h'_i\leq \delta h'_{i+1}$.
\end{proof}
\begin{theorem}
For fixed  $n \in \N, 0\ne d\in \Z$  there exist only finitely many integral
 Hermite reduced matrices $A \in M_n^{\sym}(\Z)$ with $\det(A)=d$.

Moreover, for $n \in \N$ fixed 
there exists a constant $c(n) \in \R$ such that for
any integral Hermite reduced $A=(a_{ij}) \in M_n^{\sym}(\Z)$ of
determinant $d$ with anisotropic $Q_A$ one has $\vert
a_{ij}\vert \leq c(n) d$.
\end{theorem}
\begin{proof}
  We set $d=\det(A)$ and consider first the case that $Q_A$ is anisotropic. We have then
  $\vert Q_A(\x)\vert \ge 1$ for all $\x \in \Z^n\setminus \{0\}$ and
  hence $Q_P(\x)\ge 1$ for all $\x \in \Z^n\setminus\{0\}$ for any
minimal   majorant $P$ of $A$. The successive minima $\mu_j(P)$ of
such a $P$ satisfy
  then 
  \begin{equation*}
   \mu_1(P)\dots \mu_n(P)\le \gamma_n^n \det(P) =\gamma_n^n d, 
  \end{equation*} by Theorem\ref{hermites_theorem_definite}, where
  $\gamma_n$ is Hermite's constant.
Since we have $\mu_j(P)\ge 1$ for all $j$ this implies $\mu_j(P)\le
\gamma_n^n d$ for all $j$. If $P$ is Minkowski reduced the
coefficients $p_{ij}$ of $P$ are bounded in absolute value by $c'(n)d$
for a suitable constant $c'(n)$ by Theorem
\ref{comparison_reduction_constants}.
From this we get 
\begin{eqnarray*}
  \vert a_{ij}\vert&=&\vert Q_A(\e_i+\e_j)-Q_A(\e_i)-Q_A(\e_j)\vert\\
&\le & Q_P(\e_i+\e_j)+Q_P(\e_i)+Q_P(\e_j)\\
\le c(n)d
\end{eqnarray*} with a suitable constant $c(n)$.

In the case that $Q_A$ is isotropic we consider again a Minkowski
reduced minimal majorant $P$ of the integral Hermite reduced matrix
$A$. Let ${\f}^{(1)},\ldots,\f^{(n)}\in \Z^n$ be successive minimum
vectors of $Q_P$ with components $f_j^{(i)}$, let $\delta,\epsilon>0$
be such that $P\in \mcS_n(\delta, \epsilon)$, where $\delta, \epsilon$
depend only on $n$ by Theorem \ref{reduziert_siegelbereich}. By Lemma
\ref{coefficients_successiveminima} there exists then a constant
$c'(n)$ with $\vert f_j^{(i)}\vert \le c'(n)$ for all $i,j$.  By the lemma
there exists $\epsilon_1=\epsilon_1(n)>0$ depending only on $n$ with $P^{-1}[J]
\in \mcS_n(\delta,\epsilon_1)$, and we have $P=A
P^{-1}A=(P^{-1}[J])[T]$, where $T=JA\in M_n^{\sym}(\Z)$ with $\vert
\det(T)\vert=\vert d\vert$. The vectors $T\f^{(i)}$ are therefore
successive minimum vectors of the lattice $T\Z^n$ with respect to
$P^{-1}[J]$, with the lattice $T\Z^n$ having index $\vert d\vert$ in
$\Z^n$. By Lemma \ref{coefficients_successiveminima} the coefficients
of the  $T\f^{(i)}$ in terms of the  basis of $T\Z^n$ consisting of
the $T\e_i$ are then bounded in absolute value by a constant depending
 only on $n$, and 
since $T\Z^n$ has index $\vert d \vert $ in $\Z^n$ the coefficients of
the matrix $F'=TF$ with columns $T\f^{(i)}$ are bounded in absolute
value by some $c'(n,d)$. But then  the coefficients of
$T=F'F^{-1}$ and hence the coefficients of $A=JT$ are bounded in
absolute value by a suitable $c(n,d)$ too.
\end{proof}
\section{Compactness results for anisotropic forms}
For anisotropic $A$ we can also prove two further boundedness results.
\begin{theorem}
  Let $A\in M_n^{\sym}(\Z)$ with anisotropic $Q_A$.
There exists a constant $c(A)$ such that for all $S\in =_{Q_A}(\R)$
there exists $T\in  O_{Q_A}(\R;\Z^n):=GL_n(\Z)\cap O_{Q_A}(\R)$ such that all
coefficients of $ST$ are bounded in absolute value by $c(A)$.

In other words, there exists a compact fundamental domain for the
action of $O_{Q_A}(\R;\Z^n)$ on $O_{Q_A}(\R)$ by translations.
\end{theorem}
\begin{proof}
By the previous theorem there are only finitely many Hermite reduced
integral matrices $A=A_1[U_1],A_2=A[U_2], \ldots, A_r=M[U_r]$ (with
$U_i \in GL_n(\Z)$) which are integrally equivalent to $A$. for each
of these we choose Minkowski  reduced minimal majorants $P_i \in
\mfH(A_i)$ and have $P_i[U_i^{-1}] \in \mfH(A)$. Since $O_{Q_A}(\R)$
acts transitively on $\mfH(A)$ we obtain $S_i \in O_{Q_A}(\R)$ with
$P_i[U_i^{-1}]=P[S_i]$ which implies $P_i=P[S_iU_i]$ for $1 \le i \le r$. 

Let now $S\in O_{Q_A}(\R)$ be given and choose $U \in GL_n(\Z)$ such
that $P[SU]=P[S][U]$ is a Minkowski reduced majorant of $A[S][U]=A[U]$. Then
$A[U]$ is Hermite reduced
and integrally equivalent to $A$, hence equal to 
$A_j=A[U_j]$ for some $j$, which gives $U U_j^{-1}=:T \in
O_{Q_A}(\R;\Z^n)$ . Moreover, as in the proof of the anisotropic case
in the previous theorem we can bound the coefficients of $P[S][U]$ by
$c_1(n)d$ for a suitable constant $c_1(n)$ depending only on $n$. 

With $V:=(S_jU_j)^{-1}SU$ we have $P[SU]=P_j[V]$.  Since the
coefficients of $P[SU]$ are bounded and $P_j$ is positive definite,
the coefficients of $V$ are bounded by some constant $c'(n,d)$. Since
there are only finitely many possibilities (not depending on $S$) for
the $S_j, U_j$, this 
implies that  the coefficients of $ST=S_jU_jVU_j^{-1}$ are bounded by some
constant depending on $A$ but not on $S$, which proves the assertion. 
\end{proof}
\begin{remark}
  In the isotropic case Siegel proved that there is a fundamental
  domain of finite volume with respect to the Haar measure on the
  orthogonal group $O_{Q_A}(\R)$. A compact fundamental domain does not exist
  in that case.
\end{remark}
\begin{theorem}
  Let $A\in M_n^{\sym}(\Z)$ with $Q_A$ anisotropic.

 There exists a constant $c_1(A)$ such that for any 
  $\nullvek \ne \x\in \R^n$ there exists $T\in O_{Q_A}(\R;\Z^n)$
 with $\|T\x\|^2 \le c_1(A)Q_A(\x)$.

In other words, each nonzero $O_{Q_A}(\R;\Z^n)$-orbit in $\R^n$ contains vectors of
bounded length.
\end{theorem}
\begin{proof}
  We choose a vector $\y\in \R^n$ with $Q_A(\x)=Q_A(\y)$ and such that
  the coefficients of $\y$ are bounded by a constant multiple of
  $\vert Q_A(\x)\vert^{1/2}$ (e.g., by multiplying a fixed vector of $Q_A$-value
  in $\{\pm 1\}$ by $\vert Q_A(\x)\vert^{1/2}$ ). 

By Witt's
  extension 
  theorem (Theorem \ref{extension_theorem}) there exists $S\in
  O_{Q_A}(\R)$ with $S\x=\y$, and by the previous theorem we
  find a constant $c(A)$ and $T\in O_{Q_A}(\R;\Z^n)$ such  that the
  coefficients of $TS^{-1}$ are bounded in absolute value by
  $c(A)$. Then
$T\x= TS^{-1}\y$  is as asserted.
\end{proof}

\chapter{Quadratic Lattices over Discrete Valuation Rings}
As is well known from algebraic number theory, an important tool for
the study of number rings is the study of their completions with
respect to prime ideals, i.e., the study of local fiels and their
valuation rings. Our principal interest in this chapter therefore is the theory of
quadratic forms or modules over local fields and their valuation rings.
Since many of the results are valid in the more general context of local
rings and their completions, we will work in this more general context
whenever this doesn't complicate the exposition and will indicate
differences in results and proofs where these occur. Since
over a local ring all finitely generated projective modules are free
we will restrict our attention to free modules.

\section{Local rings, discrete valuation rings, local fields}
We recall some definitions and results from commutative algebra:
\begin{definition}
A discrete valuation ring $R$  is an integral domain satisfying the
following equivalent conditions:
\begin{enumerate}
\item $R$ is a principal ideal domain with precisely one  
class of prime elements modulo units.
\item There is a surjective map $v$ from the nonzero elements of the  field of fractions $F$
  of $R$ to $\Z$ , called a discrete (additive) valuation, such that 
  \begin{enumerate}
  \item $v(ab)=v(a)+v(b)$ for all $a,b \in F^\times$. 
\item $v(a+b)\ge\min(v(a),v(b))$ for all $a,b \in F^\times$ with $a+b \ne 0$.
\item $R=\{a\in F^\times\mid v(a) \ge 0\}\cup \{0\}$.
  \end{enumerate}
\item $R$ is a local ring with a unique non zero prime ideal $P$.
\end{enumerate}
The field $k:=R/P$ is called the residue field of $R$, the field of
fractions $F$ is called a discretely valued field.

The discrete valuation ring and its field of fractions are called
complete if they are complete metric spaces with respect to the metric
derived from the absolute value given by $\vert a \vert_v=c^{-v(a)}$
for $a\ne 0$, $\vert 0\vert_v=0$, where $0<c<1$ is arbitrary.
\end{definition}
\begin{definition}
 A valuation ring is an integral domain $R$ with field of fractions
 $F$ such that for all $0 \ne a \in F$ one has $a \in R$ or $a^{-1}\in R$.
\end{definition}
\begin{theorem}
Let $R$ be a valuation ring with field of fractions $F$.
  \begin{enumerate}
  \item Every finitely generated ideal in $R$ is
    principal, every finitely generated $R$-submodule of $F$ is free
    of rank $1$.
\item The ideals are totally ordered by inclusion, the same is true
  for the $R$-submodules of $F$.
\item A valuation ring is a discrete valuation ring if and only if it
  is noetherian.
\item There is a totally ordered abelian group $G$ and a surjective
  map $v$ from $F^\times$ 
  to $G$ , called an (additive) valuation, such that 
  \begin{enumerate}
  \item $v(ab)=v(a)+v(b)$ for all $a,b \in F^\times$. 
\item $v(a+b)\ge\min(v(a),v(b))$ for all $a,b \in F^\times$ with $a+b \ne 0$.
\item $R=\{a\in F^\times\mid v(a) \ge 0\}\cup \{0\}$.
  \end{enumerate}
  \end{enumerate}
\end{theorem}
\begin{theorem}
  Let $R$ be a complete discrete valuation ring with field of
  quotients $F$ and maximal ideal $P$ and additive valuation $v$.

Then the following are equivalent:
\begin{enumerate}
\item $F$ is a local field (i.e., it is locally compact with respect
  to the metric induced by $v$).
\item $R$ is compact with respect to the metric induced by $v$.
\item $R/P$ is finite.
\end{enumerate}
\end{theorem}
\begin{remark}
  Whenever convenient we will extend the valuation $v$ to all of $F$
  by setting $v(0)=\infty$ and write  $v(b)<v(0)$ for all $b \ne 0$ accordingly.
\end{remark}
\begin{theorem}
 Let $R$ be a discrete valuation ring with field of quotients $F$ and
 additive valuation $v$.

Then there is an up to valuation preserving isomorphism unique
complete discrete valuation ring $\hat{R}\supseteq R$ (the completion
of $R$) with field of
quotients $\hat{F}\supseteq F$ (the completion of $F$ with respect to
$v$) and valuation $\hat{v}$ extending $v$ such that $R$ resp. $F$ is
dense in  $\hat{R}$ resp. $\hat{F}$ with respect to the metric induced
by  $\hat{v}$.

One has $R/P\cong \hat{R}/\hat{P}$, where $\hat{P}\supseteq P$ is the maximal
ideal of $\hat{R}$, in particular, $\hat{F}$ is a local field if and
only if $R/P$ is finite.
\end{theorem}
\begin{example}
For any prime number  $p$ the ring $\Z_{(p)}=\{\frac{a}{b}\in \Q\mid
a,b\in \Z, p\nmid b\}$ is a discrete valuation ring with maximal ideal
$P=p\Z_{(p)}$. Its completion is the ring $\Z_p$ of $p$-adic integers,
with field of quotients $\Q_p$, the field of $p$-adic
rationals. Similarly, for a number field $K$ with ring of integers
$\fo_K$, the localization of $\fo_K$ at a prime ideal $P$ is a
discrete valuation ring, and the completion of $K$ with respect to the
corresponding valuation is denoted by $K_P$.

It is well known that all non archimedean local fields (i.e, local
fields different form $\R,\C$) arise in this way either from an
algebraic number field or an algebraic function field, i.e., a
separable extension of the field of rational functions over a finite field. 
\end{example}
A key property of complete local rings is the validity of Hensel's
lemma, i.e., the fact that zeros or factorisations of polynomials
modulo sufficiently high powers of the maximal ideal can be lifted to
genuine zeroes respectively factorisations:
\begin{theorem}[Hensel's Lemma]\label{hensellemma}
Let $R$ be a local ring with maximal ideal $P$ and residue field $k=R/P$.
\begin{enumerate}
\item Let
$f \in R[X]$ and $a \in R$ be such that $f(a) \in f'(a)^2P$,
where
$f'$ is the formal derivative of $f$.

Then for all $n \in \N$ there is $a_n \in R$ with
$a_{n+1}-a_n \in f'(a)P$ and $f(a_{n+1})\in f(a)P^n$.

If $R$ 
complete with respect to the $P$-adic topology, there is $b \in R$ with
$f(b)=0$ and 
$b-a \in f'(a)P$, 
in particular  $b \in a+P$.
\item Let $f \in R[X]$ be a monic polynomial of degree $n$ such that the reduction
  $\bar{f}\in k[X]$ of $f$ modulo $P$ has a factorization $\bar{f}=\tilde{g} \tilde{h}$ into coprime
  polynomials $\tilde{g}, \tilde{h} \in k[X]$ of degrees $r,
  n-r$, assume $R$ to be noetherian and complete with respect to the $P$-adic
  topology and put $k=R/P$.
Then there are polynomials $g,h \in R[X]$ of degrees $r,n-r$
  with $\bar{g}=\tilde{g}, \bar{h}=\tilde{h}, f=gh$.
\end{enumerate}
 \end{theorem}
 \begin{proof}
   The usual proof by induction on $n$ of a) for (complete) discrete valuation rings carries
   over to the present more general situation:
We write $f(a)=(f'(a))^2y$ with $y \in P$ and put $a_1:=a$. Let $n \ge
1$ and assume that $a_n$ with the asserted properties has been found,
assume in addition that  $f'(a_n)R=f'(a)R$  holds, in particular we
have $f(a_n) \in (f'(a_n))^2P$.
   
We put $a_{n+1}=a_n-f'(a_n)y \in a_n+P$ and obtain
\begin{eqnarray*}
  f(a_{n+1})&=&f(a_n)-f'(a_n) f'(a_n)y+f(a_n)yz_1\\
&=&f(a_n)yz_1\\
&\in &f(a_n)P
\end{eqnarray*}
with $z_1 \in R$ by looking at the Taylor expansion of $F$ around
$a_n$.

Similarly, we have $f'(a_{n+1})=f'(a_n)+f'(a_n)yz_2\in f'(a_n)(1+P)$
with $z_2 \in R$, so that  $f'(a_{n+1})R=f'(a_n)R$ and $f(a_{n+1})\in
(f'(a_{n+1}))^2P$ hold and $a_{n+1}$ is as required. If $R$ is
complete, the sequence of the $a_n$ converges to a limit $b$, and
since $\cap_{n=1}^\infty P^n=\{0\}$ holds for a complete local ring
(using the definition of \cite{atiyah, eisenbud}),
one has 
$f(b)=0$.

In a similar way, the usual proof of b) for discrete valuation rings
is easily modified for the present situation.
 \end{proof}
 \begin{remark}
   \begin{enumerate}
   \item
   If $R$ is a discrete valuation ring with valuation $v$, the
   condition in a) above becomes  $v(f(a))>2v(f'(a))$, and the
   assertions for $a_n$ resp. for $b$  become $v(a_n-a)\ge v(f(a))-v(f'(a))$ and $v(f(a_n))\ge
v(f(a)+(n-1)(v(f(a))-2v(f'(a)))\ge v(f(a))+n-1$ respectively
$v(b-a)\ge v(f(a))-v(f'(a))$.
\item A different proof for complete local rings can be found in \cite{bourbaki,eisenbud}.
   \end{enumerate}
 \end{remark}

\section{Orthogonal decomposition of lattices over valuation rings}
\begin{theorem}
Let $R$ be a valuation ring with field of fractions $F$, let $(V,Q)$
be a regular quadratic space over $F$ and let $\Lambda$ be a free 
$R$-lattice on $V$ with $b(\Lambda,\Lambda)=:I\subseteq R$.

Then $\Lambda$ contains an $I$-modular sublattice $K$ which splits off
orthogonally in $\Lambda$, the orthogonal complement $K^\perp$ is free.   
\end{theorem}
\begin{proof}
The ideal $I\subseteq R$ is finitely generated (by the $b(v_i,v_j)$
where the $v_i,v_j$ run through  a basis of $\Lambda$), hence
principal, say $I=aR$ with $a \in R, a \ne 0$.

If there exists $v \in \Lambda$ with $b(v,v)R=I$, the one dimensional
lattice $K=Rv$ is as desired. Otherwise, we can find linearly
independent $v,w \in \Lambda$
with $b(v,w)=a$ and $b(v,v)\in aP, b(w,w)\in aP$. We put $K=Rv+Rw$ and
see that $K^\#=a^{-1}K$, i.e., $K$ is $I$-modular.

In both cases, Lemma \ref{modular_split} shows that $K$ splits off
orthogonally in $\Lambda$.   In the special case that $R$ is a
discrete valuation ring, the freeness of $K^\perp$ follows immediately
from the main theorem on modules over principal ideal domains.  
For a general valuation ring, since $K$ and $\Lambda$ are free, its complement
$K^\perp$ is stably free, which for modules over  a valuation ring (or
more generally a Bezout ring) implies free.
\end{proof}
\begin{corollary}
Any lattice $\Lambda$ as in the theorem has a decomposition into an
orthogonal sum of modular lattices of rank $1$ or $2$. If $2$ is a
unit in $R$ it has an orthogonal basis.    
\end{corollary}
\begin{proof}
This follows by induction from the proof of the theorem.  
\end{proof}
\begin{defcorollary}[Jordan decomposition]\label{jordan_decomposition}
Let $\Lambda$ be as in the theorem with $b(\Lambda,\Lambda)=I\subseteq
R$.

Then there are nonzero principal ideals $J_1=I\supsetneq J_2\supsetneq \dots \supsetneq
J_r$ in $R$   and a decomposition $\Lambda=K_1\perp \dots \perp K_r$
into an orthogonal sum of $J_i$-modular lattices $K_i$.

Such a decomposition is called a Jordan decomposition.
\end{defcorollary}
\begin{proof}
  Again the assertion follows from the theorem by induction.
\end{proof}
\begin{remark}\label{jordan_remark}
  \begin{enumerate}
  \item Our proof is practically the same as the usual proof (see
    e. g. \cite{omeara}) for discrete valuation rings. It seems not to
    be possible to generalize it further to an arbitrary local ring, at
    least not easily.
\item The fact that this decomposition is possible for arbitrary
  valuation rings has been noticed by Zemel \cite{zemel}. 
\item If $2$ is not a unit in $R$ we will call the lattice $\Lambda$ totally even if it has a
  Jordan decomposition into even modular lattices. Equivalently, $Q(x)
  \in b(x,\Lambda)$ holds for all $x \in \Lambda$.
  \item The name ``Jordan decomposition'' for an orthogonal
    decomposition as above has probably first beeen used by O'Meara. It is
    unclear whether there is a historical reason for this or whether
    he just used this name in analogy to the Jordan normal form of an
    endomorphism and the corresponding decomposition of the underlying
    vector space into invariant subspaces in linear algebra. 
  \end{enumerate}
\end{remark}
\begin{lemma}
 Let $\Lambda$ as above have a Jordan decomposition as above with
 ideals $J_i=c_iR$.  

Then one has $\Lambda^\#=c_1^{-1}K_1\perp\dots\perp c_r^{-1}K_r$ and
$c_r\Lambda^\#\subseteq \Lambda$. 
\end{lemma}
\section{Hensel's lemma and lifting of representations}
For quadratic forms and their isometries we have the following
specialized version of Hensel's Lemma.
 \begin{theorem}[Kneser's Hensel Lemma for quadratic forms]\label{hensellemma_kneser}
 Let $R$ be a local integral domain with maximal ideal $P$ and field
 of fractions $F$, let $(V,Q)$ and $(W,Q')$ be  finite dimensional quadratic spaces
 over $F$ with associated symmetric bilinear forms $b,b'$ and let
 $L,M$ be free submodules of $V,W$ respectively. 
Let $I\subseteq R$ be an ideal with $I Q'(M)\subseteq P$.

Let $f:L\to W$ be an $R$-linear map with $f(L)\subseteq M^\#:=\{z \in
W\mid b'(z,M)\subseteq R\}$ and write $\wtbp_f$ for the linear
map from $W$ to $\Hom_R(L,F)$ with $\wtbp_f(z)(x)=b'(f(x),z)$ for all
$x \in L$.

Assume that one has $Q'(f(x))\equiv Q(x)\bmod I$
for all $x \in L$
and $L^*=\wtbp_f(M)+PL^*$.

Then there is an $R$- linear map $f':L\to W$ with $f'(x) \equiv f(x) \bmod
I M$ and $Q'(f'(x))\equiv Q(x) \bmod IP$ for 
all $x \in L$. Moreover, $f'$ also satisfies the condition above,
i.e., we have  $L^*=\wtbp_{f'}(M)+PL^*$, so that we can iterate the
above procedure and improve the congruence for $Q'(f'(x))$ as much as desired.

If $R$ is  complete with respect to the $P$-adic topology there exists
an $R$-linear isometric map $\phi:(L,Q)\to (W,Q')$ satisfying
$\phi(x)\equiv f(x) \bmod IM$ for all $x \in L$. 
 \end{theorem}
 \begin{proof}
 The proof is an easy modification of the original proof of Kneser in
 \cite{kneserbook}:

There exists (see Lemma \ref{bilinearforms}) an  
$I$-valued $R$-bilinear form $\beta$ on $L$ such that
$\beta(x,x)=Q'(f(x))-Q(x)$ for all $x \in L$, and for fixed $y \in L$
we have the $R$-linear form $\lambda_y:L\to I$ given by
$\lambda_y(x)=\beta(x,y)$ for $x \in L$. 

By assumption we have
$\Hom_R(L,I)=IL^*=\wtbp_f(IM)+\Hom_R(L,IP)$, so there exists a
vector $g'(y)\in IM$ satisfying $\wtbp_f(g'(y))(x)\equiv -\beta(x,y)
\bmod IP$ for all $x \in L$. We let $y$ run through a basis
$(y_1,\ldots,y_r)$ of $L$ and denote by $g:L\to IM$ the $R$-linear map
with $g(y_j)=g'(y_j)$ for $1\le j \le r$, it satisfies $\wtbp_f(g(y))(x)\equiv -\beta(x,y)
\bmod IP$ for all $x,y \in L$.

Setting $f'=f+g$  we have $f' \equiv f \bmod IM$ and 
\begin{eqnarray*}
  Q'(f'(x))&=&Q'(f(x))+b'(f(x),g(x))+Q'(g(x))\\
&=&Q(x)+\beta(x,x)+\wtbp_f(g(x))(x)+Q'(g(x))\\
&\equiv&Q(x)\bmod IP,
\end{eqnarray*}
as asserted. Finally, for $x \in L$ and $z \in M$ we have
$b'(f'(x),z)=b'(f(x),z)+b'(g(x),z)$ with $b'(g(x),z) \in
b'(IM,M)\subseteq IQ'(M)\subseteq P$, so that $\wtbp_{f'}(z)\in
\wtbp_f(z)+PL^*$ holds for all $z \in M$, and we see that indeed
$L^*=\wtbp_{f'}(M)+PL^*$is true.

If $R$ is complete with respect to the $P$-adic topology, we can then
obtain a sequence of linear maps $f_k$ which converges to an isometric
linear map $\phi$.
\end{proof}
\begin{remark}
The proof remains valid if we omit the condition that $R$ should be an
integral domain and replace the field of fractions $F$ by the total
ring of fractions $\tilde{R}$. 
\end{remark}
\begin{corollary}\label{isometry_mod_P}
 With the notations as in the theorem assume that 
 $\rk(L)=\dim(V)=\dim(W)$, that $(V,Q)$ and $(W,Q')$ are regular, and that $R$ is complete with
 respect to the $P$-adic topology. 
Let $J\subseteq R$ be an ideal with $JQ(L^\#)\subseteq R$,
let $f:L\to K\subseteq W$ be an
 $R$-linear isomorphism satisfying $Q'(f(x))\equiv Q(x) \bmod
 JP^k$ for all $x \in L$ for some $k>0$.

Then there is an isometry $\phi:L\to K'$ with $\phi(x) \equiv f(x)
\bmod P^{k}J K^\#$ for all $x \in L$.
In particular, if $k>0$ satisfies $JP^k K^\#\subseteq K$ the
quadratic modules $(L,Q), (K,Q')$ are isometric.  

In particular, two free regular quadratic modules $(L,Q),(K,Q')$ over $R$ (i.e.,
$Q(L), Q'(K) \subseteq R, L^\#=L, K^\#=K$) are isometric if and only if
they are isometric modulo $P$. 
\end{corollary}
\begin{proof}
We take $M=K^\#$ in the theorem and obtain the assertion. 
\end{proof}
\begin{remark}
An ideal $J$ as above always exists:   Let $(v_1,\ldots,v_n)$ be a basis
of $L^\#$. The $R$-module $J'$ generated
by $Q(L^\#)$ is then generated by the $Q(v_j),b(v_i,v_j)$. 
The ideal $J$ generated by the product of the denominators of the
$Q(v_j),b(v_i,v_j)$, when these are written as fractions 
of elements of $R$, is then as required.

If $R$ is a discrete valuation ring, we can write $J'=P^{-\ell}$ for
some integer $\ell$ above
and have $J=P^\ell$. One calls the ideal $P^\ell$ the level of $L$.
\end{remark}

\begin{corollary}\label{hensel_for_evenunimodular}
Let $R$ be a local integral domain with maximal ideal $P$ and field
 of fractions $F$, let $(V,Q)$ and $(W,Q')$ be  finite dimensional quadratic spaces
 over $F$ with associated symmetric bilinear forms $b,b'$ and let
 $L,M$ be free submodules of $V,W$ respectively. 
 \begin{enumerate}
 \item Assume that $(L,Q)$ is a regular quadratic module (phrased
   differently:  An
   even unimodular $R$-lattice) and let $f:L\to M$ be an $R$-linear
   map satisfying $Q'(f(x))\equiv Q(x) \bmod P$ for all $x \in
   L$. 

Then there exists an $R$-linear isometric map $\phi:L\to M$
   with $\phi(x)\equiv f(x) \bmod PM$ for all $x \in L$. 

Equivalently,
   an isometry modulo $P$ of $L$ into $M$ can be lifted to an isometry
   from $L$ into $M$ whose reduction modulo $PM$ it is.
\item Assume that $(M,Q')$ is a regular quadratic module over $R$ (an
  even unimodular $R$-lattice) and let $f:L \to M$ be an injective $R$-linear map  
satisfying $Q'(f(x))\equiv Q(x) \bmod P$ for all $x \in
   L$ such that $f(L)$ is a primitive submodule of $M$. 

Then there exists an $R$-linear isometric map $\phi:L\to M$
   with $\phi(x)\equiv f(x) \bmod PM$ for all $x \in L$. Equivalently,
   an isometry modulo $P$ of $L$ into $M$ with primitive image can be lifted to an isometry
   from $L$ into $M$ whose reduction modulo $PM$ it is.
\end{enumerate}
\end{corollary}
\begin{proof}
  \begin{enumerate}
  \item  Let $(v_1,\ldots,v_r)$ be  an $R$-basis of $L$. The Gram
    matrix of $b'$ with respect to $(f(v_1),\ldots,f(v_r))$ has
    entries congruent modulo $P$ to the $b(v_i,v_j)$ in $R$ and hence
    determinant in $R^\times$. The $R$-lattice $(f(L),Q')$ is
    hence regular, and application of the previous corollary to $L$
    and $K=f(L)$ gives the assertion.
\item We have $PQ'(M)\subseteq P$ and $f(L) \subseteq M^\#=M$ by
  assumption. Let $\varphi \in L^*$. Since $f$ is injective we can
  write $\varphi={}^tf(\psi)$ for some linear form $\psi$ on
  $f(L)\subseteq M$, and since $f(L)$ is a direct
  summand of $M$ by assumption  we can extend $\psi$ to a linear form
  $\tilde{\psi}$ 
  on $M$. Since $(M,Q')$ is regular there exists $z\in M$ with
  $b'(z,y)=\tilde{\psi}(y)$ for all $y\in M$, in particular we have
  $b'(f(x),z)=\varphi(x)$ for all $x \in L$. We have shown that
  $\widetilde{b'}_f(M)=L^*$, so that application of the theorem gives
  the assertion.
  \end{enumerate}
\end{proof}
\begin{remark}
If we denote by $X$ the set of isometric $R$- embeddings of $L$ into $M$
in either of the above two situations and by $\bar{X}$ the set of
isometric $R/PR$-embeddings of $L/PL$ into $M/PM$, the corollary asserts that
the reduction map from $X$ to $\bar{X}$ is surjective. From a
geometric point of view this is often expressed by saying that $X$ is
smooth over $R$.
\end{remark}
\begin{corollary}\label{jordan_uniqueness}
Let $R$   be a complete valuation ring with quotient field $F$, let
$(V,Q)$ be a non degenerate quadratic space over $F$, let $L$ be an
$R$-lattice on $V$ which is totally even (see Remark \ref{jordan_remark}).

Then any Jordan decomposition $L=L'_0\perp \dots\perp L'_s$ of $L$ has
 even components $L'_i$ with $s=t$ and $L'_i$ is isometric to
$L_i$ for all $i$.
\end{corollary}
\begin{proof}
We may assume that $L_0$ is not zero.
Of course $L'_0$ has to be  even and nonzero if that is true
for $L_0$. Since $L_1$ is even, the radical of the quadratic space
$L/PL$ over $R/P$ with the modulo $P$ reduced quadratic form is equal
to its bilinear radical, and both 
$L_0/PL_0, L'_0/PL'_0$ are are isometric to the the quotient of $L/PL$
modulo its (bilinear or quadratic) 
radical, hence are isometric to each other. By the previous corollary
we have $L_0\cong L'_0$. By the Witt Theorem for local rings (Theorem
\ref{extension_theorem_localring}) we have that $L_1\perp \dots \perp
L_t \cong L'_1\perp\dots\perp L'_s$, and the assertion follows by induction
on the dimension of $V$.
\end{proof}
\begin{remark}
  It is known that the decomposition is not unique if $2 \in P$ and
  there are modular components which are not even.
  \end{remark}
\begin{corollary}
Let $R$ be a complete local ring with maximal ideal $M$ and $a\in
R^*$, assume that $2 \ne 0$ in $R$.

Then $a$ is a square in $R$ if and only if $a$ is congruent to a
square modulo $4P$.  
\end{corollary}
\begin{proof}
Let $F$ be the total ring of fractions of $R$ and $V=Fv$ be a free
module of rank $1$ over $F$ with quadratic form $Q$ given by
$Q(cv)=ac^2$, let $L=Rv$. Let similarly $W=Fw$ with $Q'(w)=1$ and
$K=Rw$.
We have $L^\#=R \frac{v}{2}$ with $4PQ(L^\#)\subseteq P$ and $4P L^\#
\subseteq L$. If $a \equiv b^2 \bmod 4P$ holds, the linear map $f$
given by $f(v)=bw$ satisfies $Q'(f(x)) \equiv Q(x) \bmod 4P$ for all
$x \in L$ and can be lifted to an isometry, which implies that $a$ is
a square in $R$. The other direction is trivial.   
\end{proof}
\begin{remark}
The assertion of the corollary can also be proven using the first
version of Hensel's Lemma for the polynomial $X^2-a$.  
\end{remark}
\begin{theorem}
Let $R$ be a complete local ring with maximal ideal $P$,
residue field $k=R/P$ and  field of quotients $F$. Let $(M,Q)$ be a
free  quadratic module over $R$ with associated symmetric
bilinear form $b$ and denote by $\bar{M}$ the
$k$-vector space $M/PM$ with the modulo $P$ reduced quadratic form
$\bar{Q}$.

Let $a\in R^\times$ be such that $\bar{a} =a+P \in k$ is represented
by $(\bar{M},\bar{Q})$. If $\cha(k)=2$ assume in addition that there
is such a representation by a vector generating a regularly embedded
subspace of $\bar{M}$. Then $a$ is represented by $M$.

In particular, if $(M,Q)$ is regular of rank $\ge 2$ and $k$ is finite, all units in
$R$ are represented by $(M,Q)$.   
\end{theorem}
\begin{proof}
By assumption, there exist $y\in M$ with $Q(y)+P=\bar{a}$ and $z \in
M$ with $b(z,y)\in R^\times$. Let $L=Rv$ be a free module of rank $1$
with quadratic form $Q_1$ given by $Q_1(v)=a$. Then Theorem
\ref{hensellemma_kneser} is applicable to $L,M$ with $I=P$, and we
obtain an isometric linear map $\phi:L\to M$ which gives $x=\phi(v)\in
M$ with $Q(x)=a$.

If $(M,Q)$ is regular, the reduction $\bar{M}$ is regular too, and if
$k$ is finite it represents all of $k$ by Corollary
\ref{universality_finitefields}, and  a representing one dimensional
subspace is regularly embedded since $\bar{M}$ is regular.
\end{proof}
\begin{theorem}
  With notations as above assume that $(\bar{M},\bar{Q})$ contains a
  hyperbolic plane. 

Then $M$ contains a regular hyperbolic plane $H$ and can be decomposed
as $M=M'\perp H$.

In particular, $M$ splits off a regular hyperbolic plane if $k$ is
finite, $\rk(M)\ge 3$ and $(\bar{M},\bar{Q})$ is regular or half regular.  
\end{theorem}
\begin{proof}
  This is again a direct consequence of Theorem
  \ref{hensellemma_kneser}. The assertion for finite $k$ follows from
  Corollary \ref{isotropy_finitefields} which guarantees an isotropic
  vector in $\bar{M}$, by the
  (half)-regularity assumption the space generated by it is regularly
  embedded into $\bar{M}$. Theorem \ref{isotropic_hyperbolic}  gives
  then a hyperbolic plane contained in $\bar{M}$.
\end{proof}
\section{Maximal lattices}
\begin{theorem}\label{anisotropic_maximal}
  Let $R$ be a complete valuation ring with field of fractions $F$,
  maximal ideal $P$ and residue field $k$.

Let $(V,Q)$ be a regular or half regular anisotropic quadratic space  over $F$, let $I$ be an
ideal in $R$.

Then $M_I:=\{x\in V \mid Q(x) \in I\}$ is an $R$-module.

If $R$ is a discrete valuation ring, $M_I$ is the unique $I$-maximal
lattice on $V$. 
\end{theorem}
\begin{proof}
Let $v,w \in M_I$ and assume that $a:=b(v,w)\not\in I$ and hence
$Ra\not\subseteq I$ holds. Since
$R$ is a valuation ring, this implies that we have $I\subsetneq Ra$
and hence $Q(v)\in Pa, Q(w) \in Pa$. If $v,w$ were linearly dependent
$b(v,w)\in I $ would follow, since then one of the vectors were an
$R$-multiple of the other one. The reduction of the free $R$-module $Rv+Rw$ with the
quadratic form $a^{-1}Q$ modulo $P$ is then a regular quadratic space
over $k$ with determinant $-1$, hence isometric to a hyperbolic plane.
By Corollary \ref{isometry_mod_P}, $Rv+Rw$ with $a^{-1}Q$ is a regular
hyperbolic plane, in particular isotropic, which contradicts the
assumption that $(V,Q)$ is anisotropic.

So we have $b(v,w) \in I$ for all $v,w \in M_I$, and we see that $M_I$
is an $R$-module.

If $R$ is a discrete valuation ring, it is a principal ideal domain
and any $R$-submodule of a finite dimensional $F$-vector space is
free. That $M_I$ has rank $\dim(V)$ is trivial.
\end{proof}
\begin{remark}
  \begin{enumerate}
  \item It is not clear whether $M_I$ is free 
\item If $R$ is a discrete valuation ring 
  we write $M_k:=M_{P^k}$  for $k \in \N_0$.
  \end{enumerate}
\end{remark}
\begin{corollary}
Let $R$ be a complete discrete valuation ring and $(V,Q)$ a regular or
half regular quadratic
space over its field of fractions $F$, let $I\subseteq R$ be an ideal.

Then all   $I$-maximal lattices on $V$ are isometric.
\end{corollary}
\begin{proof}
By scaling the quadratic form we can assume $I=R$.
If $\Lambda$ is a maximal lattice on $V$ we have  $\Lambda=H\perp
\Lambda_1$, where $H$ is regular hyperbolic of rank $2 \ind(V,Q)$ and
$\Lambda_1$ is maximal on the anisotropic space $V_1$ by Theorem
\ref{hyperbolic_maximal}. Since the
isometry class of $V_1$ is uniquely determined by  Theorem
\ref{Witt_decomposition}, the assertion follows from the previous theorem.
\end{proof}
\begin{theorem}\label{uniquespace_dim2}
 Let $R$ be a complete discrete valuation ring with field of
   fractions $F$ and finite residue field $k$, let $\pi$ be a prime
   element of $R$. 

Then there is up to isometry precisely one anisotropic quadratic space
of dimension $2$ over $F$ on which the maximal lattice is regular. 
\end{theorem}
\begin{proof}
  For each such space  (with $M_k$ as above) the quadratic space
  $M_0/\pi M_0$ is the unique anisotropic regular quadratic space of
  dimension $2$ over $k$. This determines the isometry class of $(M_0,Q)$
  and hence that of  $(V,Q)$.
\end{proof}
\begin{theorem}\label{anisotropic_local}
   Let $R$ be a complete discrete valuation ring with field of
   fractions $F$ and finite residue field $k$, let $\pi$ be a prime
   element of $R$.

Then every regular or half regular anisotropic quadratic space $(V,Q)$
over $F$ has dimension $\le 4$, and up to isometry there exists precisely one anisotropic regular
quadratic space of dimension $4$ over $F$. This space has determinant
$1$ and isuniversal.
\end{theorem}
\begin{proof}
With notations as in the previous theorem  the  quadratic forms $Q$ respectively $\pi^{-1}Q$ 
induce $k$-valued quadratic forms  on the $k$ vector spaces $M_0/M_1$ and $M_1/M_2$ 
which
are anisotropic. Hence these $k$-vector spaces are of dimension $\le 2$ by Corollary
\ref{isotropy_finitefields}.
Since one has
\begin{equation*}
  \dim_F(V)=\dim_k(M_0/M_2)=\dim_k(M_0/M_1)+\dim_k(M_1/M_2), 
\end{equation*}
the first  part of the assertion follows.

If $V$ has dimension $4$, we see that both $M_0/M_1$ and $M_1/M_2$
have dimension $2$ over $k$. We let $L\subseteq M_0$ be a regular preimage of rank
$2$ of $M_0/M_1$ and see that $L$ splits $M_0$, so one has $M_0=L
\perp L'$ with a sublattice $L'\subseteq M_1$ of rank $2$. Since $L$
is anisotropic modulo $P$ we have $M_1\cap L=\pi L, M_1=L'\perp \pi
L$. This implies that $M_1/M_2$ and $L'/\pi L'$ with the quadratic
forms induced by $\pi^{-1} Q$ are regular $2$- dimensional and
isometric over $k$, which determines the isometry class of $(L',
\pi^{-1}Q)$ and hence that of $(M_0,Q)$ and of $(V,Q)$. The existence
of a space as described also follows.
\end{proof}

\chapter{Quadratic Forms over Global Fields and their Integers}\label{globalchapter}
We are now ready to see how the properties of quadratic forms over
number fields and number rings (ore slightly more generally global
fields and rings of integers in these) can be studied with the help of the
local theory developed in the previous chapter, in particular we will
prove the local-global principle of Minkowski and Hasse.

In this chapter $F$ will be a global field, i.e., an algebraic number
field or an algebraic function field, with $\cha(F)\ne 2$. We write
$\Sigma_F$ for the set of places of $F$ 
(i.e., the set of equivalence classes of non trivial valuations) and
identify $\Sigma_F$ with a set of representatives for these
equivalence classes, using $\vert a\vert_v$ to denote the $v$-value of
$a\in F$. For $v\in \Sigma_F$ we denote by $F_v$ the
completion of $F$ with respect to the valuation $v$ and use $\vert
\cdot\vert_v$ for the extension of the valuation to $F_v$ as well.
The
group of ideles of $F$ is denoted by $J_F$, with $J_F^2$ denoting the
subgroup of squares of ideles. $F^\times$ is identified with the set
of principal ideles. The ring of adeles is $\A_F$ and $F$ is
identified with its image in $\A_F$ under the diagonal embedding.

\smallskip
If $F$ is a number field we let $R$ be 
its ring of integers. More generally, for some
finite set $T$ of places of $F$ containing all archimedean
places we consider the ring 
$R=R_T:=\{a\in F^\times\mid \vert a\vert_v \le 1\text{ for all }
v\not\in T\}$ and call it a ring of integers in  
$F$ or the $T$-integers in $F$. If $F$ is a global function field we
consider similarly $R_T$ as above for any finite non empty set of
places of $F$ and call these rings again rings of integers in $F$. 
We denote then by $R_v\subseteq F_v$ the completion of $R$ with
respect to $v$ for $v\not \in T$ (equivalently: The closure of $R$ in
$F_v$) and set $R_v=F_v$ for $v \in T$.
 
The most important special case will
of course be $F=\Q$ with $R=\Z$ or $R=\Z[\frac{1}{a}]$ for some
nonzero $a\in \Z$.

\section{The local-global-principle of Minkowski and Hasse}
\begin{theorem}[Strong Minkowski-Hasse Theorem]
Let $(V,Q)$ be a regular quadratic space of dimension $n$ over $F$ and assume that the
completions $V_v=F_v\otimes V$ are isotropic with respect to the
natural extension of $Q$ to $V_v$ for all non trivial valuations $v$
of $F$.

Then $(V,Q)$ is isotropic. 
\end{theorem}
\begin{proof}
 We will use without proof three results of algebraic number theory,
 the first two of which
are special cases of theorems from global class field theory. The last
one is actually valid for all fields,
for $F=\Q$ it is an easy consequence of the chinese remainder theorem.

 \begin{itemize}
 \item Let $a \in F$ be a square in all completions $F_v$. Then $a$ is
   a square in $F$.
\item Let $E/F$ be a quadratic extension of $F$ and assume that $a\in
  F$ is a norm in all extensions $E_w/F_v$ for valuations $v$ of $F$
  with extension $w$ to $E$.
\parindent=0pt

Then $a$ is a norm of some $b\in E$.
\item (Weak approximation theorem) Let $T$ be a finite set of
  inequivalent non trivial valuations of $F$ and let $a_v\in F_v$ be given for all
  $v \in T$. Then for any $\epsilon>0$ there exists $a \in F$ with
  $\vert a-a_v\vert_v<\epsilon$ for all $v \in T$.
 \end{itemize}

We can assume that $V$ represents $1$, so that for $n=2$ there is an orthogonal
basis $(x_1,x_2)$ of $V$ with $Q(b_1w_1+b_2w_2)=b_1^2-db_2^2$ for some
$d \in F^\times$. If $(V_v,Q)$ is isotropic for all $v$ we see that
$d$ is a square in all $F_v$, hence in $F$, so $(V,Q)$ is hyperbolic
and hence isotropic.

For $n=3$ we choose again an orthogonal  basis $(x_1,x_2,x_3)$ of $V$
with $Q(x_1)=1, Q(x_2)=-d, Q(x_3)=-c$. If $b$ is a square, $(V,Q)$ is
hyperbolic and we are done. Otherwise we put $U=F x_1+Fx_2, W=Fx_3$
Since $(V_v,Q)$ is isotropic for all $v$, the equation $a_1^2
-da_2^2=c$ is solvable in $F_v$ for all $v$, so that $c$ is a norm in
all local extensions $E_w/F_v$ with $E=F(\sqrt{d})$. Then $c$ is a norm
in $E/F$, so $c=a_1^2-da_2^2$ is solvable in $F$, so $(V,Q)$ is
isotropic.

For $n = 4$ assume first that $\det(V,Q)$ is a square. Over $F_v$
the space $(V_v,Q)$ is isotropic, hence spits off a hyperbolic plane,
and since the determinant is a square, the complement in this
splitting is a hyperbolic
plane as well. But then $V_v$ contains a two dimensional totally
isotropic subspace, so for any $3$-dimensional $U\subseteq V$ all
completions $U_v$ are isotropic. Since the ternary case of the theorem
is already established, we are done in this case.

Assume now that $n=4$ and that $\det(V,Q)$ is not a square, let
$E=F(\sqrt{d})$. The vector space $V_E:=E\otimes V$ (with the natural
extension of $Q$ to it) over $E$ has then also the property that all
its completions are isotropic. Since its determinant is a square, it
is isotropic. An isotropic vector of $V_E$ can be written as
$z+\sqrt{d}y$ with $z,y \in V$, and $Q(z+\sqrt{d}y)=0$ implies
$b(z,y)=0, Q(z)=-dQ(y)$. If $Q(z)=Q(y)=0$, the space $(V,Q)$ is
isotropic, otherwise we have a basis $(x_1=z,x_2=y,x_3,x_4)$ of $V$
and put $U=Fz+Fy, W=Fx_3+Fx_4$. Then we see that
$d=\det(V,Q)=\det(U,Q)\det(W,Q)=-d \det(W,Q)$, so $-\det(W,Q)$ is a
square and $(W,Q)$ is a hyperbolic plane. So $(V,Q)$ is isotropic in
that case as well.

We now start induction on $n=\dim(V)$.
Let $n>4$ and assume the assertion has been proven for
$\dim(V)<n$. Write $V=U\perp W$, where $U=Fx_1+Fx_2$ has dimension
$2$. If $W_v$ is isotropic for all non trivial valuations $v$ of $F$,
we are done by the inductive assumption. Otherwise, the set $T$ of
$v\in \Sigma_F$ for which $(W_v,Q)$ is anisotropic is finite, since a
maximal lattice on $W$ is unimodular at almost all $v$ and of rank
$\ge 3$. Since $(V_v,Q)$ is isotropic for all $v$, we can find $c_v
\in F_v$ with $c_v =Q(a_vu_1+b_v u_2)\in Q(U_v), -c_v \in Q(W_v)$ for all $v \in T$. By
Hensel's lemma and its corollary for squares in complete local rings,
any $c \in F$ with $\vert c-c_v\vert_v$ small enough will be a
multiple of $c_v$ by a square
in $F_v$.  By the weak approximation theorem we can find $a,b \in F$
which are arbitrarily close to $a_v,b_v$ respectively at all $v \in T$
and make $\vert Q(au_1+bu_2)-a_v\vert_v$ as small as we want for all
$v \in T$. This gives us a vector $x \in U$ with $Q(x)=c \in
c_v(F_v^\times)^2$ for all $v \in T$, so the space $Fx+W$ is isotropic
at all $v \in T$ and hence at all $v \in \Sigma_F$. By the inductive
assumption it is isotropic, so $(V,Q)$ is isotropic.
\end{proof}

As an immediate consequence we obtain
\begin{theorem}[Weak Minkowski Hasse theorem]
Let $(V,Q)$ and $(W,Q')$ be regular quadratic spaces over $F$ and
assume that $(W_v,Q'_v)$ is represented by $(V_v,Q_v)$  for all $v \in
\Sigma_F$, i.e., that there exist isometric embeddings $\phi_v:W_v\to V_v$ for
all $v \in \Sigma_F$. 

Then $(W,Q')$ is represented by $(V,Q)$.

In particular, if 
$(V_v,Q)$ and $(W_v,Q')$ are isometric for all $v\in
\Sigma_F$,the quadratic spaces $(V,Q)$ and $(W,Q')$ are isometric.
\end{theorem}
\begin{proof}
We prove this by induction on $n=\dim(W)$. For $n=1$ we have $W=Fx$
with $Q(x)=a$ and  $a \in
Q(V_v)$ for all $v$, so the space $V\perp Fy$  with $Q(y)=-a$ is
isotropic over all $F_v$, hence isotropic over $F$ by the previous
theorem, and it follows that $V$ represents $a$.

Let $n>1$ and assume that the assertion is proved for all $W$ with
$\dim(W)<n$. Let $0 \ne a \in Q(W)$. Then $a$ is represented  by all
$(V_v,Q)$, hence represented by $(V,Q)$ by our result for 
the case $n=1$. We can then write $V=Fx\perp V'$ and $W=Fy\perp W'$
with $Q(x)=Q(y)=a$. If $\phi_v:W_v\to V_v$ is an isometric embedding,
we can find by Witt's theorem a $\psi_v\in O(V_v,Q)$ with $\psi_v\circ
\phi_v(y)=x$ and hence $\psi_v\circ\phi_v(W') \subseteq V'_v$. Thus $W'$
is represented by all $(V'_v,Q)$ and therefore by $(V',Q)$ by the
inductive assumption, and $W=Fy\perp W'$ is represented by $V=Fx\perp V'$.
\end{proof}

\begin{corollary}[Theorem of Meyer]
Let $(V,Q)$ be a quadratic space over $\Q$ which is indefinite (i.e.,
isotropic over $\R$). Then $(V,Q)$ is isotropic. 
\end{corollary}
\begin{proof}
Since every (regular) quadratic space over $\Q_p$ (for a prime $p$) is
isotropic, this follows from the strong Minkowski-Hasse theorem.  
\end{proof}
\begin{corollary}
 There is a well defined injective group homomorphism $L:W(F)-to \prod_{v \in \Sigma_F}W(F_v)$
 satisfying $L([(V,Q)]=([(V_v,Q)])_{v\in \Sigma_F}$ for regular
 quadratic spaces $(V,Q)$ over $F$.  
\end{corollary}
\begin{proof}
Obviously, if $(V,Q)$ and $(W,Q')$ are Witt equivalent, their
completions at$v \in \Sigma_v$ are Witt equivalent as well, so we can
define the map as given above. It is, again obviously, a group
homomorphism, and that the kernel is $\{0\}$ follows from the weak
Minkowski-Hasse theorem.  
\end{proof}
\begin{remark}
  \begin{enumerate}
  \item One can show that a tuple $([(V_v,Q_v)])_{v \in \Sigma_F}$ of Witt classes of quadratic spaces of equal
dimension on the right hand side is in the image of our map $L$ if and
only if 
\begin{itemize}
\item There exists $d \in F^\times$ with $d
  (F_v^\times)^2=\det(V_v,Q_v)$ for all $v \in \Sigma_F$
\item Almost all of the Hasse invariants of the $(V_v,Q_v)$ are $1$
  and their product equals $1$.
\end{itemize}
The necessity of these conditions follows from the Hilbert reciprocity
law, for their sufficiency see O'Meara's book \cite{omeara}.
\item A precise description of the Witt group of $\Q$ along with a
 different and more elementary proof of the weak Hasse-Minkowski
 theorem can be found in the books by Cassels, Kneser, Lam,
 W. Scharlau \cite{cassels,kneserbook,lam_qf,scharlaubook}.
  \end{enumerate}
\end{remark}

\section{Lattices over $\Z$}\label{Z-lattice_section}

\begin{definition}
Let $\Lambda_1, \Lambda_2$ be $R$-lattices on regular quadratic spaces
$(V_1,Q_1)$, $(V_2,Q_2)$ over $F$.

One says that $(\Lambda_2,Q_2)$ is in the genus of $(\Lambda_1,Q_1)$
if their completions $\Lambda_i\otimes R_v$ are isometric as quadratic
$R_v$-modules for all $v \in \Sigma_F$, one writes then
$(\Lambda_2,Q_2) \in \gen((\Lambda_1,Q_1))$ (or vice versa).   
\end{definition}
\begin{remark}
  \begin{enumerate}
  \item 
If $(\Lambda_1,Q_1), (\Lambda_2,Q_2)$ are in the same genus, their
underlying quadratic spaces are isometric by the weak Hasse-Minkowski
theorem. We will therefore in general assume that lattices in the same
genus have the same underlying quadratic space. 
\item An integral local-global principle is not true: Lattices may be
  in the same genus without being isometric.
A simple example for this are the $\Z$-lattices with Gram matrix
$\bigl(\begin{smallmatrix}2 &0\\0&110\end{smallmatrix}\bigr)$ 
respectively 
$\bigl(\begin{smallmatrix}10 &0\\0&22\end{smallmatrix}\bigr)$.
\item A genus of quadratic lattices consists of full isometry
  classes. 
\item The adelic orthogonal group $O_{(V,Q)}(\A_F)$ of the quadratic
  space $(V,Q)$ operates transitively on the set of lattices in the
  genus of a given lattice $\Lambda$  on $V$: A given
  $\phi=(\phi_v)_{v\in \Sigma_F}\in O_{(V,Q)}(\A_F)$ satisfies (by
  definition of the adele group)
  $\phi_v(\Lambda_v)=\Lambda_v$ for almost all $v \in \Sigma_F$, so
  there exists (by Lemma \ref{lattices_localglobal_dedekind}) a unique lattice $\Lambda'=\phi(\Lambda)$ on $V$
  with $\Lambda'_v=\phi_v(\Lambda_v)$ for all $v \in \Sigma_F$, which
  is then in the genus of $\Lambda$. By definition, all lattices on $V$ in the
  genus of $\Lambda$ can be obtained is this way. The stabilizer of
  the isometry class of $\Lambda$ under this action is the set
  $O_{(V,Q)}(F)O_{V,Q}(\A_F,\Lambda)$, where  $O_{V,Q}(\A_F,\Lambda)$ denotes the
  set of all adeles $(\phi_v)_v$ with $\phi_v(\Lambda_v)=\Lambda_v$
  for all $v \in \Sigma_F$, i.e., the stabilizer of the lattice
  $\Lambda$ under this group action. The set of isometry classes in
  the genus of $\Lambda$ is then in bijection with the double cosets
  $O_{(V,Q)}(F)\psi O_{V,Q}(\A_F,\Lambda)$ in $O_{(V,Q)}(\A_F)$.
  \end{enumerate}
\end{remark}
\begin{lemma}
  Let $\Lambda_1,\Lambda_2$ be lattices on the regular quadratic space
  $(V,Q)$ over $F$ in the same genus. Then $(d_1R^\times)^2:=\det(\Lambda_1,Q)=\det(\Lambda_2,Q)=:d_2(R^\times)^2$.
\end{lemma}
\begin{proof}
By definition of the genus we must have
$d_1(R_v^\times)^2=d_2(R_v^\times)^2$ for all $v \in \Sigma_F$, so in
particular $d_1R_v^\times =d_2R_v^\times$ for all $v \in \Sigma_F$,
which implies $d_1R=d_2R$. Since $d_1,d_2$ differ by the  square of a $c
\in F^\times$, we must in fact have $d_1(R^\times)^2=d_2/R^\times)^2$
as asserted.
\end{proof}
We turn now attention to the special case $R=\Z, F=\Q$.
The following theorem generalizes a result which we already obtained
for positive definite lattices in chapter $3$.
\begin{theorem}[Hermite]
 Let $\Lambda$ be a $\Z$-lattice on the regular quadratic space
 $(V,Q)$ over $\Q$ of dimension $n$ with associated symmetric bilinear
 forms $b, B=\frac{1}{2}b$, let $d=\det_B(\Lambda)$,  let 
\begin{equation*}
\mu:=\mu(\Lambda):=\min\{\vert Q(x)\vert \mid x \in \Lambda
 \setminus \{0\}.
\end{equation*}
Then 
\begin{equation*}
\mu(\Lambda)\le \bigl(\frac{4}{3}\bigr)^{\frac{n-1}{2}} \vert d \vert^{\frac{1}{n}}.   
\end{equation*}
\end{theorem}
\begin{proof}
By scaling the quadratic form with a suitable integer we may assume
$Q(\Lambda) \in \Z$ and see that the minimum $\mu$ is indeed assumed.
If $\mu=0$, we are done, otherwise we can proceed as in the proof of
Theorem \ref{hermites_theorem_definite}.  
\end{proof}
\begin{theorem}\label{finiteness_classnumber}
 There are only finitely many isometry classes of integral quadratic
 $\Z$-lattices $(\Lambda,Q)$ of rank $n$ and fixed determinant $d$. 
\end{theorem}
\begin{proof}
The assertion is trivial for $n=1$. We let $n>1$ and assume the assertion to be
proven for lattices of rank $<n$. Let $0\ne x\in
\Lambda$ be a primitive vector with $\vert Q(x) \vert =\mu$ and $B(x,
\Lambda)=a\Z$.
If $\mu \ne 0$ we put $M=Rx$ and $N=(\Z x)^\perp=(Fx)^\perp \cap \Lambda$.
By Theorem \ref{duals_and_indices} b) we have $\det(N,Q)a^2=\mu d$, in
particular $\det(N,Q)$ is a divisor of $\mu d$ and hence bounded in
absolute value by a constant times a power of $d$. By the inductive
assumption, there is only a finite number of possible isometry classes
of $(N,Q)$. Since we have $M\perp N \subseteq \Lambda \subseteq
M^\#\perp N^\#$ and $M\perp N$ has finite  index in $M^\#\perp N^\#$,
there are only finitely many possibilities for the isometry class of
$\Lambda$.

In the case $\mu=0$ we look at the Gram matrix of $(\Lambda,Q)$ with
respect to a basis of $\Lambda$  beginning with the vector $x$ and see
that we can divide both the first column and the first row by $a$, so
that $a^2 \mid d$ must hold.

We find then $y' \in \Lambda$ with $B(x,y')=a$ with
$B(x,\Lambda)=a\Z$.  Replacing $y'$ by a suitable $y=y'-cx$
we can achieve $\vert Q(y)\vert=\vert Q(y')-2ca\vert \le a$, so the
rank $2$-sublattice $M=\Z x+\Z y$ has  a Gram matrix with entries
bounded by $d$, so that there are only finitely many possible isometry
classes for $M$. Again by b) of Theorem \ref{duals_and_indices} we see
that $N=(Fx+Fy)^\perp\cap \Lambda$ has determinant and hence by the
inductive assumption isometry class from a finite set.
We obtain as above that there are only finitely many possibilities
for the isometry class of $\Lambda$. 
\end{proof}
\begin{remark}
  By the work of Humbert \cite{humbert} the Theorem remains valid if
  one replaces $\Z$ by the ring of integers $\fo$ of a number field
  and the determinant by the norm of the volume ideal of an
  $\fo$-lattice  Humbert's proof is for classes of symmetric matrices
  (equivalently, for free lattices) but carries over to the more
  general situation. 
\end{remark}
\begin{corollary}
The number of isometry classes in a fixed genus of $\Z$-lattices is finite.
\end{corollary}
\begin{proof}
If $\Lambda_1, \Lambda_2$ are in the same genus, we assume them to be
on the same space and have
$\det((\Lambda_1,Q))=c^2\det((\Lambda_2,Q))$ for some $c \in
\Q$. Since the completions are isometric over all $\Z_p$, the number
$c$ is a unit in all $\Z_p$, hence must be $\pm 1$, so that the
lattices have the same determinant.
By the previous theorem there can be only finitely many isometry
classes of this determinant.   
\end{proof}
\begin{remark}
  By the previous remark the assertion of the Corollary  remains valid if
  one replaces $\Z$ by the ring of integers $\fo$ of a number field.
\end{remark}
\begin{example} For simplicity of notation we will write $\langle
  a_1,\ldots,a_n \rangle$ for a lattice with diagonal Gram matrix with
  entries $a_1, \ldots, a_n$ in the diagonal and $I_n=\langle
  1,\ldots,1\rangle, -I_n=\langle -1,\ldots,-1\rangle$.
  \begin{enumerate}
  \item for $2\le n \le 5$ our bound for the minimum of $\Lambda$ is
    $<2$ for $d=1$.

For $\mu=1$, the one dimensional sublattice $\Z x$ generated by a
minimal vector splits off orthogonally, and its complement has again
determinant $\pm 1$ with respect to $B$, so that $\Lambda$ has an
orthgonal basis of vectors $x_i$ with $Q(x_i)=\pm 1$. If both signs
occurred here, the lattice was isotropic and had minimum $0$, so
$(\Lambda,Q)$ is up to multiplication of the quadratic form by $-1$
isometric to the cube lattice $I_n$ generated by the standard basis
vectors of $\R^n$ equipped with the standard scalar product as $B$.

For $\mu=0$ we see that the lattice $M=\Z x+\Z y$ in the proof of the theorem is
a regular hyperbolic plane $H$ or has Gram matrix $\bigl(
\begin{smallmatrix}
  0&1\\1&1
\end{smallmatrix}\bigr)$, which is equivalent to $\langle 1,-1
\rangle$. In both cases it splits off orthogonally.  For $n=2$ this
finishes the 
discussion, for $n=3$ the orthogonal complement is spanned by a vector
$z$ with $Q(z)=\pm 1$. Since $H \perp \langle \pm 1 \rangle$ is
isometric to $\langle 1,-1,\pm 1\rangle$ we have $\Lambda$ isometric to
$\langle 1,1,-1\rangle$ or to $\langle 1,-1,-1\rangle$. For $n=4 $ we
can have $\Lambda $ isometric to $H \perp H$ or to $I_r \perp -I_s$
with $r+s=4$ and $0 \ne r \ne s$. For $n=5$ we get $I_r \perp -I_s$
with $r+s=5, 0\ne r\ne s$. We see that $H$ and $H\perp H$ are
distinguished from the other possibilities by the $2$-adic behaviour,
namely by $Q(x)=B(x,x)$ having only even values, and that the
remaining cases are distinguished by rank and signature at the real
place. In any case we see that the $2$-adic information and the
signature determine the genus and that each genus consists of a single
isometry class.
\item For $2 \le n \le 4$ and $d=2$ the bound for the minimum is again
  $1$.
By similar arguments as above we obtain the possibilities $\langle
\pm 1,\pm 2\rangle$ with $\langle 1,-2\rangle$ isometric to $\langle
-1,2\rangle$,  $H \perp \langle \pm 2 \rangle,
\langle \pm 1, \pm 1, \pm 2\rangle, \langle \pm 1, \pm 1,
\pm 1, \pm 2\rangle$, where the isometry class depends only on the
number of $+$-signs and the number of $-$-signs but not  on whether the sign
occurs in front of a $1$ or a $2$. This shows that the genus of the
lattice is determined by the signature at the real place and the
$2$-adic class (even values for $Q(x)$ or odd and even values) and that each
genus consists  of a single isometry class.
  \end{enumerate}
\end{example}
\section{Representations}
We recall that an isometric embedding $\phi:(M,Q') \to (\Lambda, Q)$
of quadratic modules over the ring $R$
is also called a representation of $(M,Q')$ by $(\Lambda,Q)$, it is a
primitive representation if the image of $\phi$ is a direct summand in
$\Lambda$. If $R$ is a principal ideal domain with field of quotients
$F$ and $V,W$ are the $F$-vector spaces obtained from $\Lambda, M$
respectively by extending scalars to $F$, a representation $\phi$ as
above is primitive if and only if $\phi(M)=\Lambda \cap \phi(W) $
holds. The same is true if $R$ is the ring of integers of the global
field $F$, since then the $R$-module $\phi(M)$ is a direct summand in
$\Lambda$ if and only if its completion at $v \in \Sigma_F$ is a
direct summand in $\Lambda_v$ for all $v\in \Sigma_F$ and the
intersection condition localizes as well. We say that $(M,Q')$ is
represented (primitively) by $(\Lambda,Q)$ if there exists a
representation $\phi:(M,Q') \to (\Lambda, Q)$.

\begin{theorem}\label{representations_localglobal}
Let $(M,Q'), (\Lambda,Q)$ be $R$-lattices on the quadratic spaces $W,V$ and assume that for all $v
\in \Sigma_F$ there is a representation $\phi_v:M_v \to
\Lambda_v$. Then there is a representation $\phi:M \to \Lambda'$,
where $\Lambda'$ is a lattice on the space $V$ of $\Lambda$ in the
genus of $\Lambda$. If all the representations $\phi_v$ are primitive,
the representation $\phi$ can also be chosen to be primitive.
\end{theorem}
\begin{proof}  By the Minkowski-Hasse theorem there exists an isometric embedding (a representation)
  $\phi:W\to V$ so that we can assume that $M$ is a lattice in $V$
  (but in general not on $V$). For almost all $v\in \Sigma_F$ we have
  $M_v \subseteq \Lambda_v$, so the set $T$ of places $v$ where $M_v
  \not\subseteq \Lambda_v$ is finite. By assumption, for each such $v$
  there exists a lattice $N_v\subseteq \Lambda_v$ which is isometric
  to $M_v$, and by Witt's extension theorem there an isometry can be
  extended to a map
  $\psi_v \in O_V(F_v)$ with $\psi_v(N_v)=M_v$. We let $\Lambda'$ be the
  lattice on $V$ with $\Lambda_v'=\Lambda_v$ for all $v \not\in T$
  and $\Lambda_v'=\psi_v(\Lambda)$ for all $v \in T$ and have
  $M_v\subseteq \Lambda'_v$ for all $v \in \Sigma_F$, hence
  $M\subseteq \Lambda$.

If all the $\phi_v$ are primitive, we enlarge $T$ to a still finite
set $T^\ast$ by demanding $M_v$
to be a primitive sublattice of $\Lambda_v$ for all $v \not\in T$ and
choose primitive sublattices $N_v$ of $\Lambda_v$ for all $v \in T^*$,
leaving the rest of the proof unchanged.
\end{proof}

\begin{corollary}
  Let $(\Lambda,Q)$ be an $R$-lattice for which the genus contains only
  one isometry class. Then $\Lambda$ represents (primitively) all
  $R$-lattices $(M,Q')$ which are represented (primitively) locally
  everywhere by $\Lambda$.

In particular:
\begin{enumerate}
\item (Theorem of Euler) The integer $a \in \Z$ is a sum of two coprime integral squares
  if and only if all its prime factors are congruent to $1$ mod $4$
  and it is not divisible by $4$. It is a sum of two arbitrary
  integral squares if and only if all primes $\equiv 3 \bmod 4$ divide
  $a$ to an even power.
\item (Theorem of Gauß-Legendre) The integer $a \in \Z$ is a sum of three integral squares if and
  only if it is not of the form $4^\nu(8k+7)$ with $\nu, k \in
  \N_0$. It is a sum of three coprime integers if in addition it is
  not divisible by $4$.
\item (Theorem of Lagrange) All positive integers $a$ are sums of four
  integral squares. A representation by relatively prime integers
  exists if and only if $a$ is not divisible by $8$.
\end{enumerate}
\end{corollary}
\begin{proof}
 The first statement is a trivial consequence of the theorem. For the
 classical results about representation of integers as sums of squares
 we have to check that the conditions given are equivalent to the
 condition that $a$ is represented (primitively) locally everywhere.
 This is reduced by Hensel's lemma to routine computations modulo $8$
 for representation over $\Z_2$ and modulo $p$ for $\Z_p$ with $p$ odd. 
\end{proof}
\begin{remark}
 We can obtain analogous results for representation of integers in the form
 $x^2+2y^2$ or as $x^2+y^2+2z^2$. If the class number of the genus of
 the representing lattice is bigger than $1$ it is possible to obtain
 results on representation of sufficiently large numbers by analytic
 methods.
\end{remark}
\section{Lattices over $\Z$, continued}\label{Z-lattice_section_cont}
\begin{definition}
  A positive definite $\Z$-lattice $(\Lambda,Q)$ is called
  decomposable if it contains nontrivial sublattices $M,N$ with $\Lambda=M\perp
  N$, indecomposable otherwise.
\end{definition}
\begin{theorem}[Eichler, Kneser]
Any positive definite $\Z$-lattice $(\Lambda,Q)$ has a unique
(up to order) decomposition into an orthogonal sum of indecomposable sublattices.   
\end{theorem}
\begin{proof}
The existence of such a decomposition is trivial, we have to prove the
uniqueness. We call a vector $x\in \Lambda$ indecomposable if it can
not be written as $x=y+z$ with $b(x,y)=0, x,y \in \Lambda, x\ne 0 \ne
y$. The lattice $\Lambda$ is obviously generated by its indecomposable
vectors. We call two 
indecomposable vectors $x,y$ connected, if there exists  a chain
$x=x_1,\ldots,x_r$ of indecomposable vectors of $\Lambda$ with
$b(x_i,x_{i+1})\ne 0$ for $1\le i<r$. Being connected defines an
equivalence relation among the indecomposable vectors of $\Lambda$ and
any two equivalence classes are mutually orthogonal.

If we have a decomposition $\Lambda=\perp_{i=1}^t L_i$, any indecomposable
vector lies in some $L_j$ and and if $x,y$ are
connected indecomposable vectors of $\Lambda$ they lie in the same 
$L_j$, so for each $j$ the set of indecomposable vectors in $L_j$ consists of
full equivalence classes of connected indecomposable vectors of
$\Lambda$. If $L_j$ contained more than one equivalence class it could
not be indecomposable.

This implies that in the
decomposition $\Lambda=\perp_{i=1}^t L_i$, each $L_j$ contains a
unique equivalence class $S_j$ of connected indecomposable vectors of
$\Lambda$ which generates it. So the $L_j$ are the sublattices
generated by the equivalence classes of connected indecomposable
vectors of $\Lambda$ and therefore uniquely determined. 
\end{proof}

\begin{example}
  \begin{enumerate}
  \item Denote by $A_n$ the lattice $\{\x \in \Z^n \mid
    \sum_{i=0}^nx_i=0\}$ with the standard scalar product of
    $\Q^{n+1}$ as symmetric bilinear form $b$ and
    $Q(\x)=\frac{1}{2}\sum_{i=0}^nx_i^2$. Since it is the orthogonal
    complement of $(1,\ldots,1)$ in $\Z^{n+1}$ it has determinant
    $n+1$ by Theorem \ref{duals_and_indices}. The vectors $\e_i-\e_j$
    with $i \ne j$, where $(\e_1,\ldots,\e_{n+1})$ is 
    the standard basis of $\Q^{n+1}$, are indecomposable and connected
    and generate
    $A_n$, so this lattice is indecomposable. It occurs as a root
    lattice in the theory of Lie algebras.
\item Denote by $D_n$ the lattice $\{\x \in \Z^n\mid \sum_{i=0}^n x_i
  \equiv 0 \bmod 2\}$, again with quadratic form
  $Q(\x)=\frac{1}{2}\sum_{i=0}^nx_i^2$ so that the standard scalar
  product is the associate symmetric bilinear form. The lattice is of
  index $2$ in $\Z^n$, so $\det_b(A_n)=4$. It is generated by the
  indecomposable vectors $\pm \e_i \pm \e_j$ with $i \ne j$ which  generate
  the lattice and are connected for $n>2$, so for $n>2$ it is also indecomposable. 
It occurs as a root
    lattice in the theory of Lie algebras. For $n=3$ it is isometric
    to the lattice $A_3$ above.
\item for $n \equiv 0 \bmod 8$ denote by $D_n^+=\Gamma_n$ the lattice
  $D_n+\Z\frac{1}{2}\sum_{i=1}^n\e_i$. One has $Q(D_n^+)\subseteq \Z$
  and $(D_n^+:D_n)=2$, so $D_n^+$ is an even unimodular lattice in
  $\Q^n$. The indecomposable vectors      $\pm \e_i \pm \e_j$ with $i
  \ne j$ generate
  a sublattice of full rank and are connected, so this lattice is indecomposable
  too. For $n=8$ this lattice is usually denoted by $E_8$ because it
  occurs with this notation as a root lattice in the theory of Lie
  algebras. It has the indecomposable sublattices $E_7:=\{\x \in E_8
  \mid x_7=x_8\}$ of rank $7$ and determinant $2$ and $E_6:=\{\x \in
  E_8\mid x_6=x_7=x_8\}$ of rank $6$ and determinant $3$ which occur
  as root lattices too and together with $E_8$ and the lattices $D_n$
  and $A_n$ from above exhaust the list of indecomposable root
  lattices.

 For $n=16$ we see that we know two even unimodular
  lattices, one of them ($D_{16^+}$) indecomposable and one of them
    ($E_8 \perp E_8$) decomposable, so in particular they are not
    isometric.

It can be shown that these two represent all isometry classes of even
unimodular positive definite $\Z$-lattices of rank $16$. It can also
be shown that even unimodular positive definite $\Z$-lattices can occur
only with ranks which are divisible by $8$.
  \end{enumerate}
\end{example}

\chapter{Clifford Algebra, Invariants, and Classification}

To a quadratic module one can construct a natural associative algebra,
the Clifford algebra, 
which provides important insights and is a useful tool for
classification.

For our purpose it will be sufficient to treat the case of quadratic
spaces over fields of characteristic zero. 
Slightly more general, in this chapter $F$ is a field of characteristic different from
$2$. Modifications that include the case of even characteristic can be
found in \cite{kneserbook,knus,scharlaubook}; we omit the details
since we are mainly interested in the number theoretic aspects of the
theory of quadratic forms. As earlier we write $(V,Q)\cong
[a_1,\ldots,a_m]$ for a quadratic space $(V,Q)$ with an orthogonal
basis $(v_1,\ldots,v_m)$ satisfying $Q(v_i)=a_i$ for $1\le i \le m$.

\section{Quaternion Algebras and Brauer group}
\begin{definition}
  A central simple algebra of dimension $4$ over $F$ is called a
  quaternion algebra over $F$.
\end{definition}
\begin{theorem}
  \begin{enumerate}
  \item If $A$ is a quaternion algebra over $F$ there exist linearly
    independent elements
    $x,y\in A$ with $x^2\in F^\times, y^2\in F^\times, xy=-yx$. The elements
    $1,x,y,xy$ form then a basis of $A$ as $F$-vector space.
\item Conversely, let $a,b \in F^\times$ be given. Then there exist a
  quaternion algebra over $F$ and linearly independent vectors $x,y
  \in A$ with $x^2=a, y^2=b, xy=-yx$.
  \end{enumerate}
\end{theorem}
\begin{proof} Wedderburn's theorem about central simple algebras
  implies that $A$ is either a division algebra or isomorphic to
  $M_2(F)$.
In the latter case, the matrices $x=\bigl(
\begin{smallmatrix}
  1&0\\0&-1
\end{smallmatrix}\bigr), y=\bigl(
\begin{smallmatrix}
  0&1\\1&0
\end{smallmatrix}\bigr)$ are as requested. 

  If $A$ is a division algebra, 
every element $x$ of $A\setminus
  F$ has an irreducible minimal polynomial of degree $2$ over $F$, and
  changing $x$ by an element of $F$ we can assume the minimal
  polynomial to be of the form $X^2-a$ with $a \ne 0$, i.e., $x^2=a \in F^\times$.

With $A^0:= \{0\}\cup    \{x\in A\setminus F\mid x^2\in F\}$ we have
then $A=F+A^0$, in particular we can find $x,y\in A^0$ such that
$1,x,y$ are linearly independent. Since $1,x,y$ can not generate a
(division) subalgebra of dimension $3$ over $F$, the vectors
$1,x,y,xy$ must then be an $F$-basis of $A$, and $xy+yx$ commutes with
all basis vectors and must hence be in the center $F$ of $A$. From
this we see that we can subtract a suitable multiple of $x$ from $y$
in order to obtain $0\ne y_1\in A^0$ with $xy_1+y_1x=0$, so that $x$
and $y_1$ are as requested.  We notice in passing that $xy+yx\in F$
implies $(x+y)^2\in F$, i.e., $x+y \in A^0$.

Conversely, if $a,b\in F^\times $ are given one can define a
$4$-dimensional $F$-algebra $A$ generated by linearly independent vectors
$1=e_0,e_1,e_2,e_3$   by prescribing a multiplication table with
$e_0e_j=e_je_0=e_j, e_0^2=1, e_1^2=a,e_3^2=-ab, e_1e_2=-e_1e_2=e_3,
e_2e_3=-be_1=-e_3e_2, e_3e_1=-ae2=-e_1e_3$ and extending this
multiplication distributively to the whole vector space; it is then
easily checked by explicit calculation that $A$ with this law of
multiplication satisfies the axioms for an $F$-algebra and is central simple.
\end{proof}

\begin{lemma}
  Let $A$ be a quaternion algebra over $F$ with basis
  $1=e_0,e_1,e_2,e_3=e_1e_2$ satisfying $e_1^2,e_2^2 \in F^\times,
  e_2e_1=-e_3$ as above. Then 
  \begin{equation*}
\{0\}\cup    \{x\in A\setminus F\mid x^2\in F\}=Fe_1+Fe_2+Fe_3=:A^0,
  \end{equation*}
and the map given by $a+x\mapsto \overline{a+x}:=a-x$ for $a\in F, x\in A^0$ is an
involution of the first kind on $A$ satisfying $x\bar{x}=:n(x)\in F,
x+\bar{x}=\tr(x)$ for all $x \in A$.
\end{lemma}
\begin{proof}
  This is checked by explicit calculation. The fact that $A^0$ is a
  $3$-dimensional vector space over $F$ can also be seen from the
  proof of the theorem above if $A$ is a division algebra, if $A$ is
  the matrix ring $M_2(F)$  it is obvious that $A^0$ is the space of
  matrices of trace $0$.
\end{proof}
\begin{definition}
  With notations as in the previous lemma we call $A^0$ the space of
  pure quaternions, the map $x \mapsto \bar{x}$ the (quaternionic)
  conjugation on $A$ and the maps $x \to n(x), x\mapsto \tr(x)$ the
  quaternion norm resp.\ trace. The quaternion algebra with generators
  $e_0=1,e_1,e_2,e_3=e_1e_2=-e_2e_1$ satisfying $e_1^2=a, e_2^2=b$ as above is
  denoted by $(a,b)_F=\bigl(\frac{a,b}{F}\bigr)$ and the vectors
  $e_0,\ldots,e_4$ are called a standard basis of $(a,b)_F$.
\end{definition}
\begin{remark}
  \begin{enumerate}
  \item  In the terminology of the theory of (central simple) algebras the quaternion norm and
  trace are the reduced norm and trace in the algebra $A$.
\item For the matrix algebra $M_2(F)$ the quaternion norm is the
  determinant, the quaternion trace the matrix trace and the
  conjugation the map
  \begin{equation*}
    \begin{pmatrix}
      a&b\\c&d
    \end{pmatrix}\mapsto
    \begin{pmatrix}
      d&-b\\-c&a
    \end{pmatrix}.  
\end{equation*}
\item If $A=(a,b)_F$ is a quaternion division algebra and $\delta \in A^0$
  with $\delta^2=a\in F^\times$ one can embed $A$ into the matrix ring
  $M_2(F(\delta))$ over the quadratic extension
  $E=F(\delta)=F(\sqrt{a})$ as the subgroup generated by the matrices
  \begin{equation*}
    \tilde{e_0}=1_2, \tilde{e_1}=
    \begin{pmatrix}
      \delta&0\\0&-\delta
    \end{pmatrix},\quad
    \tilde{e_2}=\begin{pmatrix}
    0&b\\1&0
    \end{pmatrix}.
  \end{equation*}
The quaternion norm and trace coincide then with the determinant and
the matrix trace in this representation.
In particular for the Hamilton quaternions ${\mathbb H}=(-1,-1)_\R$
over $R$ one obtains the usual identification of ${\mathbb H}^1:=\{x
\in {\mathbb H}\mid n(x)=1\}$ with the special unitary group $SU_2(\C)$. 
\end{enumerate}
\end{remark}
We recall that the Brauer group $Br(F)$ of $F$ is by definition the
set of isomorphism classes of central simple $F$-algebras modulo the
equivalence relation $A\otimes M_r(F)\sim B\otimes M_s(F)$ (similarity
of algebras), where the
group multiplication is induced by the tensor product of
algebras. Each such class is represented by a unique division algebra
with center $F$.
\begin{lemma}
  The class of a quaternion division algebra over $F$ in the Brauer group of
  $F$ has order $2$.
\end{lemma}
\begin{proof}
  Denoting by $A^{op}$ the opposed algebra of $F$, i. e., $A$ with
  multiplication in opposite order, it is well known that $A\otimes
  A^{op}$ is isomorphic to a matrix ring over $F$ for any central
  simple algebra $A$ over $F$. The properties of the quaternionic
  conjugation show that a quaternion algebra is isomorphic to its
  opposed algebra, which implies the assertion.
\end{proof}
\begin{lemma}
 Let $A$ be a quaternion algebra over $F$.
The quaternion norm $n$ defines a quadratic form on $A$ with
associated symmetric bilinear form $b(x,y)=\tr(x\bar{y})$, and the
quadratic space $(A,n)$ over $F$ has square determinant. It is isotropic if and only if $A \cong
M_2(F)$ holds.  
\end{lemma}
\begin{proof}
  Since the norm is multiplicative a quaternion $x$ is invertible if
  and only if its norm is nonzero, so $A$ is a division algebra if and
  only if $(A,n)$ is anisotropic. The rest of the assertion is obvious. 
\end{proof}
\begin{remark}
  Since over a finite field $F$ every quadratic space of dimension
  $\ge 3$ is isotropic we see that the only quaternion algebra over a
  finite field $F$ is the matrix ring $M_2(F)$ (in fact,  by a well
  known theorem of
  Wedderburn there exists  no non trivial division algebra with center $F$).  
\end{remark}
\begin{lemma}
Let $a,b,c,d \in F^\times$.
then the following are equivalent:
\begin{enumerate}
\item $[1,-a,-b,ab]\cong [1,-c,-d,cd]$
\item $[-a,-b, ab]\cong  [-c,-d,cd]$
\item The quaternion algebras $(a,b)_F, (c,d)_F$ are isomorphic.
\end{enumerate}
\end{lemma}
\begin{proof}
  The equivalence of a) and b) follows from Witt's cancellation law,
  and c) obviously implies b). For the reverse direction denote by
  $e_0=1,e_1,e_2,e_3$ a standard basis of $(a,b)_F$. By assumption
  there exists 
  a linear isometry $\tau$ from $(a,b)_F$ to $(c,d)_F$ with $\tau(1)=1
  and \tau((a,b)_F^0)=(c,d)_F^0$. With $f_1:=\tau(e_1),
  f_2:=\tau(e_2), f_3=f_1f_2$ we have $\tr
  (f_1\bar{f_2})=\tr(e_1\bar{e_2})=0$ and hence $f_1f_2=-f_2f_1$. With
  $f_i^2=-n(f_i)=n(e_i)=-e_i^2$ for $i=1,2$ we see that
  $1,f_1,f_2,f_3$ is a standard basis of $(c,d)_F$ with $f_1^2=a,
  f_2^2=b$ and hence $(c,d)_F\cong (a,b)_F$.
\end{proof}
\begin{lemma}\label{quaternion_clifford_dim2}
  Let $a,b,c,d \in F^\times, ab \in cd(F^\times)^2$.

Then $[a,b]\cong [c,d]$ (as quadratic spaces) is equivalent to
$(a,b)_F\cong (c,d)_F$ (as algebras). 
\end{lemma}
\begin{proof}
  By assumption the isometry of quadratic spaces is equivalent to the
  isometry $[-a,-b,-cd]\cong [-c,-d,cd]$, which is equivalent to the
  algebras being isomorphic.
\end{proof}
\begin{lemma}\label{quaternion_clifford_dim3}
  Let $a,b,c,\alpha,\beta,\gamma \in F^\times$ with $\alpha\beta\gamma
  \in abc (F^\times)^2$.
Then $[a,b,c]\cong[\alpha,\beta,\gamma]$ (as quadratic spaces)  is
equivalent to $(-bc,-ac)_F\cong (-\beta\gamma,-\alpha\gamma)_F$ (as
algebras).
\end{lemma}
\begin{proof}
  Scaling the quadratic forms with $abc$ resp.\ $\alpha,\beta,\gamma$
  shows that the isometry of quadratic spaces is equivalent to the
  isometry $[bc,ac,ab]\cong[\beta\gamma,\alpha\gamma,\alpha\beta]$,
  which in turn is equivalent to the isomorphy of algebras is question.
\end{proof}
\begin{lemma}
Let $a,b,\alpha,\beta \in F^\times$.  The map sending the pair $(a,b)$ to the
quaternion algebra $(a,b)_F$ has the following properties:
\begin{enumerate}
\item $(a,b)_F\cong (a\alpha^2,b\beta^2)_F$.
\item $(a,b)_F\cong (b,a)_F$.
\item $M_2(F)\cong (1,a)_F\cong (a,-a)_F\cong (a,1-a)_F$ (with the
  last isomorphy supposing $a\ne 1$).
\item $(a,a)_F\cong(a,-1)_F$.
\end{enumerate}
\end{lemma}
\begin{proof}
  All assertions follow from a consideration of the associated
  quadratic spaces.
\end{proof}
\begin{lemma}
  Let $a,b,c \in F^\times$, denote by $\sim$ the equivalence relation (similarity)
  between central simple algebras defining the Brauer group.
  
Then
\begin{equation*}
  (a,bc)_F\sim (a,b)_F\otimes (a,c)_F,\quad (ab,c)_F\sim
  (a,c)_F\otimes ((b,c)_F.
\end{equation*}
In particular, the map sending a pair $(a,b)\in
F^\times\times F^\times$ to the class of the quaternion algebra
$(a,b)_F$ in the $2-$ torsion subgroup of the Brauer group $Br(F)$ is
a Steinberg symbol.
\end{lemma}
\begin{proof}
  We show the assertion by proving 
  \begin{equation*}
    (a,b)_F\otimes (a,c)_F\cong (a,bc)_F\otimes (c,-ac^2)_F.
  \end{equation*}
We let $e_0=1,e_1,e_2,e_3$ resp.\ $f_0=1,f_1,f_2,f_3$ be standard
bases for the quaternion algebras on the left hand side and put $A=F
1\otimes 1+F e_1\otimes 1+F e_2\otimes f_2+Fe_3\otimes f_2$ and $B=
F1\otimes 1 +F1\otimes f_2+F\otimes e_1\otimes f_3+F(e_1\otimes f_1)$.
Then $A,B$ are isomorphic to the quaternion algebras on the right hand
side and commute. From this we obtain an algebra homomorphism
$A\otimes B \to (a,b)_F\otimes (a,c)_F$, which must be an isomorphism
since $A\otimes B$ is simple. The definition of Steinberg symbols
shows that we have indeed defined such a symbol. 
\end{proof}
\section{The Clifford Algebra}
\begin{deftheorem}
  Let $(V,Q)$ be a quadratic space over $F$. Then there is a unique up
  to unique isomorphism $F$- algebra $C(V,Q)=C(V)=C_V$ with a linear
  map $i:V\to C(V,Q)$ such that
  \begin{enumerate}
  \item $i(v)^2=Q(v)1\in C(V,Q)$ for all $v \in V$.
\item If $C'$ is any $F$-algebra with an $F$-linear map $i':V\to C'$
  satisfying $(i'(v))^2=Q(v)1 \in C'$ for all $v \in V$, there is a
  unique algebra homomorphism $\phi:C(V,Q)\to C'$ with $i'=\phi\circ
  i$.

The algebra $C(V,Q)$ is called the Clifford algebra of $(V,Q)$, it is
generated by the elements $i(v)$ for $v \in V$. It has
a natural $\Z/2\Z$-grading $C(V,Q)=C_0(V,Q)+C_1(V,Q)$, where the even
Clifford algebra 
$C_0(V,Q)$ is generated by products of even length of the $i(v)$ and
$C_1(V,Q)$ is generated (as vector space over $F$) by the products of
odd length and is a module over $C_0(V,Q)$.   

In particular one has $C_1(V,Q)C_1(V,Q) \subseteq C_0(V,Q)$.
  \end{enumerate}
\end{deftheorem}
\begin{proof}
  Let $I$ be the two sided ideal of the Tensor algebra
  $T(V)=\oplus_{j=0}^\infty V^{\otimes j}$ of $V$
  generated by the elements $x\otimes x -Q(x)1$. Then it is easily
  checked that $C(V,Q):=T(V)/I$ with the map $i:v \mapsto v+I$ has the
  desired property. From this it is clear that the $i(V)$ generate
  $C(V,Q)$. We let $T_0(V),T_1(V)$ respectively be the sum of the
  $V^{\otimes j}$ for even resp.\ odd $j$ and obtain a
  $\Z/2\Z$-grading of $T(V)$. Then we have $I=I\cap
  T_0(V)\oplus I\cap T_1(V)$, so that $C_0(V,Q)=T_0(V)/I\cap T_0(V),
  C_1(V,Q)=T_1(V)/I\cap T_1(V)$ give the desired decomposition of $C(V,Q)$.
\end{proof}
\begin{lemma}
  Identifying $x\in V$ with its image $i(x)\in C(V,Q)$ we have
  \begin{enumerate}
  \item $xy+yx =b(x,y)1 \in C(V,Q)$.
\item $x\in V$ is invertible in $C(V,Q)$ if and only if $Q(x)\ne 0$
  holds.
\item If $\{x_1,\ldots,x_r\}$ is a generating set of the vector space
  $V$, the $x_1^{\epsilon_1}\dots x_r^{\epsilon_r}$ with $\epsilon_i
  \in \{0,1\}$ generate $C(V,Q)$ as vector space over $F$. 
  \end{enumerate}
\end{lemma}
\begin{proof}
  Obvious.
\end{proof}
\begin{lemma}[Functoriality of the Clifford algebra]\label{clifford_functoriality}
  Let $(V,Q)$ be a quadratic space over $F$.
  \begin{enumerate}
  \item To   $\phi \in O(V)$ there is a unique algebra automorphism
    $C(\phi)$ of $C(V)$ with $C(\phi)(i(v))=i(\phi(v))$ for all $v
    \in V$, and one has $C(\psi\circ\phi)=C(\psi)\circ C(\phi)$, for
    $\phi,\psi \in O(V)$, $C(\id_V)=\id_{C(V)}$.

The map $C(-\id_V)$ is the identity map on $C_0(V)$ and $-\id$ on
$C_1(V)$.
\item On $C(V)$ one has an involution $x\mapsto \bar{x}$ with
  $\overline{i(v)}=-i({v})$ for all $v \in V$ and $\overline{v_1\dots
    v_r}=(-1)^rv_r\dots v_1$ for $v_1,\ldots,v_r \in V$ (identifying $v$
  with $i(v)\in C(V)$). 
  \end{enumerate}
\end{lemma}
\begin{proof}
  Existence and properties of $C(\phi)$ are direct consequences of the
  universal property of the Clifford algebra.

For b) the universal property gives us an algebra homomorphism  
$\iota: C(V) \to C(V)^{op}$ leaving the $i(v)$ fixed, it becomes an
anti-automorphism of order $2$ if we view it as a map from $C(V)$ to
itself.
Putting $\bar{x}=\iota(C(-\id_V)(x))$ we obtain the desired map. 
\end{proof}
\begin{example}\label{clifford_lowdimensions}
  \begin{enumerate}
\item Let $Q$ be identically zero. Then the Clifford algebra $C(V,Q)$
  is isomorphic to the exterior algebra $\bigwedge V$ because in this
  case the universal property of the Clifford algebra is the same as that used in the
  definition of the exterior algebra by a universal property.
  \item Let $V=Fx\cong [a]$ with $a \in F$. Then $C(V,Q)\cong
    F[X]/(X^2-a)$. In particular, if $a$ is not a square in $F$, the
    Clifford algebra is isomorphic to the quadratic extension field
    $F(\sqrt{a})$. The involution $x \mapsto \bar{x}$ from Lemma
    \ref{clifford_functoriality} of the Clifford algebra maps to the non trivial
    $F$-automorphism of the quadratic extension. We will also use the notation $(a)_F$ for this algebra.
\item $V=Fx\perp Fy\cong [a,b]$ with $a,b \in F^\times$. Then $C(V,Q)$ is generated by $1,x,y,
  xy$ with $x^2=a, y^2=b, xy=-yx$ and hence $(xy)^2=-ab$. 
Therefore, we have  an algebra homomorphism from $(a,b)_F$ to $C(V,Q)$
sending the standard basis vectors to $1,x,y,xy$. Since the quaternion
algebra $(a,b)_F$ is simple this must be an isomorphism and we have
that $C(V,Q)$ is isomorphic to the 
  quaternion algebra $(a,b)_F$ over $F$.  The involution $x \mapsto \bar{x}$ from Lemma
    \ref{clifford_functoriality} of the Clifford algebra maps to the quaternionic conjugation.
\item Let $(V,Q)\cong [a,b,c]$ with $a,b,c \in F^\times$. We claim
  that one has then
  \begin{eqnarray*}
    C_0(V)&\cong& (-ac,-bc)_F\\
C(V)&\cong&C_0(V)\otimes (-abc)_F,
  \end{eqnarray*}
where $(-abc)_F$ denotes (as above) the algebra $F[X]/(X^2+abc)$.

To prove this, let $x_1,x_2,x_3$ be an orthogonal basis of $(V,Q)$
with $Q(x_1)=a,Q(x_2)=b,Q(x_3)=c$. The even Clifford algebra
$C_0(V,Q)$ is then generated as $F$-vector space by $1, f_1=x_1x_3,
f_2=x_3x_2,f_3=c x_1x_2$ with $f_1^2=-ac, f_2^2=-bc,
f_1f_2=f_3=-f_2f_1$. As above this gives us an algebra homomorphism
from $(-ac,-bc)_F$ onto $C_0(V,Q)$ which must be an isomorphism.
The involution $x \mapsto \bar{x}$ from Lemma
    \ref{clifford_functoriality} of the Clifford algebra corresponds
    to the quaternionic conjugation under this isomorphism.

The subalgebra $F1+Fx_1x_2x_3 \cong (-abc)_F\subseteq C(V,Q)$ commutes
with $C_0(V)$ and generates $C(V,Q)$ together with $C_0(V,Q)$. The
universal property of the tensor product gives us an algebra
homomorphism from $(-ac,-bc)_F\otimes (-abc)_F$ onto $C(V,Q)$ which is an
isomorphism since $(-ac,-bc)_F\otimes (-abc)_F$ is simple.
\end{enumerate}
\end{example}
\begin{theorem}
 Let $A,B$ be finite dimensional $F$-algebras with $\Z/2\Z$ gradings
 $A=A_0\oplus A_1, B=B_0\oplus B_1$.

Then the tensor product of $F$-vector spaces $A\otimes B$ has  a
unique structure as an $F$-algebra denoted by $A\widehat{\otimes} B$ with the property
\begin{equation*}
  (a_i\otimes b_j)(a'_k\otimes b'_l)=(-1)^{jk}a_ia'_k \otimes b_jb'_l
  \end{equation*}
for  all $a_i\in A_i,a'_k\in A_k,b_j\in B_j, b'_l\in B_l,
  i,j,k,l\in \{0,1\}$.
The algebra  $A\widehat{\otimes} B$ is graded  by setting 
\begin{equation*}
   (A\widehat{\otimes} B)_0=(A_0\otimes B_0) \oplus (A_1\otimes B_1),\quad
   (A\widehat{\otimes} B)_1=(A_0\otimes B_1) \oplus (A_1\otimes B_0) 
\end{equation*} as  vector spaces and is called the {\em graded tensor product} of $A$
and $B$.
It is characterized up to unique isomorphism by the following universal
 property:

If $f:A\to C, g:B\to C$  are graded algebra homomorphisms of $A,B$ to the
$\Z/2\Z$ graded $F$-algebra $C$ satisfying 
\begin{equation*}
  f(a_i)g(b_j)=(-1)^{ij}g(b_j)f(a_i) \text{ for } a_i\in A_i, b_j \in
  B_j, i,j \in \{0,1\}
\end{equation*}
there is a unique graded algebra homomorphism $\phi:A\widehat{\otimes}B\to
C$ satisfying
\begin{equation*}
  \phi (a\otimes 1_B)=f(a),\quad \phi(1_A\otimes b)=g(b) \text{ for
    all } a\in A,b \in B.
\end{equation*}
\end{theorem}
\begin{proof}
  This an easy consequence of the universal property of the tensor
  product of vector spaces, see e. g. \cite{lam_qf}.
\end{proof}
\begin{theorem}\label{cliffordalgebra_decomposition}
  Let $(V,Q)$ be a quadratic space over $F$ with an orthogonal
  decomposition $V=U_1\perp U_2$. Then $C(V,Q)$ is isomorphic to the
  graded product $C(U_1,Q\vert_{U_1})\widehat{\otimes}C(U_2,Q\vert_{U_2})$.
\end{theorem}
\begin{proof}
  We have natural homomorphisms  $\alpha_j:C(U_j,Q\vert_{U_j})\to C(V,Q), j=1,2$.
Since $U_1,U_2$ are orthogonal these satisfy $\alpha_1(u_1)\alpha_2(u_2)=-\alpha_2(u_2)\alpha_1(u_1)$
for $u_1\in U_1,u_2\in U_2$ (identifying vectors with their images in
the Clifford algebra as usual), which implies
$\alpha_1(x_i)\alpha_2(y_j)=(-1)^{ij}\alpha_2)(y_j)\alpha_1(x_i)$  for
 $x_i\in C_i(U_1,Q\vert_{U_1}), y_j \in C_j(U_2,Q\vert_{U-2})$ with
$i,j \in \{0,1\}$.
By the universal property of the graded tensor product we obtain a
graded algebra homomorphism
$\phi:C(U_1,Q\vert_{U_1})\widehat{\otimes}C(U_2,Q\vert_{U_2})\to
C(V,Q)$ satisfying $\phi(x\otimes 1)=\alpha_1(x), \phi(1\otimes
y)=\alpha_2(y)$ for $x \in C(U_1,Q\vert_{U_1}),y\in
C(U_2,Q\vert_{U_2})$.

On the other hand, we have an $F$-linear map $j:V\to
C(U_1,Q\vert_{U_1})\widehat{\otimes}C(U_2,Q\vert_{U_2})$ given by
$j(u_1+u_2):=u_1\otimes 1+1\otimes u_2$ for $u_1 \in U_1, u_2\in U_2$,
and $j$ satisfies $j(u_1+u_2)^2=(Q(u_1)+Q(u_2))1$.

The universal property of the Clifford algebra gives us an $F$-algebra
homomorphism $\psi:C(V,Q) \to
C(U_1,Q\vert_{U_1})\widehat{\otimes}C(U_2,Q\vert_{U_2})$  satisfying 
$\psi(u_1+u_2)=j(u_1+u_2)=u_1\otimes 1+1\otimes u_2$ for $u_1 \in
U_1, u_2\in U_2$, and it is easily checked that $\phi,\psi$ are
inverses of each other. 
\end{proof}
\begin{corollary}
  Let $(V,Q)$ be a quadratic space over $F$ of dimension $n$. Then
  $C(V,Q)$ has dimension $2^n$ and $C_0(V,Q)$ has dimension $2^{n-1}$
  as vector space over $F$.

In particular the generators $x_1^{\epsilon_1}\dots x_n^{\epsilon_n}$
with $\epsilon_j\in \{0,1\}$ 
of $C(V,Q)$ for a basis $(x_1,\ldots x_n)$ of $V$ over $F$ constitute a
basis of $C(V,Q)$ as a vector space over $F$ and the mapping $i:V\to
C(V,Q)$ is injective.
\end{corollary}
\begin{proof}
  Since the Clifford algebra of a one dimensional space has dimension
  $2$ over $F$ this follows from the theorem. 
\end{proof}
\begin{remark}
It can be shown that in the more general context of Clifford algebras
of finitely generated projective quadratic modules over a commutative ring the mapping $j$
from the module to its Clifford algebra is injective, see
\cite{kneserbook, knus}. In fact, it is enough to require that the
quadratic form $Q$ can be written as $Q(v)=\beta(v,v)$ for a not
necessarily symmetric bilinear form $\beta$. It seems to be hard to
find a counterexample if this condition is violated. 
\end{remark}
\begin{theorem}\label{clifford_centralsimple}
 Let $(V,Q)$ be a non degenerate quadratic space over $F$ of dimension
 $n$, put $m=\lfloor \frac{n}{2}\rfloor$ and $\delta
 (F^\times)^2=(-1)^m \det_B(V)$.
 \begin{enumerate}
 \item If $n$ is even, $C(V,Q)$ is central simple over $F$ and
   isomorphic to a tensor product of quaternion algebras. The second
   Clifford algebra $C_0(V,Q)$ has center isomorphic to
   $F[X]/(X^2-\delta)$ and is central simple over $F(\sqrt{\delta})$ if
   $\delta$ is not a square in $F$.
\item Let $n$ be odd.
Then $C_0(V,Q)$ is central simple over $F$ and isomorphic to a tensor product
of  quaternion algebras. The center of $C(V,Q)$ is isomorphic to
   $F[X]/(X^2-\delta)\cong (\delta)_F$ and $C(V,Q)$ is central simple over $F(\sqrt{\delta})$ if
   $\delta$ is not a square in $F$. Moreover, one has $C(V,Q)\cong
   C_0(V,Q)\widehat{\otimes}(\delta)_F$ as graded algebras.  
 \end{enumerate}
\end{theorem}
\begin{proof}
  We have already proven the assertion for $n=1,2,3$ and proceed by
  induction for general $n>3$, assuming the assertion to be proven for
  smaller dimensions.

We split $V$ into a sum $V=U \perp W$ with non
degenerate $U,W, \dim(U)=3$ and have $C(V,Q)\cong C(U,Q) \widehat{\otimes}
C(W,Q)$.
By Example \ref{clifford_lowdimensions}  and Theorem
\ref{cliffordalgebra_decomposition} we have $C(U,Q)\cong A_U\otimes 
(-\det_B(U))_F$, where $A_U$ is a trivially graded quaternion algebra
over $F$ and is isomorphic to the second Clifford algebra
$C_0(U,Q)$.  Again by Theorem
\ref{cliffordalgebra_decomposition} we have $(-\det_B(U))_F\widehat{\otimes}
C(W,Q)\cong C(W_1,Q_1)$, where the $(n-2)$ dimensional space
$(W_1,Q_1)$ is the orthogonal sum of 
$(W,Q)$ and a one dimensional space of determinant $-\det_B(U)$, so
that $C(V,Q)\cong A_U\otimes C(W_1,Q)$ (where the tensor product may
be taken ungraded since $A_U$ is trivially graded). 
By the
inductive assumption the assertion 
follows.
\end{proof}
\begin{remark}\label{brauerwall}
The graded center $\hat{Z}(A)$ of a $\Z/2\Z$-graded algebra $A$ is
defined as $ \hat{Z}(A)=\hat{Z}_0(A) \oplus \hat{Z}_1(A)$ where
$z\in \hat{Z}_i(A)\subseteq A_i$ satisfies $za=(-1)^{ij}az$ for all $a \in
A_j(A)$, see \cite{wall,lam_qf}. A graded algebra is then called a central simple graded
algebra if it is simple and has trivial graded center. As in the
theory of central simple algebras, the tensor product of  central simple graded
algebras is again a central simple graded
algebra, and one can define the Brauer-Wall group of central simple
graded algebras, see again \cite{wall,lam_qf}. 
In particular, one proves as above by induction that the
Clifford algebra of a non degenerate quadratic space is a central simple graded
algebra and one may consider its class in the Brauer-Wall
group. Notice that \cite{wall} treats the case that $F$ has
characteristic $2$ as well, with slightly different definitions.
\end{remark}
\section{The invariants of Clifford-Witt and Hasse}
\begin{definition}
  Let $(V,Q)$ be a non degenerate quadratic space over the field $F$.
  \begin{enumerate}
  \item The Clifford-Witt invariant $cw(V)=cw(V,Q)$ is the class of $C(V,Q)$ in
    the Brauer group $Br(F)$ if $n=\dim(V)$ is even and the class of
    $C_0(V,Q)$ in $Br(F)$ if $n$ is odd.
\item Let $(V,Q) \cong [a_1,\ldots,a_n]$. Then the Hasse invariant
  $s(V)=s(V,Q)$ of
  $(V,Q)$ is the class  of
  \begin{equation*}
    \bigotimes_{1\le i<j\le n}(a_i,a_j)_F
  \end{equation*}
in $Br(F)$.
  \end{enumerate}
\end{definition}
\begin{remark}
  \begin{enumerate}
  \item   Since both the Clifford-Witt and the Hasse invariant are a product of classes
  of quaternion algebras in the Brauer group, they are in fact
  elements of the $2$-torsion subgroup $Br_2(F)$ of the Brauer group
  of $F$.
\item The definition of the Hasse invariant seems to depend on the
  choice of an orthogonal basis of $V$. We will show that this is not
  the case.
\item The name Clifford-Witt invariant is a combination of the name Witt
  invariant that has been T. Y. Lam in \cite{lam_qf}
   and the names Clifford invariant and Witt invariant used by W. Scharlau, who
defines a (more general) Witt invariant related to $K$-theory by its relation to the Hasse
invariant and a Clifford invariant in the same way as we do and then
shows that they coincide when both are defined.
 The invariant in this form seems to have first been
  used by C.\ T.\ C.\ Wall in his article ``Graded Brauer algebras''
  of 1964, without giving it a name there (it is called the invariant
  $D$ of the Clifford algebra).  Witt used a related but somewhat
  different invariant.
  \end{enumerate}
\end{remark}
\begin{theorem}
 Let $(V,Q),(W,Q')$ be non degenerate quadratic spaces over $F$ and
 write $\disc_B(V)=(-1)^{\lfloor \dim(V)/2\rfloor}\det_B(V)$.
Then 
\begin{equation*}
  cw(V\perp W)=cw(V)cw(W)\cdot
  \begin{cases}
    (\disc_B(V),\disc_{B'}(W))_F &\dim(W)\equiv \dim(V)\bmod 2\\
(-\disc_B(V),\disc_{B'}(W))_F& \dim(V)\text{ odd}, \dim(W)\text{ even}
  \end{cases}.
\end{equation*}
\end{theorem}
\begin{proof}
Let first $V\cong [d]$ have dimension $1$.
If $\dim(W)$ is odd, we have
\begin{eqnarray*}
  C(V\perp W)&\cong& (d)_F\widehat{\otimes}C(W)\\
&\cong& (d)\widehat{\otimes}(\disc_{B'}(W))\widehat{\otimes}C_0(W)\\
&\cong& (\disc_B(V),\disc_{B'}(W))_F\otimes C_0(W),
\end{eqnarray*}
which implies $cw(V\perp W)= (\disc_B(V),\disc_{B'}(W))_Fcw(V)cw(W)$.

If $\dim(W)\ge 2$ is even
we set $W=W_1\perp [a]$ and have,
 using Example \ref{clifford_lowdimensions},
 \begin{eqnarray*}
 C(V\perp  W)&=&C(W_1\perp [a,d])\\
&\cong&C_0(W_1)\widehat{\otimes}(\disc_{B'}(W_1)_F\widehat{\otimes}
    (a)_F\widehat{\otimes}(d)_F\\
&\cong&C_0(W_1)\widehat{\otimes}(-d\cdot \disc_{B'}(W_1),-ad)_F\widehat{\otimes}(-ad\cdot
    \disc_{B'}(W_1))_F
\end{eqnarray*} 
and hence $C_0(W\perp [d])\cong C_0(W_1)\otimes (-d\cdot
\disc_{B'}(W_1),-ad)_F$.

On the right hand side of the assertion we compute
\begin{align*}
  cw(W)cw(V)&(-d,
  -\disc_{B'}(W))_F\\
&=C_0(W_1)\widehat{\otimes}(\disc_{B'}(W_1))_F\widehat{\otimes}(a)_F(-d,
                  -\disc_{B'}(W))_F\\
&= cw(W_1)\cdot(\disc_{B'}(W_1),a)_F\cdot(-d, -a\cdot \disc_{B'}(W_1))_F\\
&=cw(W_1)(a,\disc_{B'}(W_1))_F(-d,\disc_{B'}(W_1)_F(-d,-a)_F\\
&=cw(W_1)(-ad,\disc_{B'}(W_1))_F(-d,-ad)_F\\
&=cw(W_1)(-d\disc_{B'}(W_1),-ad)_F,
\end{align*}
and we are done with the case $\dim(V)=1$.

We use now induction on $\dim(V)+\dim(W)$. If both dimensions have
opposite parity we may assume that $\dim(V)$ is odd.  We write $V=V_1\perp[d]$ and have
\begin{align*}
 c&w(V\perp W)\\&= cw(V_1\perp W \perp [d])\\
&=cw(V_1\perp W)(-d, \disc_B(V_1)\disc_{B'}(W))_F\\
&=cw(V_1)cw(W)(\disc_B(V_1),\disc_{B'}(W))_F(-d,  \disc_B(V_1)\disc_{B'}(W))_F\\
&=cw(V)(-d,\disc_B(V_1))_Fcw(W)(\disc_B(V_1),\disc_{B'}(W))_F(-d,
    (\disc_B(V_1)\disc_{B'}(W))_F\\
&=cw(V)cw(W)(\disc_B(V_1),-d)_F(\disc_B(V_1),\disc_{B'}(W))_F(-d, \disc_B(V_1)\disc_{B'}(W))_F\\
&=cw(V)cw(W)(\disc_B(V_1),\disc_{B'}(W))_F(-d,\disc_{B'}(W))_F\\ 
&=cw(V)cw(W)(-\disc_B(V),\disc_{B'}(W))_F
\end{align*}
as asserted.

If both dimensions have the same parity we assume first that both are
even and write 
$V=V_1\perp[d]$ again and obtain 
\begin{align*}
  c&w(V\perp W)\\
&=cw(V_1\perp W)(d,\disc_B(V)\disc_{B'}(W))_F\\
&=cw(V_1)cw(W)(-\disc_B(V_1),\disc_{B'}(W))_F(d,\disc_B(V)\disc_{B'}(W))_F\\
&=cw(V)cw(W)(d,\disc_B(V_1))_F)(-\disc_B(V_1),\disc_{B'}(W))_F(d,\disc_B(V)\disc_{B'}(W))_F\\
&=cw(V)cw(W)(-d \disc_B(V_1),\disc_{B'}(W))_F\\
&=cw(V)cw(W)(\disc_B(V),\disc_{B'}(W))_F
\end{align*}
as asserted. If finally both dimensions are odd, we obtain in
the second step above a factor
$(\disc_B(V_1),-\disc_{B'}(W))_F$ instead of $(-\disc_B(V_1),\disc_{B'}(W))_F$  and proceed analogously.
\end{proof}
\begin{corollary}
 The Clifford-Witt invariant $cw(V,Q)$ depends only on the Witt class of
 $(V,Q)$. In particular, if $(V,Q)$ is hyperbolic or the orthogonal
 sum of a hyperbolic space and a one dimensional space, the Clifford-Witt
 invariant $cw(V,Q)$ is trivial.
\end{corollary}
\begin{proof}


From Example \ref{clifford_lowdimensions} we see that a hyperbolic
plane has trivial Clifford-Witt invariant.
The corollary then  follows from the theorem since hyperbolic spaces have
discriminant $1$.
\end{proof}
\begin{proposition}
The Clifford-Witt invariant and the Hasse invariant are related by
  \begin{equation*}
    s(V)=cw(V)\cdot
    \begin{cases}
      1& \dim(V)\equiv 1,2 \bmod 8\\
(-1,-{\det}_B(V))_F &\dim(V)\equiv 3,4 \bmod 8\\
(-1,-1)_F&\dim(V)\equiv 5,6 \bmod 8\\
(-1,{\det}_B(V))_F&\dim(V)\equiv 0,7 \bmod 8\\
    \end{cases}.
  \end{equation*}
In particular, the Hasse invariant is independent of the choice of an
orthogonal basis of $(V,Q)$.
\end{proposition}
\begin{proof}
We write $V=V_1\perp[a]$ and use induction on $\dim(V)$; the assertion
is clear for $\dim(V)=1,2$. Assume $\dim(V)>2$ and the assertion to be
proven for all smaller dimensions. We have always $s(V)=s(V_1)(a,
{\det}_B(V_1))_F$ by the multiplicative properties of the symbol
$(\quad,\quad)_F$.
If $\dim(V_1)$ is even, we have $cw(V)=cw(V_1)(-a,
(-1)^{\dim(V)/2}\det_B(V))_F$, if $\dim(V_1)$ is odd, we have $cw(V)=cw(V_1)(a,
(-1)^{(\dim(V)-1)/2}\det_B(V))_F$.  Upon inserting the relations given
for $V_1$ by the inductive assumption and comparing the factors we
obtain the assertion.
\end{proof}
\begin{remark}
  Let $(V,Q)$ be a non degenerate quadratic space of dimension $n=2m$ or
  $n=2m+1$ over $F$ and $W$ a hyperbolic space of dimension $2m$ over
  $F$, put $\tilde{V}=W$ if $\dim(V)=2m$ is even and $\tilde{V}\cong
  W\perp [1]$ if $\dim(V)$ is odd. Then $cw(V)=cw(V\perp \tilde{V})$
  and $cw(V\perp \tilde{V})$ is the Brauer class of the Clifford
  algebra of $V\perp \tilde{V}$. One could use this construction to
  define the Clifford-Witt invariant without a distinction of cases
  between even and odd dimensions. Witt defined his invariant in \cite{witt}
  similarly but used a space $V'\cong [-1]\perp\dots\perp[-1]$ of dimension
  $n$ instead of our $\tilde{V}$. 
\end{remark}
\section{Classifying invariants over local and global fields}
\begin{lemma}
   Let $F=F_v$ be a local field of characteristic $\ne 2$ with
   valuation $v$. The subgroup of
   the Brauer group $Br(F)$ generated by the classes of quaternion
   algebras is isomorphic to $\{\pm 1\}$ if $F\not\cong \C$ and is
 trivial for $F=\C$.

Under this isomorphism the Brauer class of the algebra $(a,b)_F$ is mapped to
the Hilbert symbol $(a,b)_v$.
\end{lemma}
\begin{proof} The case $F=\C$ is trivial, for $F=\R$ it is well known
  that the Hamilton quaternions are up to isomorphy the only
  quaternion division algebra over $F$.  We therefore assume that $F$
  is not archimedean in the sequel.

By Theorem \ref{anisotropic_local} there is a unique isometry class of anisotropic
quadratic spaces of dimension $4$ over $F$,and this space has
determinant $1$ and is universal. It is therefore isometric to
$[1]\perp[-a]\perp[-b]\perp[ab]$ for some $a,b \in F^\times$ and is
hence the quadratic space given by the quaternion algebra $(a,b)_F$
with the quaternion norm as quadratic form. This quaternion algebra is
therefore the unique quaternion division algebra over $F$.  
\end{proof}
\begin{remark}
For a local field $F$ we identify the the Brauer class of the
quaternion algebra $(a,b)_F$ with the Hilbert symbol $(a,b)_v\in \{\pm
1\}$.
\end{remark}
\begin{theorem}\label{classification_local}
 Over a local field $F$ of characteristic $\ne 2$ regular quadratic
 spaces $(V,Q)$ are classified up to isometry by dimension,
 determinant, and Clifford-Witt (or Hasse) invariant.

In particular, for given dimension and determinant   there exist at
most two isometry classes of regular quadratic spaces over $F$.
\end{theorem}
\begin{proof}
  Dimension $1$ is trivial, dimension 2 follows from Lemma
  \ref{quaternion_clifford_dim2} and Example
  \ref{clifford_lowdimensions}, dimension 3 from Lemma
  \ref{quaternion_clifford_dim3} and Example
  \ref{clifford_lowdimensions}, both without any assumption on $F$.



Assume now that $\dim(V)=\dim(W)=n>3$ and that the assertion is true
for spaces of dimension smaller than $n$. The spaces $V\perp [-1],
W\perp [-1]$ are isotropic, having dimension $\ge 5$,  hence $V,$
represent $1$ and can be written as $W\cong [1]\perp V', W\cong
[1]\perp W'$ with $s(V')=s(W'), \det(V')=\det(W')$ the inductive
assumption implies $V'\cong W'$, and the assertion follows.   
\end{proof}
\begin{theorem}
  Let $F=F_v$ be a non archimedean local field of characteristic $\ne
  2$, let $(V,Q)$ be a non degenerate quadratic space over $F$ of
  dimension $\ge 2$, not isometric to a hyperbolic plane.

Then there exists a non degenerate quadratic space $(W,Q')$ over $F$
of the same dimension and determinant as $(V,Q)$ with $s(W)=-s(V)$
(hence $cw(W)=-cw(W)$). 
\end{theorem}
\begin{proof}
  If $\dim(V)=2$ we put $(V,Q)\cong [a,b]$ and have
  $s(V)=cw(V)=(a,b)_v$, by assumption $-ab$ is not a square in $F$.
If $-ab$ is a unit choose $c$ to be a prime element $\pi$ in $F$, if
$-ab$ is (up to squares) a prime element, let $c\in F$ be a unit for
which $F(\sqrt{c})$ is the (unique)  unramified quadratic extension of
$F$.

In both cases we have $(ca,cb)_v=(ca,-ab)_v=(c,-ab)_v(a,b)_v$ with
$(c,-ab)_v=-1$, so $(W,Q')\cong [ca,cb]$ is as desired.

Let now $\dim(V)\ge 3$ and let $c$ be the non square unit of $F$ given
above, let $U\cong [1,\pi], U'\cong [c,c\pi]$, where $\pi$ is a prime
element, we have $\det(U)=\det(U')$ and $cw(U)=-cw(U')$.

The space $V\perp U$ has dimension $\ge 5$ and is hence isotropic and
universal, so it represents $c$ and can be written as $V_1\perp [c]$,
where $V_1$ has dimension $4$ and is hence universal as well, so it
represents $c\pi$. we can therefore write $U\perp V\cong U'\perp V'$,
where $V'$ has the same dimension and determinant as $V$ but opposite
Clifford-Witt invariant.  
\end{proof}
\begin{corollary}
  For $n\ge 3$ there exists a (up to isometry unique) regular
  quadratic space over the non archimedean local field $F$ with
  $\cha(F)\ne 2$ with given determinant and Clifford-Witt invariant.
\end{corollary}
\begin{proof}
  Obvious.
\end{proof}
\begin{theorem}
  Let $F$ be a number field, fix $n \in \N$ and let non degenerate
  quadratic spaces $(V_v,Q_v)$ of dimension $n$ over $F_v$ be given
  for all places $v$ 
  of $F$, denote by $S_v$ resp.\ $cw_v$ the Hasse resp.\ Clifford-Witt
  invariants for the completion $F_v$.

Then a quadratic space $(V,Q)$ over $F$ with all completions isometric
to the given spaces  $(V_v,Q_v)$ exists if and only if one has
\begin{enumerate}
\item There exists $d\in F$ with $d(F_v^\times)^2=\det(V_v)$ for all
  places $v$ of $F$.
\item The Hasse invariants $s_v(V_v)$ (equivalently: The Clifford-Witt
  invariants) are equal to $1$ for almost all
  places $v$.
\item The product of the Hasse (or the Clifford-Witt) invariants
  $s_v(V_v)$ (resp.\  $cw_v(V_v)$ over all places $v$ is $1$.
\end{enumerate}
\end{theorem}
\begin{proof}
  The necessity of the conditions follows immediately from the
  reciprocity theorem for the Hilbert symbol.
We have to prove the existence of the global space if the conditions for the
local spaces are satisfied; this is trivial for $n=1$.

Let $n\ge 2$ and write $(V_v,Q_v)\cong [a_1^{(v)},\ldots,a_n^{(v)}]$
for all $v$ in the set of places $\Sigma_F$  of $F$.
Let $T\subseteq \Sigma_F$ be a finite set containing all archimedean
places and all places $v$ with $s_v(V_v)=-1$.
By the weak approximation theorem of algebraic number theory we find
$a_1,\ldots,a_{n-1} \in F$ such that $a_i$ is in the square class of
$a_i^{(v)}$ at all $v\in T$ for $1\le i \le n-1$.

The space $W\cong [a_1,\ldots,a_{n-1},da_1a_2\dots a_{n-1}]$ over $F$
of determinant $d(F^\times)^2$ 
satisfies $W_v\cong V_v$ for all $v \in T$ and hence
$s_v(W_v)=s_v(W_v)$ for all $v \in T$.

The set $S=\{v \in \Sigma_F\mid s_v(W_v)=-s_v(V_v)\}$ is therefore a
finite subset of $\Sigma_F\setminus T$ satisfying $W_v\cong V_v$ for
all $v\not\in S$, since for these $v$ both spaces have Hasse invariant
$1$ (and the same determinant). On the other hand, for $v\in S$ the
spaces $V_v$ and $W_v$ have the same determinant and are not
isometric, so that in particular neither of these spaces over $F_v$ can be a
hyperbolic plane. 
 
If $S\ne \emptyset$ we have $S=\{v\in \Sigma_F\setminus T\mid
s_v(W_v)=-1\}$, and 
one sees
\begin{equation*}
  \prod_{v\in S}s_v(W_v)=\prod_{v\in S}s_v(W_v)\prod_{v\notin
    S}s_v(V_v)=\prod_{v\in \Sigma_F}s_v(W_v)=1,
\end{equation*}
so that $S$ has even cardinality.

Again by weak approximation we find $\beta\in F^\times$ which is a
square at all archimedean places and not a square at all $v \in S$,
and by the Hilbert reciprocity law we can find $\alpha \in F^\times$
satisfying $(\alpha,\beta)_v=-1$ if and only if $v \in S$.

Let $U,U'$ be binary quadratic spaces over $F$ with $U\cong
[1,-\beta], U'\cong [\alpha,\-alpha\beta]$. We have $s_v(U') = (\alpha,
-\alpha
\beta)_v=(\alpha,-1)_v(\alpha,-\beta)_v=(\alpha,\beta)_v$
for $v \notin S$, hence $U_v \cong U'_v$ for $v \notin S$ and  $U_v
\not\cong U'_v$ for $v \in S$.

For $n\ge 3$ and $v \in S$ we obtain as in the proof of Theorem
\ref{classification_local} that $U_v'$ is represented by $U_v\perp
W_v$. If $n=2$ holds, the space $U_v \perp W_v$ represents $\alpha$
since it is universal, and we can write $U_v\perp W_v\cong [\alpha,
c_1,c_2,c_3]$.
If $-\alpha\beta$  were not represented by $ [c_1,c_2,c_3]$, the space $[
c_1,c_2,c_3,\alpha\beta]$ would be anisotropic, hence of square
determinant, so we had $\alpha\beta \in c_1c_2c_3(F^\times)^2$. From
this we find $\det(U_v\perp W_v)=\alpha c_1c_2c_3
(F^\times)^2=\beta(F^\times)^2=-\det(U_v)$, which implies that $W_v$
is a hyperbolic plane, a contradiction. So $U'_v$ is represented by
$U_v\perp W_v$ in the case $n=2$ too. Since $U_v\cong U'_v$ for all
$v\notin S$ we see that $U'$ is represented by $U\perp W$ locally
everywhere, and by the Minkowski Hasse local global principle we can
write $U\perp W\cong U'\perp W'$ with some space $W'$ over $F$
of dimension $n$ and determinant $d$  which satisfies
$s_v(W_v')=s_v(V_v)$ for all $v \in \Sigma_F$ as desired. 
\end{proof}
\begin{corollary}
  For any given signature $(n_+,n_-)$ there exists at most one genus
  of even unimodular $\Z$-lattices. Such a genus exists if and only if
  one has $n_+\equiv n_- \bmod 8$.
\end{corollary}
\begin{proof}
An even unimodular $\Z$-lattice $\Lambda$ of signature $(n_+,n_-)$ has
determinant $(-1)^{n_{-}}$. The isometry class of
$\Lambda_p$  is by Hensel's lemma determined by the isometry class of
its reduction $\Lambda_p/p\Lambda_p$ as (regular) quadratic space over $\F_p$,
which for odd $p$ is determined by the determinant. This shows that
the isometry class of the completion of the space $V$ supporting $\Lambda$ is determined
by the signature at all places of $\Q$ except perhaps at the prime
$p=2$. But then the Hilbert reciprocity law implies that the Hasse invariant
at $2$ is also determined, which by Theorem \ref{classification_local}
implies  that there is only one possibility for the isometry class of
the completion $V_2$ at the prime $2$, and by the Minkowski-Hasse
theorem the class of $V$ is determined by the fixed signature.
Finally, an even unimodular lattice has $\Z_p$-maximal completions at
all primes $p$, and since all $\Z_p$-maximal lattices on the same
space $V_p$ are isometric, the genus of such a lattice is uniquely
determined, if it exists.

For the existence question we have by the Hilbert reciprocity law
that $s_2(V_2)=(-1)^{\frac{n_-(n_--1)}{2}}$. Moreover, the space $V_2$
supports an even unimodular lattice of determinant $(-1)^{n_-}$ and
dimension $n_-+n_+=n$ if and
only if $n=2m$ is even and $V_2$ is hyperbolic, in which case its
Hasse invariant is $(-1)^{\frac{m(m -1)}{2}}$. This is equivalent to  $n_-\equiv
m \bmod 4$, hence to $n_++n_-=2m\equiv 2n_-\bmod 8$, which is finally
equivalent to $n_+\equiv n_- \bmod 8$ as asserted.
\end{proof}
\begin{remark}
  There exist several different proofs of the last corollary. The
  existence of such a lattice also follows from the explicit
  construction of the positive definite $E_8$ root-lattice.
\end{remark}

\chapter{The Maßformel of Smith, Minkowski and Siegel}
In this chapter we continue the study of quadratic lattices over a
ring of integers $R$  in a global field $F$ of characteristic different from
$2$, in particular over the ring of integers of a number field. We
keep the notations of chapter \ref{globalchapter}. If $R=R_T$ for some
fixed set $T$ of places of $F$ we will usually omit $T$ in the notations.  
A (quadratic) $R$-module $\Lambda$ for such an $R$ can always be
viewed as  a lattice on the vector space $V=\Lambda \otimes_R F$ over
$F$, with the natural extension of the quadratic form $Q$ and its
associated symmetric bilinear form $b$ to $V$. 

\smallskip
As a consequence of the Minkowski-Hasse theorem we had seen in Theorem
\ref{representations_localglobal} that a lattice $M$ which is
represented locally everywhere by the lattice $\Lambda$ is represented
globally by some lattice in the genus of $\Lambda$. Siegel's Maßformel
(or mass formula or measure formula), also called Siegel's main
theorem for (integral) quadratic forms, gives a quantitative version
of this statement, expressing a weighted average over the genus of
$\Lambda$ of the numbers or measures of representations of $M$ by the
lattices in the genus as a product of local data. This extended
earlier work of Smith and of Minkowski on the average of the inverses of the
numbers of isometric automorphisms of the lattices in a genus.

We take the occasion to add a comment on the naming of the formula: In
Siegel's three fundamental articles ``Über die analytische Theorie der
quadratischen Formen'' \cite{siegel1,siegel2,siegel3} written in german in Annals of Mathematics,
Siegel calls it the Massformel, following earlier terminology of
Eisenstein und Minkowski, who defined and investigated  the ``Maß
eines Genus'' (english: the measure of a genus). Since the printer of
the Annals apparently had no type for the special german letter ß, the
word was not printed as Maßformel but as above as Massformel, which in
later english texts 
changed to mass formula. It is, however, concerned with measure
(german: Maß), and not with mass (german: Masse) and would better be
translated as measure formula if one doesn't want to use the original
german word.   
\section{Class and genus of a representation}     
\begin{definition}
 Let $R$ be a commutative ring and $\phi_i:(M,Q')\to (\Lambda_i,Q_i)$ for $i=1,2$ be
 representations of quadratic $R$-modules. One says that
 $\phi_1,\phi_2$ belong to the same {\em class of representations} or that
 they are equivalent over $R$ if there is an isometry
 $\psi:(\Lambda_1, Q_1) \to (\Lambda_2,Q_2)$ with $\phi_2=\psi \circ \phi_1$; one writes then
 $\phi_2 \in \cl(\phi_1)$.

If $R$ is a ring of integers of the global field $F$ the representations
$\phi_1,\phi_2$ are said to belong to the same {\em genus (more precisely:
$R$-genus or $T$-genus for $R=R_T$) of
representations} if their coefficient extensions to the completions $R_v$ are
equivalent over $R_v$ for all $v \in \Sigma_F$, one writes then
$\phi_2 \in \gen(\phi_1)$ or $\phi_2 \in \gen_T(\phi_1)$.
\end{definition}
Obviously, genera of representations consist of full classes of
representations and the relations of belonging to the same class or to
the same genus of representations are equivalence relations. If $\phi$
is a primitive representation then all representations in its class
and in its genus are primitive.

Representations in the same class are by isometric quadratic modules.
Similarly, representations in
the same genus are representations by lattices in the same genus of
lattices.

If $\phi:(M,Q')\to (\Lambda,Q)$ is a representation we can instead
consider the representation $i=i_\phi:(\phi(M),Q\vert_{\phi(M)})
\to (\Lambda,Q)$ of the lattice $\phi(M)$ in the class of $M$, where
$i$ is the inclusion mapping.
The following definition and lemma allow us to restrict attention to
inclusions (sometimes also called special representations) of
quadratic $R$-modules whenever that is convenient:
\begin{deflemma}
Let $M_1,M_2$ be submodules of the quadratic $R$-modules $(\Lambda_1,Q_1),(\Lambda_2,Q_2)$ and
denote by $i_1,i_2$ their respective inclusions into $\Lambda_1,\Lambda_2$.

The inclusions $i_1,i_2$ are in the same class (or equivalent over
$R$) if there exists an
isometry $\psi:\Lambda_1 \to \Lambda_2$ with $\psi \circ
i_1=i_2$. If $R$ is a ring of integers in a global field, they
are in the same genus if they are equivalent over all $R_v$.

\smallskip 
Two representations $\phi_1:(M,Q')\to (\Lambda_1,Q_1)$, $\phi_2:(M,Q')\to
(\Lambda_2,Q_2)$ of the quadratic $R$-module $(M,Q')$ are in the same class
if and only if the inclusions $i_{\phi_1},i_{\phi_2}$ of the
$\phi_j(M)$ into $\Lambda_j$ are in the same class of inclusions as
defined above, and analogously for the genus of representations resp.\ inclusions.   
\end{deflemma}
\begin{proof}
Obvious.
\end{proof}

\begin{lemma} Let $(M,Q')$ and $(\Lambda_j, Q_j)$ for $j=1,2$ be
  quadratic lattices on the quadratic spaces $(W,
  Q'),(V_1,Q_1),(V_2,Q_2)$ over $F$. 
  Let $\phi_j:M \to \Lambda_j$ for $j=1,2$ be representations in the
  same genus of representations.

Then there is a lattice $\Lambda'$ on $V_1$ and a representation
$\phi':M\to \Lambda'$ in the class of the representation $\phi_2$.

If one has here $M\subseteq \Lambda_1$ and $\phi_1$ is the inclusion
mapping from $M$ into $\Lambda_1$, the representation $\phi'$ above
can also be chosen to be the inclusion of $M$ into $\Lambda'$. 
\end{lemma}
\begin{proof}
By assumption $(V_1,Q_1)$ is isometric to $(V_2,Q_2)$ over all
completions $F_v$, by the Hasse-Minkowski theorem there exists an
isometry $\sigma:V_2\to V_1$ over $F$. Then $\phi':=\sigma \circ \phi_2:M \to
\Lambda':=\sigma(\Lambda_2)$ is a representation in the class of
$\phi_2$ as requested.   

If we assume now that $M$ is a sublattice of $\Lambda_1$ and that
$\phi_1$ is the inclusion mapping, we choose $\sigma:V_2\to V_1$ and
$\phi'=\sigma \circ \phi_2$ as
above. The isometric map $\phi':M\to \Lambda'\subseteq V_1$ extends to
an isometry from the vector space $U\subseteq V_1$ over $F$ generated
by $M$ onto the space $U'$ generated by $\phi'(M)\subseteq V_1$. We
denote by $\rho'$ its inverse and find by the theorem of Witt an
isometry $\rho \in O(V_1,Q_1)$ extending $\rho'$. The map
$\phi''=\rho\circ \phi':M \to \Lambda'':=\rho(\Lambda')\subseteq V_1$
is then the identity on $U$, so that $\phi''$ and $\Lambda''$ are as requested.  
\end{proof}
The next theorem transfers the study of the classes in the genus of
a fixed representation into the study of double cosets in adelic
orthogonal groups.
\begin{theorem}
  Let $R$ be a ring of integers in the global field $F$, let $\Lambda$
  be an $R$-lattice on the non degenerate quadratic space $(V,Q)$ over
  $F$. Let $V=W\perp U$ be an orthogonal splitting of $V$ into non
  degenerate subspaces $W,U$, write $O_U(F),O_V(F),O_W(F)$ for the
  orthogonal groups of the spaces $(U,Q),(V,Q), (W,Q)$ and
  $O_U(\A_F),O_V(\A_F),O_W(\A_F)$ for their adelizations, consider the
  (adelic) orthogonal groups of $U,W$ as subgroups of the (adelic)
  orthogonal group of $V$ (acting trivially on the respective
  orthogonal complement).

Let $M\subseteq \Lambda$ be a lattice on $W$ and denote
  by  $i_{\Lambda'}=i_{M,\Lambda'}$ the inclusion map of $M$ into
  any lattice $\Lambda'$ on $V$ containing $M$.

Then $O_U(\A_F)$ acts transitively on the set of inclusions
$i_{M,\Lambda'}$ in the genus of the inclusion $i_{M,\Lambda}$ by
$i_{M,\Lambda}\mapsto i_{M,\phi(\Lambda)}$ for $\phi \in O_U(\A_F)$.
This action induces a transitive action  on the set of classes of
representations in the genus of  $i_{M,\Lambda}$. 

The stabilizer of $i_{M,\Lambda}$ is the group
\begin{equation*}
 O_{U}(\A_F;\Lambda):=O_U(\A_F)\cap O_\Lambda(\A_F)=\{\phi \in O_U(\A_F) \mid \phi(\Lambda)=\Lambda\}.
\end{equation*}
For $\phi,\phi' \in O_U(\A_F)$ the inclusions
$i_{M,\phi(\Lambda)},i_{M,\phi'(\Lambda)}$ are in the same class if
and only if 
\begin{equation*}
O_U(F)\phi O_{U}(\A_F;\Lambda)=O_U(F)\phi'
O_{U}(\A_F;\Lambda),
\end{equation*}
 i.e., one has a bijection between classes of
representations in the genus of $i_{M,\Lambda}$ and double cosets
$O_U(F)\phi O_{U,\Lambda}(\A_F)$ with $\phi \in O_U(\A_F)$.
\end{theorem}
\begin{proof}
  Obvious.
\end{proof}
\begin{remark}
The theorem implies that we have a bijection between classes of
representations in the genus of $i_{M,\Lambda}$ and double cosets
$O_U(F)\phi O_{U}(\A_F;\Lambda)$ in $O_U(\A)$.
\end{remark}
\begin{theorem}
  Let $R$ be the ring of integers of the number field $F$, let $M$ be
  an $R$- lattice on the non degenerate quadratic space $(W,Q')$
  over $F$. Then for any given dimension $n$ and fixed $d\in \N$ there
  are only finitely many classes of representations of $(M,Q')$ into
  $(\Lambda,Q)$, where $\Lambda$ is an integral $R$-lattice on a non
  degenerate quadratic space $(V,Q)$ of dimension $n$ over $F$ such
  that the norm of the volume ideal of $(\Lambda,Q)$ equals $d$. 
\end{theorem}
\begin{proof}
  We can restrict attention to inclusions $M\to \Lambda$. With
  $V=W\perp U$ we put $K=\Lambda\cap U$. By Theorem
  \ref{duals_and_indices} (which, as noticed in Section
  \ref{dedekindsection}) stays valid for lattices over $\fo$ we have
  $\vol_{b} K \mid \vol_b(M) \vol_b\Lambda$ by b) of that Theorem. By
  the version of Theorem \ref{finiteness_classnumber} for number rings
  (see the remark after that theorem) there are only finitely many
  possibilities for the isometry class of $K$, denote
  representatives by $(K_i,Q_i)$ and by $U_i$ the $F$-space generated
  by $K_i$, put $V_i=W\perp U_i$. If $\psi:K \to K_i$ is an isometry,
  the inclusion of $M$ into the lattice
  $(\id_W,\psi)(\Lambda)\subseteq M^\#\perp K_i^\#\subseteq V_i$
  (where the duals are taken with regard to the spaces of the lattices
  $M,K_i$) is in the class of the inclusion $M\to\Lambda$ and has
  image in one of the finitely many lattices lying between $M\perp
  K_i$ and  $M^\#\perp K_i^\#$. Taken together, the given inclusion
  $M\to \Lambda$ is in the class of one of finitely many inclusions,
  as asserted.   
\end{proof}
It is of interest to study representations of degenerate lattices as
well, e.g. representations of $0$. Indeed with a suitable primitivity
condition we can extend the finiteness statement above to that
situation.
\begin{definition}
Let $\phi:M\to \Lambda$ be a representation of lattices on quadratic
spaces $(W,Q'), (V,Q)$ and extend $\phi$ to an isometric map (also denoted by
$\phi$) from $W$ to $V$.

The representation $\phi$ is said to be of {\em imprimitivity bounded by
$c\in \N$} if one has $(\phi(W)\cap\Lambda:\phi(M))\le c$. We also call
the index  $(\phi(W)\cap\Lambda:\phi(M))$ the {\em imprimitivity index} of
the representation $\phi$.
\end{definition}
\begin{remark}
  Obviously $\phi$ is primitive if and only if it is of imprimitivity
  bounded by $1$. Since representations in the same class have the
  same imprimitivity index and the imprimitivity index can assume arbitrarily
  large values for representations of a degenerate lattice, it is
  clear that a finiteness statement as in the theorem above can not
  hold for representations of degenerate lattices unless one imposes a
  bound on the imprimitivity index.
\end{remark}
\begin{theorem}
 Let $M$ be a lattice on the quadratic space $(W,Q')$
  over the number field $F$. Then for any given dimension $n$ and fixed $c,d\in \N$ there
  are only finitely many classes of representations of $(M,Q')$ into
  $(\Lambda,Q)$ of imprimitivity bounded by $c$, where $\Lambda$ is an integral lattice on a non
  degenerate quadratic space $(V,Q)$ of dimension $n$ over $F$ such
  that the norm of the volume ideal of $(\Lambda,Q)$ equals $d$.    
\end{theorem}
\begin{proof} Again we have to consider only inclusions $i_{M,\Lambda}:M\to \Lambda$.
  We can moreover restrict attention to primitive representations: If
  $i_{M\Lambda}:M\to \Lambda$ is (without loss of generality) an
  inclusion of imprimitivity index   bounded by $c$ and $M'=W\cap
  \Lambda$, then $M$ is one of the finitely many sublattices of $M'$
  of index $\le c$, so if there are only finitely many possibilities
  for the class of the inclusion $i_{M',\Lambda}$, the same holds for
  the class of $i_{M,\Lambda}$.

For simplicity we treat the case $\R=\Z$ first and indicate the
necessary changes for general $R$ in the end.

We assume first that $\rad(M)=M$, i.e., $Q\vert_W=0$, $(M,Q)$ is
totally isotropic. With
$\tilde{b}^{(W)}:V\to W^*$ as usual $\tilde{b}^{(M)}(\Lambda)$ is a
sublattice of full rank of  $M^*$ (which we identify with $\{f\in W^* \mid
f(M)\subseteq R\}$), so there are a basis $(y_1^*,\ldots,y_r^*)$ of
$M^*$ and $c_1,\ldots c_r \in R$ with $c_i\mid c_{i+1}$ such that the
$c_iy_i^*$ form a basis of $\tilde{b}^{(M)}(\Lambda)$. We choose $z_i
\in c_i^{-1}\Lambda\subseteq V$ with $\tilde{b}^{(M)}(z_i)=y_i^*$ and denote by
$(x_1,\ldots,x_r)$ the basis of $M$ dual to $(y_1^*,\ldots,y_r^*)$, by
primitivity of $M$ it can be extended to a basis of $\Lambda$. We
have then $b(x_i,\Lambda)=c_i\Z$ and see by looking at the Gram matrix
of this basis of $\Lambda$ that $c_1\dots c_r$ divides
$\det(\Lambda)$. Since $\det(\Lambda)$ is fixed this leaves only
finitely many possibilities for the $c_i$.
 
The $c_iz_i$ are determined only modulo elements of the orthogonal
complement $K\supseteq M$ of $M$ in $\Lambda$, so changing $c_iz_i$ by
vector in $\sum_{j=1}^i\Z x_j\subseteq M$ for $1\le i \le r$ we can achieve
$0\le b(c_iz_i,c_jz_j)\le c_i$ for $j \le i$. The lattice $M'=M+\sum_{i=1}^r
\Z c_iz_i\subseteq \Lambda $ is then a non degenerate sublattice of
$\Lambda$ of rank $2r$ whose Gram matrix with respect to the basis
$(x_1,\ldots,x_r,z_1,\ldots,z_r$ is from a finite supply, so there are
only finitely many possible isometry classes for $M'$. Moreover, for
each of these $M'$ there are only finitely many possible classes of
inclusions $M\to M'$. By the previous
theorem each of the possible non degenerate $M'$  has only finitely
many classes of inclusions  into lattices $\Lambda$ of rank $n$ and determinant $d$.

To sum up, each inclusion $M\to \Lambda$ is in the class of one of the
finitely many composite inclusions $M\to M'\to \Lambda$ constructed
above, which proves the assertion in the case of totally isotropic $M$
over $\Z$. 

For general $M$ we write $M=\rad(M)\perp M_1$ with non degenerate $(M_1,Q)$.
By the previous theorem there are only finitely many classes of
inclusions $i_{M_1,\Lambda}:M_1 \to \Lambda$ with $\Lambda$ integral of rank $n$ and
determinant $d$, these lead to finitely many possible isometry classes
for the orthogonal complement $K$ of $M_1$ in $\Lambda$. Our result in
the totally isotropic case implies that we have to consider only
finitely many classes of inclusions $\rad(M)\to K$. Finally, any
inclusion $i_{M,\Lambda}:M\to \Lambda$ can be written as
$(i_{\rad(M),K},i_{M_1,\Lambda})$ as above, so there are only finitely
many classes of such inclusions, as asserted.

The case of general $R$ requires some modifications in the argument
for the totally isotropic case since we have used the existence of
bases at several points.

We have now $M^*=\oplus \mfa_iy_i^*, \tilde{b}^{(M)}(\Lambda)= \oplus
\mfc_i \mfa_i y_i^*$ with linearly independent $y_i*\in W^*$, fractional ideals
$\mfa_i$ from a fixed set of representatives of the ideal classes in
$R$ and integral ideals $\mfc_i\subseteq R$. We choose 
again $z_i\in (\mfa_i\mfc_i)^{-1}\Lambda \subseteq V$ with
$\tilde{b}^{(W)}(z_i)=y_i^*$ and $x_1,\ldots,x_r\in W$ 
with $y_i^*(x_j)=\delta_{ij}=b(z_i,x_j)$. We have then
$M=\oplus_{i=1}^r \mfa_i^{-1}x_i$ and
$b(x_i,\Lambda)=\mfa_i\mfc_i$. Over each completion all the ideals are
principal, so we can argue locally as in the case $R=\Z$ to see that
$\mfc_1\dots\mfc_r$ divides the volume ideal of $\Lambda$, thus there
are again only finitely many possibilities for these ideals.  Changing
the $z_i$ as above modulo $\mfc_i^{-1}\mfa_i^{-1}\sum_{j=1}^i \mfa_j^{-1}x_j\subseteq  $ we
can change $b(z_i,z_j)\in (\mfa_i\mfa_j\mfc_i\mfc_j)^{-1}$ modulo
$(\mfc_i\mfa_i\mfa_j)^{-1}$ for $j \le i$, which allows us to reduce
the $b(z_i,z_j)$ to a finite set of representatives of
$(\mfa_i\mfa_j\mfc_i\mfc_j)^{-1}$ modulo
$(\mfc_i\mfa_i\mfa_j)^{-1}$. There are hence only finitely many
possibilities for the isometry class of $M \oplus \sum_{i=1}^r
\mfa_i\mfc_i z_i \subseteq \Lambda$, and the rest of the proof
proceeds as above.      
\end{proof}
\begin{theorem}
 Let $R=R_v$  be the ring of integers in the local non archimedean
 field $F=F_v$, let $M$ be a lattice on the quadratic space $(W,Q')$
 over $F$, let $c \in \N$. Then there are only finitely many classes of
 representations of  imprimitivity bounded by $c$ into non degenerate
 quadratic lattices
 $(\Lambda,Q)$ of fixed rank $n$ and determinant $d$.
\end{theorem}
\begin{proof}
  As in the global case we can restrict attention to primitive
  inclusions $i_{M,\Lambda}:M\to\Lambda$. By considering the possible
  Jordan decompositions of $\Lambda$ of fixed rank and determinant one
  sees that there are only finitely many possibilities for the
  isometry class of $\Lambda$. The rest of the proof  goes through as in the global case.  
\end{proof}
\begin{corollary}
  For $M$ as in the theorem and $\Lambda$ fixed there are only
  finitely many classes of representations of $M$ by $\Lambda$.

All representations of $M$ by $\Lambda$ for which the image of $M$ is
regularly embedded into $\Lambda$ are in the
same $R=R_v$-class, in particular, if $M$ is regular, there is only
one class of representations of $M$ by $\Lambda$. 
\end{corollary}
\begin{proof}
  The finiteness assertion is a consequence of the theorem. The
  assertion for regularly embedded  $M$ follows from Theorem \ref{extension_theorem_localring}.
\end{proof}
\begin{theorem}
  Let $R$ be the ring of integers in the number field $F$, let $M$ be
  an $R$ lattice on the quadratic space $(W,Q')$ over $F$. Then the
  number of genera of representations of $M$ by lattices $\Lambda'$ in
  the genus of a fixed non degenerate quadratic lattice $(\Lambda,Q)$
  with index of imprimitivity bounded by some fixed integer $c$ is finite.
\end{theorem}
\begin{proof}
  It is easily seen that for almost all places $v$ of $F$ the
  representations of $M_v$ by $\Lambda_v$ are regularly embedded and
  that the number of genera of representations is the product over all
  $v$ of the numbers $c_v(M_v,\Lambda_v)$ of representations of $M_v$
  by $\Lambda_v$ (with the boundedness condition on the
  imprimitivity). Since almost all of these numbers are $1$ the
  assertion follows. 
\end{proof}
\section{Representation and measures in the adelic orthogonal group}
Let $(V,Q)$ be a quadratic space over the number field $F$, let $\Lambda$ be
a lattice on $V$. The adelic (special) orthogonal group 
$O_V(\A)=\{ \phi=(\phi_v)_{v\in \Sigma_F}\mid \phi_v \in O_V(F_v),
\phi_v(\Lambda)=\Lambda \text{ for almost all }v\}$ (resp. $SO_V(\A)$) of $(V,Q)$
carries the restricted direct product topology coming from the
topological groups $SO_V(F_v)$, i.e., one has a basis of open sets of
the form $\prod_vU_v$, where each $U_v$ is open in $O(F_v)$ and
almost all $U_v$ are equal to $O_{V}(F_v;\Lambda)=\{\phi \in
O(F_v)\mid \phi(\Lambda)=\Lambda\}$. With respect to this topology   $O_V(\A)$
is known to be a locally compact group and as such carries an up to
scalar multiplication unique left invariant measure or {\em Haar
measure} \cite{bourb_int}. Moreover, if $(V,Q)$ is non degenerate
(which we assume in the sequel), the group is
semisimple and the Haar measure is biinvariant. We normalize the Haar
measure on the adele group $\A=\A_F$ by requiring that $\A_F/F$ has
measure $1$. Using some algebraic
geometry a Haar measure on $G(\A)$  can be defined (for any linear algebraic group $G$ defined
over $F$)  as a converging product  $\mu=\prod_v\mu_v$ over all places
of $F$ of local Haar measures, using
a left differential form $\omega$ ( a so 
called gauge form) on $G$ defined over $F$;
a gauge form  is unique up to $F$-multiples. The local integrals and
measures are then defined by identifying  coordinate neighborhoods in
$G(F_v)$ with subsets of $F_v^{\dim G}$ and the Haar
measure on $\A_F^{\dim G}$ induced from the Haar measure on $\A_F$
fixed above. Changing $\omega$ by a factor $c$
changes the local measures $\mu_v$ by the factor $\vert c \vert_v$, so that the product
formula of algebraic number theory implies that the Haar measure on
the adelic orthogonal group attached to a gauge form $\omega$ defined
over $F$ is independent of the choice of the gauge
form $\omega$.   This unique Haar measure $\mu$ on $G$ attached to a gauge
form is then called the 
{\em Tamagawa measure} $\mu$, and it makes sense to define the {\em Tamagawa
number} of $G$ as $\tau(G)=\mu(G(F)\backslash G(\A))$. It is known that
$\tau(G)=1$ for connected simple simply connected groups and that
$\tau(SO_V)=2$ for non degenerate $(V,Q)$ independent of the (non
degenerate) quadratic space $(V,Q)$. We will take this for granted in
the sequel, details can be found in
\cite{weil_aag,kottwitz_tamagawa}. 
It will be convenient to write 
 \begin{equation*}
  O_V(\A_f)=\{ \phi=(\phi_v)_{v\in (\Sigma_F\setminus\infty}\mid \phi_v \in O_V(F_v),
\phi_v(\Lambda)=\Lambda \text{ for almost all }v\} 
\end{equation*} for the subgroup of {\em finite adeles} in the adelic
orthogonal group and  $O_V(F_\infty)=\prod_{v \in \infty}O(F_v)$ for
the archimedean component of $O_V(\A)$, we have then
$O_V(\A)=O_V(F_\infty) \times O_V(\A_f)$. We may also view both
$O_V(F_\infty)$ and $O_V(\A_f)$ as subgroups of $O_V(\A)$ by setting
the respective other components equal to $1$. Analogous notations are
used for the special orthogonal group.

\medskip
We have to explain in more detail what we mean by $\mu(G(F)\backslash
G(\A))$.
\begin{definition}
Let $G$ be a group acting transitively on the topological space $S$ which carries a
$G$-invariant Borel measure. A Borel set $\mcF\subseteq S$ is called a
{\em fundamental  domain}  for the action of $G$ if $S$ is the disjoint
union of its translates $g\mcF$ for $g \in G$. The measure of $\mcF$
is then also called the measure of the homogeneous space $G\backslash S$.
\end{definition}
\begin{remark}
  \begin{enumerate}
  \item    Frequently one requires  a fundamental domain to be connected and to
  satisfy some smoothness condition on the boundary. This will play no
  role in our case, so we omit such conditions.
\item It is known that two fundamental domains for the action of $G$
  on $S$ have the same measure, so the measure of $G\backslash S$ is a
  well defined quantity.
  \end{enumerate}
\end{remark}
\begin{definition}
Let $(V,Q)$ be a quadratic space over $F$  and let $\Lambda$ be an
$R$-lattice on $V$, let $N\in \N$. 
  
Then 
\begin{align*}
  H&:=O_{V}(F;\Lambda,N\Lambda)\\
&:=\{\phi \in O_V(F)\mid \phi(x)\equiv x \bmod
  N\Lambda \text{ for all } x\in \Lambda \}\subseteq O_V(F;\Lambda) 
\end{align*}
is called a {\em principal congruence subgroup} in $O_V(F)$.

Principal congruence subgroups in $O_V(F_v), O_V(\A), O_V(\A_f)$ (where $v$ is a
non archimedean place of $F$) and in the
corresponding special orthogonal groups are defined analogously.

A subgroup $H$ of $O_V(F)$ (resp.  $O_V(F_v), O_V(\A),O_V(\A_f)$ or the respective
special orthogonal groups) is called a
congruence subgroup if it contains a principal congruence subgroup
with finite index.
\end{definition}
\begin{remark}
  \begin{enumerate}
\item Since for $\phi \in O_V(F)$ one has $O_{V}(F;\phi\Lambda,N\phi\Lambda)=\phi
  O_{V}(F;\Lambda,N\Lambda)\phi^{-1}$, all conjugates of congruence subgroups
  are congruence subgroups (globally, locally, adelically).
  \item  In the non archimedean local and in the adelic case (restricted to
 the finite adeles $\A_f$) the congruence subgroups are
 precisely the compact open subgroups of $O_V(F_v), O_V(\A_f)$, they
 form a basis of open neighborhoods of the identity element. 

\item A congruence subgroup $H$ of $O_V(\A)$ is of the shape $H=O_V(F_\infty)
\times H_f$, where $H_f$ is a congruence subgroup of $O_V(\A_f)$.
  \end{enumerate}
\end{remark}
\begin{lemma} Let $(V,Q)$ be a non degenerate quadratic space over the
  number field $F$ with ring of integers $R$, let $\Lambda$ be a
  lattice on $V$. Let $V=U\perp W$ be an orthogonal splitting of $V$
  into non degenerate subspaces of $(V,Q)$.

Then $O_{U}(F;\Lambda)$ is a congruence subgroup of $O_U(F)$. The
analogous assertion is true for the non archimedean local and adelic (special)
orthogonal groups.  
\end{lemma}
\begin{proof}
 Let $K_1=U\cap \Lambda$ and
  denote by $K_2\supseteq K_1$ the orthogonal projection $\pr_U(\Lambda)$ of $\Lambda$
  to $U$, let $N \in \N$ satisfy $NK_2\subseteq N_1$. For $\phi \in
  O_{U}(F;K_2,NK_2)$ we have then $\phi(z)-z \in \Lambda$ for all $z\in
  K_2$.
We consider $x=y+z \in \Lambda$ with $y\in W, z \in K_2$ and extend
$\phi$ to all of $V$ by letting it act trivially on $W$, obtaining
$\phi(x)=x+(\phi(z)-z)\in \Lambda$. This gives $O_{U}(F;K_2,NK_2)\subseteq
O_{U}(F;\Lambda)\subseteq O_{U}(F;K_2)$, so  $O_{U}(F;\Lambda)$  is a
congruence group. The proof carries over to the non archimedean local and to the
adelic situation without change.
\end{proof}
\begin{lemma}\label{commensurability}
  Let $\Lambda_1,\Lambda_2$ be lattices on $V$ and $N_1,N_2 \in
  \N$. Then the principal congruence subgroups
  $O_{V}(F;\Lambda_1,N_1\Lambda_1),O_{V}(F;\Lambda_2,N_2\Lambda_2)$ are commensurable, i.e., their
  intersection has finite index in both of them. 

The same assertion holds for non archimedean local and adelic congruence subgroups of the
(special) orthogonal group.
\end{lemma}
\begin{proof}
It is enough to prove the assertion for $N_1=N_2=1$ since
$O_{V}(F;\Lambda,N\Lambda) $ has finite index in $O_{V}(F;\Lambda)$ for any $N$.

One can find nonzero $a,b \in \Z$ with $ab\Lambda_1\subseteq
a\Lambda_2\subseteq \Lambda_1$. Then $O_{V}(F;\Lambda_1,ab\Lambda_1)$ is contained
in 
$O_{V}(F;\Lambda_2)\cap
O_{V}(F;\Lambda_1)$ and has finite index 
in $O_{V}(F;\Lambda_1)$. In the same way one sees that the intersection
has finite index in $O_{V}(F;\Lambda_2)$. 

It is clear that the proof works for the non archimedean local and the adelic case as well.
\end{proof}
\begin{remark}\label{commensurability_remark}
  For the applications towards Siegel's theorem the Haar measure on
  local and adelic (special) orthogonal groups is actually needed only
  for measurable sets which are finite unions of translates of
  congruence subgroups (often called {\em congruence sets}).  Since
  all congruence subgroups have finite index over  intersections
  of finitely many of them, any congruence set is a finite  union of cosets of a single
  congruence subgroup $H$, i.e., it is of the shape $\cup_{i=1}^r
  \phi_i H$ and has measure $r \mu(H)$. In \cite{kneserbook} this is
  used to define an invariant measure (or more correctly: an invariant
  measuring function) on congruence sets in the non archimedean local
  situation by setting
  $\mu(H_0)=1$ for some fixed congruence subgroup $H_0$, defining 
  \begin{equation*}
    \mu(H):=\frac{(H:H_0\cap H)}{(H_0:H_0\cap H)}
  \end{equation*}
for congruence subgroups $H$ and $\mu(\cup_{i=1}^r
  \phi_i H)=r\mu(H)$ for an arbitrary congruence set; the latter
  number is easily seen to be independent of the representation of the
  congruence set as a finite union of cosets.  It is also clear that
  on congruence sets  it agrees with the Haar 
  measure normalized to $\mu(H_0)=1$ and that any such left invariant
  measuring function must be a scalar multiple of this $\mu$.

Moreover, it is not difficult to
  prove that it is also right invariant without using this property
  for the Haar measure:
For $\phi \in  O_V(F)$ the function $\mu'$ on congruence sets given by
$\mu'(X):=\mu(X\phi)$  is a left invariant measuring function as well,
hence $\mu'=c(\phi)\mu$ with $c(\phi\circ \psi)=c(\phi)c(\psi)$. Since
the orthogonal group is generated by reflections $\tau$ which satisfy
$\tau^2=\id_V$, one must have $c(\phi)$=1 for all $\phi$. The problem is then to choose
  the normalizations of the local measures in such a way that their
  product over all places (including the archimedean ones, where the
  measure has to be defined by analytic means) is the (uniquely
  normalized) Tamagawa measure. 

It should be noted that in this way one can also define a biinvariant
measuring function on the set of all finite unions of cosets of
congruence subgroups of the global orthogonal group $O_V(F)$ for a
number field $F$. For the congruence subgroups this is then up to a
constant multiple the inverse of the measure of a fundamental domain
for the action of the congruence subgroup on $O_V(F_\infty)$ by left translations. 
\end{remark}
\begin{theorem}
Let $(V,Q)$ be a non degenerate quadratic space over the number field
$F$ with ring of integers $R$, let $H$ be  a congruence subgroup of
$O_V(\A)$, assume that a fundamental domain $\mcF_\infty$ for the action of
$O_V(F)\cap H$ by left translations on  $O_V(F_\infty)$ is given,
write $H_f=H\cap O_V(\A_f)$ for the finite part of $H$. 

Then $\mcF_\infty \times H_f$ is a fundamental domain both for the
action of $O_V(F)$ on $O_V(F)H$ and for the action of $O_V(F)\cap H$ on $H$.
\end{theorem}
\begin{proof}
 For $\sigma_1,\sigma_2 \in O_V(F)$ one has $\sigma_1(\mcF_\infty \times
 H_f)\cap   \sigma_2(\mcF_\infty \times  H_f)=\emptyset$ since
 otherwise one had $\sigma_2^{-1}\sigma_1 \in O_V(F) \cap H $ 
with $\sigma_2^{-1}\sigma_1 \mcF_\infty \cap \mcF_\infty \ne
\emptyset$, in contradiction to the assumption that $\mcF_\infty$ is a
fundamental domain for the action of $O_V(F)\cap H$ on
$O_V(F_\infty)$.

Since one has 
\begin{eqnarray*}
  H=O_V(F_\infty) \times H_f&=&\cupdot_{\sigma \in H\cap
                                O_V(F)}\sigma(\mcF_\infty \times
                                H_f)\\
O_V(F) H&=&\cup_{\sigma \in O_V(F)}\sigma(\mcF_\infty \times H_f),
\end{eqnarray*}
the assertion follows.
\end{proof}
\begin{remark}
  If $F$ is totally real and $(V,Q)$ totally definite, the group
  $O_V(F_\infty)$ is compact and $O_V(F)\cap H$ is finite, so that it
  is not difficult to find a fundamental domain $\mcF_\infty$ of
  finite measure as  assumed above. One has then
  $\mu(O_V(F_\infty))=\vert H\cap O_V(F)\vert \cdot \mu(\mcF_\infty)$.

Siegel has shown \cite{siegel_einheiten} that in the indefinite case  a fundamental domain of
finite volume exists too, a more general version of this statement for
arithmetic groups is due to Borel and Harish-Chandra, see \cite{borel_harish}.
\end{remark}
\begin{corollary}
 In the situation of the theorem let $\phi=(\phi_\infty,\phi_f)\in O_V(\A)$ and assume that
 $\mcF_{\infty,\phi}$ is a fundamental domain for the action of
 $O_V(F)\cap \phi H\phi^{-1}$ on $O_V(F_\infty)$ by left translation. 

Then
 $\mcF_{\infty,\phi}\times \phi_f H_f$ is a fundamental domain for the action of $\phi H
 \phi^{-1} \cap O_V(F)$  on $O_V(F_\infty)\times \phi_f H_f$ and for the
 action of $O_V(F)$ on the double coset $O_V(F)\phi H\subseteq O_V(\A)$. 
\end{corollary}
\begin{proof}
  Obvious.
\end{proof}
\begin{theorem}[Siegel's main theorem on quadratic forms, first version]
Let $(V,Q)$ be a non degenerate quadratic space over the number field
$F$ with ring of integers $R$, let $H$ be  a congruence subgroup of
$SO_V(\A)$ and let 
\begin{equation*}
  SO_V(\A)=\cupdot_{i=1}^r SO_v(F) \phi_i H
\end{equation*}
be a finite disjoint coset decomposition.
Assume that  fundamental domains $\mcF_{\infty,\phi_i}$ for the action of
$SO_V(F)\cap \phi_i H\phi_i^{-1}$ by left translations on
$SO_V(F_\infty)$ are given for $1\le i \le r$ and let $\mu$ be the
Tamagawa measure on $SO_V(\A)$, factored as $\mu=\mu_\infty \cdot
\mu_f$ with Haar measures $\mu_\infty$ on $SO_V(F_\infty)$ and
$\mu_f=\prod_{v\not\in \infty}\mu_v$
on $SO_V(\A_f)$.

Then one has 
\begin{equation*}
  \sum_{i=1}^r \mu_\infty(\mcF_{\infty,i}) =\frac{2}{\mu_f(H_f)}.
\end{equation*}
 In particular, if $F$ is totally real and $(V,Q)$ totally definite
 and if in addition  $H_f=\prod_{v \not\in \infty}H_v$ one has
 \begin{equation*}
   \sum_{i=1}^r \frac{1}{\vert SO_V(F)\cap \phi_i H\phi_i^{-1}
     \vert}=\frac{2}{\mu_\infty(SO_V(F_\infty))\prod_{v\not\in \infty}\mu_v(H_v)}.
 \end{equation*}
\end{theorem}

\section{Computation of the local measures, non-archimedean
  places}\label{nonarchimedean_haarmeasure}
In this section $F$ is a non archimedean local field with ring of
integers $R$, prime element $\pi$, maximal ideal $P=(\pi)$ and residue
field $R/P=\F_q$.

We want to express the local measures in terms of numbers of solutions
of certain congruences. An important tool will be a quantitative refinement of
Theorem \ref{hensellemma_kneser} (Hensel's Lemma):
\begin{lemma}\label{hensel_quantitative}
  Let $(V,Q)$ and $(W,Q')$ be  finite dimensional quadratic spaces
 over $F$ with associated symmetric bilinear forms $b,b'$ and let
 $L,M$ be lattices of ranks $\ell,m$ in  $V,W$ respectively. 
Let $k \in \N$ be such that $P^kQ'(M)\subseteq P$.  

Let $f:L\to W$ be an $R$-linear map with $b(f(L),M)\subseteq R$ and write $\wtbp_f$ for the linear
map from $W$ to $\Hom_R(L,F)$ with $\wtbp_f(z)(x)=b'(f(x),z)$ for all
$x \in L$.

Assume that one has 
\begin{enumerate}
\item $L^*=\wtbp_f(M)+PL^*$
\item $Q'(f(x))\equiv Q(x)\bmod P^k$
for all $x \in L$
\end{enumerate}

Then the number of modulo $P^{k+1}M$ distinct  $R$- linear maps $f':L\to W$ with $f'(x) \equiv f(x) \bmod
P^kM$ and $Q'(f'(x))\equiv Q(x) \bmod P^{k+1}$ for 
all $x \in L$ is equal to $q^{\ell m-\frac{\ell (\ell+1)}{2}}$. 

All of these 
also satisfy the condition 
$L^*=\wtbp_{f'}(M)+PL^*$
\end{lemma}
\begin{proof}
As in the proof of Theorem \ref{hensellemma_kneser} we write $f'=f+g$
where $g$ maps $L$ into $P^kM$  and satisfies
\begin{equation*}
  \wtbp_f(g(y))(x)\equiv -\beta(x,y)
\bmod P^{k+1}
\end{equation*}
 for all $x,y \in L$.

Now, for any $\F_q$-linear map $\bar{h}:L/PL \to M/PM$ we can define an
$\F_q$-valued quadratic form $\bar{Q}_{\bar{h}}$ on $L/PL$ by
\begin{equation*}
  \bar{Q}_{\bar{h}}(\bar{x}):=-\bar{b}((\overline{f(x)},\bar{h}(x)),
\end{equation*}
and the modulo $P^{k+1}$ distinct maps $g$ as above can be obtained as
$g=\pi^kh$, where $h:L\to M$ is considered modulo $P$. 
As in the proof of Theorem \ref{hensellemma_kneser} we see that the
linear 
map $\bar{h} \mapsto \bar{Q}_{\bar{h}}$ from $\Hom_{\F_q}(L/PL,M/PM)$ to the space
$X$ of $\F_q$-valued quadratic forms on $L/PL$ is surjective. Since $\Hom_{\F_q}(L/PL,M/PM)$
has dimension $\ell m$ and the space $X$ has dimension
$\frac{\ell(\ell +1)}{2}$, the assertion follows.
\end{proof}
\begin{corollary}
  Let $\Lambda$ be a lattice on the non degenerate quadratic space
  $(V,Q)$ over $F$ of dimension $n$, let $k\in \N$ be such that
  $P^kQ(\Lambda^\#)\subseteq P$ holds.

Then with $O_V(F;\Lambda,P^k\Lambda^\#)=\{\phi \in O_V(F;\Lambda) \mid
\phi(x)-x \in P^k\Lambda^\# \text{ for all } x \in \Lambda\}$ and
$A(\Lambda, \Lambda\bmod P^k\Lambda^\#):=\vert\{\phi\in \Hom_R(\Lambda,
  \Lambda/P^k\Lambda^\#) \mid Q(\phi(x))\equiv Q(x)\bmod P^k \text{ for
  all } x \in \Lambda\}\vert$
one has
\begin{equation*}
  (O_V(F;\Lambda,P^k
  \Lambda^\#):O_V(F;\Lambda,P^{k+1}\Lambda^\#))=q^{n\frac{n-1}{2}}
\end{equation*}
and 
\begin{equation*}
(O_V(F;\Lambda):O_V(F;\Lambda, P^k\Lambda^\#))=A(\Lambda,\Lambda \bmod P^k\Lambda^\#).
\end{equation*}
If $\mu$ is a Haar measure
  on $O_V(F)$ one has 
\begin{equation*}
  \frac{\mu(O_V(F;\Lambda,P^k\Lambda^\#))}{\mu(O_V(F;\Lambda,P^{k+1}\Lambda^\#))}=q^{\frac{n(n-1)}{2}}.
\end{equation*}
\end{corollary}
\begin{proof}
This follows directly from Theorem \ref{hensellemma_kneser} and the lemma.  
\end{proof}
\begin{proposition}
Let $(V,Q)$ be a non degenerate quadratic space over $F$ of dimension
$n$, let $\Lambda$ be an $R$- lattice on $V$ with $Q(\Lambda)\subseteq
R$. Identify $\End_F(V)$
with $M_n(F)$ and $\End_R(\Lambda)$ with $M_n(R)$ with respect to some fixed
$R$-basis of $\Lambda$, equip $F$ with the (additive) Haar measure $\mu_F$
under which $R$ has measure $1$. 
 
\medskip
Let $\omega_\mcX=\bigwedge_{i,j}dx_{ij}$ denote the gauge form on
$\mcX:=\End_F(V)$ which induces (with the given choice of $\mu_F$ and
the given coordinate mappings) the  
Haar measure $\mu_\mcX$ on $\mcX$ with $\mu_\mcX(\End_R(\Lambda))=1$.

Let $\mu_{\mcS}$
denote the Haar measure on $\mcS:=M_n^{\rm sym}(F)$ with
$\mu(M_n^{\rm sym}(R))=1$, denote by $\omega_\mcS$ the gauge form on $\mcS$
inducing $\mu_\mcS$ and write $\mcS^*=\{S\in \mcS\mid S={}^tXS_0X
\text{ for an } X \in GL_n(F)\}$.   

Denote by $S_0\in \mcS$ the Gram matrix of
$\Lambda$ with respect to the given basis of $\Lambda$ and put
$\mcX_S^{(S_0)}:=\{X\in \mcX \mid {}^tXS_0X=S\}$ for $S \in \mcS$.

Then there exists a differential form $\omega$ of degree
$\frac{n(n-1)}{2}$ on $\mcX$ such that for
the measure $\mu_S^{(S_0)}$ associated to the restriction of $\omega$ to
$\mcX_S^{(S_0)}$ (for $S\in \mcS^*$) one has 
\begin{equation*}
\int_{\mcX}f(X)d\mu_\mcX=\int_{\mcS^*}\bigl(\int_{\mcX_S^{(S_0)}}f(X) d\mu_S^{(S_0)}\bigr)d\mu_{\mcS}
\end{equation*}
for any integrable function $f$ on $\mcX$.

The measure $\mu_{S_0}:=\mu_{S_0}^{(S_0)}$ is then a Haar measure on
$\mcX_{S_0}^{(S_0)}=O_V(F)$, induced by the gauge form
$\omega\vert_{\mcX_{S_0}^{(S_0)}}$ on $O_V(F)$.
\end{proposition}
\begin{proof}
 This follows from the theory of differential forms on algebraic
 varieties and their associated measures. An explicit construction is
 given in \cite{boege}.
\end{proof}
\begin{remark}
  If we replace $Q$ by $cQ$ and hence $S_0$ by $cS_0$, we change
  $\mu_{S_0}$ to $\mu_{cS_0}=\vert c
  \vert_v^{-\frac{n(n+1)}{2}}\mu_{S_0}$. Similarly, changing the Gram matrix
  $S_0$ to $S_0'={}^tUS_0 U$ with $U \in GL_n(F)$ changes the measure
  by a factor of $\vert \det(U)\vert_v^{-n-1}$ (the map sending $S \in
  M_n^{\sym}(F)$ to ${}^t U S U$ is easily seen to have determinant $(\det(U))^{n+1}$). 
\end{remark}
\begin{proposition}
  With notations as above one has for $k\in \N$ satisfying
  $P^kQ(\Lambda^\#)\subseteq P, P^k\Lambda^\#\subseteq \Lambda$:
  \begin{eqnarray*}
    \mu_{S_0}(O_V(F;\Lambda))&=&\vert 2 \vert_v^{-n}q^{-kn\frac{n-1}{2}}A(\Lambda,\Lambda
    \bmod P^k\Lambda)\\
&=&\vert 2 \vert_v^{-n}q^{-kn\frac{n-1}{2}}\vert{\det}_b(\Lambda)\vert_v^{-n}A(\Lambda,\Lambda
    \bmod P^k\Lambda^\#).
  \end{eqnarray*}
\end{proposition}
\begin{proof}
  We prove the first equality, the second one is then a trivial
  consequence.

We take for $f$ the characteristic function of $\{\phi
\in \End_R(\Lambda)\mid Q(\phi(x))\equiv Q(x) \bmod P^k \text{ for
  all } x \in \Lambda\} $ and notice that this property depends only
on $\phi$ modulo $P^k\Lambda$ (in fact, only on $\phi$ modulo
$P^k\Lambda^\#$). If $X$ is the matrix associated to $\phi$ 
this is equivalent to ${}^tXS_0X \equiv S_0 \bmod P^k$, where the
 diagonal entries, which are even,  are congruent modulo $2P^k$; we write ${}^tXS_0X
 \equiv_{\rm ev} S_0$ for this property in the sequel. 

We have then (omitting superscripts $(S_0)$ in the sequel) for $S\in \mcS^*$
\begin{equation*}
  \int_{\mcX_S}f(X) d\mu_{\mcX_S}=
  \begin{cases}
    0 & S\not\equiv_{\rm ev} S_0 \bmod P^k\\
\mu_{S_0}(O_V(F;\Lambda)) & S\equiv_{\rm ev} S_0 \bmod P^k
  \end{cases}.
\end{equation*}
 From the previous proposition we get then 
\begin{equation*}
  \int_{\mcX}f(X)d\mu_{\mcX}=\mu_{S_0}(O_V(F;\Lambda)) 
\int_{S \equiv_{\rm ev}  S_0\bmod P^k} d\mu_{\mcS}.
\end{equation*}
The integral on the right hand side without the extra congruence
condition on the diagonal elements equals $q^{-kn\frac{n+1}{2}}$, the
extra congruence condition gives an additional factor of $\vert 2
\vert_v^{n}$, so that we arrive at 
\begin{equation*}
  \int_{\mcX}f(X)d\mu_{\mcX}=\mu_{S_0}(O_V(F;\Lambda)) \vert 2 \vert_v^{n}q^{-kn\frac{n+1}{2}}. 
\end{equation*}
We compare this with 
\begin{eqnarray*}
  \int_{\mcX}f(X)d\mu_{\mcX}&=&\mu_{\mcX}(\{X\in M_n(R)\mid {}^t X
  S_0X\equiv_{\rm ev} S_0 \bmod P^k\}\\
&=&q^{-kn^2}A(\Lambda, \Lambda \bmod P^k\Lambda),
\end{eqnarray*}
using that the validity of the congruence here depends only on $X$
modulo $P^kM_n(R)$, 
and obtain the assertion.
\end{proof}
\begin{remark}
  \begin{enumerate}
  \item 
The formula above is valid for any $k$ satisfying the
  assumptions. The independence of the right hand side of $k$ (as
  long as $k$ satisfies the assumptions) can also be obtained as a consequence of our
  quantitative form of Hensel's lemma, 
\item The proof shows that the factor $\vert 2 \vert_v^{-n}$ does not
  occur if one replaces the condition $Q(\phi(x))\equiv Q(x) \bmod
  P^k$ in the definition of $A(\Lambda, \Lambda \bmod P^k\Lambda^\#)$
  by the condition ${}^tXS_0X\equiv S_0\bmod P^k$, where $X$ is the
  matrix of $\phi$ with respect to the given basis. This power of $2$
  occurring for dyadic places $v$ therefore does not occur in
  treatments working with the symmetric bilinear form only instead of the
  quadratic form.
  \end{enumerate}
\end{remark}
\begin{corollary}\label{local_measure_computed}
  With notations as above let $\Lambda_0$ be a fixed lattice on $V$
  with Gram matrix $S_0$, write $\mu=\mu_{S_0}$ for the Haar measure
  on $O_v(F)$ as above. Then we have for any lattice $\Lambda$ on $V$ with
  $\Lambda \subseteq \Lambda^\#$ and $k\in \N$ with
  $P^kQ(\Lambda^\#)\subseteq P, P^k\Lambda^\#\subseteq \Lambda$
\begin{equation*}
  \mu(O_V(F;\Lambda))=\vert 2 \vert_v^{-n} (\Lambda:P^k\Lambda^\#)^{\frac{1-n}{2}}
  A(\Lambda, \Lambda \bmod P^k\Lambda^\#)    \vert \det S_0
  \vert_v^{-\frac{n+1}{2}} 
\end{equation*}
and
  \begin{equation*}
 \mu(O_V(F, \Lambda,
 P^k\Lambda^\#))(\Lambda:P^k\Lambda^\#)^{\frac{n-1}{2}}= \vert 2 \vert_v^{-n}  \vert \det S_0 \vert_v^{-\frac{n+1}{2}}
  \end{equation*}
independent of the choice of $\Lambda$. 
\end{corollary}
\begin{proof}
 We fix a basis of $\Lambda$ and write the Gram matrix of $\Lambda$ as
 ${}^tUS_0S$ with $U \in GL_n(F)$.
Using $\mu=\vert \det(U) \vert_v^{n+1} \mu_S$ and
$(\Lambda:P^k\Lambda^\#)=q^{kn}\vert {\det}_b(\Lambda)\vert_v$ we can
then rewrite the formula from the proposition as 
\begin{eqnarray*}
 \mu(O_V(F;\Lambda))&=&\vert 2 \vert_v^{-n}\vert
                        \det(U)\vert_v^{n+1}(\Lambda:P^k\Lambda^\#)^{\frac{1-n}{2}}\vert
                        {\det}_b(\Lambda)\vert_v^{-\frac{n+1}{2}}A(\Lambda,\Lambda
                        \bmod P^k\Lambda^\#)\\
&=&\vert 2\vert_v^{-n} (\Lambda:P^k\Lambda^\#)^{\frac{1-n}{2}}\vert
    \det(S_0) \vert_v^{-\frac{n+1}{2}} A(\Lambda,\Lambda \bmod
    P^k\Lambda^\#). 
\end{eqnarray*}
The second formula follows from this upon inserting
\begin{equation*}
(O_V(F;\Lambda):O_V(F;\Lambda, P^k\Lambda^\#))=A(\Lambda,\Lambda \bmod P^k\Lambda^\#).
\end{equation*}



\end{proof}

\section{Computation of local measures, definite archimedean places}
We postpone the discussion of the local measure on the orthogonal
group $O_{(V,Q)}(F)$ in the case that that $F=\C$ or that $F=\R$ and $Q$
is indefinite and assume in this section that $F=F_v=\R$ and that
$(V,Q)$ is a positive definite quadratic space over $F$. Since $Q$ and
$-Q$ lead to the same orthogonal group this covers the negative
definite case as well.

We choose as in the previous section a basis of $V$ and denote by
$S_0$ the Gram matrix of $(V,Q)$ with respect to this basis. As in the
non archimedean case we construct a Haar measure $\mu=\mu_{S_0}$ on
$O_V(F)$. For this, we first take on $F=\R$ the Lebesgue measure. We identify $\mcS=M_n^{\sym}(F)$ with
$\R^{\frac{n(n+1)}{2}}$ and obtain  as measure $\mu_{\mcS}$ on $\mcS$
induced by the differential form $\bigwedge_{i\le j} ds_{ij}$
the Lebesgue measure on $\R^{\frac{n(n+1)}{2}}$.  From this we obtain
as in the non-archimedean case  measures $\mu_S^{(S_0)}$ on $\mcX_S^{(S_0)}$
for each $S \in \mcS$. Identifying $O_V(F)$ with $\mcX_{S_0}^{(S_0)}$ we set
$\mu =\mu_{S_0}=\mu_{S_0}^{(S_0)}$, this Haar measure on $O_V(F)$ is induced (with
notations as in the non-archimedean case) by the gauge form
$\omega\vert_{\mcX_{S_0}^{(S_0)}}$. Again 
  as there we see that  $S'_0={}^t U S_0U$ for $U\in GL_n(F)$ leads to
  $\mu_{S'_0}=\vert \det U\vert_v^{-n-1}\mu_{S_0}$. In particular, if we write
  $S_0={}^t U U$ with $U \in GL_n(F)$  we obtain $\mu_{S_0}=\vert
  \det(S_0)\vert_v^{-\frac{n+1}{2}}\mu_{1_n}$, where $1_n$ denotes the
  identity matrix.

Fixing $S_0=1_n$ we omit superscripts $(S_0)$ in the sequel. To proceed
similarly as in the non-archimedean case we choose now
$f(X)=\exp(-\tr({}^tX X)$ as our test function $f$
on $\End(V)=M_n(\R)$ and have
$\int_{\mcX}f(X)d\mu_{\mcX}=\pi^{\frac{n^2}{2}}$, using
$\int_{-\infty}^{\infty}\exp(-x^2)dx=\sqrt{\pi}$.  

In order to make the measure $\mu_\mcS$ on $\mcS^*$ explicit we
parametrize $\mcS$ by the set of lower triangular matrices
$X=(x_{ij})_{i,j}$, mapping $X$ to $S={}^tX X$. We have then by a
straightforward calculation 
\begin{equation*}
\omega_\mcS:=  \bigwedge_{i\le j}ds_{ij}=2^n\prod_{i=1}^n x_{ii}^i \bigwedge_{i\ge j} dx_{ij},
\end{equation*}
and $\mu_{\mcS}$ is the measure induced by this differential form. The
equation $\omega_\mcX=\omega\wedge \omega_\mcS$ gives then 
\begin{equation*}
  \omega=2^{-n}\prod_{i=1}^nx_{ii}^{-i}\bigwedge_{i<j}dx_{ij}
\end{equation*}
up to summands containing one of the $dx_{ij}$ with $i\ge j$.
For the computation of
\begin{equation*}
\int_{\mcX}f(X)d\mu_\mcX=\int_{\mcS^*}\bigl(\int_{\mcX_S^{(S_0)}}f(X) d\mu_S^{(S_0)}\bigr)d\mu_{\mcS}
\end{equation*}
we notice that for $S={}^tXX\in \mcS^*$ we have $\mcX_S=O_V(F)X$. The
integrand in the inner integral above is constant on $\mcX_S$, and the
volume of $O_V(F)X$ with respect to the restriction of $\omega$ to
$\mcX_S$ is $\det(X)^{-1}$ times the volume of $O_V(F)$. This can be
checked by looking at $X=\lambda 1_n$ for some $\lambda \ne 0$ since
the quotient of these volumes is some character of $GL_n(F)$ and hence
a function of the determinant only.
We arrive at 
\begin{eqnarray*}
  \pi^{\frac{n^2}{2}}&=&\int_{M_n(F)} f(X)\\
&=&
\int_{\mcS^*}\bigl(\int_{\mcX_S^{(S_0)}}f(X)
d\mu_S^{(S_0)}\bigr)d\mu_{\mcS}\\
&=&\mu_{1_n}(O_V(F))\int\prod_{i=1}^n2^n x_{ii}^{i-1} 
\exp(-\sum_{i\le j}x_{ij}^2)\bigwedge_{i\ge j}dx_{ij}. 
\end{eqnarray*}
We use again $\int_{-\infty}^{\infty}\exp(-x^2)dx=\sqrt{\pi}$ for the
integration with respect to the variables $x_{ij}$ with $i>j$, for the
integration with respect to the $x_{ii}$ we use $\int_{-\infty}^\infty
2x^i\exp(-x^2)dx=\Gamma(\frac{i+1}{2})$
and obtain 
\begin{equation*}
\mu_{1_n}(O_V(F))=\pi^{\frac{n(n+1)}{4}}\prod_{i=1}^n(\Gamma(\frac{i}{2}))^{-1}
\end{equation*}
and hence 
\begin{equation*}
\mu_{1_n}(SO_V(F))=\frac{1}{2}\pi^{\frac{n(n+1)}{4}}\prod_{i=1}^n(\Gamma(\frac{i}{2}))^{-1}.
\end{equation*}
\begin{remark}
  \begin{enumerate}
\item The differential forms considered in this section are the same
  as in the non-archimedean situation of the previous section.
  \item  By the well known formula ${\rm
   vol}(S^{n-1})=2\pi^{\frac{n}{2}}(\Gamma(\frac{n}{2}))^{-1}$ for the volume of
 the $(n-1)$-dimensional unit sphere we can express the measure of
 $SO_V(F)$ also as
 \begin{equation*}
 \mu_{1_n}(SO_V(F))=2^{-n}\prod_{i=2}^n{\rm vol}(S^{i-1}).
\end{equation*}
 \end{enumerate}
\end{remark}
\section{The global Tamagawa measure}

We now  put together the local computations. For this, let $F$ be a
totally real number field of degree $r$ over $\Q$ with discriminant
$D_F$ and different $\mcD_F$, let 
$\mu_{F_\A}=\prod_{v\in \Sigma_F}\mu_{F_v}$ be the invariant measure on
$\A_F$ with $\mu_{F_\A}(\A_F/F)=1$. We choose the usual normalization of
the $\mu_{F_v}$ with $\mu_{F_v}(R_v)=\vert \mcD_F\vert_v^{\frac{1}{2}}$ for all non
archimedean $v$ and $\mu_{F_v}$ equal to the Lebesgue measure for the real
places of $F$. In particular, for non archimedean $v$ the measure on
$F_v$ used
now is $\vert \mcD_F\vert_v^{\frac{1}{2}}$ times the measure used in
Section \ref{nonarchimedean_haarmeasure}, and since $SO_V(F_v)$ has
dimension $n(n-1)/2$, the measure on $SO_V(F_V)$ is changed by a
factor $\vert \mcD_F\vert_v^{\frac{n(n-1)}{4}}$.

By the product formula
the product of these factors over all non archimedean places of $F$ is
then  $\vert D_F\vert^{-\frac{n(n-1)}{4}}$, where  $\vert D_F\vert$
is the ordinary absolute value of the field discriminant.

Let $(V,Q)$ be a
totally definite quadratic space over $F$, let $\Lambda_0$ be a fixed
free lattice on $V$ with Gram matrix $S_0$ with respect to a fixed
basis. Using our present choice of the measures $\mu_{F_v}$ on the $F_v$ we define
local measures $\mu_v=(\mu_{S_0})_v$ on $SO_V(F_v)$ and obtain the
Tamagawa measure $\mu_{\A}$ on $SO_V(\A)$. 

Combining our local results with the first version of the
Minkowski-Siegel theorem we obtain:
\begin{theorem}[Siegel's main theorem on quadratic forms, second version]\label{massformel}
Let $\Lambda$ be an integral lattice on $V$, denote by
$\Lambda_1,\ldots,\Lambda_h$ representatives (on $V$) of the proper (with
respect to $SO(V)$) integral equivalence classes in the genus of
$(\Lambda,Q)$.

For non archimedean $v$ and $k=k_v$ satisfying
$P_ v^kQ(\Lambda_v)\subseteq P_v, P_v^k\Lambda_v^\#\subseteq \Lambda_v$
put
\begin{equation*}
  \alpha_v(\Lambda)=\frac{1}{2}\vert 2 \vert_v^{-n}q_v^{-k\frac{n(n-1)}{2}}\vert
  {\det}_b(\Lambda_v)\vert_v^{\frac{1-n}{2}}A(\Lambda_v,\Lambda_v\bmod P_v^k\Lambda_v^\#)
\end{equation*}
where $q_v$ is the order of the residue field $R_v/P_v$.
Then one has
\begin{equation*}
  \sum_{i=1}^h\frac{1}{\vert SO_V(F;\Lambda_i)}\vert=\vert
  D_F\vert^{\frac{n(n-1)}{4}}2^{r+1}\pi^{-r\frac{n(n+1)}{4}}\prod_{i=1}^n(\Gamma(\frac{i}{2}))^{r}\prod_{v\not\in \infty}(\alpha_v(\Lambda))^{-1}.
\end{equation*}
\end{theorem}
\begin{proof}
  In our first version of Siegel's main theorem we use
  $H=SO_V(\A_F;\Lambda)$ and have $SO_V(F)\cap
  \phi_iH\phi_i^{-1}=SO(F;\Lambda_i)$ with
  $\phi_i\Lambda=\Lambda_i$. The local measures occurring there have
  been computed in 
  the two preceding sections, with the normalization of the
  Haar measure discussed above introducing  the extra factor $\vert
  D_F\vert^{\frac{n(n-1)}{4}}$. We notice that the product of the
  factors $\vert \det(S_0)\vert_v^{-\frac{n+1}{2}}$ over all places
  $v$ is equal to $1$ by the product formula and obtain the assertion.   
\end{proof}
\begin{remark}
  \begin{enumerate}
\item The $\alpha_v(\Lambda)$ are called the {\em local densities} for $\Lambda$.
  \item The factors $\vert 2\vert_v^{-n}$ occurring non trivially in the $\alpha_v$
    for the dyadic places $v$ of $F$ (i.e., the places where the
    $v$-value of $2$ is not $1$) combine in the product to a total
    factor of $2^{nr}$. It is therefore not unusual to omit these
    factors in the definition of the local densities at the non
    archimedean places and instead to insert a factor of $2^n$ for each
    of the real places.
\item   It is easily checked that $ \sum_{i=1}^h\frac{1}{\vert
    SO_V(F;\Lambda_i)}\vert=2\sum_{i=1}^{h'}\frac{1}{\vert
    O_V(F;\Lambda_i')}\vert$,  where $\Lambda_1',\ldots,
  \Lambda_{h'}'$ are a set of representatives of the isometry classes
  in the genus of $\Lambda$ (with respect to $O_V(F)$), since for any
  $\Lambda_i$ admitting no 
  automorphism of determinant $-1$ the isometry class of $\Lambda_i$
  splits into two proper isometry classes.  The latter sum is called
  the {\em measure} or {\em Maß}, often also written as ``mass'',
  $m(\Lambda)=m(\gen(\Lambda))$ of the genus of  $\Lambda$. We have
  then
\begin{equation*}
m(\gen(\Lambda))=\vert
  D_F\vert^{\frac{n(n-1)}{2}}2^{r}\pi^{-r\frac{n(n+1)}{4}}\prod_{i=1}^n(\Gamma(\frac{i}{2}))^{r}\prod_{v\not\in \infty}(\alpha_v(\Lambda))^{-1}.
\end{equation*}
Another way to check this is to consider  the subgroup
\begin{equation*}
SO_V(\A_F)\cup \{\phi=(\phi_v)_{v \in \Sigma_F}\in O_v(\A_F) \mid
\det(\phi_v)=-1 \text{ for all } v \in \Sigma_F\}
\end{equation*} 
of $O_V(\A_F)$ in which $SO_V(\A_F)$ has index $2$. A fundamental
domain for the action of $SO_V(F)$ on $SO_V(\A_F)$  is then also a
fundamental domain for the action of $O_V(F)$ on this group.
  \end{enumerate}
\end{remark}
If $H\subseteq SO_V(\A_F)$ is a congruence subgroup of
$SO_V(\A_F;\Lambda)$ for a lattice $\Lambda$ as above and we have a
double coset decomposition $SO_V(\A_F)=\cup_{j=1}^tSO_V(F)\psi_jH$, we
can use our two versions of Siegel's theorem to express the sum of the
inverses of the $\vert SO_v(F)\cap \psi_jH\psi_j^{-1}$ by the left
hand side of the theorem above and the index of $H$ in
$SO_V(\A_F;\Lambda)$. 

For example we get the following result first proven by van der Blij
in classical matrix notation:
\begin{corollary}
 Let $\Lambda$ be a lattice as above and fix $x_0\in \Lambda^\#$,
 write $SO_V(F, x_0+\Lambda)=\{\phi \in SO_V\mid
 \phi)x_0+\Lambda)=x_0+\Lambda\}$ and define $SO_V(\A_F,x_0+\Lambda)$
 analogously.
Then $SO_V(\A_F,x_0+\Lambda)$ is a congruence subgroup of
$SO_V(\A_F,\Lambda)$.

For a lattice $\Lambda' \in \gen(\Lambda)$ and $x\in (\Lambda')^\#$
say that $x+\Lambda'$ is in the proper class of $x_0+\Lambda$ if there
exists $\phi \in SO_V(F)$ with $\phi(x_0+\Lambda)=x+\Lambda'$ and that   
 that $x+\Lambda'$ is in the proper genus of $x_0+\Lambda$ if there
exists $\phi \in SO_V(\A_F)$ with $\phi_v(x_0+\Lambda)=x+\Lambda'_v$
for all $v \in \Sigma_F$. 

Let $\{x_j+\Lambda'_j\}$ be a set of
representatives of the (finitely many, say $t$) classes in the genus
of $x_0+\Lambda$ and write $s_v$ for the index of
$SO_V(F_v,x_0+\Lambda)$in $SO_V(F;\Lambda)$.

Then $s_v=1$ for almost all $v$, and with $s=\prod_vs_v$ we have
\begin{equation*}
  \sum_{j=1}^t\frac{1}{\vert SO_V(F;x_j+\Lambda'_j)}\vert= s
  \sum_{i=1}^h\frac{1}{\vert SO_V(F;\Lambda_i)\vert}.
\end{equation*}
 \end{corollary}
 \begin{proof}
   Obvious.
 \end{proof}
\begin{example}
Let $F=\Q$ and let $\Lambda$ be  an even unimodular $\Z$-lattice on the
positive definite quadratic space $(V,Q)$ over $\Q$ (i.e., we have
$Q(\Lambda)\subseteq \Z$ and ${\det}_b(\Lambda)=1$).
It is known that such lattices only occur with rank $n=2m$ divisible by
$8$ and that the $E_8$ root lattice represents the unique isometry
class of such lattices in dimension $8$.

For the computation of the local factors we can take $k=k_p=1$ for all
primes $p$, which gives 
\begin{equation*}
  \alpha_p(\Lambda)=\vert 2
  \vert_p^{-n}q^{-\frac{(n(n-1)}{2}}\vert SO_{\Lambda/p\Lambda}(\F_p)\vert,
\end{equation*}
where the factor $\vert 2
  \vert_p^{-n}$ is equal to $2^n$
  for $p=2$ and equal to $1$ for the
  odd primes $p$.

Since the quadratic space
$\Lambda/p\Lambda$ over $\F_p$ is an
orthogonal sum of $m=\frac{n}{2}$
hyperbolic planes  we have (see
Section 13 of \cite{kneserbook})
\begin{equation*}
  \vert
  SO_{\Lambda/p\Lambda}(\F_p)\vert =p^{\frac{n(n-1)}{2}}(1-p^{-m})\prod_{i=1}^{m-1}(1-p^{-2i}),
\end{equation*}
which gives 
\begin{equation*}
  \prod_p(\alpha_p(\Lambda))^{-1}=2^{-n}\zeta(m)\prod_{i=1}^{m-1}\zeta(2i).
\end{equation*}
For $n=8$ this gives
\begin{equation*}
  \prod_p(\alpha_p(\Lambda))^{-1}=\frac{\pi^{16}}{2^{11}3^{8}5^37}
\end{equation*}
 using the well
known values
$\zeta(2)=\frac{\pi^2}{6},
\zeta(4)=\frac{\pi^4}{90}, \zeta(6)=\frac{\pi^6}{945}$.
The remaining factor
$2^2\pi^{-\frac{n(n+1)}{4}}\prod_{i=1}^n\Gamma(\frac{i}{2})$
evaluates to 
\begin{equation*}
  \frac{3^3 5}{2^2 \pi^{16}},
\end{equation*}
and we obtain $\vert
SO_V(\Q;\Lambda)\vert=\vert SO(E_8)\vert = 2^{13}3^55^27$,
which is indeed in agreement with
the known value for the order of the
special orthogonal group of the
$E_8$ root lattice.

The same computation in rank $16$
can be used to prove that there are
only two isometry classes of even
unimodular lattices of rank $16$,
represented by the lattices
$E_8\perp E_8$ and $D_{16}^+$
discussed in Section
\ref{Z-lattice_section_cont}.

The measure (or Maß) of the unique genus of even unimodular lattices 
then grows rapidly with the rank and is of the order of magnitude of $
4\cdot 10^7$ already in rank $32$, which implies that the number of
isometry classes is at least $80$ million in this rank, so that an
explicit classification is no longer possible or useful. The even
unimodular lattices of rank $24$ have been classified, there are
exactly $24$ isometry classes.    
\end{example}

Having established the maßformel for the measure of a genus we now
need some preparations for the formula for representation measures
(Darstellungsmaße).

\begin{lemma}\label{firstlemma_representationmeasures} Let $(V,Q)$ be a non degenerate quadratic space over the
  non-archimedean local field $F=F_v$ with an orthogonal splitting
  $V=U\perp W$ into nonzero non degenerate subspaces $U,W$, write $n,u,w$ for
  their respective dimensions. Let $\Lambda$ be an integral
  lattice on $V$ and $M=W\cap \Lambda$ a primitive sublattice of
  $\Lambda$ on $W$, denote by $i_{M,\Lambda}:M\to\Lambda$ the inclusion map of
  $M$ into $\Lambda$. Let $k\in\N$ satisfy $P^kQ(\Lambda^\#)\subseteq P,
  P^k\Lambda^\#\subseteq \Lambda$. Then 
  \begin{equation*}
    \frac{(O_V(F;\Lambda):O_V(F;\Lambda, P^k\Lambda^\#))}
    {(O_U(F;\Lambda):O_U(F;\Lambda,P^k\Lambda^\#))}=A(i_{M,\Lambda},\Lambda
    \bmod P^k\Lambda^\#),
  \end{equation*} where $A(i_{M,\Lambda},\Lambda \bmod P^k\Lambda^\#) $ denotes  the number of $R$-linear maps 
$\phi:M\to \Lambda/P^k\Lambda^\#$ with $\phi =\psi \bmod
P^k\Lambda^\#$  for an isometry $\psi:M\to\Lambda$ in the
  class of $i_{M,\Lambda}$. 
\end{lemma}
\begin{proof}
By the isomorphism theorem of group theory we have 
\begin{equation*}
  (O_U(F;\Lambda):O_U(F;\Lambda,P^k\Lambda^\#))=(O_U(F;\Lambda)O_V(F;\Lambda,P^k\Lambda^\#):O_V(F;\Lambda,P^k\Lambda^\#)),
\end{equation*} and get
\begin{equation*}
    \frac{(O_V(F;\Lambda):O_V(F;\Lambda, P^k\Lambda^\#)}
    {(O_U(F;\Lambda):O_U(F;\Lambda,P^k\Lambda^\#))}=(O_V(F;\Lambda):O_U(F;\Lambda)O_V(F;\Lambda,P^k\Lambda^\#)).
  \end{equation*}
We put $O_V(F;\Lambda,(M,P^k\Lambda^\#)):=\{\phi \in O_V(F;\Lambda)\mid
\phi(x)-x \in P^k\Lambda^\# \text{ for all } x\in M\}$ and claim that 
$O_U(F;\Lambda)O_V(F;\Lambda,P^k\Lambda^\#)=O_V(F;\Lambda,M,P^k\Lambda^\#)$
holds.

\medskip
The inclusion from left to right is trivial, for the other direction
let $\phi \in O_V(F;\Lambda,(M,P^k\Lambda^\#))$ be given. Since $M$ is
primitive in $\Lambda$ there exists a primitive sublattice $K\subseteq
\Lambda$ with $\Lambda=M\oplus K$, write $\sigma=\phi\oplus \id_K$
which is a linear map, but not necessarily an isometry. For $x=y+z\in
\Lambda$ with $y\in M, z\in K$ we have $y':=\sigma(y)-y=\phi(y)-y\in
P^k\Lambda^\#$ and hence $Q(\sigma(x))=Q(y+y'+z)\equiv Q(x) \bmod
P^k$, and for $x\in M$ we even have $Q(\sigma(x))=Q(\phi(x))$.

 By Hensel's Lemma there exists an isometric map $\rho:
\Lambda \to \Lambda^\#$ with $\rho(x)\equiv \sigma(x) \bmod
P^k\Lambda^\#$ for all $x \in \Lambda$, from which we see that 
$\rho \in O_V(F;\Lambda,M,P^k\Lambda^\#)$ holds. 
More precisely, going through the proof of Theorem
\ref{hensellemma_kneser} we see that the bilinear form $\beta$ used
there can be chosen to satisfy $\beta(x,y)=0$ for all $y \in M$. If
$(e_1,\ldots,e_n)$ is a basis of $\Lambda$ for which the first $w$
vectors form a basis of $M$ and the remaining $u$ vectors form a
basis of $K$, we can then choose the linear map $g$ in that proof to
satisfy $g(e_1)=\dots=g(e_w)=0$ and hence $g\vert_M=0$, i.e., the
isometric map
$\rho$ mentioned above can be constructed in such a way that it
satisfies $\rho\vert_M=\phi\vert_M$. 

We have therefore $\phi\rho^{-1} \in O_U(F;\Lambda)$, which proves the
claimed equality.

\medskip
The quotient of group indices in the assertion of the Lemma is
therefore equal to
$(O_V(F;\Lambda):O_V(F;\Lambda,(M,P^k\Lambda^\#)))$.
But obviously $\phi,\psi \in O_V(F;\Lambda)$ are in the same coset
modulo $ O_V(F;\Lambda,(M,P^k\Lambda^\#))$ if and only if they induce
the same map $M\to \Lambda/P^k\Lambda^\#$, which proves the assertion.
\end{proof}
\begin{lemma}
 With notations as before let $S_0$ be the Gram matrix of $(V,Q)$ with
 respect to some fixed basis of $V$   and analogously $T_0$ a Gram
 matrix of $(W,Q)$, denote by $\mu_V, \mu_U$ the Haar measures on
 $O_V(F), O_U(F)$ constructed with respect to these Gram matrices as
 in Section \ref{nonarchimedean_haarmeasure}.

Then one has
\begin{equation*}
  \frac{\mu_V(O_V(F;\Lambda,P^k\Lambda^\#))}{\mu_U(O_U(F;\Lambda,P^k\Lambda^\#))}=
q^{kw\frac{1-u-n}{2}}\vert{\det}_b(M)\vert_v^{\frac{1-u}{2}}\vert{\det}_b(\Lambda)
  \vert_v^{-\frac{w}{2}}\frac{\vert
    \det(S_0)\vert_v^{-\frac{n+1}{2}}\vert 2 \vert_v^{-n}}{\vert
    \det(T_0)\vert_v^{-\frac{u+1}{2}}\vert 2 \vert_v^{-u}}.
\end{equation*}
\end{lemma}
\begin{proof}
 We denote by $L$ the orthogonal projection of $\Lambda$  onto $U$. By
Lemma \ref{ortho_decompositions} 
we have $L^\#=\Lambda^\# \cap U$ and by Theorem \ref{duals_and_indices}
 ${\det}_b(\Lambda)={\det}_b(M){\det}_b(L)$, which implies
 $(\Lambda:P^k\Lambda^\#)=(M:P^kM^\#)(L:P^kL^\#)$.   

We claim that
\begin{equation*}
  O_U(F;\Lambda, P^k\Lambda^\#)=O_U(F;L, P^kL^\#).
\end{equation*}
The inclusion from left to right is trivial, for the other direction
let $\phi \in O_U(F;L,P^kL^\#)$ and $x=y+z\in \Lambda$ with $y\in W, z
\in L$. We see that $\phi(x)=y+\phi(z)=y+z+z'$ with $z'\in
P^kL^\#=P^k\Lambda^\#\cap U$, which implies $\phi(x)-x\in
P^k\Lambda^\#$ for all $x \in \Lambda$ and proves our claim.

\medskip
We can now apply Corollary \ref{local_measure_computed} and obtain 
\begin{eqnarray*}
  \frac{\mu_V(O_V(F;\Lambda,P^k\Lambda^\#))}{\mu_U(O_U(F;\Lambda,P^k\Lambda^\#))}&=&
\frac{(\Lambda:P^k\Lambda^\#)^{\frac{1-n}{2}}}{(L:P^kL^\#)^{\frac{1-u}{2}}}
 \frac{\vert 2\vert_v^{-n}\vert \det(S_0)\vert^{-\frac{n+1}{2}}}{\vert
 2\vert_v^{-u} \vert \det(T_0)^{-\frac{u+1}{2}}}\\
&=&(\Lambda:P^k\Lambda^\#)^{-\frac{w}{2}}(M:P^kM^\#)^{\frac{1-u}{2}}\frac{\vert 2\vert_v^{-n}\vert \det(S_0)\vert^{-\frac{n+1}{2}}}{\vert
 2\vert_v^{-u} \vert \det(T_0)^{-\frac{u+1}{2}}}\\
&=&q^{kw\frac{1-u-n}{2}}\vert{\det}_b(M)\vert_v^{\frac{1-u}{2}}\vert{\det}_b(\Lambda)\vert_v^{-\frac{w}{2}}\frac{\vert 2\vert_v^{-n}\vert \det(S_0)\vert^{-\frac{n+1}{2}}}{\vert
 2\vert_v^{-u} \vert \det(T_0)^{-\frac{u+1}{2}}}.
\end{eqnarray*}
\end{proof}
\begin{proposition}
  With notations as above 
let 
\begin{equation*}
 \alpha_v(i_{M,\Lambda}):=\frac{\vert
   \det(S_0)\vert_v^{\frac{n+1}{2}}\mu_V(O_V(F;\Lambda))}{\vert \det(T_0)\vert_v^{\frac{u+1}{2}}\mu_U(O_U(F;\Lambda))}
\end{equation*}
and write $\alpha_v^*(\Lambda,M)$ for the sum of the
$\alpha_v(i_{M,\Lambda_i)})$, where 
the inclusions $i_{M,\Lambda_i}$ run through a set 
  of representatives of the classes of primitive representations of
  $M$ by lattices $\Lambda_i$ on $V$ in the class of
  $\Lambda$. Similarly, write $A^*(M,\Lambda \bmod P^k\Lambda)$ for the
  same sum of the $A(i_{M,\Lambda_i},\Lambda \bmod P^k\Lambda)$.
Then $A^*(M, \Lambda/P^k\Lambda^\#)$ is equal to the number of $R$-linear maps
$\phi:M\to \Lambda/P^k\Lambda^\#$ satisfying $Q(\phi(x))\equiv
Q(x)\bmod P^k$ for all $x \in \Lambda$ for which $\phi(M)$ is a direct
summand in $\Lambda/P^k\Lambda^\#$
and we have 
\begin{equation*}
  \alpha_v(i_{M,\Lambda})=  q^{kw\frac{1-u-n}{2}}\vert{\det}_b(M)
\vert_v^{\frac{1-u}{2}} \vert{\det}_b(\Lambda)\vert_v^{-\frac{w}{2}}
\vert 2\vert_v^{u-n}
A(i_{M,\Lambda},
\Lambda \bmod P^k\Lambda^\#),
\end{equation*}
\begin{equation*}
\alpha_v^*(\Lambda,M)=  q^{kw\frac{1-u-n}{2}}\vert{\det}_b(M)
\vert_v^{\frac{1-u}{2}} \vert{\det}_b(\Lambda)\vert_v^{-\frac{w}{2}}
\vert 2\vert_v^{u-n}
A^*(M, \Lambda \bmod P^k\Lambda^\#).
\end{equation*}
\end{proposition}
\begin{proof}
The first assertion is a direct consequence of the two lemmas above.

For the second assertion we observe that the quotient in the second lemma
above is independent of the class of the representation at hand. The
assertion therefore follows upon summation over the classes of
primitive representations.
\end{proof}
\begin{remark}
The $\alpha_v^*(\Lambda,M)$ are called the {\em primitive local
representation densities}.  
\end{remark}
\begin{theorem}[Siegel's main theorem for representations, definite case]
  Let $F$ be a totally real number field of degree $r$ over $\Q$ and
  $(V,Q), (W,Q')$  totally definite
  quadratic spaces over $F$ with $\dim(W)=w<\dim(V)=n$. Let $\Lambda$
  be a lattice on $V$ and $M$ 
  a lattice on $W'$ that is represented locally everywhere primitively
  by $\Lambda$, let (without loss of generality) $\phi:M\to \Lambda$
  be a primitive representation 
  of $M$ by $\Lambda$,
  let $\Lambda_1,\ldots,\Lambda_h$ be a set of representatives of the
proper   classes in the genus of $\Lambda$.
  \begin{enumerate}
  \item For a lattice $\Lambda_j\in \gen(\Lambda)$  let
    $r(\gen(\phi),\Lambda_j)$ denote the number of 
    representations of $M$ by $\Lambda_j$ which are in the genus of
    the representation $\phi$ (all of which are primitive), let 
    \begin{equation*}
r(\gen(\phi)) =
\dfrac{\sum_{j=1}^h\dfrac{r(\gen(\phi),\Lambda_j)}{\vert
    SO_V(F;\Lambda_j)\vert}}{\sum_{j=1}^h\dfrac{1}{\vert
    SO_V(F;\Lambda_j)\vert}}.
\end{equation*}  
Then
\begin{equation*}
  r(\gen(\phi))=\vert
  D_F\vert^{-\frac{(n+u-1)(n-u)}{4}}\pi^{\frac{r(n+u+1)(n-u)}{4}}\prod_{j=u+1}^n
  (\Gamma(\frac{j}{2}))^{-r}\prod_{v\not\in \infty}\alpha_v(\phi).
\end{equation*}
\item Let $r^*(\Lambda_j,M)$ denote the number of primitive
  representations of $M$ by $\Lambda_j$, let 
  \begin{equation*}
    r^*(\gen(\Lambda),M)=\dfrac{\sum_{j=1}^h\dfrac{r^*(\Lambda_j,M)}{\vert
    SO_V(F;\Lambda_j)\vert}}{\sum_{j=1}^h\dfrac{1}{\vert
    SO_V(F;\Lambda_j)\vert}}.
\end{equation*}
Then 
\begin{equation*}
 r^*(\gen(\Lambda),M)=\vert
  D_F\vert^{-\frac{(n+u-1)(n-u)}{4}}\pi^{\frac{r(n+u+1)(n-u)}{4}}\prod_{j=u+1}^n
  (\Gamma(\frac{j}{2}))^{-r}\prod_{v\not\in \infty}\alpha_v^*(\Lambda,M).  
\end{equation*}
\end{enumerate}
If one replaces proper classes by classes and $SO_V$ by $O_V$ in the
definitions of $ r^*(\gen(\phi)),r^*(\gen(\Lambda),M)$, these numbers
remain unchanged.
\end{theorem}
\begin{proof} We know that we can restrict attention to inclusion
  mappings $i_{M,\Lambda_j}$. Let such an $i_{M,\Lambda_j}$ be
  given. Then $SO_V(F;\Lambda)$ operates transitively on the set of
  representations of $M$ by $\Lambda_j$ which are in the proper class
  of $i_{M,\Lambda_j}$, and the stabilizer of $i_{M,\Lambda_j}$ is
  $SO_U(F;\Lambda_j)$. If we sum here over all classes of
  representations of $M$ by $\Lambda_j$ which are in the genus of
  $i_{M,\Lambda_j}$ (representing each such class by an inclusion) we find  
\begin{equation*}
\dfrac{r(\gen(\phi),\Lambda_j)}{\vert SO_V(F;\Lambda_j)\vert}=\frac{1}{SO_U(F;\Lambda_j)}.  
\end{equation*}
With a double coset decomposition 
\begin{equation*}
  SO_U(\A_F)=\bigcup_{j=1}^t SO_U(F) \phi_j SO_U(\A_F;\Lambda)
\end{equation*}
we can write here $SO_U(F;\phi_j(\Lambda))=SO_U(F)\cap \phi_j
SO_U(\A_F;\Lambda) \phi_j^{-1}$, and our first version of Siegel's
main theorem, applied to the congruence subgroup
$H=SO_U(\A_F;\Lambda)$ of $SO_U(\A_F)$ gives
\begin{equation*}
  \sum_{j=1}^h\dfrac{r(\gen(\phi),\Lambda_j)}{\vert
    SO_V(F;\Lambda_j)\vert}=\frac{2}{\mu_{U,\infty}(SO_U(F_\infty))\prod_{v\not\in\infty}\mu_{U,v}(SO_U(F_v;\Lambda))},
\end{equation*} where the $\mu_{U,v}$ are the local factors of the
Tamagawa measure for $SO_U(\A_F)$. We divide this equation by the
equation given by the first version of Siegel's main theorem for the
congruence subgroup $SO_V(\A_F,\Lambda)$ of $SO_V(\A_F)$ and obtain,
in view of the definition of the $\alpha_v(i_{M,\Lambda})$ and
inserting the measures of the infinity components already computed, the
assertion in a). 
Notice that the factors $\vert \det(S_0)\vert_v^{-\frac{n+1}{2}},\vert
\det(T_0)\vert_v^{-\frac{u+1}{2}}$ occurring in the local measures for
all places $v$ including the archimedean ones cancel out because of
the product formula. 

The assertion in b) then follows upon summation over
the genera of primitive representations.

That changing from $SO$ to $O$ doesn't change anything is obvious
since numerator and denominator both change by a factor of $2$.
\end{proof}
\begin{remark} As in Theorem \ref{massformel}
  the factors $\vert 2\vert_v^{u-n}$ occurring non trivially in the $\alpha_v$
    for the dyadic places $v$ of $F$ (i.e., the places where the
    $v$-value of $2$ is not $1$) combine in the product to a total
    factor of $2^{r(n-u)}=2^{rw}$. It is again  not unusual to omit these
    factors in the definition of the local densities at the non
    archimedean places and instead to insert a factor of $2^{n-u}=2^{w}$ for each
    of the real places.
\end{remark}
\begin{example}
  We let $(\Lambda,Q)=I_4$ be the four dimensional cube lattice whose Gram
  matrix is diagonal with diagonal entries $2$. The lattice is
  unimodular over $\Z_p$ for all odd primes $p$, for $p=2$ we have
  $\Lambda_2^\#=\frac{1}{2}\Lambda$ with
  $Q(\Lambda_2^\#)\Z=\frac{1}{4}\Z$, so that we can choose $k=k_2=3$
  at the prime $2$. For $M$ we choose the $1$-dimensional lattice with
  Gram matrix $2t$ for some $t\in\N$. For odd primes $p$ we have to
  count the primitive solutions
  modulo $p$ of $\sum_{i=1}^4 x_i^2=t$, where primitive is here the
  same as nonzero modulo $p$. This is most easily done writing the
  completion $\Lambda_p$ as a sum of two hyperbolic planes, we find (see also Section 13
  of \cite{kneserbook}) $p^3-p=p(p^2-1)$ for $p\nmid t$ and $p^3+
  p^2-p-1=(p+1)(p^2-1)$.

At the prime $2$ we have to count the number of solutions of
$\sum_{i=1}^{4}x_i^2 \equiv t \bmod 8$ where the $x_i$ are considered
modulo $4$ and where at least one of the $x_i$ is odd. For $t=1$ we
have $\sum_{i=1}^{4}x_i^2 \equiv 1 \bmod 8$ if and only if precisely
one of the $x_i$ is odd and either one or all of the remaining
variables are zero modulo $4$, which gives a total of $4 \cdot 2 \cdot
(3+1)=32$ solutions. Since the lattice $(\Lambda_2, tQ)$ is easily seen
to  be isometric to $(\Lambda_2,Q)$ over $\Z_2$ for all odd $t$, one
obtains the same result for all odd $t$, alternatively one can count
the solutions for those $t$ directly as well. For $t$ congruent to $2$
or $6$ modulo $8$ precisely $2$ of the $x_i$ have to be odd and we
obtain $48$ solutions, for $t\equiv 4 \bmod 8$ all $x_i$ have to be
odd and we obtain $16$ solutions, for $t\equiv 0 \bmod 8$ there are no
solutions.

For $t=1$ our formula gives then with $\alpha_2(I_4,t):=\alpha_2(\Lambda,M)=1$ and
$\alpha_p(I_4,t):=\alpha_p(\Lambda,M)=(1-p^{-2})$ for $p\ne 2$ and using $\zeta(2)=\frac{\pi^2}{6}$
\begin{eqnarray*}
  r(1,I_4)&=&\pi^2 \prod_p\alpha_p(\Lambda,M)\\
&=&\pi^2 \frac{4}{3}\zeta(2)^{-1}\\
&=&8,
\end{eqnarray*}
in agreement with the obvious count for this number.

For $p\nmid t$ we have the same values of $\alpha_p$ as above, for
$p\mid t$ and $2\ne p$ we have $\alpha_p^*(I_4,t)=\frac{p+1}{p}\vert t\vert_p^{-1}\alpha_p(I_4,1)$,
for $p=2$ we obtain
\begin{equation*}
  \alpha_2^*(I_4,t)=\begin{cases}
3 & t\equiv 2,6 \bmod 8\\
2& t \equiv 4 \bmod 8
\end{cases}
\end{equation*}
and hence with $t=2^s t', s \nmid t'$ 
\begin{equation*}
  r^*(I_4,t)=t'\prod_{2\ne p \mid t}\frac{p+1}{p}
  \begin{cases}
  8 &2 \nmid t\\
24& t\equiv 2,6\bmod 8\\
16&t\equiv 4 \bmod 8  
  \end{cases}.
\end{equation*}
The formula for the number $r(I_4,t)=\sum_{d^2\mid t} r^*(I_4,\frac{t}{d^2})$ of all representations of $t$ by
$I_4$ looks smoother, we obtain
\begin{equation*}
  r(I_4,t)=8 \sum_{d\mid t, 4\nmid d}d,
\end{equation*}
a formula which was first obtained by Jacobi with the help of a study
of the
analytic properties of the theta series $\sum_{t=0}^\infty
r(I_4,t)\exp(\pi i t z)$.
\end{example}

So far  we have only considered primitive representations, let us now
turn towards imprimitive representations.
\begin{lemma}
With notations as before let $M'$ be an integral lattice on $W$ and
denote by $M_1,\ldots,M_t$ the integral overlattices of $M'$. 

Then $r(\Lambda,M')=\sum_{i=1}^t r^*(\Lambda,M_i)$.
\end{lemma}
\begin{proof}
  Obvious.
\end{proof}
\begin{lemma}
With notations as above let $v$ be a non archimedean place of $F$ and
write $A(M',M;\Lambda \bmod P^k\Lambda^\#)$ for the number of $R$-linear
maps   $\phi: M'\to \Lambda /P^k\Lambda^\#$ satisfying
$\phi(M)=\phi(W)\cap \Lambda$ and $Q(\phi(x))\equiv Q(x)\bmod P^k$ for
all $x\in M'$, let $\alpha_v(\Lambda;M',M)$ be obtained from
$A(M',M;\Lambda \bmod P^k\Lambda^\#)$ as we obtained
$\alpha_v^*(\Lambda;M)$ from $A^*(M;\Lambda\bmod P^k\Lambda^\#)$ above.

Then for large enough $k$ 
\begin{equation*}
  A(M',M;\Lambda \bmod P^k\Lambda^\#)=(M:M')^{1-u}A^*(M;\Lambda/P^k\Lambda^\#),
\end{equation*} and one has 
\begin{equation*}
 \alpha_v(\Lambda;M',M )=\alpha_v^*(\Lambda;M).
\end{equation*}
\end{lemma}
\begin{proof}
  The second formula follows from the first one.
By the elementary divisor theorem there is a basis $f_1, \ldots, f_w$
of $M$ such that $\pi_v^{s_1}f_1,\ldots,\pi_v^{s_w}f_w$ with some $0\le
  s_1\le\dots\le s_w$ is a basis  of $M'$, where $\pi_v$ is a prime
  element of $R_v$. Let $T$ denote the Gram matrix of $Q$ with respect
  to the $f_i$.

Let  $\phi:W\to V$ be an $R$-linear (hence $F$-linear) map such that
$\phi(M)\subseteq \Lambda$ is primitive.
 Then $\phi$ restricted to $M'$
is an isometry modulo $P^k$ if and only if the  Gram matrix
$T^{(\phi,M')}$ of $(M',Q)$ with respect to the $\phi(\pi_v^{s_i}f_i)$ satisfies 
$t^{(\phi,M')}_{ij}\equiv_{ev} \pi_v^{s_i+s_j}t_{ij} \bmod P^k$, where
$\equiv_{ev}$ denotes as earlier that the congruence is modulo $2P^k$
for $i=j$.
Equivalently, we may write 
$t^{(\phi,M)}_{ij}\equiv_{ev} t_{ij} \bmod P^{k-(s_i+s_j)}$, where
$T^{(\phi,M)}$ denotes the Gram matrix of $(M,Q)$ with respect to the $\phi(f_i)$.
 The number of modulo $P^k$ different such Gram matrices  $T_\ell$ is
 $\prod_{i\le j}q^{s_i}q^{s_j}=(M:M')^{w+1}$. 

The proof of
 Hensel's Lemma (Theorem \ref{hensellemma_kneser}) shows that for
 large enough $k$  the number of modulo $P^k\Lambda^\#$ distinct $R$-linear
 maps $\psi:M\to \Lambda$ for which $\psi(M)$ is primitive in
 $\Lambda$ and $T^{(\psi,M)}\equiv_{ev}T_\ell\bmod P^k$ for one of
 these $T_\ell$ holds is $(M:M')^{w+1}A^*(M;\Lambda \bmod
 P^k\Lambda^\#)$. In other words,  all these matrices occur equally
 often among the Gram matrices modulo $P^k$ associated to refinements
 modulo $P^k\Lambda^\#$ of the maps  counted by $A^*(M;\Lambda \bmod 
 P^{k-2s_w}\Lambda^\#)$. 
Furthermore, each of these $\psi$ yields upon
 restriction to $M'$ one of the maps $\phi$ counted in $ A(M',M;\Lambda \bmod
 P^k\Lambda^\#)$, and each of those $\phi$ occurs as the restriction of
 $(M':M)^n$ different $\psi$. Summing up and noticing $1-u=1+w-n$ we see that indeed
\begin{equation*}
  A(M',M;\Lambda \bmod P^k\Lambda^\#)=(M:M')^{1-u}A^*(M;\Lambda/P^k\Lambda^\#),
\end{equation*} holds.

Inserting this result into the formula relating the $\alpha_v(\quad),\alpha_v^*(\quad)$
with the $A(\quad), A^*(\quad)$ one obtains the second part of the assertion.
\end{proof}
\begin{remark}
For both lemmas  the argument above is valid for the case $W=V$ as well.
\end{remark}
\begin{theorem}
  With notations as above one has 
  \begin{equation*}
    r(\gen(\Lambda),M')=\dfrac{\sum_{j=1}^h\dfrac{r(\Lambda_j,M')}{\vert
    SO_V(F;\Lambda_j)\vert}}{\sum_{j=1}^h\dfrac{1}{\vert
    SO_V(F;\Lambda_j)\vert}}    
  \end{equation*}
satisfies
\begin{equation*}
 r(\gen(\Lambda),M')=\vert
  D_F\vert^{-\frac{(n+u-1)(n-u)}{4}}\pi^{\frac{r(n+u+1)(n-u)}{4}}\prod_{j=u+1}^n
  (\Gamma(\frac{j}{2}))^{-r}\prod_{v\not\in \infty}\alpha_v(\Lambda,M').  
\end{equation*}
In the case $V=W$ (i.e, $u=0$ in the formula above) one has to adjust
here the definition of $\alpha_v(\Lambda,M')$ by a factor $\frac{1}{2}$
on the right hand side.
\end{theorem}
\begin{proof}
  The first of the above lemmas expresses the number of representations
  of $M'$ by $\Lambda$ in terms of the primitive representation
  numbers of the integral overlattices of $M'$. The resulting sum of
  numbers of primitive representations can then be expressed using the
  primitive local densities $\alpha_v^*$ for these overlattices, and
  application of the second lemma finishes the proof.
Notice that the factor $\frac{1}{2}$ already occurred in the formula
  for the Maß (measure) of a genus, it occurs here for the same reason.
\end{proof}

\chapter{Spin Group and Strong Approximation}
In this chapter $F$ is a field of characteristic not $2$ and $(V,Q)$
is a non degenerate quadratic space over $F$. In the definition of the
Clifford group we will follow the by now standard twisted approach introduced by
Atiyah and Bott, modifying the original definition of Chevalley.

\section{Clifford group and spin group}
\begin{definition}
  The Clifford group $\Gamma(V,Q)=\Gamma_V$ is defined by
  \begin{equation*}
    \Gamma_V:=\{x \in C(V,Q)^\times \mid C(-\id)(x) v x^{-1} \in V
    \text{ for all } v\in V\}.
  \end{equation*}
\end{definition}
\begin{lemma}
 $\Gamma_V$ acts on $V$ by $x.v=C(-\id)(x) v x^{-1}$, and the
 homomorphism $\rho:\Gamma_V\to GL(V)$ defined by this group action has
 kernel $F^\times\subseteq \Gamma_V$.
\end{lemma}
\begin{proof}
  It is obvious that  $x.v=C(-\id)(x) v x^{-1}$ defines a group action
  of $\Gamma_V$ on $V$ If $\x \in \Gamma_V$ is in the kernel of $\rho$
  we write $x=x_0+x_1$ with $x_i \in C_i(V,Q)$ and have
  $x_0v-x_1v=vx_0+vx_1$ and isolating the components in the grading we
  see $x_0v=vx_0, x_1v=-vx_1$. This implies $x_0z=zx_0$ for all $z\in
  C(V,Q)$ and $x_1z_j=(-1)^ju_jx_1$ for $z_j\in C_j(V,Q)$. We see that
  $x$ is in the graded center of $C(V,Q)$ as defined in Remark
  \ref{brauerwall}, and since this graded center is trivial we obtain
  the assertion.
\end{proof}
\begin{deflemma}
  For $x\in \Gamma_V$ one has 
  \begin{equation*}
    n(x):=\bar{x}x\in F
  \end{equation*}
and $\bar{x}x=x\bar{x}$. Moreover,  $n$ is a group homomorphism from
$\Gamma_V$ to $F^\times$ satisfying $n(x)=-Q(x)$ for $x\in V$ with $Q(x)\ne 0$.

We call $n(x)$ the norm or the Clifford norm of $x$.
\begin{proof}
We apply the standard involution $x\mapsto \bar{x}$ of
$C(V,Q)$ to the equation $C(-\id)(x) v x^{-1}=w\in V$ 
 to get
$\bar{x}^{-1}vC(-\id)\bar{x}=w$.
From this we get $v=C(-\id)(\bar{x}x)v (\bar{x}x)^{-1})$, i.e.,
$\bar{x}x$ is in the kernel of the homomorphism $\rho$ above and
therefore in $F^\times$.

We see then  $x(\bar{x}x)=(\bar{x}x)x$ and thus $\bar{x}x=x\bar{x}$,
which easily gives that $n$ is a group homomorphism. The definition of
$\bar{x}$ finally implies $n(x)=-Q(x)$ for $x\in V$ with $Q(x)\ne 0$.
\end{proof}
\end{deflemma}
\begin{lemma}
  For $x\in V$ with $Q(x)\ne 0$ we have $x\in \Gamma_V$ and
  $\rho(x)(v)=\tau_x(v)$ for all $v \in V$. 
\end{lemma}
\begin{proof}
  We have
  \begin{eqnarray*}
    C(-\id)(x) v x^{-1}&=&-xv \frac{x}{Q(x)}\\
&=&\frac{(vx-b(x,v))x}{Q(x)}\\
&=&v-\frac{b(x,v)}{Q(x)}x\\
&=&\tau_x(v)
  \end{eqnarray*} for all $v\in V$.
\end{proof}
\begin{theorem}
  The homomorphism $\Gamma_V/F^\times \to GL(V)$ induced by $\rho$ is an
  isomorphism onto $O_V(F)$. 
\end{theorem}
\begin{proof}
 Let $v\in V$ with $Q(v)\ne 0$. We have 
 \begin{equation*}
  -Q(C(-\id)(x) v x^{-1})= n(C(-\id)(x) v x^{-1})=-n(C(-\id)(x))Q(v)n(x)^{-1}
 \end{equation*}
since $n$ is a
   group homomorphism and $v\in \Gamma_V$ by the previous lemma. But
   $\overline{C(-\id)(x)}=C(-\id)(x)$ implies $n(C(-\id)(x))=n(x)$ and
   we obtain $Q(\rho(x)(v))=Q(v)$ for all anisotropic $v$. If $v \in V$
   is isotropic with $w=\rho(x)(v)$, we have $v=\rho(x^{-1})(w)$, so $w$ has to be
   isotropic as well, and we see that indeed $\rho(x) \in O_V(F)$.

To see that $\rho$ is surjective we recall that by Witt's generation
theorem \ref{SatzvonWitt} $O_V(F)$ is generated by symmetries $\tau_x$
with anisotropic $x$, which are in the image of $\rho$ by the previous lemma.
\end{proof}
\begin{corollary}
Any $x \in \Gamma_V$ is homogeneous, i.\ e., it is either in
$C_0(V,Q)$ or in $C_1(V,Q)$. More precisely, it can be written as a
monomial $v_1\dots v_r$ with $v_i \in V$ and  $\det(\rho(x))=(-1)^r$.  
\end{corollary}
\begin{proof}
  Obvious.
\end{proof}
\begin{definition}
The Spin group $\spin(V,Q)=\spin_V(F)$ of $(V,Q)$ is $\{x \in
\Gamma_V\cap C_0(V,Q)\mid n(x)=1\}$. The group $\pin_V(F)$ is   $\{x \in
\Gamma_V\mid n(x)=1\}$.

The group $O_V'(F)\subseteq SO_V(F)$ is the image of $\spin_V(F)$
under $\rho$.
\end{definition}
\begin{lemma}
 The group $O_V'(F)$ contains the commutator subgroup of $O_V(F)$. 
\end{lemma}
\begin{proof}
  This is obvious since $(\tau_{y_1}\circ \dots \circ
  \tau_{y_r})^{-1}=\tau_{y_r}\circ \dots \circ \tau_{y_1}$ holds.
\end{proof}
\section{Spinor norm, spinor genus, and strong approximation}
\begin{definition}
Let $\phi \in SO_V(F), \phi=\rho(x)$ with $x \in
\Gamma_V\cap C_0(V,Q)$.
The {\em spinor norm} $\theta(\phi)$ of $\phi$ is the square class $n(x)(F^\times)^2$.
\end{definition}
\begin{lemma}
  Let $\phi =\tau_{y_1}\circ \dots \circ \tau_{y_r}\in SO_V(F)$ with
  anisotropic vectors $y_i \in V$. Then $\theta(\phi)=Q(y_1)\dots
  Q(y_r)(F^\times)^2$.

In particular, the square class $Q(y_1)\dots
  Q(y_r)(F^\times)^2$ is independent of the choice of a product
  decomposition $\phi =\tau_{y_1}\circ \dots \circ \tau_{y_r}\in SO_V(F)$.
\end{lemma}
\begin{proof}
  Obvious.
\end{proof}
\begin{remark}
  If one scales the quadratic form $Q$ by a factor $c\in F^\times$,
  the orthogonal group obviously doesn't change, and the spinor norm
  of a $\phi \in SO_V(F)$ doesn't change either. It is of course
  possible to define the spinor norm for orthogonal transformations of
  determinant $-1$ as well, but it will then not be invariant under
  scaling of the quadratic form.
\end{remark}
\begin{example}
  \begin{enumerate}
  \item Let $F=\R$ and $(V,Q)$ be positive definite. Then
    $\theta(\phi)=(\R^\times)^2$ for all $\phi \in SO_V(F)$ since all
    the $Q(v)$ are positive, and the
    restriction of the homomorphism $\rho$ to $\spin_V(\R)$ is
    surjective. Since its kernel is $\{\pm 1\}$, the spin group is in
    this case a double cover of the special orthogonal group.
\item  If $F$ is $\C$ or another algebraically closed field, the
  surjectivity of $\rho\vert_{\spin_V(F)}$ is obvious and we have
  again a double covering of $SO_V(F)$ by $\spin_V(F)$. In particular,
  in the sense of the theory of algebraic groups, the spin group (as
  an algebraic group) is a double cover of the special orthogonal
  group. If $F$ is a field that is not algebraically closed, e.\ g.\
  $F=\Q$, the image  $O_V'(F)$ of the $F$-points $\spin_V(F)$ can
  nevertheless be smaller than $SO_V(F)$.  
  \end{enumerate}
\end{example}
\begin{definition}
  Let $F$ be a number field, let $\Lambda,\Lambda'$ be lattices on
  $V$.

$\Lambda'$ is said to be in the {\em spinor genus} of $\Lambda$ and one
writes $\Lambda'\in \spn(\Lambda)$ if there are $\sigma \in O_V(F)$
and $\phi=(\phi_v)_{v\in \Sigma_F}\in O'_V(\A_F)$ with
$\Lambda'=\sigma(\phi(\Lambda))$.

If here $\sigma$ can be chosen in $SO_V(F)$ we say that $\Lambda'$ is
in the proper spinor genus $\spn^+(\Lambda)$.
\end{definition}
\begin{remark}
 It is clear that lattices in the same spinor genus are a fortiori in
 the same genus and that a spinor genus consists of full isometry
 classes. It turns out that in some cases the concepts of spinor genus
 and isometry class coincide.
\end{remark}
\begin{theorem}[Strong approximation]\label{strong-approximation}
Let $F$ be a global field and $(V,Q)$ be a non degenerate quadratic
space over $F$, let $S\subseteq \Sigma_F$ be a finite set
of places such that $(V_v,Q_v)$ is isotropic for at least one $v\in
S$. 

Then $\spin_V$ admits strong approximation with respect to $S$, i,\
e.,\\ $\spin_V(F)\prod_{v\in S}\spin_V(F_v)$ is dense in $\spin_V(\A_F)$. 
\end{theorem}
\begin{proof}
  This has been proven for all almost simple simply connected
  classical groups over number fields by Kneser in
  \cite{kneserapprox}, a first form of the result is due to Eichler
  \cite{eichler_qfog}. For the general proof and historical remarks
  see \cite{platonov_rapinchuk}.
\end{proof}
\begin{remark}
 The group $\spin_V(\A_F)$ of adelic points of the spin group is as
 usual the restricted direct product of the $\spin_V(F_v)$ with
 respect to a family of compact subgroups $H_v$ for the non
 archimedean places $v$ of $F$, which can be taken (equivalently) as
 \begin{equation*} H_v=\spin_V(F_v;\Lambda):=\{x\in \spin_V(F_v)\mid
   \rho(x)(\Lambda) \subseteq \Lambda\}\end{equation*} 
 for a fixed (but arbitrary) lattice $\Lambda$ on $V$ or as
 $H_v=\{\sum_j\alpha_jz_j\in \spin_V(F_v)\mid \alpha_j \in R_v\}$, where the $z_j$ comprise
 some fixed basis of $C_0(V,Q)$ as vector space over $F$ and $R_v$ is
 the valuation ring of $F_v$. A basis of open neighborhoods of $1$ in
 $\spin(\A_F)$ is then given by the $\prod_vU_v$, where each $U_v$ is
 open in $\spin_v(F_v)$ and may be chosen to be a principal congruence
 subgroup of  $\spin_v(F_v)$ for non archimedean $v$, and where
 $U_v=H_v$ for almost all $v$.
\end{remark}
To give the strong approximation property a more concrete shape we
formulate the following corollary:
\begin{corollary}\label{strongapprox_corollary}
Let $F$ and $S\subseteq \Sigma_F$ be as in the theorem, assume that
$S$ contains all archimedean places of $F$ and let $v_1,\ldots,v_r \in
\Sigma_F\setminus S$ be given.

Let moreover $g_j \in \spin_V(F_{v_j})$, $\phi_j \in O_V'(\F_{v_j})$
and  $s_j \in \N$ be given for
$1\le j\le r$. Let  $\spin_V(F_v;\Lambda, \pi^{s_j}\Lambda)\subseteq
\spin_V(F_v, \Lambda)$ for a
lattice $\Lambda$ on $V$ denote the  congruence subgroup
consisting  of all $g\in \spin_V(F_v)$
satisfying $\rho(x)-x \in \pi^{s_j}\Lambda$ for all $x \in \Lambda$.
Then
\begin{enumerate}
\item There exists $g \in\spin_V(F)$ with $g \in
  g_j\spin_V(F_{v_j};\Lambda,\pi_v^{s_j}\Lambda)$ for $1\le j\le r$
  and $g \in \spin_V(F_v;\Lambda)$ for all $v \notin
  (S\cup\{v_1,\ldots,v_r\})$.
\item There exists $\sigma \in O_V'(F)$ satisfying  
$\sigma \in \phi_j(O_V'(F_v)\cap SO_V(F_v;\Lambda,\pi^{s_j}\Lambda)$
for $1\le j\le r$ and $\sigma \in SO_V(F_v;\Lambda)$ for all $v\in
\Sigma_F\setminus (S\cap\{v_1,\ldots,v_r\})$.
\end{enumerate}
\end{corollary}
\begin{proof}
  This is obtained by writing out the statement of the theorem in
  terms of the topology of $\spin_V(\A_F)$.
\end{proof}
\begin{theorem}[Eichler]\label{spinorgenus_indefinite}
  Let $F$ be a number field and $\Lambda$ a lattice on the non
  degenerate quadratic space $(V,Q)$ over $F$, assume that $(V_v,Q_v)$
  is indefinite for at least one archimedean place of $F$.

Then the (proper) spinor genus of $\Lambda$ consists of only one
(proper) class.
\end{theorem}
\begin{proof}
  Let $\Lambda'$ be a lattice on $V$ in the spinor genus of $\Lambda$,
  we may assume that $\Lambda'=\phi(\Lambda)$ for some
  $\phi=(\phi_v)_v\in O_V'(\A_F)$. Let $S=\infty$ denote the set of
  archimedean places of $F$ and denote by $v_1,\ldots,v_r$ the
  finitely many $v \in \Sigma_F\setminus S$ with $\Lambda_v'\ne
  \Lambda$.

By Corollary \ref{strongapprox_corollary} there exists $\sigma \in
O_V'(F)$satisfying $\sigma \in \phi_{v_j}SO_V(F_v;\Lambda)$ for $1\le
j \le r$ and $\phi \in SO_V(F_v;\Lambda)$ for all non archimedean $v
\notin \{v_1,\ldots,v_r\}$. We have then
$\sigma(\Lambda_v)=\Lambda_v'$ for all non archimedean $v$, hence
$\sigma(\Lambda)=\Lambda'$, and $\Lambda'$ is in the class of
$\Lambda$. The argument for the proper spinor genus is identical.
\end{proof}
The main advantage of the concept of spinor genus is that questions
about spinor genera can essentially be decided on a local  level. 
We want to show next  that the number of spinor genera in a given genus of
lattices can be explicitly computed as an index of idele groups. For
this we need some lemmas.
\begin{lemma}
Let $F=F_v$ be a non archimedean local field for which $2$ is a unit
in the valuation ring $R=R_v$ (a nondyadic local field). Let $\Lambda$
be a unimodular lattice on the non degenerate quadratic space $(V,Q)$
over $F$ of dimension $\ge 2$.
Then $\theta(SO_V(F;\Lambda))=R^\times(F^\times)^2$. 
\end{lemma}
\begin{proof}
As in the proof of Theorem \ref{SatzvonWitt} one proves by induction
on $\dim(V)$ that
$O_V(F;\Lambda)$ is generated by symmetries $\tau_x$ with respect to
vectors $x \in \Lambda$ with $Q(x)\in R^\times$; indeed in all places
where we had the condition that $Q(z)\ne 0$ for some vector $z$ we can
replace it here by $Q(z)\in R^\times$, the fact that $2$ and hence $4$
is a unit in $R$ is crucial for this. Since a unimodular $R$-lattice
of rank $\ge 2$ represents all units (here we use again that $2$ is a
unit), the assertion follows.   
\end{proof}
\begin{deflemma}
 Let $(V,Q)$ be a non degenerate quadratic space over the number field
 $F$, let $\phi=(\phi_v)_v \in SO_V(\A_F)$. Then $\theta(\phi_v)\cap
 R_v^\times\ne \emptyset$ for almost all places $v$ of $F$, and there
 is an idele $\alpha=(\alpha_v)_v\in J_F$ with $\alpha_v \in
 \theta(\phi_v)$ for all $v \in \Sigma_F$, and all such ideles form a
 square class modulo $J_F^2$.

This square class $\alpha J_F^2\subseteq J_F$ is denoted by
$\theta(\phi)$ and is called the idele spinor norm of $\phi$. 
\end{deflemma}
\begin{proof}
 Let $\Lambda$ be a lattice on $V$. Since $\phi_v \in
 SO_V(F_v;\Lambda_v)$ holds 
 for almost all $v$ and $\Lambda_v$ is unimodular for almost all
 $v$, the assertion follows from the previous lemma.  
\end{proof}
\begin{theorem}\label{spinorgenera_number}
 Let $(V,Q)$ and $F$ be as before, $\Lambda$ a lattice on $V$, let
 $\phi=(\phi_v)_v\in SO_V(\A_F)$.

Then  $\phi(\Lambda)$ is in the proper spinor genus of $\Lambda$
if and only if 
\begin{equation*}
  \theta(\phi)\in \theta(SO_V(F))\theta(SO_V(\A_F;\Lambda))
\end{equation*}
holds.

The number of proper spinor genera in the genus of $\Lambda$ is 
equal to the group index
\begin{equation*}
  (J_F^V:\theta(SO_V(F))\theta(SO_V(\A_F;\Lambda))),
\end{equation*}
where $J_F^V=\theta(SO_V(\A_F))$.
\end{theorem}
\begin{proof}
  $\phi(\Lambda)\in \spn^+(\Lambda)$ is equivalent to $\phi \in
  SO_V(F)O_V'(\A_F)SO_V(\A_F;\Lambda)$. This is again equivalent to
  $\theta(\phi)\in \theta(SO_V(F))\theta(SO_V(\A_F;\Lambda))$, notice
  for this that in $SO_V(F)O_V'(\A_F)SO_V(\A_F;\Lambda)$ the order of
  the factors can be exchanged since $ O_V'(\A_F)$ contains all
  commutators in $SO_V(\A_F)$.

The number of proper spinor genera is equal to the number of double
cosets in the decomposition
\begin{equation*}
  SO_V(\A_F)=\bigcup_j SO_V(F)\phi_j O_V'(\A_F)SO_V(\A_F;\Lambda).
\end{equation*}
By the same argument as above, the map associating to the double coset
$SO_V(F)\phi O_V'(\A_F)SO_V(\A_F;\Lambda)$  the coset of
$\theta(\phi)$ in the factor group
\begin{equation*}
J_F^V/\theta(SO_V(F))\theta(SO_V(\A_F;\Lambda))
\end{equation*}
 is a bijection.
\end{proof}
We can make the last formula a little more explicit:
\begin{lemma}
 Let $J_F^V$ be as in the theorem and assume $\dim(V)\ge 3$. Then
 \begin{equation*}
   J_F^V=\{\alpha=(\alpha_v)_v \in J_F \mid \alpha_v>0 \text{ for all
     real } v \text{ with } (V_v,Q_v) \text{ definite}\}
\end{equation*}
and $\theta(SO_V(F))=F^\times\cap J_F^V$.
\end{lemma}
\begin{proof}
For archimedean $v$ it is obvious that $\theta(SO_V(F_v))=F^\times$ if
$v$ is complex or $(V_v,Q_v)$ indefinite and that
$\theta(SO_V(F_v))$ consists of the positive elements of
$F_v^\times=\R^\times$ if $v$ is real and  $(V_v,Q_v)$ definite.

Let now $v$ be non archimedean.  If $\dim(V)\ge 4$ holds, the
completion $(V_v,Q_v)$ is universal and hence
$\theta(SO_V(F_v))=F^\times$. Let $\dim(V)=3$ and $a \in
F_v^\times$, by scaling the quadratic form we may assume that
${\det}_B(V)$ is the square class of a prime element.  If $a$ is a
unit in $R_v$, we have $-a\notin
{\det}_B(V_v)$, hence  the $4$-dimensional space $V_v \perp [-a]$ over $F_v$
has determinant $\ne (F_v^\times)^2$ and must be isotropic,
which implies $a \in Q(V_v)$. On the other hand, ${\det}_B(V)=\pi_v
(F_v^\times)^2$ for a prime element $\pi_v$ of $R_v$ implies that at
least one prime element is in $Q(V_v)$, and we see that
$\theta(SO_V(F_v))=F_v^\times$.

For the statement about the global spinor norm group $\theta(SO_V(F))$
the inclusion  $\theta(SO_V(F))\subseteq F^\times\cap J_F^V$ is clear.

Denote by $T$ the set of all real places $v$ for which $(V_v,Q_v)$ is
definite.
Let $\dim(V)\ge 4$ and let $a\in F^\times$ be positive at all $v \in
T$, let $0\ne b\in Q(V)$. Then $ab$ is positive and
hence in $Q(V_v)$ at all $v\in T$, and $ab \in Q(V_v)$ for the
remaining $v\in \Sigma_F$ since $(V_v,Q_v$ is universal for these
$v$. By the Minkowski-Hasse Theorem we have $ab\in Q(V)$, hence $ab^2
\in \theta(SO_V(F))$ and $a\in  \theta(SO_V(F))$. 

Let now $\dim(V)=3$ and $a$ be as above, by scaling the quadratic form
we may assume ${\det}_B(V)=(F^\times)^2$.  Denote by $S$ the finite set of
all non archimedean $v$ for which $(V_v,Q_v)$ is anisotropic. 
Using weak approximation we choose $b \in F^\times$  such that $\-b
\notin (F_v^\times)^2, -ab  \notin (F_v^\times)^2$ hold for all $v
\in S\cup T$, in particular, $-b$ and hence $-ab$ are negative at all
$v\in T$. The $4$-dimensional spaces $W_1=V\perp[-b], W_2=V\perp[-ab]$
are then isotropic over all completions $F_v$, and the Minkowski-Hasse
Theorem implies $b\in Q(V), ab\in Q(V)$, from which we obtain $a\in
\theta(SO_V(F))$.   
\end{proof}
\begin{lemma}
  Let $F=F_v$ be a non archimedean local field, $\cha(F)\ne 2$, let
  $(V,Q)$ be a non degenerate quadratic space of dimension $\ge 3$
  over $F$ and $\Lambda$   be a maximal lattice on $V$.

Then $\theta(SO_V(F_v;\Lambda))\supseteq R_v^\times$.
\end{lemma}
\begin{proof}
  If $(V,Q)$ is anisotropic $\Lambda$ is the set of all $x \in V$ with
  $Q(x)\in R$ by Theorem \ref{anisotropic_maximal} and hence
  $SO_V(F_v;\Lambda)=SO_V(F_v)$, which implies
  $\theta(SO_V(F_v;\Lambda))=F_v^\times$.
 Otherwise we can write $\lambda=H\perp \Lambda'$ with a hyperbolic
 plane $H$ by Theorem \ref{hyperbolic_maximal} and we see
 $\theta(SO_V(F_v;\Lambda))\supseteq R_v^\times$ as asserted since
 already the orthogonal group of $H$ has all units as spinor norms.
\end{proof}
\begin{theorem}
 Let $(V,Q)$ be a non degenerate quadratic space over $\Q$.

Then all $\Z$-maximal lattices on $V$ belong to the same proper spinor genus. 

In particular, for any signature $(n_+,n_-)$ with $n_+\ne 0 \ne n_-$
there exists at most one class of even unimodular $\Z$-lattices of
this signature.
\end{theorem}
\begin{proof}
Since $\Q$ has class number $1$ this follows from Theorem
\ref{spinorgenera_number} and the lemmas following it.  
\end{proof}
\begin{remark}\label{spinornorm_computations}
The spinor norm groups of the automorphism groups of local lattices
have been computed, following the initial work of  Eichler and Kneser, in great detail by
Hsia, Earnest, and Beli.  In particular, it has been proven that the spinor norm group of a
unimodular $\Z_2$-lattice of rank $\ge 3$ contains (as in the case of
maximal lattice treated above) all units, so that a genus of
unimodular $\Z$-lattices (odd or even) contains only one proper spinor genus.
\end{remark}

\newpage
\section{Neighboring lattices and anzahlmatrices} 
\begin{theorem}\label{arithmetically_indefinite}
  Let $F$ be a number field and $(V,Q)$ a non degenerate quadratic
  space. Let $v_0 \in \Sigma_F$ be a non archimedean place such for
  which $(V_{v_0},Q)$ is isotropic, let $\Lambda_1, \Lambda_2$ be
  lattices on $V$ in the same spinor genus.

Then there exists a lattice $\Lambda_2'$ in the isometry class of
$\Lambda_2$  with $(\Lambda_2)'_v=(\Lambda_1)_v$ for all $v \in \Sigma_F\setminus\{v_0\}$. 
\end{theorem}
\begin{proof}
  without loss of generality we can assume that there exist $\phi \in
  O_V'(\A_F)$ with $\Lambda_2=\phi(\Lambda_1)$.  Let $S:=\infty \cup
  \{v_0\}\subseteq \Sigma_F$, denote by $v_1,\ldots,v_n$ the finitely
many places $v\in \Sigma_F$ with $(\Lambda_1)_v\ne (\Lambda_2)_v$.
By Corollary \ref{strongapprox_corollary} there exist $\sigma \in
O_V'(F)$ with $\sigma \in \phi_{v_j}^{-1} O_V'(F_{v_j};\Lambda_2)$ for
$1\le j \le n$ and $\sigma \in O_V'(F_v;\lambda_2)$ for all $v \not\in
S\cup\{v_1,\ldots,v_n\}$.

Then $\sigma(\Lambda_2)=:\Lambda_2'$ is as required. 
\end{proof}
\begin{remark}
The lattices $\Lambda_1,\Lambda_2$ or the space $(V,Q)$ supporting
them are called arithmetically indefinite (``arithmetisch indefinit'')
following Eichler, who introduced this idea in \cite{eichler_spin}.   
\end{remark}
\begin{definition}[Kneser]
  Let $(V,Q)$ be a non degenerate quadratic space over the number
  field $F$ and let $v\in \Sigma_F$ be a non archimedean place of $F$
  and $P$ the associated maximal ideal in the ring of integers $R$
  of $F$.

Lattices $\Lambda,\Lambda'$ on $V$ are called $v$-neighbors or
$P$-neighbors if one has 
\begin{equation*}
  \Lambda/(\Lambda\cap \Lambda')\cong R/P \cong
  \Lambda'/(\Lambda\cap \Lambda').
\end{equation*}
 If $\Lambda=\Lambda_1,\ldots,\Lambda_h$ is a set of representatives of the
isometry classes of  lattices in the genus of $\Lambda$ denote by
$N_v(\Lambda_i,\Lambda_j) =N_P(\Lambda_i,\Lambda_j)$ the number of
lattices in the isometry class of $\Lambda_j$ which are $P$-neighbors
of $\Lambda_i$.

The $h\times h$ matrix $(N_P(\Lambda_i,\Lambda_j))_{i,j}$ is called
$P$-neighborhood matrix
of the genus of $\Lambda$.   
\end{definition}
\begin{remark}
   The $P$-neighborhood matrix defined above is a special case of the
   anzahlmatrix associated to a fixed system of elementary divisors
   for an ideal complex as defined by Eichler in \cite{eichler_qfog}.  
\end{remark}
As noticed by Eichler the anzahlmatrices satisfy a certain symmetry
relation. We formulate and prove this in a slightly more general context:
\begin{proposition}[Eichler,\cite{eichler_qfog}]
 Let $F$ be a number field or a non archimedean local field with ring
 of integers $R$, let $(V,Q)$ be a regular quadratic space over $F$.  

Let $\sim_1,\sim_2$ be two relations on the set of lattices on $V$
such that for all $\phi \in O_V(F)$ and all lattices
$\Lambda_1,\Lambda_2$ on $V$ one has $\phi(\Lambda_2)\sim_1 \Lambda_1$
if and only if $\phi^{-1}(\Lambda_1) \sim_2 \Lambda_2$ holds. Assume
moreover that for all lattices $\Lambda$ on $V$ there are only
finitely many lattices $\Lambda'$ on $V$ satisfying $\Lambda'\sim_1
\Lambda$ or $\Lambda' \sim_2 \Lambda$.

For $i=1,2$ let $N_i(\Lambda_1,\Lambda_2)$ denote the number of
lattices $M$ on $V$ which are isometric to $\Lambda_2$ and satisfy
$M \sim_i \Lambda_1$.

Then $N_1(\Lambda_1,\Lambda_2)=0$ if and only if
$N_2(\Lambda_2,\Lambda_1)=0$, and if both are nonzero one has  
\begin{equation*}
\frac{ N_1(\Lambda_1,\Lambda_2)}{N_2(\Lambda_2,\Lambda_1)}=
\frac{ (O_V(F;\Lambda_1):O_V(F;\Lambda_1)\cap
 O_V(F;\Lambda_2))}{(O_V(F;\Lambda_2):O_V(F;\Lambda_1)\cap
 O_V(F;\Lambda_2)) }.
\end{equation*}
\end{proposition}
\begin{proof}
 Let lattices $\Lambda_1,\Lambda_2$ be given with
 $\Lambda_2\sim_1\Lambda_2$, hence (setting $\phi=\id_V$ in the
 assumption) $\Lambda_1\sim_2 \Lambda_2$ and let
 $\phi_1(\Lambda_2),\ldots \phi_q(\Lambda_2)$ with $\phi_i \in O_V(F)$
 denote the different lattices $M$ in the class of $\Lambda_2$ with
 $M\sim_1 \Lambda_1$. The cosets $\phi_iO_V(F;\Lambda_2)$ are then
 distinct and uniquely determined and their union is equal to
 $X:=\{\phi \in O_V(F)\mid \phi(\Lambda_2)\sim_1 \Lambda_1$. 
Let $\psi_1(\Lambda_1),\ldots,\psi_r(\Lambda_1)$ denote the different
lattices $K$ in the class of $\Lambda_1$ with $K\sim_2 \Lambda_2$. The
cosets $\psi_kO_V(F;\Lambda_1)=O_V(F;\psi_k(\Lambda_1))\psi_k$ are
then also distinct and uniquely determined. and their union is equal
to $Y=\{\psi \in O_V(F)\mid \psi(\Lambda_1)\sim_2
\Lambda_2\}$. Moreover, by our assumption on the relations
$\sim_1,\sim_2$ we see that $Y$  is just the set of inverses
of the elements of $X$. 
Let 
\begin{equation*}
G=O_V(F;\Lambda_2)\cap
 O_V(F;\phi_1^{-1}(\Lambda_2))\cap\dots \cap O_V(F;\phi_r^{-1} (\Lambda_2))
\end{equation*}
and  consider coset decompositions
\begin{eqnarray*}
  O_V(F;\Lambda_2)&=&\bigcup_\mu ^s\sigma_\mu G\\
O_V(F;\psi_k\Lambda_1)&=&\bigcup_\nu^t G\rho_{k\nu}.
\end{eqnarray*}
Notice that remark \ref{commensurability_remark} on the elementary
construction of a biinvariant Haar measure on congruence sets implies
$(O_V(F;\psi_k\Lambda_1):G)=(O_V(F;\psi_{k'}\Lambda_1):G)$ for 
all $k,k'$, so that the index $\nu$ above indeed runs over the same
set $1\le \nu \le t$ for
all $k$.
We have $\phi_i\sigma_\mu G=\phi_{i'}\sigma_{\mu'}G$  if and only if
$i=i', \mu=\mu'$ hold and $G\rho_{k\nu}\psi_k=G\rho_{k'\nu'}\psi_{k'}$
if and only if $k=k', \nu=\nu'$. 
Since $\phi \mapsto \phi^{-1}$ defines a bijection of
$X=\bigcup_{i,\mu}\phi_i\sigma_\mu G$ onto
$Y=\bigcup_{k,\nu}G\rho_{k\nu}\psi_k$ we obtain $qs=rt$, which is the assertion.
\end{proof}
\begin{corollary}\label{neighbor_symmetry}
With notations as above write $m_j$ for the measure\\
$\mu_\infty( O_V(F;\Lambda_j) \backslash (O_V(F_\infty))) $
of the volume of a fundamental domain in
$O_V(F_\infty)$ under the action of  the arithmetic subgroup
$O_V(F;\Lambda_j)$,
where
$F_\infty=\prod_{v\in \infty}F_v$ and $\mu_\infty$ is a (fixed) Haar
measure  on $O_V(F_\infty)$ .

Then one has 
\begin{equation*}
  m_iN_P(\Lambda_i,\Lambda_j) =m_jN_P(\Lambda_i,\Lambda_j).
\end{equation*}
In particular, if $(V,Q)$ is totally definite, one has
\begin{equation*}
  \vert O_V(F;\Lambda_j)\vert N_P(\Lambda_i,\Lambda_j)=\vert
  O_V(F;\Lambda_i)\vert N_P(\Lambda_i,\Lambda_j).
\end{equation*}
Moreover, the $P$-neighborhood matrices for different maximal ideals
commute pairwise and  the algebra generated by the $P$-neighborhood matrices
for all maximal ideals $P$ can be diagonalized simultaneously.
\end{corollary}
\begin{proof}
Upon letting
$\sim_1=\sim_2$ be the relation ``is $P$-neighbor of'',
the first assertion follows from the proposition  because of
\begin{equation*}
  \frac{m_i}{m_j}=\frac{(O_V(F;\Lambda_j):O_V(F;\Lambda_i)\cap
    O_V(F;\Lambda_j)}{(O_V(F;\Lambda_i):O_V(F;\Lambda_i)\cap
    O_V(F;\Lambda_j)}. 
\end{equation*}
The assertion in the totally definite case is just a
reformulation using the finiteness of the automorphism groups of the
lattices in that case.

The last part of the assertion is obvious since by the first part all
matrices in question can be simultaneously transformed into symmetric
matrices by conjugation with a diagonal matrix.
\end{proof}
\begin{corollary}\label{venkov_cor}
With notations as above let $\Lambda_1\subseteq \Lambda_2$. Let $q$ be the number of sublattices of $\Lambda_1$
which are isometric to $\Lambda_2$ and $r$ be the number of
overlattices of $\Lambda_2$ which are isometric to $\Lambda_1$.
Then
\begin{equation*}
\frac{q}{r}=
\frac{ (O_V(F;\Lambda_1):O_V(F;\Lambda_1)\cap
 O_V(F;\Lambda_2))}{(O_V(F;\Lambda_2):O_V(F;\Lambda_1)\cap
 O_V(F;\Lambda_2)) }.
\end{equation*}
\end{corollary}
\begin{proof}
With 
$M\sim_1 L$ if and only if $M\subseteq L$, $L\sim_2 M$ if and only if
$L\supseteq M$
  this follows directly from the proposition.
\end{proof}
\begin{remark}
 Eichler proves in \cite{eichler_qfog} analogous results for
 anzahlmatrices of ideal complexes with fixed elementary divisor
 systems. These can be 
 obtained from our proposition in the same way as Corollary
 \ref{neighbor_symmetry} using similitude groups instead of groups of
 isometries. The assertion of Corollary \ref{venkov_cor} 
 can be obtained from Eichler's results, it has been formulated and
 used for classification results about positive definite lattices by
 B.\ B.\ Venkov and is sometimes called Venkov's theorem. From our
 general setting it is clear that analogous results (with unitary
 isometry groups instead of orthogonal groups) hold for hermitian lattices.  
\end{remark} 
\begin{theorem}
 Let $(V,Q)$ be a positive definite quadratic space over $\Q$ of
 dimension $\ge 5$, let
 $\Lambda, \Lambda'$ be unimodular $\Z$-lattices on $V$ which are in
 the same genus, let $p\in \N$
 be a prime.

Then there exists a chain
$\Lambda=\Lambda_0,\Lambda_1,\ldots,\Lambda_r$ of unimodular lattices
on $V$ such that $\Lambda_i,\Lambda_{i+1}$ are $p$-neighbors for $0\le
i<r$ and $\Lambda_r$ is in the isometry class of $\Lambda'$.

If $p=2$ the condition that $\Lambda,\Lambda'$ should be in the same
genus can be omitted. 
\end{theorem}
\begin{proof} From Remark \ref{spinornorm_computations} it follows that
  the genus of $\lambda$ consists of only one spinor genus.
By theorem \ref{arithmetically_indefinite} we can therefore assume that
$\Lambda_\ell=\Lambda'_{\ell}$ for all primes $\ell \ne p$.  
If $\Lambda \ne \Lambda'$ there exists $x\in \Lambda', x \not\in
\Lambda$ with $px \in \Lambda$. Since $px$ is then primitive in
$\Lambda$ we have $b(px, \Lambda)\not\subseteq p\Z_p$. It follows that
$\Lambda_x:=\{y\in \Lambda \mid b(x,y)\in \Z\}$ is a sublattice of
$\Lambda$ of index $p$, and setting $\Lambda_1:=\Z x+\Lambda_x$ we
obtain another unimodular $\Z$-lattice $\Lambda_1$ on $V$ which is a
$p$-neighbor of $\Lambda$.

From $\Lambda \cap \Lambda'\subseteq \Lambda_x\subseteq
\Lambda_1$ we see $\Lambda \cap \Lambda'\subseteq \Lambda'\subseteq
\Lambda_1$, and this inclusion is proper since $x \not\in \Lambda\cap
\Lambda', x\in \Lambda' \Lambda_1$. By induction on
$(\Lambda:\Lambda\cap \Lambda')=(\Lambda':\Lambda\cap \Lambda')$
the assertion follows.
\end{proof}
\begin{remark}
  \begin{enumerate}
  \item If $p=2$ and $\Lambda,\Lambda'$ are both even all members of
    the chain constructed above are even as well.
\item Starting from a given unimodular $\Z$-lattice $\Lambda$ one can
  algorithmically determine all $p$-neighbors of it. Continuing this
  until representatives of all classes in the genus of $\Lambda$  have
  been found one can classify genera of positive definite
  lattices. This is Kneser's neighboring lattice method from
  \cite{kneser_klassenzahlen_definit}. It has been generalized by
  various authors and is now part of some computer algebra packages,
  in particular MAGMA.  
  \end{enumerate}
\end{remark}


 \backmatter

\begin{thebibliography}{MVW}
\bibitem{atiyah} M. Atiyah, MacDonald: Introduction to Commutative
  Algebra, Addison-Wesley, Reading Mass. 1969
\bibitem{boege} S. Böge: Schiefhermitesche Formen über Zahlkörpern und
  Quaternionenschiefkörpern, J. f. die reine u. angew. Math. (Crelles
  Journal) 221.
\bibitem{borel_harish} A. Borel, Harish-Chandra:
Arithmetic subgroups of algebraic groups. 
Ann. of Math. (2) 75 (1962), 485–535. 
\bibitem{bourbaki} N. Bourbaki: Alg\`ebre Commutative, Reprint of the 1998 original. Springer-Verlag, Berlin, 2007.
\bibitem{bourb_int} N. Bourbaki: Integration, Paris 1965
\bibitem{cassels} J. W. S. Cassels: Rational Quadratic Forms, London
  Mathematical Society Monographs, 13. Academic Press, London-New York, 1978. 
\bibitem{chevalley} C. Chevalley: L'Arithm\'etique dans les Alg\`ebres
  de Matrices, Hermann, Paris 1936
\bibitem{curtisreiner} C. Curtis, I. Reiner: Representation Theory of
  Finite Groups and Associative Algebras, Wiley 1966
\bibitem{dieud}J. Dieudonn\'e: La G\'eom\'etrie des Groupes
  Classiques, Ergebnisse der Mathematik und ihrer Grenzgebiete, Band
  5. Springer-Verlag, Berlin-New York, 1971 
\bibitem{earnest_hsia_spinornorms2}  A. G. Earnest, J. S.Hsia:
Spinor norms of local integral rotations. II. {Pacific
  J. Math.} {61} (1975), no. 1, 71–86 
\bibitem{eichler_qfog} M. Eichler: Quadratische Formen und Orthogonale
   Gruppen, Grundlehren d. math. Wiss. 63, Springer Verlag 1952.
\bibitem{eichler_spin} M. Eichler: Die \"Ahnlichkeitsklassen
  indefiniter Gitter.  Math. Z.  {\bf 55}  (1952), 216--252.  
\bibitem{eisenbud} D. Eisenbud: Commutative Algebra. With a view
  toward algebraic geometry. Graduate Texts in Mathematics,
  150. Springer-Verlag, New York, 1995 
\bibitem{ev} J. Ellenberg, A. Venkatesh: Local-global principles for
  representations of quadratic forms.  {\em Invent. Math.}  {\bf 171}  (2008),
  257--279.
\bibitem{hsia_spinornorms1} J. S. Hsia: Spinor norms of local integral
  rotations I, Pacific J. Math. 57 (1975), 199–206. 
\bibitem{hsia_spinor} J. S. Hsia: Representations by spinor genera.
  {\em Pacific J. Math.}  {\bf 63} (1976), 147--152.
\bibitem{humbert} P. Humbert: R\'eduction de formes quadratiques dans
  on corps alg\'ebrique fini, Comm. Math. Helv. 23 (1949), 50-63.
\bibitem{kaplansky} I.  Kaplansky: Modules over Dedekind rings and
  valuation rings, Trans. Amer. Math. Soc. 72 (1952), 327–340. 
\bibitem{kneserbook} M. Kneser: Quadratische Formen. Quadratische
  Formen. Revised and edited in collaboration with Rudolf Scharlau. Springer-Verlag, Berlin, 2002.
\bibitem{kneserapprox} M.Kneser: Starke Approximation in algebraischen
  Gruppen I, J. Reine Angew. Math. 218 (1965), 190–203. 
\bibitem{kneser_klassenzahlen_definit} M. Kneser: Klassenzahlen definiter quadratischer
  Formen,  Arch. Math.  8  (1957), 241--250.
\bibitem{knus} M. -A. Knus: Quadratic and Hermitian forms over Rings,
  Hermitian forms over rings. Grundlehren der Mathematischen
  Wissenschaften 294. Springer-Verlag, Berlin, 1991. 
\bibitem{kneser_klassenzahlen} M.\ Kneser:
  Klassenzahlen indefiniter quadratischer Formen.
  {\em Arch.\ Math.}  {\bf 7} (1956) 323 -- 332.
\bibitem{kneser_mz} M. Kneser: Darstellungsma\ss e indefiniter
quadratischer Formen. {\em Math. Z.}  {\bf 77}  (1961), 188--194.

\bibitem{kottwitz_tamagawa} R. Kottwitz: Tamagawa numbers, 
Ann. of Math. (2) 127 (1988), 629-646 
 \bibitem{lam_qf} T. Y. Lam: Introduction to Quadratic Forms over
   Fields, Graduate Studies in Math. 67, AMS, Providence 2004
\bibitem{lll} A. K. Lenstra, H. W. Lenstra, H. W, L. Lovász: Factoring
  polynomials with rational coefficients. Math. Ann. 261 (1982),
  no. 4, 515–534.  
\bibitem{omeara} O. T. O'Meara: Introduction to Quadratic
  Forms. Introduction to quadratic forms. Die Grundlehren der
  mathematischen Wissenschaften, Bd. 117 Springer-Verlag,
  Berlin-Göttingen-Heidelberg 1963  
\bibitem{minkowski_quaternaire} H. Minkowski: Sur la reduction des
  formes quadratiques positives quaternaires, Ges. Abh. I (1911), 149-156
\bibitem{platonov_rapinchuk} V. Platonov, A. Rapinchuk: Algebraic
  groups and number theory. Translated from the 1991 Russian original
  by Rachel Rowen. Pure and Applied Mathematics, 139. Academic Press,
  Inc., Boston, MA, 1994. 
\bibitem{scharlaubook} W. Scharlau: Quadratic and Hermitian
  Forms. Grundlehren der Mathematischen Wissenschaften
  270. Springer-Verlag, Berlin, 1985.
  \bibitem{siegel1} C. L. Siegel: Über die analytische Theorie der
    quadratischen Formen, Ann. of Math (2) 36 (1935), 527-606
  \bibitem{siegel2} C. L. Siegel: Über die analytische Theorie der
    quadratischen Formen II, Ann. of Math (2) 37 (1936), 230-263
  \bibitem{siegel3} C. L. Siegel: Über die analytische Theorie der
    quadratischen Formen III, Ann. of Math (2) 38 (1937), 212-291
  \bibitem{siegel_einheiten} C. L. Siegel:
Einheiten quadratischer Formen:
Abh. Math. Sem. Hansischen Univ. 13 (1940), 209–239.
\bibitem{Tammela} P. Tammela: Minkowski's fundamental reduction domain
  for positive quadratic forms of seven variables. J. of Soviet
  Math. 16 (1981), 836-857
\bibitem{vdw} B. L. van der Waerden:  Die Reduktionsthorie der
  positiven quadratischen Formen, Acta math. {\bf 96} (1956)
\bibitem{wall} C.\ T.\ C.\ Wall: Graded Brauer groups,  J.\ f.\ d.\
  reine und angew.\ Mathematik {\bf 213}, 1964
\bibitem{weil_aag} A. Weil: Adeles and Algebraic Groups,
  Birkhäuser/Springer 1982 (PM 23)
\bibitem{witt} E. Witt: Theorie der quadratischen Formen in beliebigen
  Körpern, J.\ f.\ d.\ reine und angew.\ Mathematik {\bf 176}
\bibitem{zemel} S. Zemel: On lattices over valuation rings of
  arbitrary rank, J. of Algebra 423 (2015), 812-852
\end{thebibliography}


\end{document}